\documentclass[twoside,12pt]{Classes/CUEDthesisPSnPDF}
\usepackage[ngerman,english]{babel}

\ifpdf
    \pdfcatalog { /PageMode (/UseOutlines)
                  /OpenAction (fitbh)  }
\fi

\title{Graph Relations and Constrained Homomorphism Partial Orders}

\ifpdf
  \author{\href{mailto:ylong@mis.mpg.de}{Yangjing Long}}
  \collegeordept{\href{http://www.eng.cam.ac.uk}{Department of Engineering}}
  \university{\href{http://www.cam.ac.uk}{University of Cambridge}}
\else
  \author{Yangjing Long}
  \collegeordept{Department of Engineering}
  \university{University of Cambridge}
  \crest{\includegraphics[bb = 0 0 292 336, width=30mm]{UnivShield}}
\fi
\degree{Doctor of Philosophy}
\degreedate{Yet to be decided}

\usepackage{graphicx,color,amsmath,amssymb}
\usepackage{hyperref}
\usepackage{subfig}
\usepackage{titlesec}
\usepackage{fancyhdr}
\usepackage{calc}
\usepackage{framed}
\usepackage{picinpar}
\usepackage{natbib}
\usepackage{amssymb}
\usepackage[all,cmtip]{xy}
\usepackage{epstopdf}
\usepackage{newfile}
\usepackage{german}

\usepackage{datetime}
\newdateformat{myformat}{\THEDAY{. }\monthname[\THEMONTH], \THEYEAR}
\usepackage{CJKutf8}

\usepackage{bbm}
\usepackage[a4paper]{geometry}
\usepackage{subfig}
\usepackage{float}
\usepackage{enumitem}
\setlist[enumerate,2]{ref=\theenumi(\alph*)}

\floatstyle{boxed}
\restylefloat{figure}

\usepackage{calc}
\renewcommand{\chaptermark}[1]{\markboth{\chaptername
 \ \thechapter.\ #1}{}}
\renewcommand{\sectionmark}[1]{\markright{\thesection\ #1}}
\fancypagestyle{plain}{%
   \fancyhead{} 
}

\usepackage{fix-cm}
\usepackage{color}
\usepackage{amssymb}
\usepackage{amsmath}
\usepackage{multirow}
\usepackage{algorithm}
\usepackage{algorithmic}
\usepackage{amsthm}
\usepackage{makeidx}
\usepackage{booktabs}
\usepackage{array}
\usepackage{multirow}
\usepackage[dvipsnames]{xcolor}

\makeatletter
\def\th@plain{%
  \thm@notefont{}
  \itshape 
}
\def\th@definition{%
  \thm@notefont{}
  \normalfont 
}
\makeatother

\titleformat{\chapter}[display]
{}
{\titlerule[1pt]%
 \vspace{1pt}%
 \titlerule
 \vspace{1pc}%
 \fontsize{30}{10}\selectfont \color{gray1}\chaptertitlename \hspace{ 4pt}\thechapter}
{1pc}
{\color{black}\titlerule
 \vspace{5pc}%
  \fontsize{37}{10}\selectfont}

\titlespacing*{\chapter}{0pt}{0pt}{100pt}

\usepackage{lipsum,xcolor}
\makeatletter
\def\subsection{\@ifstar\unnumberedsubsection\numberedsubsection}
\def\numberedsubsection{\@ifnextchar[
  \numberedsubsectionwithtwoarguments\numberedsubsectionwithoneargument}
\def\unnumberedsubsection{\@ifnextchar[
  \unnumberedsubsectionwithtwoarguments\unnumberedsubsectionwithoneargument}
\def\numberedsubsectionwithoneargument#1{\numberedsubsectionwithtwoarguments[#1]{#1}}
\def\unnumberedsubsectionwithoneargument#1{\unnumberedsubsectionwithtwoarguments[#1]{#1}}
\def\numberedsubsectionwithtwoarguments[#1]#2{%
  \ifhmode\par\fi
  \removelastskip
  \vskip 3ex\goodbreak
  \refstepcounter{subsection}%
  \hbox to \hsize{%
    \colorbox{black!20}{%
      \hbox to 1cm{\hss\bfseries\Large\thesubsection.\ }%
      \vtop{%
        \advance \hsize by -1cm
        \advance \hsize by -2\fboxrule
        \advance \hsize by -2\fboxsep
        \parindent=0pt
        \leavevmode\raggedright\bfseries\Large
        #2
        }%
      }}\nobreak
  \vskip 2mm\nobreak
  \addcontentsline{toc}{subsection}{%
    \protect\numberline{\thesubsection}%
    #1}%
  \ignorespaces
  }
\def\unnumberedsubsectionwithtwoarguments[#1]#2{%
  \ifhmode\par\fi
  \removelastskip
  \vskip 3ex\goodbreak
  \refstepcounter{subsection}%
  \hbox to \hsize{%
    \colorbox{black!20}{%
      \vtop{%
        \advance \hsize by -2\fboxrule
        \advance \hsize by -2\fboxsep
        \parindent=0pt
        \leavevmode\raggedright\bfseries\Large
        #2
        }%
      }}\nobreak
  \vskip 4mm\nobreak
  \addcontentsline{toc}{subsection}{%
   \protect\numberline{\thesubsection}%
    #1}%
  \ignorespaces
  }
\makeatother
\titlespacing*{\subsection}{0pt}{0pt}{60pt}

\usepackage{framed, amsthm, amsmath}
\usepackage{color}
\definecolor{gray}{rgb}{0.9,0.9,0.9}
\definecolor{gray1}{rgb}{0.7,0.7,0.7}
\definecolor{gray2}{rgb}{0.8,0.8,0.8}

\usepackage[T1]{fontenc}

\titleformat{\subsubsection}
{\bfseries\scshape\large}{\theparagraph}{1em}{}

\makeatletter
\renewcommand\section{\@startsection{section}{1}{\z@}%
                                  {-3.5ex \@plus -1ex \@minus -.2ex}%
                                  {2.3ex \@plus.2ex}%
                                  {\normalfont\LARGE\bfseries}}
\makeatother
\makeindex

\let\origdoublepage\cleardoublepage
\newcommand{\clearemptydoublepage}{%
  \clearpage
  {\pagestyle{empty}\origdoublepage}%
}
\let\cleardoublepage\clearemptydoublepage

\def\Fraisse{Fra\"{\i}ss\' e}
\newcommand{\Homo}[0]{\ensuremath{\textsc{Hom}}}

\newcommand{\FulRel}[0]{\ensuremath{\textsc{Ful\mbox{-}Rel}}}
\newcommand{\sHom}[0]{\ensuremath{\textsc{Sur\mbox{-}Hom}}}
\newcommand{\RelHom}{\mathop{\to^{PR}}}

\newcommand{\CSP}[0]{\ensuremath{\textsc{H\mbox{-}CSP}}}
\newcommand{\COL}[0]{\ensuremath{\textsc{H\mbox{-}COL}}}

\newcommand{\Graphs}[0]{\ensuremath{\mathsf{Graphs}}}
\newcommand{\ConnGraph}[0]{\ensuremath{\mathsf{ConnGraph}}}
\newcommand{\DiGraphs}[0]{\ensuremath{\mathsf{DiGraphs}}}
\newcommand{\DiCycle}[0]{\ensuremath{\mathsf{DiCycle}}}
\newcommand{\DiCycles}[0]{\ensuremath{\mathsf{DiCycles}}}
\newcommand{\Cycle}[0]{\ensuremath{\mathsf{Cycle}}}
\newcommand{\Cycles}[0]{\ensuremath{\mathsf{Cycles}}}

\newcommand{\Path}[0]{\ensuremath{\mathsf{Path}}}
\newcommand{\Matrices}[0]{\ensuremath{\mathsf{DRM}}}
\newcommand{\LocInL}{\mathop{<^{LI}}}

\newcommand{\Hom}{\ensuremath{\to}}

\newcommand{\AnyHomo}{$*$-morphism}
\newcommand{\MonoHomo}{M-morphism}
\newcommand{\EmbedHomo}{E-morphism}
\newcommand{\FullHomo}{F-morphism}
\newcommand{\VSurHomo}{VS-morphism}
\newcommand{\SurHomo}{S-morphism}
\newcommand{\LocInHomo}{LI-morphism}
\newcommand{\LocSurHomo}{LS-morphism}
\newcommand{\LocBiHomo}{LB-morphism}
\newcommand{\FRelHomo}{R-morphism}
\newcommand{\RelHomo}{PR-morphism}
\newcommand{\SurHomeq}{\mathop{\Homeq^{S}}}
\newcommand{\UnorPath}[0]{\ensuremath{\mathsf{UnorPath}}}

\newcommand{\EmbedHom}{\mathop{\to^E}}
\newcommand{\AnyHom}{\ensuremath{\to^*}}
\newcommand{\SurHom}{\ensuremath{\to^S}}
\newcommand{\VSurHom}{\ensuremath{\to^{VS}}}

\newcommand{\FullHom}{\ensuremath{\to^F}}
\newcommand{\LocSurHom}{\ensuremath{\to^{LS}}}
\newcommand{\LocInHom}{\ensuremath{\to^{LI}}}
\newcommand{\LocBiHom}{\ensuremath{\to^{LB}}}

\newcommand{\Pfin}{\ensuremath{P_{\mathrm{fin}}}}
\newcommand{\subsetLeq}[1]{\mathop{\leq^{\mathrm{dom}}_#1}}

\newcommand{\PtoInclussion}{E}

\newcommand{\PrimesToCycles}{E}

\newcommand{\AnyLeq}{\ensuremath{\leq^*}}

\newcommand{\SurLeq}{\ensuremath{\leq^S}}
\newcommand{\LocSurLeq}{\ensuremath{\leq^{LS}}}
\newcommand{\LocInLeq}{\ensuremath{\leq^{LI}}}
\newcommand{\LocBiLeq}{\ensuremath{\leq^{LB}}}
\newcommand{\EmbedLeq}{\ensuremath{\leq^E}}
\newcommand{\EmbedGeq}{\ensuremath{\geq^E}}
\newcommand{\MonoLeq}{\ensuremath{\leq^M}}
\newcommand{\FullLeq}{\mathop{\leq^F}}
\newcommand{\FullGeq}{\mathop{\geq^F}}
\newcommand{\RelLeq}{\mathop{\leq^{PR}}}
\newcommand{\FRelLeq}{\mathop{\leq^{R}}}
\newcommand{\ESurLeq}{\mathop{\leq^{ES}}}
\newcommand{\VSurLeq}{\mathop{\leq^{VS}}}

\newcommand{\K}{\ensuremath{\mathcal{K}}}

\let\Homeq\sim
\newcommand{\AnyHomeq}{\ensuremath{\Homeq^{*}}}
\newcommand{\Matrixeq}{\ensuremath{\Homeq^{M}}}
\newcommand{\drm}{\mathrm{drm}}
\newcommand{\FullHomeq}{\mathop{\Homeq^{F}}}

\newcommand{\Releq}{\mathop{\Homeq^{PR}}}
\newcommand{\FReleq}{\mathop{\Homeq^{R}}}

\renewcommand{\year}{2014}
\newcommand{\mefirst}{Yangjing}
\newcommand{\melast}{Long}
\newcommand{\me}{Yangjing Long}

\newcommand{\mytitle}{Graph Relations and Constrained Homomorphism Partial Orders}
\newcommand{\mytitleDE}{Relationen von Graphen und Partialordungen durch spezielle Homomorphismen}

\newoutputstream{pages}

\openoutputfile{pages.num}{pages} 
\addtostream{pages}{\arabic{page}}
\closeoutputstream{pages}

\newcommand{\muls}{\ast}
\newcommand{\mulw}{\star}
\newcommand{\mulm}{\circledast}

\newcommand{\overbar}{\overline}
\DeclareMathOperator{\domain}{dom}
\DeclareMathOperator{\image}{img}

\begin{document}
\theoremstyle{plain}
\newtheorem{thm}{\Large T\normalsize heorem}[section]
\newtheorem{pro}[thm]{\Large P\normalsize roposition}
\newtheorem{prop}[thm]{\Large P\normalsize roposition}
\newtheorem{lem}[thm]{\Large L\normalsize emma}
\newtheorem{lemma}[thm]{\Large L\normalsize emma}
\newtheorem{cor}[thm]{Corollary}
\newtheorem{corollary}[thm]{Corollary}
\newtheorem{con}[thm]{Conjecture}
\newtheorem{obser}[thm]{Observation}
\newtheorem{obs}[thm]{Observation}
\theoremstyle{definition}
\newtheorem{defi}{\Large D\normalsize efinition}[section]
\newtheorem{defn}{\Large D\normalsize efinition}[section]
\newtheorem{definition}{\Large D\normalsize efinition}[section]
\newtheorem{example}{Example}[section]
\theoremstyle{remark}
\newtheorem{rem}{Remark}[section]

\setlength{\abovecaptionskip}{20pt} 

\thispagestyle{empty}
{\Large
\begin{center}
\textbf{Graph Relations and \\Constrained Homomorphism Partial Orders}
\end{center}}
\vspace*{0.7cm}
\begin{center}
Der Fakult\"at f\"ur Mathematik und Informatik\\
der Universit\"at Leipzig\\
eingereichte
\end{center}
\vspace*{0.7cm}
\begin{center}
D I S S E R T A T I O N
\end{center}
\vspace*{0.7cm}
\begin{center}
zur Erlangung des akademischen Grades
\end{center}

\begin{center}
DOCTOR RERUM NATURALIUM\\
(Dr.rer.nat.)
\end{center}
\vspace*{0.5cm}
\begin{center}
im Fachgebiet 
\end{center}
\begin{center}
Mathematik
\end{center}
\begin{center}
vorgelegt 
\end{center}
\vspace*{0.7cm}
\begin{center}
von Bachelorin \me\\
geboren am 08.07.1985 in Hubei (China)
\end{center}
\vspace{1cm}

\begin{center}
\selectlanguage{german}
Leipzig, den \myformat\today

\end{center}



\setcounter{secnumdepth}{3}
\setcounter{tocdepth}{3}

\frontmatter 
\pagenumbering{roman}

\begin{dedication} 

\begin{CJK*}{UTF8}{gbsn}
献给我挚爱的母亲们！
\end{CJK*}

\end{dedication}




\begin{acknowledgements}      

My sincere and earnest thanks to my `Doktorv{\"a}ter' Prof.~Dr.~Peter~F. Stadler and Prof.~Dr.~J\"{u}rgen Jost. They introduced me to an interesting project, and have been continuously giving me valuable guidance and support.

Thanks to my co-authers Prof.~Dr.~Ji{\v{r}}\'{\i} Fiala and Prof.~Dr.~Ling Yang, with whom I have learnt a lot from discussions.

I also express my true thanks to Prof.~Dr.~Jaroslav Ne\v{s}et\v{r}il for the enlightening discussions that led me along the way to research.

I would like to express my gratitude to my teacher Prof.~Dr.~Yaokun Wu, many of my senior fellow apprentice, especially Dr.~Frank~Bauer and Dr.~Marc~Hellmuth, who have given me a lot of helpful ideas and advice. 

Special thanks to Dr.~Jan Hubi\v{c}ka and Dr.~Xianqing Li-jost, not only for their guidance in mathematics and research, but also for their endless care and love---they have made a better and happier me.

Many thanks to Dr. Danijela Horak helping with the artwork in this thesis, and also to Dr.~Andrew Goodall, Dr.~Johannes Rauh, Dr.~Chao Xiao, Dr.~Steve Chaplick and Mark Jacobs for devoting their time and energy to reading drafts of the thesis and making language corrections.

I am very grateful to IMPRS for providing a financial scholarship for my studies and to Research Academic Leipzig for supporting my attendance at conferences. I also wish to express my appreciation to the administrative staff at MIS MPG and Mrs. Petra Pregel for their supportiveness.

Finally, to my family and friends for their ongoing care, support and encouragement---you make my life colourful.

\end{acknowledgements}




\begin{abstracts}        


We consider constrained variants of graph homomorphisms such as embeddings, monomorphisms, full homomorphisms,
surjective homomorpshims, and locally constrained homomorphisms.
We also introduce a new variation on this theme which derives from
relations between graphs and is related to multihomomorphisms.
This gives a generalization of surjective homomorphisms and naturally leads to notions of R-retractions, R-cores, and R-cocores of graphs. Both \mbox{R-cores} and R-cocores of graphs are unique up to isomorphism and can be computed in polynomial time.

The theory of the graph homomorphism order is well developed, and from it we consider analogous notions defined for orders induced by constrained homomorphisms. 
We identify corresponding cores, prove or disprove universality, characterize gaps and dualities. We give a new and significantly easier proof of the universality of the homomorphism order by showing that even the class of oriented cycles is universal. We provide a systematic approach to simplify the proofs of several earlier results in this area.
We explore in greater detail locally injective homomorphisms on connected graphs, characterize gaps and show universality. We also prove that for every $d\geq 3$ the homomorphism order on the class of line graphs of graphs with maximum degree $d$ is universal.

\end{abstracts}



\tableofcontents
\listoffigures

\mainmatter 
\chapter{Introduction}
\ifpdf
    \graphicspath{{Introduction/IntroductionFigs/PNG/}{Introduction/IntroductionFigs/PDF/}{Introduction/IntroductionFigs/}}
\else
    \graphicspath{{Introduction/IntroductionFigs/EPS/}{Introduction/IntroductionFigs/}}
\fi




The subject of this thesis---graph relations and constrained homomorphism orders---belongs to the field of discrete mathematics, more specifically to graph theory. It has a close relationship to several other areas of mathematics stemming from algebra and mathematical structures via category theory and logic to computer science disciplines such as computational complexity.

The term `graph' was introduced by Sylvester in 1878~\cite{Sylvester1878}.
A {\em directed graph}\index{directed graph} $G$ is a pair $G=(V_G,E_G)$ such that $E_G$ is a subset of $V_G\times V_G$. We denote by $V_G$\index{$V_G$} the set of {\em vertices}\index{set of vertices} of $G$ and by $E_G$ the set of {\em edges}\index{set of edges} of $G$. The class of all finite directed graphs is denoted by $\DiGraphs$\index{$\DiGraphs$}.
An {\em undirected graph} (or simply a {\em graph}\index{graph}) $G$ is a directed graph such that $(u,v)\in E_G$ if and only if $(v,u)\in E_G$. We thus consider undirected graphs to be a special case of directed graphs. Unless explicitly stated otherwise we allow loops on vertices. The class of all finite undirected graphs is denoted by $\Graphs$\index{$\Graphs$}.

Graphs are frequently employed to model natural or artificial systems.
Let us begin the introduction with two examples taken from the real world.

In social science we can use social networks to express the relationships between individuals, groups and organizations. For example, a friendship network shows the friendship relation among individuals in a community. At the same time, travelling is an important activity in modern society as people travel frequently, for business or for pleasure. If we consider all the places that people in a certain community have ever visited, we can define a network on places in the following natural way. 
Draw an edge between place $A$ and place $B$ if and only if we can find two people who are friends, one who has visited $A$, and the other $B$.
In this way we build a new network on places generated by the friendship network and the information about people and the places they visited.

The research presented in this thesis was originally motivated by a generalization of concepts from bioinformatics. Consider protein-protein interaction networks (PPIs) in biology. Proteins contain domains that determine their function. Information on the domain content of proteins, which can be seen as a relation
$R\subseteq D\times P$ between domains and the proteins which contain them is readily available from biological databases. Candidates for interaction between domains can be derived from knowledge about the interacting proteins, i.e., a graph $G$ with vertex set $P$, and the relation $R$: If $(p,p')\in E_G$,
 $(d,p)\in R$, and $(d',p')\in R$, then the domains 
$(d,d')$ are putative interaction partners. 

In both examples, we use a network and a relation to create a new network. 
Conversely, given two networks, can we determine some meaningful relation between them? This is the idea behind the graph relations. 

More abstractly, let $G=(V_G,E_G)$ be a graph, $B$ a finite set, and
$R\subseteq V_G\times B$ a binary relation such that for every $b\in B$ there is $v \in V_G$ such that $(v,b)\in R$. Then
the graph $G\muls R$ has vertex set $B$ and edge set
\begin{equation*}
E_{G\muls R} = \left\{ (a,b)\in B\times B \mid \textrm{ there is }
      (u,v)\in E_G \textrm{ such that } (u,a),(v,b)\in R \right\}.
\end{equation*}
If there exists an $R$ satisfying the equation $G\muls R = H$, we say
there is a \emph{relation}\index{relation} from $G$ to $H$. The graphs $G$ and $H$
can be seen to be conjugates, i.e., $E_G=R^+\circ E_H\circ R$, where $R^+\subseteq B\times V_G$ is the transpose of the relation $R$ and $\circ$\index{$\circ$} denotes the \emph{composition}\index{composition} of binary relations, which is defined as follows: If $S\subseteq X\times Y$ and $T\subseteq Y\times Z$ are two binary relations, then $S\circ T=\{(x,z)\in X\times Z\mid  \textrm{ there is } y\in Y \textrm{ such that } (x,y)\in S \textrm{ and } (y,z)\in T\}$.

The concept of a relation between graphs turns out be closely related to the well-studied graph theory 
notion of a graph homomorphism. For given graphs $G=(V_G,E_G)$ and
$H=(V_H,E_H)$, a \emph{homomorphism}\index{homomorphism} $f:G\to H$ is a mapping $f:V_G\to
V_H$ such that $(u,v)\in E_G$ implies $(f(u),f(v))\in E_H$. If there is no homomorphism from $G$ to $H$, we write $G\nrightarrow H$\index{$G\nrightarrow H$}. 
Graph homomorphisms have been studied ever since the 1960s, originally as a generalization of graph colourings, with pioneering work by G.\ Sabidussi~\cite{Sabidussi1961}, Z.\ Hedrl\' in and A.\ Pultr~\cite{Hedrlin1964} subsequently seeing rapid development by P.\ Hell and J.\ Ne\v{s}et\v{r}il~\cite{Hell2004}. There is a lot of intensive research going on around this topic; for an extensive review of the subject, see~\cite{Hell2004}.

A homomorphism from a graph $G$ to a graph $H$ naturally defines a mapping $f^{1}:E_G\to E_H$ by setting
$f^{1}((u,v))=(f(u),f(v))$ for all $(u,v)\in E_G$. If both $f$ and $f^1$
are surjective, we call $f$ a \emph{surjective homomorphism}\index{surjective homomorphism}. Note that all
surjective homomorphisms $f:G\to H$ are (modulo representation) also
relations $G\muls R=H$, where $R=\{(u,f(u)) \mid u\in V_G\}$.

Obviously, relations that are not functional are not homomorphisms, but
correspond to multihomomorphisms, which in a sense are multifunctions. A \emph{multifunction}\index{multifunction} from a non-empty set $X$ to a non-empty set $Y$ is a function from $X$ to $P(Y)\setminus \emptyset$ where $P(Y)$ is the power set of $Y$~\cite{Nenthein2012}. Multifunctions between two topological spaces have been studied since 1960s~\cite{Whyburn1965}. Multifunctions were first studied in an algebraic sense in~\cite{Triphop2008}, in which they feature as multihomomorphisms from one group to another group. To the best of our knowledge, multihomomorphisms between graphs first appeared as building blocks of Hom-complexes, introduced by Lov{\'a}sz, and are related to recent exciting developments in topological combinatorics~\cite{Matousek2003}, in particular to deep results involved in a proof of the Lov{\'a}sz hypothesis~\cite{Babson2006}.

A relation with full domain thus can be regarded as a \textit{surjective multihomomorphism}, a multihomomorphism such that the pre-image of every vertex in $H$ is non-empty and for every edge $(x,y)$ in $H$ we can find an edge $(u,v)$ in $G$ satisfying $x\in f(u)$, $y\in f(v)$.

It appears that surjective multihomomorphisms between groups were first defined in~\cite{Triphop2008} and studied in~\cite{Nenthein2012}. However, surjective multihomomorphisms between graphs have not yet received attention.

\vspace{0.3cm}
In this work we show that even though the notion of relations has not been noticed by graph theorists it has its own charm. Owing to the omission of the constraint of being a mapping, relations substantially generalize the sort of situation that can hold for graph homomorphisms; a relation can ignore some part of the original graph and/or map one vertex to multiple vertices. Even relations with full domains are a generalization of surjective homomorphisms. Every relation can be decomposed in a standard way into a surjective homomorphism and an injective relation, whose transpose can be seen as a full homomorphism. It produces many interesting notions of graph classes, such as R-cores, cores, reduced graphs and cocores. We show that core and reduced graph coincide, develop the characterization of cocore and R-core, and provide effective algorithms to compute them. We also give the relations of containment among these notions and give examples which can distinguish them. We give a perhaps surprisingly simple characterization of those graphs $G$ for which all relations of $G$ to itself are automorphisms. For the computational complexity of testing for the existence of a relation, we describe a reduction of this problem to the corresponding problem for surjective homomorphisms.

%

\vspace{10pt}

Inspired by relations between graphs, we widen the scope of our study to investigate the properties of \emph{constrained homomorphisms}\index{constrained homomorphisms}, i.e., homomorphisms that satisfy further conditions such as surjectivity.
Partial orders induced by the existence of constrained homomorphisms and their properties form the second topic of this thesis. In particular we consider the following:

\begin{enumerate}
\item A homomorphism $f:G\to H$ is a {\em monomorphism}\index{monomorphism}\footnote{The name `monomorphism' originates from category theory.}, (also called \emph{injective homomorphism}\index{injective homomorphism}~\cite{Hell2004} or {\em \MonoHomo}) if it is injective.
\item A homomorphism $f:G\to H$ is a {\em full homomorphism}\index{full homomorphism}, or {\em \FullHomo}\index{\FullHomo}, if $(f(u),f(v))\in E_H$ implies $(u,v)\in E_G$.
\item A homomorphism $f:G\to H$ is an {\em embedding}\index{embedding}, or {\em \EmbedHomo}\index{\EmbedHomo} if it is both a monomorphism and a full homomorphism.
\end{enumerate}

For undirected graphs we also consider the following forms of locally constrained homomorphisms.
\begin{enumerate}
\item A homomorphism $f:G\to H$ is a {\em locally injective homomorphism}\index{locally injective homomorphism}, or {\em \LocInHomo}\index{\LocInHomo}, if for all $v\in V_G$, the restriction of the mapping $f$ to the domain $N_G(v)$ and range $N_H(f(v))$ is injective.
\item A homomorphism $f:G\to H$ is a {\em locally surjective homomorphism}\index{locally surjective homomorphism}, or {\em \LocSurHomo}\index{\LocSurHomo}, if for all $v\in V_G$, the restriction of the mapping $f$ to the domain $N_G(v)$ and range $N_H(f(v))$ is surjective.
\item A homomorphism $f:G\to H$ is a {\em locally bijective homomorphism}\index{locally bijective homomorphism}, or {\em \LocBiHomo}\index{\LocBiHomo}, if it is both a locally surjective and locally bijective homomorphism, that is, if the restrictions to the neighbourhoods are bijections.
\end{enumerate}

There are also other types of constrained homomorphism that have been studied. One example is that of a compaction~\cite{Kedem1984}.
A \emph{compaction}\index{compaction} of a graph $G$ to a graph $H$ is a homomorphism $f:G\to H$ such that for every vertex $x$ of $H$, there exists a vertex $v$ of $G$ with $f(v)=x$, and for every edge $(x,y)$ of $H$, $x\neq y$, there exists an edge $(u,v)$ of $G$ with $f(u)=x$ and $f(v)=y$. The computational complexity of compaction has been studied by Vikas~\cite{Vikas2002,Vikas2004}.

Constrained homomorphisms have been intensively studied in their structural and algorithmic aspects as well as in their applications. 
Each type of constrained homomorphism has its own history. For example, globally constrained homomorphisms have also been studied from the 1960s. The earliest literature we can find is what Hedetniemi published in 1966~\cite{Hedetniemi1966}. It is however surprising to learn from the survey paper~\cite{Fiala2008} that the notion of a locally bijective homomorphism was already well established in combinatorial topology as early as 1932~\cite{Reidemeister1932}.
The other types of locally constrained homomorphism received attention much later. Locally surjective homomorphisms were introduced by Everett and Borgatti in 1991~\cite{Everett1991} and locally injective homomorphisms by Ne\v set\v ril in 1971~\cite{Nesetril1971}.

There are many results that link graphs and orders, such as those concerning vertex-edge-face partial orders~\cite{Brightwell1997}, planar graphs, hypergraph colourings~\cite{Duffus1991}, and on-line algorithms~\cite{Kierstead1981}.
In a similar vein, the dimension of a partial order is analogous to the chromatic number of a graph~\cite{Kelly1982}.
See the survey paper~\cite{Trotter1996} by Trotter. We give another instance of partial orders on graphs.

For any fixed type of finite relational structure, homomorphisms induce an ordering of the set of all structures. In particular, given two (directed) graphs $G$ and $H$, we write $G\leq H$, or $G\to H$, if there is a homomorphism from $G$ to $H$. The relation $\leq$ is a quasi-order and so it induces an equivalence relation. 
For two (directed) graphs $G$ and $H$ such that $G\Hom H$ and $H\Hom G$, we write $G\Homeq H$ and say that {\em $G$ and $H$ are homomorphically equivalent}\index{homomorphically equivalent}. The homomorphism orders $(\Graphs,\leq)$ and $(\DiGraphs,\leq)$ are the partially ordered sets of all equivalence classes of finite undirected and directed graphs, respectively, ordered by the relation $\leq$. We shall abuse notation and follow the convenient practice of identifying an equivalence class in either of these homomorphism quasi-orders with one of its representatives, namely that (di)graph which is the core of all other (di)graphs equivalent to it. So by $(\Graphs,\leq)$, respectively $(\DiGraphs,\leq)$, we will actually be referring to the partially ordered sets induced by the quasi-ordered sets of all graphs (digraphs) on the set of graph (digraph) cores. (The notion of a graph core requires some preparation and will be defined in its proper place in Chapter~\ref{ch:homo}.)

These orders are of particular interest. For example, each of the two orders forms a distributive lattice, with the disjoint union of graphs being the supremum and the categorical product being the infimum.

The homomorphism order has been a fruitful research topic for several decades.
One of the well studied properties of the homomorphism order is density. Both $(\Graphs,\leq)$ and $(\DiGraphs,\leq)$ are dense, the former shown by Welzl~\cite{Welzl1982} and the latter by Ne\v set\v ril and Tardif~\cite{Nesetril2000}.

The other property which is extensively studied is universality. Any countable partial order is isomorphic to a suborder of $(\DiGraphs,\leq)$. This result first appeared in the context of categories in 1969~\cite{Pultr1980}. \emph{Antichains}\index{antichain} are sets of elements which are pairwise incomparable. A characterization of finite maximal antichains was given in~\cite{Nesetril1978}.

Maximal antichains are particularly relevant because of their relationship to the notion of a \emph{homomorphism duality}\index{homomorphism duality}, introduced by Ne\v set\v ril and Pultr~\cite{Nesetril1978}. An ordered pair $(F,D)$ of graphs, or directed graphs, is a \emph{duality pair}\index{duality pair} if $\{G\mid F\to G\}=\{G\mid G\nrightarrow D\}$. Moreover, Ne\v set\v ril and Tardif obtain a correspondence between duality pairs and gaps in the homomorphism order~\cite{Nesetril2000a}.

\newpage
In the second part of this work, we consider properties of constrained homomorphism orders analogous to those for the homomorphism order. We provide a systematic method for characterizing the gaps and dualities in constrained homomorphism orders by analysing properties of future-finite or past-finite orders, which has simplified the proofs in the full homomorphism order. We also give a method of building universal orders from future-finite or past-finite universal orders. Among other things, by realizing that oriented cycles are a new universal class in the homomorphism order, we prove universality of the homomorphism order on the class of line graphs and of locally injective homomorphism order on connected graphs.

\vspace{10pt}

This thesis consists of five chapters, including this introduction.

In Chapter~\ref{ch:graph} we describe notions that we require from graph theory and order theory.
In particular, we introduce basic graph-theoretic concepts and state their properties,
and define the main operations on graphs. We also give the fundamental properties of partial orders.
Chapter~\ref{ch:homo} gives a brief overview of graph homomorphisms and constrained homomorphisms. Moreover, we introduce the new notion of graph relation, and we compare it to the other well-studied types of constrained graph homomorphism. We provide the necessary background on the theory of computational complexity in Appendix A.

Further properties of relations are introduced in Chapter~\ref{ch:relation}.
We first show that every relation can be decomposed in a standard way into a surjective and an injective relation (Corollary~\ref{DeCo}). 
The equivalence classes of strong relational equivalence (Theorem~\ref{thm:equi}) and weak relational equivalence (Proposition~\ref{pro:iso}) are characterized. We then introduce the notion of R-cores, cores, reduced graphs and cocores. We show that core and reduced graph coincide, develop the characterization of cocore and R-core, and provide an effective algorithm to compute them. We also give a perhaps surprisingly simple characterization of those graphs $G$ for which all relations of $G$ to itself are automorphisms (Theorem~\ref{thm:auto1}). The computational complexity of testing for the existence of a relation between two graphs is briefly discussed. We describe the reduction of this problem to the corresponding problem for surjective homomorphisms. Finally, we briefly summarize the most important similarities and differences between weak and strong relational composition (Section~\ref{sect:weak}). Some of the results from this chapter have been published by J.~Hubi\v cka, J.~Jost, Y.~Long, P.~F.~Stadler, and L.~Yang in~\cite{Hubicka2012}.

In Chapter~\ref{Ch:order} we review existing results on the homomorphism order (in Section~\ref{sec:homoorder}) and introduce the notion of constrained homomorphism orders. Moreover, after some discussion of universality, 
a new and significantly easier proof of universality of graph homomorphisms is given (Theorem~\ref{lem:cycles}), which simplifies the proof given in~\cite{Hubicka2005}.
As a new result, we show that graph homomorphisms are universal even on the class of oriented cycles. The notions of future-finite-universal and past-finite-universal partial order are introduced (in Section~\ref{sec:universality}) for later use. We survey existing results on constrained homomorphism orders and provide several new results about cores, universality, gaps and dualities in the context of constrained homomorphisms. We derive properties of these partial orders and identify their similarities and differences. At the same time we give a simple condition for the existence of dualities in future-finite and past-finite partial orders (in Section~\ref{sec:pastfinite}). 
In the full homomorphism order (Section~\ref{sec:fullhomo}) we characterize the gaps and give a new and simple proof of the existence of left duals (Theorem~\ref{thm:fulldual}) (simplifying results independently obtained by Feder and Hell~\cite{Feder2008} and Ball, Ne\v set\v ril, Pultr~\cite{Ball2007,Ball2010}).
For locally injective homomorphisms we prove universality on connected graphs (Theorem~\ref{thm:lochomouniv})
and give a partial characterization of gaps (Theorem~\ref{thm:locingap}). We also prove the universality of all three kinds of locally constrained homomorphisms on graphs. We apply the characterization of PR-cores and FR-cores which are derived in Chapter~\ref{ch:relation} and carry over most results on the surjective homomorphism order. 
We prove that for every $d\geq 3$ the homomorphism order on the class of line graphs with maximum degree $d$ is universal (Theorem~\ref{thm:main}).
Some of the results from this chapter have been submitted~\cite{Fiala2014} or are in preparation~\cite{Fiala2012}.

\chapter[Graphs and Orders]{Graphs and Orders}
\label{ch:graph}
\ifpdf
    \graphicspath{{Chapter1/Chapter1Figs/PNG/}{Chapter1/Chapter1Figs/PDF/}{Chapter1/Chapter1Figs/}}
\else
    \graphicspath{{Chapter1/Chapter1Figs/EPS/}{Chapter1/Chapter1Figs/}}
\fi


\section{Graphs}
Graphs, as mathematical structures used to model pairwise relations between objects from a certain collection, are the main theme of this thesis. Graphs are among the most ubiquitous models of both natural and human-made structures. They can be used to model many types of relations and process dynamics in physical, biological, social and information systems. Many problems of practical interest can be represented by graphs.

\subsection*{Basic concepts}

We introduce several variants of the notion of `graph'. Our situation is the same as Hell and Ne\v{s}et\v{r}il described in~\cite{Hell2004}: ``One challenge we face is the intermingling of the various versions of graphs, directed graphs, and more general systems.
It is typical of the area to freely jump from graphs to directed graphs, allowing or disallowing loops,
as it dictated by the context.'' The main purpose of this part is to cover all the various versions of graphs used in this thesis.

\begin{defi}
A \emph{directed graph}\index{directed graph} (or \emph{digraph})\index{directed graph}\index{digraph} $G$ is a pair $G=(V_G,E_G)$ such that $E_G$ is a
subset of $V_G\times V_G$. Set $V_G$\index{$V_G$} is called the \emph{vertex set of $G$}\index{vertex set}, and each element $V_G$ is a \emph{vertex}\index{vertices, vertex} (or \emph{point}\index{point}, \emph{node}\index{node}) of graph $G$.
Set $E_G$\index{$E_G$} is called the \emph{edge set of $G$}\index{edge set}, and every element in $E_G$ is an \emph{edge}\index{edge} (or \emph{arc}\index{arc}) of the graph $G$.
\end{defi}

Directed graphs are visualized by representing vertices as points and edges as arrows. There are some examples of directed graph shown in Figure~\ref{fig:digraphs}.

\begin{figure}[!ht]
\centering
\subfloat[]{
\label{fig:digraph_a}
\begin{minipage}[c]{0.3\textwidth}
\centering
\includegraphics[width=1\textwidth]{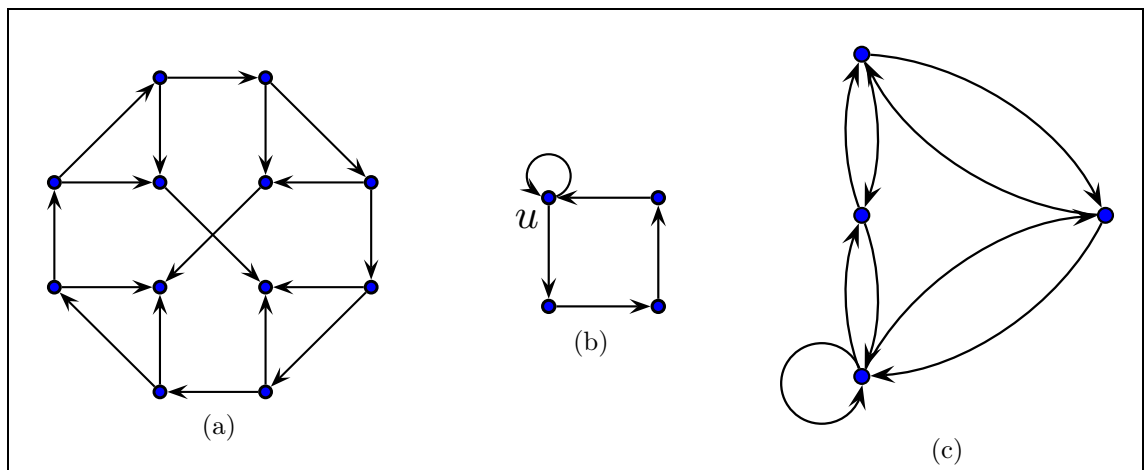}
\end{minipage}\quad 
}
\subfloat[]{
\label{fig:digraph_b}
\begin{minipage}[c]{0.3\textwidth}
\centering
\includegraphics[width=0.45\textwidth]{di2.eps}
\end{minipage}
}
\subfloat[]{
\label{fig:digraph_c}
\begin{minipage}[c]{0.3\textwidth}
\centering
 \includegraphics[width=1\textwidth]{ncarc.eps} 
\end{minipage}
}
\caption{Directed graphs.}
\label{fig:digraphs}
\end{figure}

\begin{defi}
An \emph{undirected graph}\index{undirected graph} (or simply a \emph{graph}\index{graph}), is a graph with symmetric edge set, i.e., a directed graph such that $(u,v)$ is an edge if and only if $(v,u)$ is an edge. 
\end{defi}

The graph shown in Figure~\ref{fig:digraph_c}, for instance, is an undirected graph. On the one hand, we can consider an undirected graph to be a special case of a directed graph.
On the other hand, each directed graph has a unique \emph{underlying graph}\index{underlying graph}, that is, the directed graph with the direction on edges omitted. Figure~\ref{fig:under} shows an example of the underlying graph of a directed graph.


Note that undirected graphs and directed graphs can be mutually constructed.
We have seen that we can create an undirected graph from a directed graph. We can also
orient the edges of undirected graph to be directed.
An \emph{orientation}\index{orientation} of an undirected graph is an assignment of a direction to each edge, making it into a directed graph. A directed graph is called an \emph{oriented graph}\index{oriented graph} if it is an orientation of an undirected graph. In other words, an oriented graph is a directed graph having no symmetric pair of directed edges.

\begin{figure}[!ht]
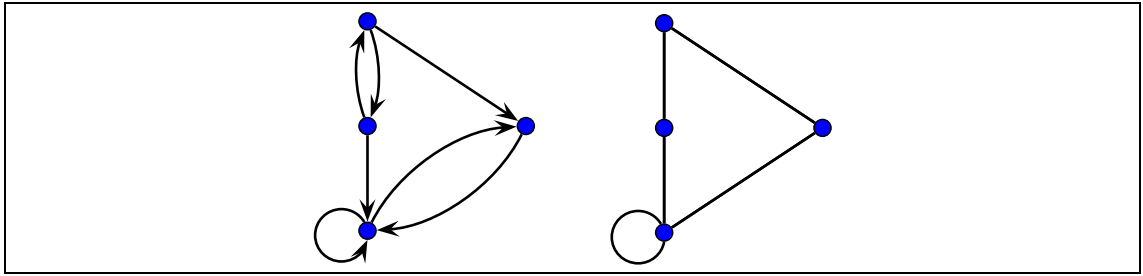

\centering
\includegraphics[width=0.2\textwidth]{ncarc2.eps} 
\qquad
\includegraphics[width=0.2\textwidth]{under.eps} 
\caption{A directed graph and its underlying graph.}
\label{fig:under}
\end{figure}

An edge that connects a vertex to itself is called a \emph{loop}\index{loop}.
For instance, the graph in Figure~\ref{fig:digraph_b} has a loop on vertex $u$. In graph theory, sometimes we only consider \emph{simple graphs}\index{simple graph}\footnote{In the literature it is often the case that the term graph refers to a simple graph.}, that is, undirected graphs without loops. But in this thesis, for the sake of convenience, we use the term \emph{graph}\index{graph} for an undirected graph with loops allowed.

Sometimes we consider multiple edges between a single pair of vertices.
A \emph{multigraph}\index{multigraph} $G$ is an ordered pair $G:=(V, E)$ with $V$ as the vertex set and $E$, a multiset of unordered pairs of vertices,
as the edge set. Figure~\ref{fig:mul} shows a simple graph, a graph, a directed graph,
an oriented graph, and a multigraph.

A directed graph is a \emph{weighted directed graph}\index{weighted directed graph} if a number (weight) is assigned to each edge. Such weights might represent, for example, costs, lengths or capacities, etc. depending on the problem at hand. Some authors call such a graph a \emph{network}\index{network}~\cite{Strang2006}.

\begin{figure}[!ht]
\centering
\subfloat[Simple graph]{
\begin{minipage}[c]{0.3\textwidth}
\centering
\includegraphics[width=0.5\textwidth]{K_32.eps}
\end{minipage}
} \quad
\subfloat[Graph]{
\begin{minipage}[c]{0.3\textwidth}
\centering
\includegraphics[width=0.5\textwidth]{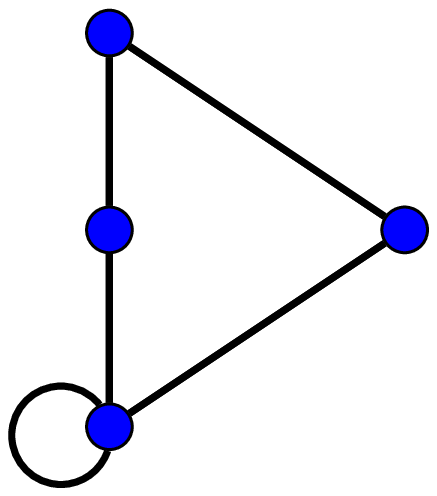}
\end{minipage}
} \quad
\subfloat[Digraph]{
\label{fig:mul_c}
\begin{minipage}[c]{0.3\textwidth}
\centering
\includegraphics[width=0.5\textwidth]{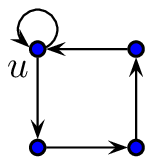}
\end{minipage}
}
\subfloat[Oriented graph]{
\begin{minipage}[c]{0.3\textwidth}
\centering
\includegraphics[width=0.6\textwidth]{orc8.eps}
\end{minipage}
}
\subfloat[Multigraph]{
\label{fig:mul_e}
\begin{minipage}[c]{0.3\textwidth}
\centering
\includegraphics[width=0.7\textwidth]{mul.eps}
\end{minipage}
}
\caption{Different classes of graphs.}
\label{fig:mul}
\end{figure}

The \emph{size of a graph $G$}\index{size of a graph}, denoted by $|G|$\index{$"| G"|$}, is the number of its vertices.
It is possible to consider an \emph{infinite graph}\index{infinite graph}, that is, a graph with infinitely many vertices. However,
in this thesis we consider only \emph{finite graphs}\index{finite graph}, that is, graphs of finite size. 
There are different classes of graphs. In this thesis, 
if not explicitly stated, by a graph we mean a finite simple graph possibly with a loop on some vertices. 



\subsubsection*{Adjacency and neighbourhoods}

Let $G=(V_G,E_G)$ be a graph, $x$ and $y$ be vertices in graph $G$. If $(x,y)\in E_G$, then we say vertex $x$ and vertex $y$ are \emph{adjacent}\index{adjacent}.
The \emph{degree}\index{degree} or \emph{valency}\index{valency} of a vertex $v$ in a graph $G$, denoted by $d(v)$\index{$d(v)$},
is the number of vertices which is adjacent to vertex $v$.
An \emph{isolated vertex}\index{isolated vertex} is a vertex with degree 0. Note that a simple loop is not an isolated
vertex in this case. For example, in Figure~\ref{fig:dis}, vertex $i$ is an isolated vertex,
vertex $o$ is an isolated loop, but not an isolated vertex.

\begin{figure}[!ht]
\centering
\includegraphics[width=0.28\textwidth]{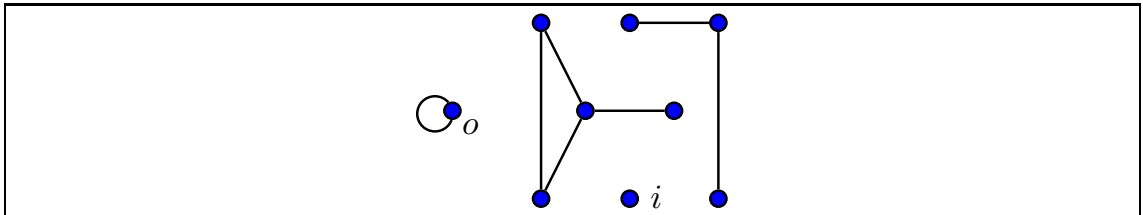}
 \caption{Disconnected graph with an isolated vertex and a loop.}
\label{fig:dis}
\end{figure}

Let $G$ be a graph, and $x$ be a vertex of graph $G$.
\begin{defi} 
The \emph{(open) neighbourhood}\index{open neighbourhood}\index{neighbourhood} of $x$ in $G$, denoted by $N_G(x)$\index{$N_G(x)$}, is the set of vertices adjacent to a vertex $x\in G$.
The \emph{closed neighbourhood}\index{closed neighbourhood} of $x$ in $G$, 
denoted by $N_G[x]$\index{$N_G[x]$}, is the open neighbourhood and $x$, i.e., $N_G[x]=N_G(x)\cup \{x\}$.
\end{defi}

For the sake of brevity we use $N(x)$\index{$N(x)$} and $N[x]$\index{$N[x]$} if it is clear which graph we refer to.
As an example, in Figure~\ref{fig:dis} $N(o)=o,N[o]=o, N(i)=\emptyset, N[i]=i$.

\subsubsection*{Subgraphs} \label{subsec:sub}
We say that $G'=(V',E')$ is a \emph{subgraph}\index{subgraph} of $G=(V,E)$, denoted by $G'\subseteq G$\index{$\subseteq$}, if $V'\subseteq V$ and $E'\subseteq E$. If $G'$ contains all edges of $G$ which join two vertices in $V'$, then 
$G'$ is an \emph{induced subgraph}\index{induced subgraph} of $G$.
Graph $G'$ is said to be the subgraph \emph{induced by $V'$} and denoted by $G[V']$\index{$G[V']$}. Induced subgraphs form a special class of subgraphs. Figure~\ref{fig:ind} shows the difference between subgraphs and induced subgraphs.

\begin{figure}[!ht]
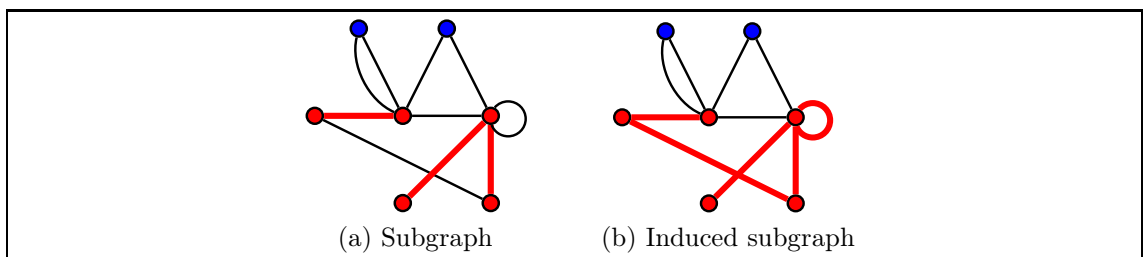

\centering
\subfloat[Subgraph]{%
  \label{fig:subg}
  \includegraphics[width=0.2\textwidth]{subg.eps}
}\qquad
\subfloat[Induced subgraph]{%
  \hspace*{4pt}%
  \includegraphics[width=0.2\textwidth]{indsubg.eps}
  \hspace*{4pt}%
}
\caption{Subgraph and induced subgraph.}
\label{fig:ind}
\end{figure}

\subsubsection*{Diameter and Connectivity}

A \emph{walk}\index{walk} in a graph $G$ is a sequence of vertices $v_0,v_1,\dots,v_k$ of $G$ such that $(v_{i-1},v_i)\in E_G$ for each $i=1,2,\dots,k$.
A \emph{path}\index{path} is a non-repeated walk, that is the vertices in a walk are distinct.
A \emph{cycle}\index{cycle} is a walk $v_0,v_1,\dots,v_k$ with $k\geq 2$ such that $v_0=v_k$ and the vertices $v_i$ for $0<i<k$ are distinct from each other and $x_0$.

\begin{figure}[!ht]
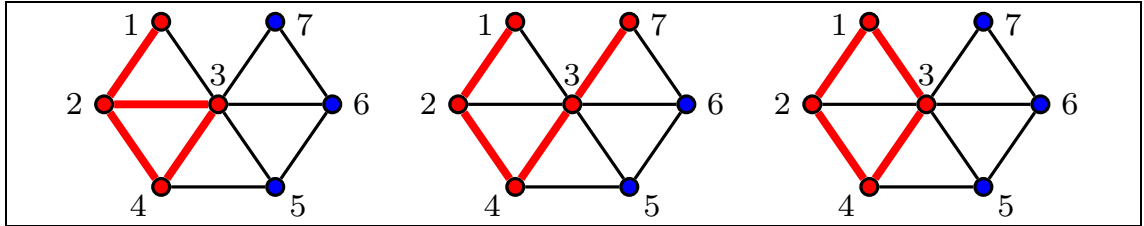

\centering
\subfloat{
\includegraphics[width=0.27\textwidth]{walk.eps}
}\quad
\subfloat{
\includegraphics[width=0.27\textwidth]{path.eps}
}\quad
\subfloat{
\includegraphics[width=0.27\textwidth]{cycle.eps}
}
\caption{An example of walk, path and cycle.}
\label{fig:walk}
\end{figure}

\begin{example}
In Figure~\ref{fig:walk}, sequence 1, 2, 4, 3, 2 is a walk (but not a path nor a cycle), sequence 1, 2, 4, 3, 7 is a
path, and sequence 1, 2, 4, 3, 1 is a cycle.
\end{example}

The \emph{distance}\index{distance} between two vertices $u,v$ in a graph, denoted by $d_G(u,v)$\index{$d_G(u,v)$}, is the number of edges in a shortest path connecting them; if there is no path connects vertices $u$ and $v$,
then the distance is infinite. The \emph{eccentricity}\index{eccentricity} $\epsilon$ of a vertex $v$ is the greatest distance between $v$ and any other vertex.
The \emph{radius}\index{radius} of a graph $G$, denoted by $rad(G)$\index{$rad(G)$}, is the minimum eccentricity
of any vertex. The \emph{diameter}\index{diameter} of a graph $G$, denoted by $diam(G)$\index{$diam(G)$}, is the maximum eccentricity of any vertex in the graph, that is, the greatest distance between any pair of vertices.

Graph connectivity is in the sense of topological space.
A graph is \emph{connected}\index{connected} if there is a path from any vertex to any other vertex in the graph.
A \emph{connected component}\index{connected component} of an undirected graph is a subgraph in which any two vertices are connected to each other by paths, and which is connected to no additional vertices in the subgraph.
For example, the graph shown in the Figure~\ref{fig:dis} has three connected components: an isolated vertex, a loop and a connected graph.

\subsection*{Graph constructions}

When we need a new graph, the most convenient way is to construct it from old graphs. There are many methods of construction, for instance, building subgraphs and induced subgraphs. In this subsection we introduce some common graph construction methods.

\subsubsection*{Complement graphs}
The \emph{complement}\index{complement} (or \emph{inverse}\index{inverse}) of a simple graph $G$, denoted by $\overbar{G}$\index{$\overbar{G}$}, is a simple graph on the same vertices such that two vertices of $\overbar{G}$ are adjacent if and only if they are not adjacent in $G$.

Figure~\ref{fig:peter} shows an example of Petersen graph (left) and its complement graph (right).
\begin{figure}[!ht]
\centering
 \includegraphics[width=0.6\textwidth]{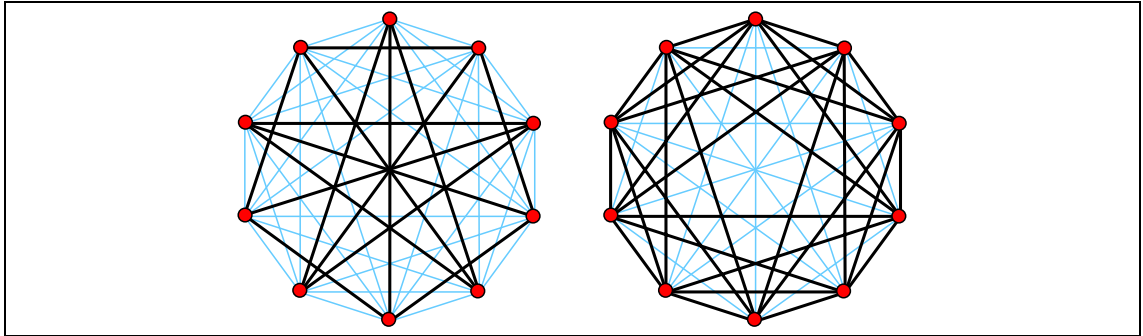}
\caption[Petersen graph]{Petersen graph and its complement\protect\footnotemark.}
\label{fig:peter}
\end{figure}

\footnotetext{This figure is from \textit{Complement graph. (2013, December 20). In Wikipedia, The Free Encyclopedia. Retrieved 09:44, February 20, 2014, from \url{http://en.wikipedia.org/w/index.php?title=Complement_graph&oldid=587001685}}.}

\subsubsection*{Vertex deletions}
Deleting or adding some vertices and edges are common methods to construct graphs. 

Let $G$ be a directed graph. If $W\subset V_G$, then we define {\em $G-W$}\index{$G-W$} to be $G[V_G\setminus W]$, i.e., the subgraph 
of $G$ obtained by deleting the vertices in $W$ and all edges adjacent with them.

\subsubsection*{Closed neighbourhood deletions} 

We define another method of vertex deletion. 
Let $N_G[x]$ be the closed neighbourhood\index{closed neighbourhood} of vertex $x$ in directed graph $G$. We define 
$\overbar{G_x}:=G-N[x]$\index{$\overbar{G_x}$}, i.e., the induced subgraph of $G$ that is
obtained by removing the closed neighbourhood of a vertex $x$.
Analogously, for a subset $S\subseteq V_G$ we define $\overbar S = G-\bigcup_{x\in S}N_G[x]$.
That is, the induced subgraph obtained by removing all vertices in $S$ and their
neighbours. Figure~\ref{fig:rem} shows an example of $\overbar G_u$.


\begin{figure}[!ht]
\centering
\begin{minipage}[c]{0.3\textwidth}
\centering
 \includegraphics[height=0.6\textwidth]{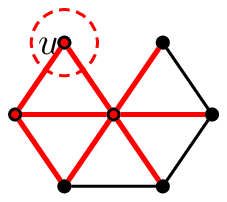} 
\end{minipage}
\hspace{0.5cm} 
 $\longrightarrow$
\begin{minipage}[c]{0.3\textwidth} 
 \centering
  \includegraphics[height=0.5\textwidth]{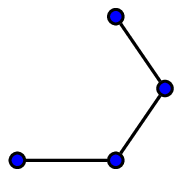}
  \end{minipage} 
\caption{Vertex deletion $\overbar G_u = P_4$.}
\label{fig:rem}
\end{figure}

Similarly we can also consider edge deleting and adding instead of vertex~\cite{Diestel2000}. However, since it is not needed in the sequel, we omit it here.

\subsubsection*{Duplications, Contractions}

\begin{defi}
The \emph{vertex duplication}\index{vertex duplication} of a graph $G$ by a vertex $x$ produces a new graph $G'=G\circ x$,
by adding a new vertex $x'$ and then fastening it to the neighbourhood of $x$ $(N_G(x')=N_G(x))$.
\end{defi}
Obviously, this operation can be iterated. 

\begin{defi}
The \emph{vertex multiplication}\index{vertex multiplication} of $G$ by the positive integer vector $h=(h_1,\dots,h_n)$ is
the graph $H=G\circ h$ whose vertex set consists of $h_i$ copies of $x_i$, and each copy of $x_i$
is adjacent to a copy of $x_j$ in $H$ if and only if $x_i$ is adjacent to $x_j$ in $G$.
\end{defi}

\begin{figure}[!ht]
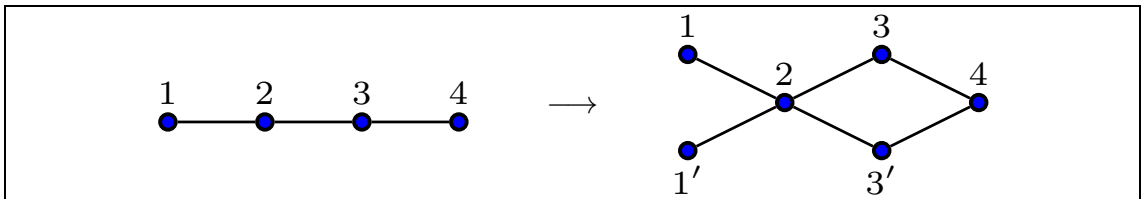

\centering
\begin{minipage}[c]{0.4\textwidth}
\centering
\includegraphics[width=0.7\textwidth]{dupl1.eps}
\end{minipage}
$\longrightarrow$
\begin{minipage}[c]{0.4\textwidth}
\centering
\includegraphics[width=0.7\textwidth]{dupl2.eps}
\end{minipage}
\caption{Vertex multiplication: $P_4\circ(2,1,2,1)$.}
\label{fig:dup}
\end{figure}

Duplication always enlarges graphs. Dually, we can also reduce a graph.

\begin{defi} 
The \emph{contraction}\index{contraction} of a pair of vertices $v_i$ and $v_j$ of a graph produces a graph in
which the two vertices $v_1$ and $v_2$ are replaced with a single vertex $v$ such 
that $v$ is adjacent to the union of the vertices to which $v_1$ and $v_2$ were originally adjacent. 
\end{defi}


\begin{figure}[!ht]
\centering
\begin{minipage}[c]{0.4\textwidth}
\centering
\includegraphics[width=0.7\textwidth]{dupl2.eps}
\end{minipage}
$\longrightarrow$
\begin{minipage}[c]{0.4\textwidth}
\centering
\includegraphics[width=0.6\textwidth]{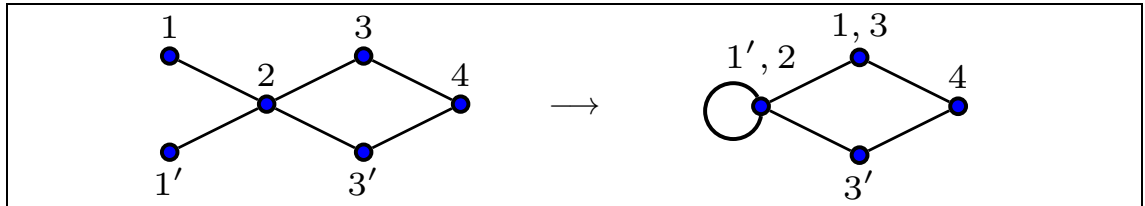}
\end{minipage}
\caption{Vertex contraction \{$1,3$\} and \{$1',2$\}.}
\label{fig:contr2}
\end{figure}

\begin{example}
We give two examples:
 \begin{enumerate}
  \item Figure~\ref{fig:dup} shows a vertex multiplication of $P_4$ by duplication vertex 1 and 3.
  \item In vertex contraction, it does not matter if $v_1$ and $v_2$ are connected by an edge or not.
If they are, we produces a loop, see Figure~\ref{fig:contr2}.
 \end{enumerate}
\end{example}

\begin{rem}
In this thesis, contraction is referred to as vertex contraction,
even though in some literature~\cite{Diestel2000} contraction is referred to as edge contraction or is ambiguous about the distinction between vertex
contraction and edge contraction. 
\end{rem}

\subsubsection*{Quotient graphs}

Forming a quotient graph is another way of constructing new graphs from old. A \emph{partition}
\index{partition} $\mathcal{P}$ of the vertices set $V_G$ is a division
into disjoint subsets: $\mathcal{P}=\{S_1,\dots, S_k\}$, i.e.\ $S_i\cap S_j=\emptyset$ for any $i,j$ and $\bigcup_{1\leq i\leq k}{S_i}=V_G$.
\emph{Quotient graph}\index{quotient graph} induced by the partition $\mathcal{P}$, written $G/\mathcal{P}$,
is the graph $\widetilde{G}=(\widetilde{V},\widetilde{E})$ with vertex set
$ \widetilde{V}= \{S_1,\dots,S_k\}$, in which
$ (S_i,S_j)\in \widetilde{E} $ if and only if
$ \exists v \in S_i, w\in S_j$ such that $(v,w) \in E.$

\begin{rem}
 In our definition we allow loops, i.e., for the case $i=j$, if there exists $(v,w)\in E$ and $(v,w)\in S_i$, then there is a loop on $S_i$ in the quotient graph.
\end{rem}

\begin{figure}[!ht]
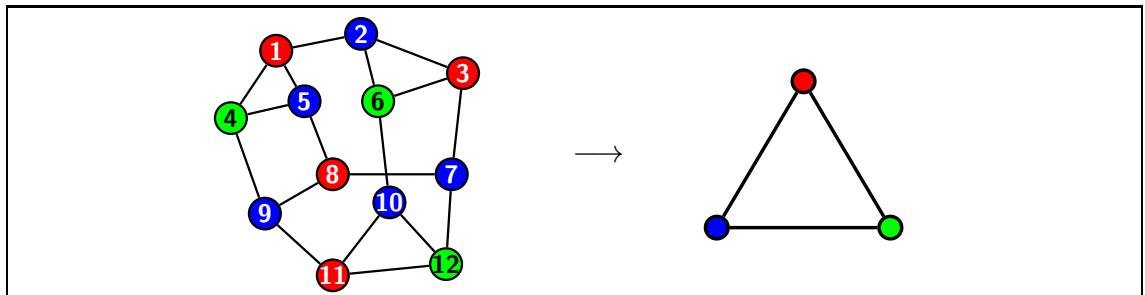

\centering
\begin{minipage}[c]{0.3\textwidth}
\centering
\includegraphics[width=0.8\textwidth]{Frucht.eps}
\end{minipage}
\hspace{0.5cm}
$\longrightarrow$
\begin{minipage}[c]{0.3\textwidth}
\centering
\includegraphics[width=0.6\textwidth]{K_3.eps}
\end{minipage}
\caption{Quotient graph of Frucht graph (left) is $K_3$.}
\label{fig:quo}
\end{figure}

If we take the partition $\mathcal{P}=\{S_1,S_2,S_3\}$ by different colours of Frucht graph, i.e.\ $S_1=\{1,3,8,11\}$ (red vertices), $S_2=\{2,5,7,9,10\}$ (blue vertices), $S_3=\{4,6,12\}$ (green vertices), then the quotient graph of Frucht graph is $K_3$ (Figure~\ref{fig:quo}).

\subsection*{Graph classes}

In graph theory, some classes of graphs are frequently considered, due to their simple but well-behaved structural properties. For instance, connected graphs are a nice class of graphs. In this subsection we introduce some of these concepts which we will need for later use.

\subsubsection*{Path graphs}
A path in a graph is formed by a sequence of distinct vertices $v_1,\dots,v_n$ such that $(v_i,v_{i+1})$ is an edge for all $1\leq i< n$. A \emph{path graph}\index{path graph} (or simply \emph{path}\index{path}) with $n$ vertices, denoted by $P_n$\index{$P_n$}, is a graph induced by a path. 
As in most graph theory literature, we do not make a distinction between a path and its path graph.

An \emph{oriented path $\overrightarrow{P_n}$}\index{oriented path} with $n$ vertices,
is a directed graph $G$ where $V_G=\{v_1,\ldots,v_n\}$ and for
every $1\leq i<n$ either $(v_i,v_{i+1})\in E_G$ or $(v_{i+1},v_i)\in E_G$ (but not both),
and there are no other edges. Thus an oriented path is any orientation of an undirected path. 
We denote by $\Path$\index{$\Path$} the class of all finite oriented paths.
A \emph{directed path}\index{directed path}\index{directed path} with $n$ vertices, denoted by $\overrightarrow{P_n}$\index{$\overrightarrow{P_n}$}, is a path which the edges are oriented in the same direction, i.e.\ $(v_i,v_{i+1})\in E_G$ or $(v_{i+1},v_i)\in E_G$ for all $1\leq i<n$.
A path\index{path} $P_n$ can be seen as an underlying graph of an oriented path $\overrightarrow{P_n}$ and a directed path.

Figure~\ref{fig:orp} shows an oriented path, a directed path and a path.

\begin{figure}[!ht]
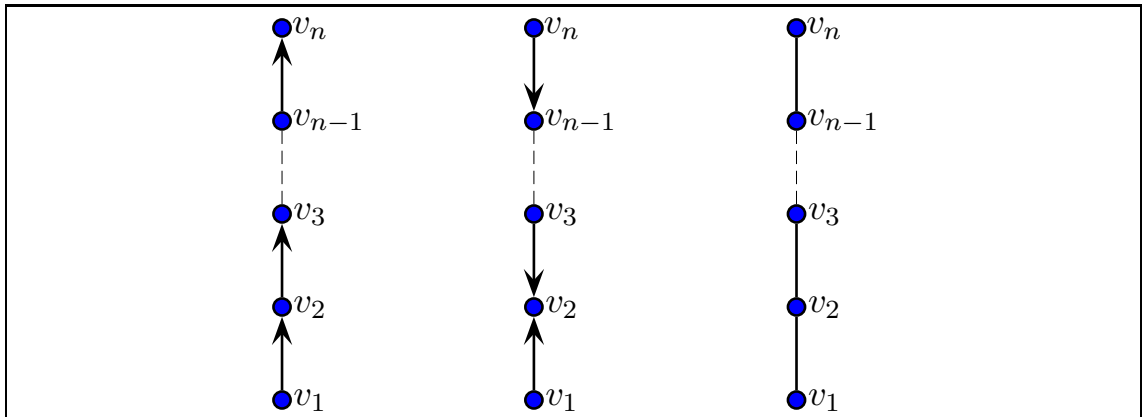

 \centering
 \includegraphics[width=0.08\textwidth]{orianted_path.eps}\hspace{20 mm}
  \includegraphics[width=0.08\textwidth]{dipath.eps} \hspace{20 mm}
  \includegraphics[width=0.08\textwidth]{pathg.eps}
 \caption{Directed path, oriented path and path.}
\label{fig:orp}
\end{figure}

\subsubsection*{Cycle graphs}
A \emph{(unoriented) cycle graph}\index{cycle graph} (or simply \emph{cycle}\index{cycle}\footnote{Like path, we also do not distinct cycle graph and cycle.}) with $n$ ($n\geq 3$) vertices $\{1,2,\dots,n\}$, denoted by $C_n$\index{$C_n$},
is a simple graph with edges $(i,i+1)$ for all $1 \leq i< n$ and $(1,n)$.
We denote by $\Cycle$\index{$\Cycle$} the class of graphs formed by all $C_k$, $k\geq 3$, and by $\Cycles$\index{$\Cycles$} the class of graphs formed by disjoint union of finitely many graphs in $\Cycle$.

\begin{figure}[!ht]
\centering
\subfloat[$C_3$]{
\begin{minipage}[c]{0.2\textwidth}
\centering
\includegraphics[width=0.7\textwidth]{K_32.eps}
\end{minipage}
} \quad
\subfloat[$C_6$]{
\begin{minipage}[c]{0.3\textwidth}
\centering
\includegraphics[width=0.8\textwidth]{C_6.eps}
\end{minipage}
} 
\caption{Cycle graphs.}
\label{fig:cyc}
\end{figure}

A \emph{directed cycle}\index{directed cycle} is formed by a cycle with edges oriented in one direction. 

\begin{defi}
\emph{Oriented cycle}\index{oriented cycle} with $n$ vertices, denote by $\overrightarrow{C}_n$\index{$\overrightarrow{C}_n$}, is an orientation of a cycle graph, with all the edges being oriented in the same direction. We denote by $\DiCycle$\index{$\DiCycle$} the class of directed graphs formed by all $\overrightarrow{C}_k$, $k\geq 3$, and $\DiCycles$\index{$\DiCycles$} the class of directed graphs formed by disjoint union of finitely many graphs in $\DiCycle$.
\end{defi}

\begin{figure}[!ht]
\centering
 \includegraphics[width=0.2\textwidth]{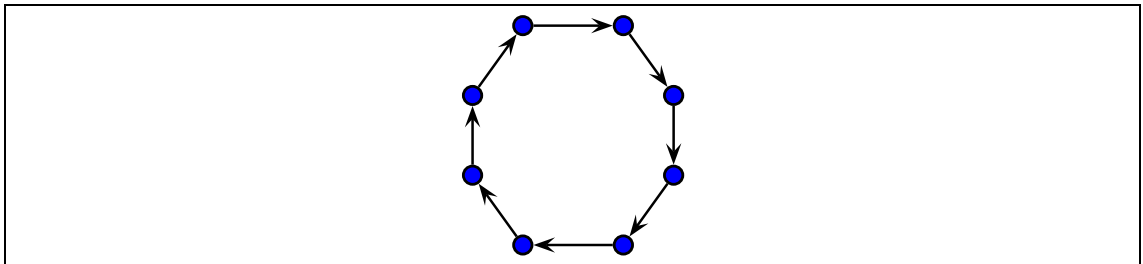}
\caption{Oriented cycle with 8 vertices.}
\end{figure}

\subsubsection*{Reflexive, transitive graphs}
A graph is \emph{reflexive}\index{reflexive} if for every vertex there is a loop.
A graph is \emph{transitive}\index{transitive} if for any pair of edges $(v_i,v_j)$ and $(v_j,v_k)$, there is an edge $(v_i,v_k)$.
A \emph{directed acyclic graph}\index{directed acyclic graph} (or shortly \emph{DAG}\index{DAG}) is a directed graph with no directed cycles.
Reflexive graphs, transitive graphs and directed acyclic graphs are representations of relational structure which we will introduce in the second part of this chapter.

\begin{figure}[!ht]
\centering
 \includegraphics[width=0.3\textwidth]{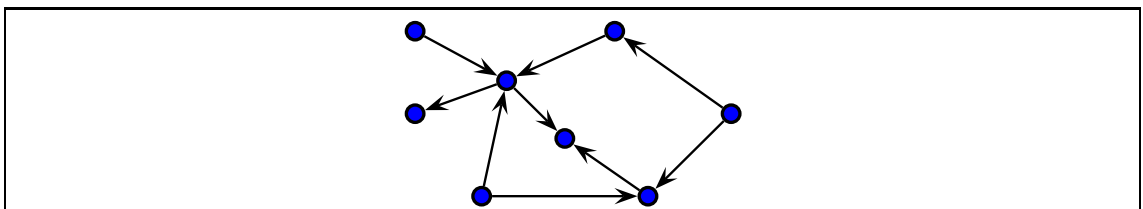}
\caption{A directed acyclic graph.}
\end{figure}

\subsubsection*{Complete graphs}
 A \emph{complete graph}\index{complete graph} with $n$ vertices, denoted by $K_n$\index{$K_n$},
is a simple graph in which every pair of distinct vertices is connected by a unique edge.
For convenience, we denote the vertices of $K_n$ by $1,2,\dots,n$\index{$1,2,\dots,n$}.

Figure~\ref{fig:complete} shows the complete graphs with vertex number 1 to 4.
 
\begin{figure}[!ht]
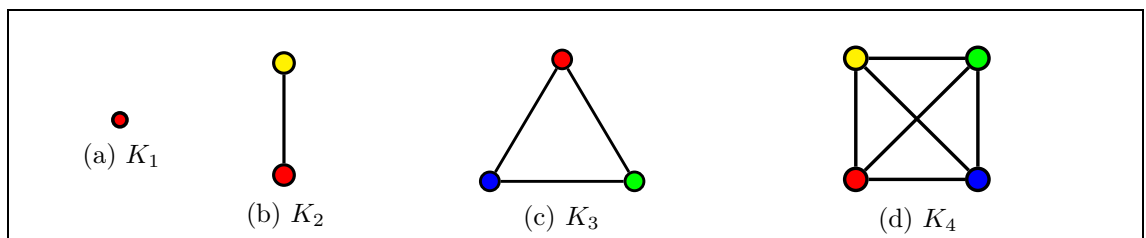

\centering
\subfloat[$K_1$]{
\begin{minipage}[c]{0.1\textwidth}
\centering
\includegraphics[width=0.2\textwidth]{K_1.eps}
\end{minipage}
} \quad
\subfloat[$K_2$]{
\begin{minipage}[c]{0.1\textwidth}
\centering
\includegraphics[width=0.25\textwidth]{K_2.eps}
\end{minipage}
} \quad
\subfloat[$K_3$]{
\begin{minipage}[c]{0.3\textwidth}
\centering
\includegraphics[width=0.5\textwidth]{K_3.eps}
\end{minipage}
}
\subfloat[$K_4$]{
\begin{minipage}[c]{0.3\textwidth}
\centering
\includegraphics[width=0.45\textwidth]{K_4.eps}
\end{minipage}
}
\caption{Complete graphs.}
\label{fig:complete}
\end{figure}

\subsubsection*{$M$-partite graphs}
A \emph{$m$-partite graph}\index{$m$-partite graph} is a simple graph whose vertices can be decomposed into $k$ disjoint
sets $U_1, \dots, U_m$, such that every edge connects a vertex in $U_i$ to one in
$U_j$ for some distinct $i,j\in \{1,2,\dots,m\}$. A \emph{complete $m$-partite graph}\index{complete $m$-partite graph}
is a $m$-partite graph where every vertex in $U_i$ is connected to every vertex in $U_j$.

See Figure~\ref{fig:cnp} for an example of complete 3-partite graph.

\begin{figure}[!ht]
 \centering
 \includegraphics[width=0.5\textwidth]{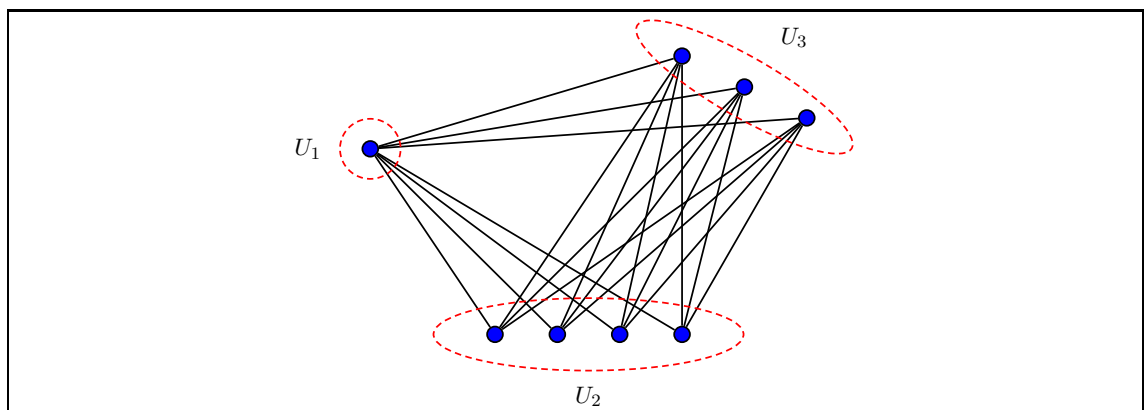}
 \caption{A complete 3-partite graph.}
\label{fig:cnp}
\end{figure}
 
\subsubsection*{Trees}

A \emph{tree}\index{tree} is a connected simple graph which has no cycles. 
A \emph{oriented tree}\index{oriented tree} is a directed tree having no symmetric pair of directed edges. Figure~\ref{tree} shows an oriented tree (left) and a tree (right).

\begin{figure}[!ht]
 \begin{minipage}[c]{0.3\textwidth}
 \centering
 \includegraphics[width=0.8\textwidth]{otree.eps}
 \end{minipage} \hspace{1cm}
\begin{minipage}[c]{0.5\textwidth}
\centering
 \includegraphics[width=1\textwidth]{tree2.eps}
 \end{minipage}
 \caption{Trees.}
\label{tree}
\end{figure}

\subsubsection*{Graph colourings}
A \emph{(vertex) colouring}\index{vertex colouring}\index{colouring} is a way of colouring
the vertices of a simple graph such that no two adjacent vertices share the same colour.
A colouring using at most $k$ colours is called a \emph{(proper) $k$-colouring}\index{proper $k$-colouring}\index{$k$-colouring}.
The \emph{chromatic number}\index{chromatic number} of a simple graph $G$, denoted by $\chi(G)$\index{$\chi(G)$}, is the smallest $k$ such that $G$ admits a $k$-colouring.

Figure~\ref{fig:complete} shows $n$-colourings of the graph $K_n$ for $n=1,\dots,4$.

\begin{example}
 It is non-trivial to determine the chromatic number of an arbitrary graph, but for some certain graphs, the chromatic numbers are known:
\begin{enumerate}
 \item $\chi(K_n)=n$.
 \item $\chi(P_n)=2$.
 \item $\chi(C_{2k+1})=3$, $\chi(C_{2k})=2$.
 \item The chromatic number of the Frucht graph (Figure~\ref{fig:quo} left) as well as the Petersen graph (Figure~\ref{fig:peter}) is 3.
 \end{enumerate}
\end{example}

\subsubsection*{Graph isomorphisms} \label{sub:iso}
An \emph{isomorphism}\index{isomorphism} $f$ of graphs $G$ and $H$ is a bijection between the vertex sets of $G$ and $H$ such that any two vertices $u$ and $v$ of $G$ are adjacent in $G$ if and only if $f(u)$ and $f(v)$ are adjacent in $H$.
In the case when $G$ and $H$ are the same graph, an isomorphism is called an automorphism. 

\begin{table}[t]
 \centering
 \bgroup
 \def\arraystretch{1.2}
\begin{tabular}{|c|c|c|} \hline
Graph $G$ & Graph $H$ & An isomorphism between $G$ and $H$\\ \hline
\multirow{8}{*}{\includegraphics[height=4cm]{Graph_iso_a.eps}} & \multirow{8}{*}{\includegraphics[height=4cm]{Graph_iso_b.eps}} & $f(a)$ = 1\\&& $f(b)$ = 6\\ &&$f(c)$ = 8\\ &&$f(d)$ = 3\\ &&$f(g)$ = 5\\ &&$f(h)$ = 2\\ &&$f(i)$ = 4\\ &&$f(j)$ = 7\\ \hline
\end{tabular}
\egroup
\caption{Graph isomorphism~\protect\footnotemark.}
\end{table}

\footnotetext{This example is from \textit{Graph isomorphism. (2014, February 18). In Wikipedia, The Free Encyclopedia. Retrieved 09:48, February 20, 2014, from \url{http://en.wikipedia.org/w/index.php?title=Graph_isomorphism&oldid=596007728}}.}

We mention that the graph isomorphism problem is one of the few standard problems in computational complexity theory belonging to NP, but not known to belong to either of its well-known (and, if P $\neq$ NP, disjoint) subsets: P and NP-complete.

\subsection*{Matrix representations of graphs}
\subsubsection*{Adjacency matrices}
There are several matrix representations of a graph. The most intuitive one is its adjacency matrix.
\emph{Adjacency matrix}\index{adjacency matrix} of a graph $G$, denoted by $A_G$\index{$A_G$}, is a $|G|\times |G|$ matrix. If there is an edge from a vertex $x$ to a vertex $y$ then the element $a_{x,y}$ is 1, otherwise it is 0. If graph $G$ is a multigraph or weighted graph, then the elements $a_{x,y}$ are taken to be the number of edges between the vertices or the weight of the edge ($x$, $y$), respectively.
The adjacency matrix is in many cases very useful. For example, in computing, it makes finding subgraphs and reversing a directed graph easy.

\begin{example}
 
The adjacency matrices of Fig.~\ref{fig:mul_c} and Fig.~\ref{fig:mul_e} are as below:

\begin{gather*}
\begin{pmatrix}
1 & 0 & 0 & 1\\
1 & 0 & 0 & 0\\
0 & 1 & 0 & 0\\
0 & 0 & 1 & 0
\end{pmatrix} \qquad 
\begin{pmatrix}
1 & 0 & 0 & 1\\
1 & 0 & 0 & 0\\
0 & 1 & 0 & 0\\
0 & 0 & 1 & 0
\end{pmatrix}
\end{gather*}
\end{example}

Suppose two (directed) graphs $G$ and $H$ whose adjacency matrices $A_G$ and $A_H$ are given. Then $G$ and $H$ are isomorphic if and only if there exists a permutation matrix $P$ such that
$P A_G P^{-1} = A_H$.

\subsubsection*{DDMs and DRMs} \label{pra:DeMat}
Let $G$ be a connected graph.
An \emph{equitable partition}\index{equitable partition} of $G$ is a partition $B_1,\dots,B_k$ of its vertex set such that any vertex in $B_i$ has the same number of neighbours in $B_j$ for any $i,j$. 
A \emph{degree decomposition matrix}\index{degree decomposition matrix} (also called \emph{degree matrix}\index{degree matrix}\footnote{It is different from the degree matrix which contains information about the degree of each vertex.} in~\cite{Fiala2008b} or \emph{degree partition matrix}\index{degree partition matrix}) of $G$ is a square matrix $M=\{m_{i,j}\}$ of order $k$ such that there is an equitable partition of $V_G$ into blocks $B_1,\dots,B_k$ satisfies $|N_G(u)\cap B_j|=m_{i,j}$ for all $u\in B_i$ and $1\leq i,j\leq k$.

Note that a graph $G$ can allow several degree matrices, with an adjacency matrix itself being the largest one.
An equitable partition is said to be \emph{finer}\index{finer} than another one when every class of its partition is a
subset of some class of the latter one, e.g., the partition into singletons is finer than any other equitable partition
of the same graph.
In this case a canonical ordering can be imposed on the blocks, so the corresponding degree matrix, called the 
\emph{degree refinement matrix}\index{degree refinement matrix}, or \emph{DRM}\index{DRM}, denoted by $drm(G)$\index{$drm(G)$}, is also defined uniquely~\cite{Fiala2005}.

There is an efficient algorithm for computing the degree refinement matrix of a given graph~\cite{Fiala2005}.

\begin{figure}[!ht]
\centering
\includegraphics[width=0.25\textwidth]{degree.eps}
\caption{An example.}
\label{fig:degree}
\end{figure}

\begin{example}
 The graph showed in Fig.~\ref{fig:degree} has degree decomposition matrices:
\begin{gather*}
\begin{pmatrix}
0 & 0 & 1 & 0 & 0 & 0 & 0\\
0 & 0 & 1 & 0 & 0 & 0 & 0\\
1 & 1 & 0 & 1 & 1 & 0 & 0\\
0 & 0 & 1 & 0 & 1 & 0 & 0\\
0 & 0 & 1 & 1 & 0 & 1 & 1\\
0 & 0 & 0 & 0 & 1 & 0 & 0\\
0 & 0 & 0 & 0 & 1 & 0 & 0
\end{pmatrix} \qquad
\begin{pmatrix} 
 1 & 0 & 0 & 0 & 0\\
 2 & 0 & 1 & 1 & 0\\
 0 & 1 & 0 & 1 & 0\\
 0 & 1 & 1 & 0 & 2\\
 0 & 0 & 0 & 1 & 0
\end{pmatrix} \quad \text{and} \quad
\begin{pmatrix}
 0 & 1 & 0\\
 2 & 0 & 1\\
 0 & 2 & 0
\end{pmatrix}
\end{gather*}
which are the adjacency matrix with partitions
$B_1=\{1,2\},B_2=\{3\},B_3=\{4\},B_4=\{5\},B_5=\{6,7\}$,
and $B_1=\{1,2,6,7\},B_2=\{3,5\},B_3=\{4\}$, respectively. The last one is the degree refinement matrix.
\end{example}

\subsubsection*{Matchings}
A \emph{matching}\index{matching} $M$ in a graph $G$ is a set of edges which do not share endpoints.
$M$ is a \emph{perfect matching}\index{perfect matching} $M$ if any vertex of $G$
can be found as an endpoint in $M$.

\begin{figure}[!ht]
 \centering
 \includegraphics[width=0.17\textwidth]{matching1.eps}
\qquad \qquad \qquad
 \includegraphics[width=0.17\textwidth]{matchin2.eps}
 \caption{Matching.}
\label{fig:matching}
\end{figure}

\begin{figure}[!ht]
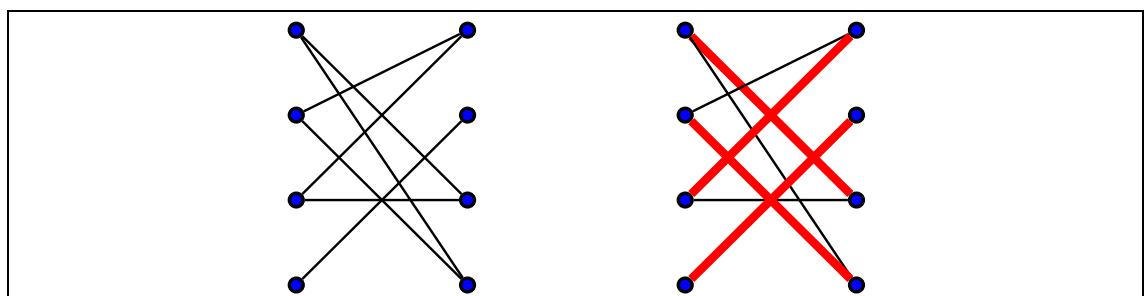

 \centering
 \includegraphics[width=0.17\textwidth]{matching4.eps}
\qquad \qquad \qquad
 \includegraphics[width=0.17\textwidth]{matching3.eps}
 \caption{Perfect matching.}
\label{fig:pmatching}
\end{figure}

Fig.~\ref{fig:matching} shows a matching, but it is not a perfect matching. Actually there is no perfect matching in the graph shown in Fig.~\ref{fig:matching},
while Fig.~\ref{fig:pmatching} shows a perfect matching.

Hall's marriage theorem, or simply Hall's theorem, gives a necessary and sufficient condition for being able to select a distinct element from each of a collection of finite sets. Here we state a simpler version.

\begin{thm}[\textbf{Hall's marriage theorem}]
\label{thm:hallo} 
Given a finite bipartite graph $G$, with bipartite sets $X$ and $Y$ of equal size, there is a perfect matching in $G$ if and only if for any subset $W$ of $X$, $|W|\leq |N_G(W)|$.
\end{thm}

\begin{proof}
If there is a perfect matching in a bipartite graph $G$, then for any subset $W$ of $X$, $|W|\leq |N_G(W)|$.
Conversely, if for any subset $W$ of $X$, $|W|\leq |N_G(W)|$, then any vertex $x$ in $X$ has at least one neighbor. Otherwise taking $W=x$ the inequality does not hold. Similarly any vertex $y$ in $Y$ has at least one neighbor, otherwise taking $W=X$ the inequation does not hold. We prove by induction on $|X|$. Suppose $|X|=n$. For $n=1$ it is trivial. Assume that the statement is true for all bipartite graphs satisfying $|X|=|Y|=k$. We want to prove that it is also true for $|X|=|Y|=k+1$.
Take vertex $v$ in $X$. We consider two casses.
Case 1: Suppose $\forall W\subseteq X\setminus \{v\}$, $|N_G(W)|\geq |W|+1$. In this case we can match $v$ with any of its neighbors and it is still a perfect matching.
Case 2: There exists a subset $W\subseteq X\setminus \{v\}$ such that $|N(W)|=|W|$. Pick a minimum cardinality set satisfying this property. By induction and the minimality of $W$, we know that $W$ can be matched to $N(W)$. Now consider $X\setminus W$, take a set $W'\subseteq X\setminus W$. Then $W'$ muss have neighbors outside $N(W)$ otherwise the inequality does not hold for $W\cup W'$. It has to have at least $|W'|$ neighbors outside $N(W)$. Becase of the arbitrariness of $W'$, the inequality holds for the graph induced by $X-W$ and $Y-N_G(W)$, by induction there is a perfect matching in the graph. So combining this matchings we have a perfect matching on $G$.
\end{proof}

\section{Ordered sets}
In the language of graph theory, a \emph{(partially) ordered set}\index{partially ordered set}\index{ordered set} is a transitive, reflexive directed acyclic graph. 
Orders are ubiquitous in everyday life. Order theory is related with many other mathematical branches, such as universal algebra, topology and category theory. It also has broad application in many other fields, such as the humanities and social sciences, computer science (e.g. programs, binary strings, information orderings)~\cite{Davey2002}.
Several introductory books on this topic are available, such as the one by Davey and Priestley~\cite{Davey2002}.
In this section we introduce the definitions and elementary facts about orders.

\newpage

\subsection*{Basic definitions}

\begin{definition}
Let $P$ be a set. An \emph{order}\index{order} (or \emph{partial order}\index{partial order})
is a binary relation $\leq$ on $P$ such that, for all $a,b,c\in P$, the following properties are satisfied:
\begin{enumerate}
 \item \emph{reflexivity}\index{reflexivity}: $a\leq a$,
 \item \emph{antisymmetry}\index{antisymmetry}: $a\leq b$ and $b\leq a$ imply $a=b$,
 \item \emph{transitivity}\index{transitivity}: $a\leq b$ and $b\leq c$ imply $a\leq c$.
\end{enumerate}
A set $P$ equipped with an order relation $\leq_P$ is called an \emph{ordered set}\index{ordered set} (or \emph{partial order set}\index{partial order set}, \emph{poset}\index{poset}~\footnote{The name of poset is an abbreviation for partially ordered set which was coined by Garrett Birkhoff in his influential book \textit{Lattice Theory}~\cite{Birkhoff1995}.}), denoted by $(P,\leq_P)$\index{$(P,\leq_P)$}, if the relation $\leq_P$ on $P$ is reflexive, antisymmetric and transitive.
A relation $\leq_P$ on a set $P$ which is reflexive and transitive but not necessarily antisymmetric is called a \emph{quasi-order}\index{quasi-order} or \emph{pre-order}\index{pre-order}. 
\end{definition}

\begin{example}
Standard examples of orders arising in mathematics include:
\begin{enumerate}
\item The set of real numbers ordered by the standard less-than-or-equal relation.
\item The set of natural numbers equipped with the relation of divisibility.
\item The set of subsets of a given set (its power set) ordered by inclusion.
\item The vertex set of a directed acyclic graph ordered by reachability.
\end{enumerate}
\end{example}

\subsubsection*{Visualization of an order}
An order can be visually represented by a \emph{Hasse diagram}\index{Hasse diagram}, which is a graph drawing where the vertices are the elements of the order and the order relation is indicated by both the edges and the relative positioning of the vertices: if two elements $x\leq y$, then there exists a path from $x$ to $y$ and $y$ sits up of $x$.

\begin{figure}[!ht]
 \centering
 \includegraphics[height=0.3\textwidth]{Lattice60.eps}
\qquad
 \includegraphics[height=0.3\textwidth]{hasse2.eps}
\caption{Hasse digrams.}
\label{fig:Hasse}
\end{figure}

\begin{example}
Figure~\ref{fig:Hasse} shows two Hasse diagrams:
\begin{enumerate}
\item The Hasse diagram of the divisor order (right): the set of all divisors of 60 ordered by divisibility,
\item the Hasse diagram of the inclusion order (left): the set of the power set of $\{1,2,3\}$ equipped with subset relation $\subseteq$.
Note that even though $\{3\} < \{1,2,3\}$ (since $\{3\} \subset \{1,2,3\}$), there is no edge directly between them because there are inbetween elements:
$\{2,3\}$ and $\{1,3\}$. However, there still exists a path from $\{3\}$ to $\{1,2,3\}$. 
\end{enumerate}
\end{example}

\subsection*{Functions between orders}

\begin{defi}
Let $(P,\leq_P)$ and $(Q,\leq_Q)$ be two orders, and $f$ is a mapping from $(P,\leq_P)$ to $(Q,\leq_Q)$:
\begin{enumerate}
\item $f$ is an \emph{order preserving}\index{order preserving} or \emph{monotone}\index{monotone}, if $a\leq_P b$ implies $f(a)\leq_Q f(b)$ in $Q$~\footnote{Like homomorphism in graphs}. 
\item $f$ is an \emph{order reflecting}\index{order reflecting}, if $f(a)\leq_Q f(b)$ implies $a\leq_P b$~\footnote{In the language of graph theory, it is a full homomorphism.}. 
\item $f$ is an \emph{order embedding}\index{order embedding} $E$ if it is both order preserving and order reflecting, i.e.\ $a\leq_P b$ if and only if $f(a)\leq_Q f(b)$ for every $a,b\in P$.
\item An \emph{order isomorphism}\index{order isomorphism} is a surjective order embedding. 
\item An \emph{order automorphism}\index{order automorphism} is an order isomorphism from an ordered set to itself.
\end{enumerate}
\end{defi}

\begin{example}
We again give some examples to illustrate these concepts:
\begin{enumerate}
 \item The mapping that maps a natural number to its successor is clearly an order preserving with respect to the natural order; any mapping from a discrete order is an order preserving, for example, consider any mapping from a set ordered by the identity order `='.
 \item Mapping each natural number to the corresponding real number gives an example for an order embedding. 
 \item There is an order isomorphism between two orders with the same Hasse diagrams.
 \item The mapping that maps a natural number to its successor is an order automorphism with respect to the natural order.
\end{enumerate}
\end{example}

\begin{defi}
Let $(P,\leq_P)$ and $(Q,\leq_Q)$ be two orders, if there is an order embedding $E$ from $(P,\leq_P)$ to $(Q,\leq_Q)$,
we say $(P,\leq_P)$ is a \emph{suborder}\index{suborder} of $(Q,\leq_Q)$.
\end{defi}


Note that the suborder or order preserving relationship does not direct correspond to the relationship of their Hasse diagram. This is simply because the Hasse digraphs miss all the transitive edges. 

The visualization of orders with Hasse diagrams has a straightforward generalization: instead of displaying lesser elements below greater ones, the direction of the order can also be depicted by giving directions to the edges of a graph. In this way, each order is seen to be equivalent to a directed acyclic graph (DAG), where the vertices are the elements of the order and there is a directed path from $a$ to $b$ if and only if $a\leq b$. For example, Figure~\ref{fig:DAC} is a directed acyclic graph representation of the inclusion order in the second example.

\begin{figure}[!ht]
 \centering
 \includegraphics[height=0.3\textwidth]{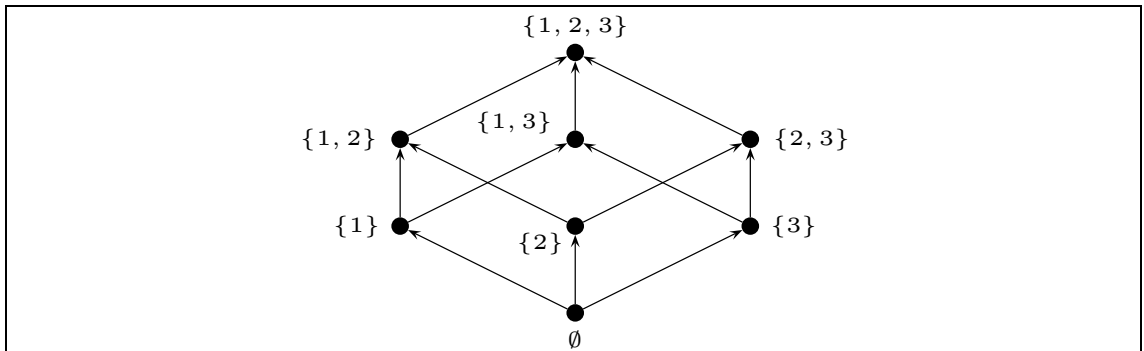}
\caption{Directed acyclic graph.}
\label{fig:DAC}
\end{figure}

\subsection*{Basic properties}
\subsubsection*{Chains and antichains}
\begin{defi}
An ordered set $(P,\leq_P)$ is a \emph{chain}\index{chain}, or \emph{totally ordered set}\index{totally ordered set}, if for all $x,y\in P$, either $x\leq_P y$ or $y\leq_P x$.
An ordered set $(P,\leq_P)$ is an \emph{antichain}\index{antichain} if $x\leq_P y$ in $(P,\leq_P)$ implies $x=_P y$.
\end{defi}

Clearly, any suborder of a chain (an antichain) is a chain (an antichain).
The Hasse diagraph of a chain is a path, and that of an antichain is a set of isolated vertices.

\begin{defi}
A \emph{increasing chain}\index{increasing chain} is a sequence of elements $(x_1,x_2,\dots)$ in the set $(P,\leq_P)$ such that $x_1\leq_P x_2\leq_P x_3\leq_P \dots$. Dually, a \emph{decreasing chain}\index{decreasing chain} is a sequence of elements $(x_1,x_2,\dots)$ in set $(P,\leq_P)$ such that $\dots \leq_P x_3\leq_P x_2\leq_P x_1$.
\end{defi}

As examples, the usual numerical order of positive integers is a (increasing) chain, and the inverse order ($\mathbb{N},\geq$) is a decreasing chain. The equality order of positive integers is an antichain.

\subsubsection*{Lattice}
Let $(P,\leq_P)$ be an ordered set and $(S,\leq_S)$ be a suborder of $(P,\leq_P)$. An element $x\in P$ is an \emph{upper bound}\index{upper bound} of $(S,\leq_S)$ if $s\leq_P x$ for all $s\in S$. If $x$ is the least element among all the upper bound, we call it \emph{least upper bound}\index{least upper bound} (or \emph{join}, \emph{supremum}).
An element $x\in P$ is a \emph{lower bound}\index{lower bound} of $(S,\leq_S)$ if $x\leq_P s$ for all $s\in S$. If $x$ is the greatest element among all the lower bound, we call it \emph{greatest lower bound}\index{greatest lower bound} (or \emph{meet}\index{meet}, \emph{infimum}\index{infimum}). Note that least upper bound and greatest lower bound are not necessary unique.

\begin{defi}
A \emph{lattice}\index{lattice} is an ordered set in which any two elements have a least upper bound and a greatest lower bound.
\end{defi}

\begin{example}
There are some examples of lattices:
\begin{enumerate}
\item The real numbers ordered by the standard less-than-or-equal relation form a lattice, under the operations of `min' and `max'.
\item For any set $A$, the collection of all subsets of $A$, ordered by inclusion, is a lattice.
\item Although the set of all divisors of 60 ordered by divisibility is a lattice, the subset $\{1,2,3\}$ so ordered is not a lattice because the pair 2,3 lacks a join, $\{2,3,6\}$ neither because 2,3 lacks a meet.
\end{enumerate}
\end{example}

As we already stated, graph-theoretically, an order is a \emph{transitive, acyclic, directed, reflexive} graph.
Note that an embedding is necessarily injective, since $f(a) = f(b)$ implies $a \leq_P b$ and $b \leq_P a$ (since an order is a reflexive graph). 

Interestingly, except for graph theoretical interpretation, orders can straightforwardly be viewed from many other perspectives. For example: model-theoretically, an order is a set equipped with a \emph{reflexive, antisymmetric, transitive} binary relation. An order embedding from $P$ to $Q$ is an isomorphism from $P$ to an elementary substructure of $Q$.
Category-theoretically, an order is a \emph{small, skeletal} category such that each homset has at most one element. An order embedding $P$ to $Q$ is a full and faithful functor from $P$ to $Q$ which is injective on objects, or equivalently an isomorphism from $P$ to a full subcategory of $Q$.


\subsubsection*{Density}

 
An order $(P,\leq_P)$ is \emph{dense}\index{dense}, if for any $x,y\in P$ with $x\leq_P y$, we can find another element $z\in P$ such that
$x\leq_P z\leq_P y$. Conversely, if for a pair $x,y$ with $x\leq_P y$ we cannot find another element $z$ in between, we say the pair $x,y$ is a \emph{gap}\index{gap}, denote it by \emph{$[x,y]$}\index{$[x,y]$}. 

For example, the real number ordered set $(\mathbb{R},\leq)$ is dense, however, the integer number ordered set $(\mathbb{Z},\leq)$ is not dense because there are gaps $[k,k+1]$ for any $k\in \mathbb{Z}$.

\subsubsection*{Local finiteness}
For a given order $(P,\leq_P)$ we call $\{y \mid y\leq_P x\}$ the {\em down-set}\index{down-set} of $x$ in $(P,\leq_P)$. When there is no confusion about particular order, we denote the down-sets by $\downarrow x$\index{$\downarrow x$}. Similarly we call the set $\{y \mid x\leq_P y\}$ the {\em up-set}\index{up-set} of $x$ in $(P,\leq_P)$ and denote it by $\uparrow x$\index{$\uparrow x$}.

We say that a countable order is {\em past-finite}\index{past-finite} if every down-set is finite. Similarly a countable order is {\em future-finite}\index{future-finite} if every up-set is finite. 
An order $(P,\leq_P)$ is said to be {\em locally-finite}\index{locally finite} if for every $a,b\in P$, $a\leq_P b$, the interval $[a,b]_P=\{c\mid a\leq_P c\leq_P b\}$ is finite. 
Both past-finite and future-finite orders are special case of {\em locally-finite}\index{locally finite}, that is, if for every $a,b\in P$, $a\leq_P b$, the interval $[a,b]_P=\{c\mid a\leq_P c\leq_P b\}$ is finite.

\subsubsection*{Duality}

We say that a pair $(a,b)$ is a {\em simple duality pair}\index{simple duality pair} in an order $(P,\leq_P)$, if $a,b\in P$, and any element $p\in P$ is either in the up-set of $a$, or in the down-set of $b$.

Furthermore, let $A$ and $B$ be two sets of $P$, we say that $(A,B)$ is a {\em generalized finite duality pair}\index{generalized finite duality pair} in $(P,\leq_P)$, if for any element $p\in P$ there either exists an $a\in A$ such that $a\leq_P p$ or there exists a $b\in B$ such that $p\leq b$.

In Chapter~\ref{Ch:order}, we will discuss more properties of orders, such as universality in Section~\ref{sec:universality}, and present some more interesting results regarding orders.



\chapter[Homomorphisms and Constrained Homomorphisms]{\fontsize{35}{10}\selectfont Homo\-morphisms and\\ Constrained Homomorphisms}
\label{ch:homo}
\ifpdf
    \graphicspath{{Chapter2/Chapter2Figs/PNG/}{Chapter2/Chapter2Figs/PDF/}{Chapter2/Chapter2Figs/}}
\else
    \graphicspath{{Chapter2/Chapter2Figs/EPS/}{Chapter2/Chapter2Figs/}}
\fi






This chapter concerns graph homomorphisms and several common variants of graph homomorphisms, and introduces a new variant of them that is key to our later work. The theory of graph homomorphisms, which underlies much of this thesis, developed from the well-known graph-theoretical problem of `graph colouring'. 
Meanwhile, many graph-theoretical problems can be generalized to notions of \emph{constrained homomorphisms}\index{constrained homomorphisms}.
This chapter is divided into four parts, in the first three of which we overview existing structural results. In Section~\ref{sec:homo} we first recall the basic properties of ordinary graph homomorphisms. In Section~\ref{sec:glo} we review several types of \emph{globally constrained homomorphisms}\index{globally constrained homomorphisms} such as monomorphism, embedding, surjective homomorphism, full homomorphism and compaction. In Section~\ref{sec:loc}
we discuss three kinds of \emph{locally constrained homomorphisms}\index{locally constrained homomorphisms}: locally injective, locally bijective and locally surjective. In the final section we introduce a new notion that we call `relation' which can be viewed as a generalization of surjective homomorphism. We develop the basic properties of `relation', as  preparation for the deeper investigation in Chapter~\ref{ch:relation}.

\section{Graph homomorphisms} \label{sec:homo}

Graph colouring is one of the most well-known problems in the field of graph theory. As an example, the four colour conjecture, which was postulated by Francis Guthrie while trying to colour a map of the counties of England, stating that four colours were sufficient to colour the map so that no regions sharing a common border receive the same colour. In terms of graph theory: all planar graphs can be coloured by 4 colours. As a generalization of graph colouring,  the study of graph homomorphisms dates from 1960s. It was pioneered by G.\ Sabidussi~\cite{Sabidussi1961}, and by Z.\ Hedrl\' in and A.\ Pultr~\cite{Hedrlin1964}, and rapidly developed by P.\ Hell and J.\ Ne\v{s}et\v{r}il~\cite{Hell2004}. In this section we give a brief summary of the main concepts and results in this area. For more details please refer to the recent monograph~\cite{Hell2004}.

\begin{defi}
Let $G=(V_G,E_G)$ and $H=(V_H,E_H)$ be two directed graphs,
a \emph{homomorphism}\index{homomorphism} of $G$ to $H$,
is a mapping $f:V_G\rightarrow V_H$ such that $(f(u),f(v))\in E_H$
whenever $(u,v)\in E_G$. 
\end{defi}

A homomorphism of $G$ to $H$ is also called an \emph{$H$-colouring of G}\index{$H$-colouring}.
If there exists a homomorphism $f:G\rightarrow H$ we write, $G\rightarrow H$\index{$\rightarrow$}, and $G\nrightarrow H$\index{$\nrightarrow$} means there is no such homomorphisms.

Let $f$ be a homomorphism of graph $G=(V_G,E_G)$ to graph $H=(V_H,E_H)$, and $S$ a subset of $V_G$, $M$ a subset of $V_H$,
we define $f(S):=\{v\in V_H\mid \exists x\in S \text{ s.t.\ } f(x)=v\}$ as the \emph{image of $S$}\index{image},
$f^{-1}(M):=\{x\in V_G\mid \exists u\in M \text{ s.t. } f(x)=u\}$ as the \emph{domain of $M$}\index{domain}.
Moreover, the \emph{image of homomorphism $f$}\index{image of a homomorphism}, denoted by $I(f)$\index{$I(f)$}, is defined as the image of $V_G$, and the \emph{domain of $f$}\index{domain of a homomorphism}, denoted by $D(f)$\index{$D(f)$}, is defined as the domain of $V_H$. 
Note that since $f$ is a mapping, the domain of $f$ is always $V_G$.


\begin{table}
\centering
\begin{tabular}{|c|c|c|}
    \hline
    Frucht graph & Graph $K_3$ & \begin{tabular}{c}An homomorphism from\\Frucht graph to $K_3$\end{tabular}
    \\[8pt]
    \hline
     \lower1.8cm\hbox{\includegraphics[width=4cm]{Frucht.eps}}
    &
     \lower1cm\hbox{\includegraphics[height=2.5cm]{K_3.eps}} 
    &
    \begin{tabular}{l}
     \\[-6pt]
     $  f(1) = \color{Red}  \text{Red}  $\\
     $  f(2) = \color{Blue} \text{Blue} $\\ 
     $  f(3) = \color{Red}  \text{Red}  $\\ 
     $  f(4) = \color{Green}\text{Green}$\\ 
     $  f(5) = \color{Blue} \text{Blue} $\\ 
     $  f(6) = \color{Green}\text{Green}$\\ 
     $  f(7) = \color{Blue} \text{Blue} $\\ 
     $  f(8) = \color{Red}  \text{Red}  $\\ 
     $  f(9) = \color{Blue} \text{Blue} $\\
     $ f(10) = \color{Blue} \text{Blue} $\\
     $ f(11) = \color{Red}  \text{Red}  $\\
     $ f(12) = \color{Green}\text{Green}$\\[6pt]  
     \end{tabular}
    \\
    \hline
\end{tabular}
\caption{Graph homomorphism.}
   \label{tab:FtoK3} 
   \end{table}


\begin{example}
In order to build an intuition about the existence of graph homomorphisms, we give some examples :
\begin{enumerate}
 \item If graph $G$ is a subgraph of graph $H$, then the subgraph structure induces a homomorphism from $G$ to $H$
(Figure~\ref{fig:subg}). However, there may have other homomorphisms of $G$ to $H$ which are not induced by subgraph structure. 
 \item There is a homomorphism of Frucht graph to complete graph $K_3$ as in Fig.~\ref{tab:FtoK3}. Because $K_3$ is a subgraph of Frucht graph, there is a homomorphism of $K_3$ to Frucht graph.
 \item $P_2\to K_3$, since $P_2$ is a subgraph of $H_3$. However $K_3\nrightarrow P_2$, that is because the size of $K_3$ is larger than $P_2$, then there are at least two vertices of $K_3$ which must be mapped to the same vertex of $P_2$, this would induce a loop. However $P_2$ is loop-free.
\end{enumerate}
\end{example}

It is easy to verify that the composition of homomorphisms is also a homomorphism. Moreover, composition of homomorphisms is associative.
Let $f:V_G\rightarrow V_H$ be a homomorphism of $G$ to $H$. If the mapping $f$ is injective, we call it an \emph{injective homomorphism}\index{injective homomorphism}. The homomorphism $f$ naturally induces an edge mapping of $E_G$ to $E_H$. If $f$ and the induced edge mapping are both surjective, we call it is a \emph{surjective homomorphism}\index{surjective homomorphism}.

\begin{pro}[\textbf{Decomposition law}~\cite{Hell2004}] 
Let $G$ and $H$ be any directed graphs. Every homomorphism $f:G\rightarrow H$ can be written as $f=i\circ s$ where $s$ is a surjective homomorphism and $i$ is an injective homomorphism.
\end{pro}

We leave the proof to the section on surjective homomorphisms that begins on page~\pageref{decomp}. 

\begin{pro} \label{pro:colouring}
A homomorphism $f:G\rightarrow K_k$ is precisely a $k$-colouring of the graph $G$.
\end{pro}

\begin{proof}
For $1\leq i\leq k$, we colour the vertices of graph $G$ in the set $f^{-1}(i)$ by colour $i$. Clearly for distinct $i$ and $j$, $f^{-1}(i)\cap f^{-1}(j)=\emptyset$. And for any $i$, $f^{-1}(i)$ is an independent set, otherwise there is a loop on vertex $i$ of $K_k$. Hence homomorphism $f$ is a proper $k$-colouring of graph $G$.
\end{proof}

Proposition~\ref{pro:colouring} explains the reason that a homomorphism from $G$ to $H$ is also called an $H$-colouring, as a graph homomorphism can be seen as a generalization of a graph colouring.

\subsection*{Homomorphisms preserve adjacency}
The fact that homomorphisms preserve vertex adjacency has interesting implications. The following propositions are derived from these adjacency preservations.

\begin{pro}
A mapping $f: V_{P_k}\rightarrow V_G$ is a homomorphism of $P_k$ to graph $G$ if and only if the sequence $f(1),f(2),\dots,f(k)$ is
a walk in graph $G$.
\end{pro}

Consequently, we obtain the following useful corollary.

\begin{cor}
If $f:G\rightarrow H$ is a homomorphism, then for any two vertices $u,v$ of graph $G$, the distance between $f(u)$ and $f(v)$ in graph $H$ is at most the distance between $u$ and $v$ in $G$.
\end{cor}

\begin{rem}
Even though the existence of homomorphisms is monotone with respect to distance, it is not monotone with respect to radius or diameter. For example, $P_2\to K_3$ and $P_2\to P_5$, but the diameter as well as the radius of $P_2$ is between $K_3$ and $P_5$. In contrast, we will see an analogous structure which is monotone with respect to diameter as well in Corollary~\ref{rad} on page~\pageref{rad}. 
\end{rem}

Similar to paths, we have the following result for homomorphisms from cycle graphs.

\begin{pro}
A mapping $f: V_{C_k}\rightarrow V_G$ is a homomorphism of $C_k$ to $G$ if and only if the sequence $f(0),f(1),\dots,f(k-1)$ is a closed walk in $G$.
\end{pro}

\begin{cor}
$C_{2k+1}\rightarrow C_{2l+1}$ if and only if $l\leq k$.
\end{cor}

\subsection*{Homomorphism equivalence}
\label{sec:hom-equ}
In the last section, we saw that the existence of homomorphisms is monotone with respect to chromatic number, so that if two directed graphs $G$ and $H$ have different chromatic numbers, then $G\to H$ and $H\to G$ cannot hold at the same time. Two directed graphs $G$ and $H$ such that each is homomorphic to the other are called \emph{homomorphically equivalent}\index{homomorphically equivalent}, and denoted by $G\sim H$\index{$G\sim H$}. Homomorphic equivalence is indeed an equivalence relation. To see this, let $G$, $H$ and $K$ be directed graphs. Since every directed graph is homomorphically equivalent to itself (self-adjoint); $G\sim H$ implies $H\sim G$ (symmetric); If $G\sim H$ and $H\sim K$, then $G\sim K$ follows from the fact that homomorphisms are closed under composition (transitive). We immediately get that homomorphically equivalent directed graphs have the same chromatic number. In this section we focus on the question of how to characterize homomorphically equivalent directed graphs.

\begin{defi}
A \emph{retraction}\index{retraction} of a directed graph $G$ to its subgraph $H$ is a homomorphism
$r:G\rightarrow H$ such that $r(x)=x$ for all $x\in V_H$.
\end{defi}

If there is a retraction of $G$ to $H$, we call $H$ is a \emph{retracted graph}\index{retracted graph} of $G$.
Recall that the a subgraph has a homomorphism to its original graph, so a graph and (any of) its retracted graphs are homomorphism equivalent. Figure~\ref{fig:Retr} shows a series of retractions.

\begin{figure}[!ht]
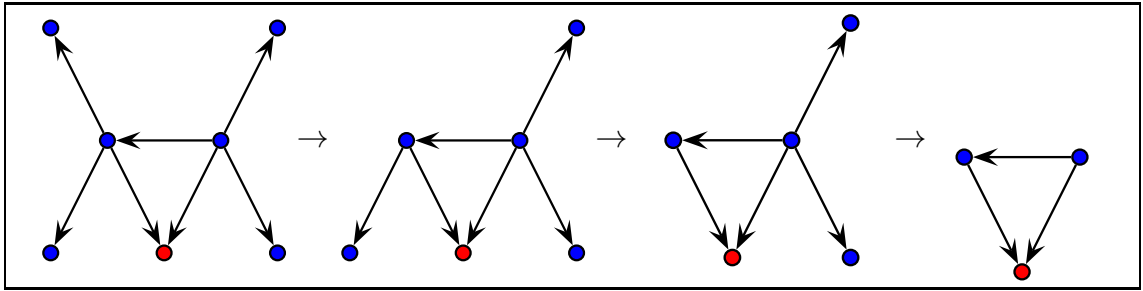

\centering
\begin{minipage}[c]{0.22\textwidth}
  \includegraphics[width=1\textwidth]{dicore1.eps}
\end{minipage}
$\rightarrow$
\begin{minipage}[c]{0.22\textwidth}
 \centering
 \includegraphics[width=1\textwidth]{dicore2.eps}
\end{minipage}
$\rightarrow$
\begin{minipage}[c]{0.22\textwidth}
 \centering
\includegraphics[width=0.8\textwidth]{dicore3.eps}
\end{minipage}
$\rightarrow$
\begin{minipage}[t][][b]{0.15\textwidth}
 \centering
\includegraphics[width=0.8\textwidth]{dicore4.eps}
\end{minipage}
\caption{A series of retractions.}
\label{fig:Retr}
\end{figure}

The most natural idea is to choose the smallest directed graph (in the sense of size) in an equivalence class as its representative. 
Firstly, given a graph $G$, we can try to find the smallest subgraph which $G$ has a retraction to. This leads to an important concept which was independently discovered and was given different names such as \emph{minimal graphs}\index{minimal graph} and \emph{unretractive graphs}\index{unretractive graph}~\cite{Fellner1982,Hell1992,Nowakowski1979}.

\begin{defi}
A \emph{core of (directed) graph $G$}\index{core of a graph} is the smallest graph (in the sense of size) which $G$ admits a retraction to.
A (directed) graph is a \emph{core}\index{core} if it does not retract to a proper subgraph.
\end{defi}

It is easy to see that every directed graph is homomorphically equivalent to a (up to isomorphism) unique core.
Thus we can denote the core of a given graph $G$ by $G_\text{core}$\index{$G_\text{core}$}.
For example, the core of first three directed graphs in Figure~\ref{fig:Retr} is the fourth directed graph.

Two graphs are homomorphically equivalent if and only of their cores are isomorphic, therefore we can use core to be the
representative for an equivalence class. However, the computational complexity of determining a core is NP-hard. As we know, the 3-colouring problem, i.e., testing whether a graph has a homomorphism to triangle $K_3$, is NP-complete. Given a graph $G$, consider $G'$ to be the disjoint union of $G$ and $K_3$. If we can determine the core of $G'$, then it is easy to decide whether $G$ is 3-colourable or not. If the core of $G'$ is $K_3$, then $G$ is 3-colourable, otherwise $G$ is not 3-colourable.

\subsection*{H-colouring problem}
In this section we introduce the main results about the existence of homomorphisms from an algorithmic perspective.

Let $H$ be a fixed directed graph. The \emph{homomorphism problem for $H$} asks whether or not an input
directed graph $G$ admits a homomorphism to $H$. Recall that a homomorphism of $G$ to $H$ is also called an
$H$-colouring of $G$. Thus the homomorphism problem for $H$ is also called the \emph{$H$-colouring problem}\index{$H$-colouring problem},
and denoted by $\COL$\index{$\COL$}.

The $H$-colouring problem can be analogously stated for any general relational system $H$, in this case the problem is also known as \emph{constraint satisfaction problem}\index{constraint satisfaction problem} (or \emph{CSP}\index{CSP}) with a template $H$, and denoted as $\CSP$\index{$\CSP$}~\cite{Hell2004}. Many studies have been carried out regarding the constraint satisfaction problem, and there is an introductory book by E.~Tsang~\cite{Tsang1993}.

We first consider the $H$-colouring problem for graphs (instead of digraphs). Clearly all graphs which have a loop have core $O$, the graph with one vertex and a loop.
The problem of $O$-colourablity is trivial.

Recall that the usual $k$-colouring problem for simple graphs is polynomial time solvable when $k\leq 2$ and is NP-complete for $k>2$.
Such a classification result implies that we have a \emph{dichotomy}\index{dichotomy} of possibilities ---
each $k$-colouring problem is polynomial time solvable or NP-complete. Analogously, $H$-clouring has also a dichotomy of possibilities.

\begin{thm}
 Let $H$ be a graph. If $H$ is bipartite or contains a loop, then the $H$-colouring problem has a polynomial time algorithm.
 Otherwise the $H$-colouring problem is NP-complete.
\end{thm}

This theorem was first proved in~\cite{Hell1990} by Hell and Ne\v{s}et\v{r}il.
There is a new proof by Bulatov~\cite{Bulatov} which follows the same idea but uses an algebraic approach.

The $H$-colouring problem for directed graphs has received much attention, and yet no graph-theoretic classification has been obtained, or conjectured~\cite{Hell2004}. Even the following dichotomy conjecture is still open.

\begin{con}
 Let $H$ be a directed graph, the $H$-colouring problem is polynomial time solvable or NP-complete.
\end{con}

There is a nice conjectured dichotomy classification for directed graphs with positive indegree and outdegree at each vertex.
The conjecture is still open though many special cases of the conjecture have been verified~\cite{Hahn2002}.

\begin{con}[Bang-Jensen, Hell~\cite{Bang1990}]
Suppose $H$ is a connected directed graph in which each vertex has positive indegree and outdegree.
If the core of $H$ is a directed cycle, then the $H$-colouring problem is polynomial.
Otherwise it is NP-complete.
\end{con}

\section[Globally constrained homomorphisms]{\fontsize{20}{10}\selectfont Globally constrained homomorphisms}\label{sec:glo}

We consider several types of (globally) constrained homomorphisms.

\begin{enumerate}
\item A homomorphism $f:G\to H$ is a {\em monomorphism}\index{monomorphism}\footnote{The name `monomorphism' is from category theory.}, (also called \emph{injective homomorphism}\index{injective homomorphism}~\cite{Hell2004} or {\em \MonoHomo}\index{\MonoHomo}) if it is injective.
\item A homomorphism $f:G\to H$ is a {\em full homomorphism}\index{full homomorphism}, or {\em \FullHomo}\index{\FullHomo}, if $(f(u),f(v))\in E_H$ implies $(u,v)\in E_G$.
\item A homomorphism $f:G\to H$ is an {\em embedding}\index{embedding}, or {\em \EmbedHomo}\index{\EmbedHomo} if it is both a monomorphism and a full homomorphism.
\item A homomorphism $f:G\to H$ is a {\em vertex surjective homomorphism}\index{vertex surjective homomorphism}, or {\em \VSurHomo}\index{\VSurHomo}, if $f(V_G)=V_H$. A homomorphism $f:G\to H$ is a {\em surjective homomorphism}\index{surjective homomorphism}, or {\em \SurHomo}\index{\SurHomo}, if it is vertex surjective and moreover for every $(u,v)\in E_H$ there is $(u',v')\in E_G$ such that $f(u')=u$ and $f(v')=v$.
\end{enumerate}


\begin{figure}[!ht]
 \centering
 \includegraphics[width=0.55\textwidth]{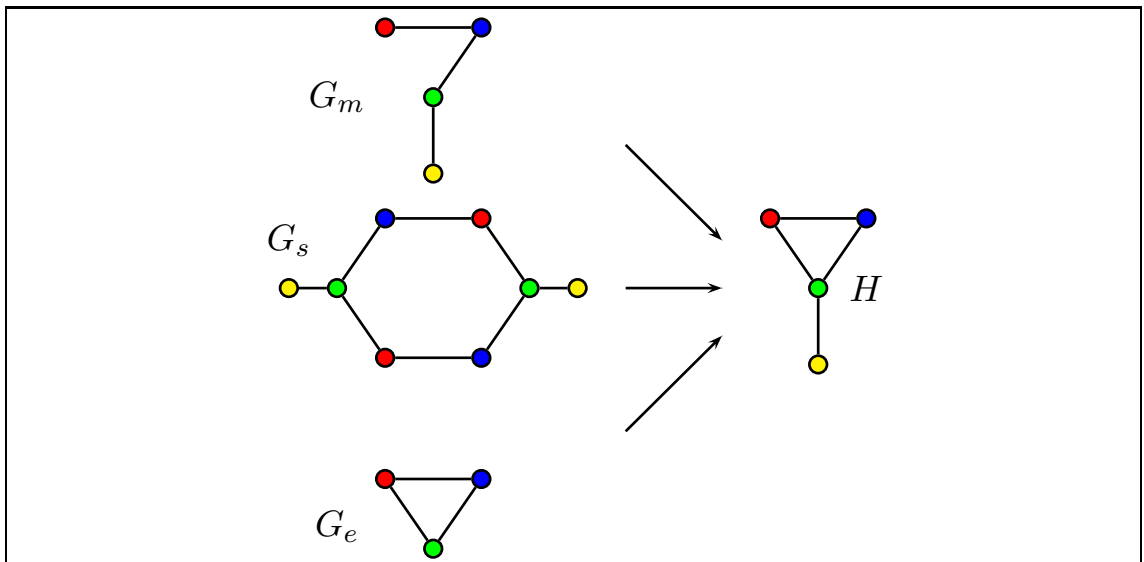}
 \caption{Monomorphism, surjective homomorphism, embedding.}
 \label{fig:glo}
\end{figure}

Figure~\ref{fig:glo} shows a monomorphism ($G_m$ to $H$), a surjective homomorphism ($G_s$ to $H$) and an embedding ($G_e$ to $H$).

Let $G$ and $H$ be directed graphs, there is a monomorphism from
$G$ to $H$ if and only if $G$ is a subgraph of $H$, there is an embedding from $G$ to $H$ if and only if $G$
is an induced subgraph of $H$. Examples can be found in Section~\ref{subsec:sub} on page~\pageref{subsec:sub}.


\subsection*{Surjective homomorphisms} \label{sec:mono}


\begin{example} 
Here are some examples of vertex or edge surjective homomorphisms.
\begin{enumerate}
\item The mapping from Frucht graph to $K_3$ shown in Figure~\ref{tab:FtoK3} is a surjective homomorphism;
\item The homomorphism from $P_3$ (Figure~\ref{fig:iplus})to $K_3$ by mapping $f(x_i)=v_i$ for $i=1,2,3$ is vertex surjective but not edge surjective (because the preimage of edge $(v_2,v_3)$ is empty);
\item The homomorphism from to $P_3$ to $G$ shown in Figure~\ref{fig:iplus} by mapping $f(x_1)=v_2$, $f(x_2)=v_3$ and $f(x_3)=v_2$ is edge surjective bot not vertex surjective (because the preimage of $v_1$ is empty).
\end{enumerate}
\end{example}

\begin{figure}[!ht]
 \centering
 \includegraphics[width=0.65\textwidth]{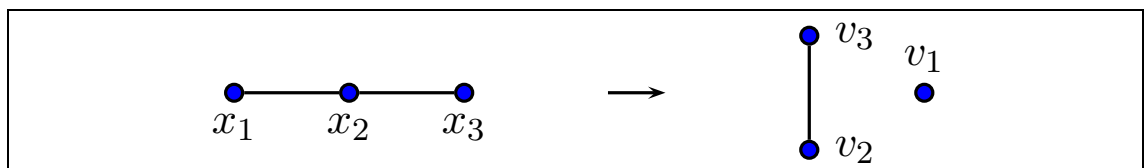}
 \caption{Example of an edge surjective but not vertex surjective homomorphism.}
 \label{fig:iplus}
\end{figure}

Note that a homomorphism that is vertex injective is also edge injective (but not conversely), and, as long as the resulting graph has no isolated vertices, a homomorphism that is edge surjective is also vertex surjective (but not conversely).
In other words, injective homomorphisms are the same as vertex injective homomorphisms, while surjective homomorphisms are, in the absence of isolated vertices, the same as edge surjective homomorphisms~\cite{Hell2004}.

\subsubsection*{Quotient graphs and homomorphic images}

Recall that given an arbitrary partition $P$ of $V_G$, with nonempty parts $V_i$, $i\in I$, we define the quotient graph \index{quotient graph} of $G$ induced by partition $P$ to be the graph $G/P$ with vertex set $\{V_i\mid i\in I\}$, in which there is a loop $(V_i,V_i)\in E_{G/P}$ if $V_i$ contains some $u,v$ with $(u,v)\in E_G$, and in which $(V_i,V_j)\in E_{G/P}$ if some $(u,v)\in E_G$ has $u\in V_i$, $v\in V_j$. Then the canonical mapping that assigns to each $u\in V_G$ the unique $i$ such that $u\in V_i$ is a homomorphism $f:G\rightarrow G/P$, it is actually surjective.

For example, if we take the partition $P=\{V_1,V_2,V_3\}$ by different colours of Frucht graph in Figure~\ref{fig:quo}, i.e.\ $V_1=\{1,3,8,11\}$ (red vertices), $V_2=\{2,5,7,9,10\}$ (blue vertices), $V_3=\{4,6,12\}$ (green vertices), then the quotient graph of Frucht graph induced by $P$ is $K_3$.

Similarly to the quotient graph, there is another concept `homomorphic image' which is also closely connected to surjective homomorphisms. Let $G$ and $H$ be directed graphs, and $f$ a homomorphism of directed graph $G$ to directed graph $H$.
We define the \emph{homomorphic image of $G$ under $f$}\index{homomorphic image}, denoted by  $f(G)$\index{$f(G)$},
as the directed graph with vertices $\{f(v)\mid v\in V_G\}$ and for which there is an edge $(f(v),f(w))$ if and only if $v,w\in E_G$.

Note that $f(G)$ is a subgraph of $H$, and that $f:G\rightarrow f(G)$ is a surjective homomorphism.
Conversely, if $f:G\rightarrow H$ is a surjective homomorphism, then $H=f(G)$.

There is a one-to-one correspondence in between quotient graphs and homomorphic images.

\begin{pro}
 Every quotient of $G$ is a homomorphic image of $G$, and conversely, every homomorphic image of $G$ is isomorphic to a
 quotient of $G$.
\end{pro}

\begin{proof}
 Suppose $G/P$ is a quotient graph of $G$ with partition $P=\{V_1,V_2,\dots,V_k\}$, then we assign the canonical mapping that assigns to each vertex $u \in V_G$ the unique $i\in \{1,2,\dots,k\}$ such that $u\in V_i$, $G/P=f(G)$.

For a homomorphic image $f(G)$, we assign an equivalence class by $u\sim v$ if $f(u)=f(v)$ for any $u,v\in V_G$. It is easily seen that
for the partition $P$ induced by $\sim$, $G/P$ is isomorphic to $f(G)$.
\end{proof}

Now we come to the proof of the decomposition law.

\begin{pro}[Hell, Ne\v set\v ril~\cite{Hell2004}] \label{decomp}
Let $G$ and $H$ be any directed graphs. Every homomorphism $f:G\rightarrow H$ can be written as $f=i\circ s$ where $s$ is a surjective homomorphism and $i$ is an injective homomorphism.
\end{pro}

\begin{proof}
 Recall that the image graph $f(G)$ is a graph with vertex set $f(V_G)$ and $u,v\in f(V_G)$ is an edge in $f(G)$ if and only if there exist $x,y\in V_G$ and $f(x)=u$, $f(y)=v$. $f(G)$ is a subgraph of $H$, thus there is a natural injective homomorphism $i$ from $f(G)$ to $H$. On the other hand, $f$ induces a quotient of $G$ with respect to partition $\theta_f$. Thus we can take $s$ to be canonical surjective homomorphism of $G$ onto $G/\theta_f = f(G)$, and $i$ the inclusion homomorphism of $f(G)$ in $H$.
\end{proof}

\subsection*{Full homomorphisms} \label{sec:full}
Consider the surjective homomorphism from the Frucht graph to $K_3$ shown in Table~\ref{tab:FtoK3}. There is no edge for example between vertices 1 and 7 in the Frucht graph, but there is an edge between their images Red and Blue in $K_3$. This is because surjective homomorphisms do not necessarily preserve non-adjacency. In this part we discuss full homomorphisms which preserve both edges and non-edges.

\begin{figure}[!ht]
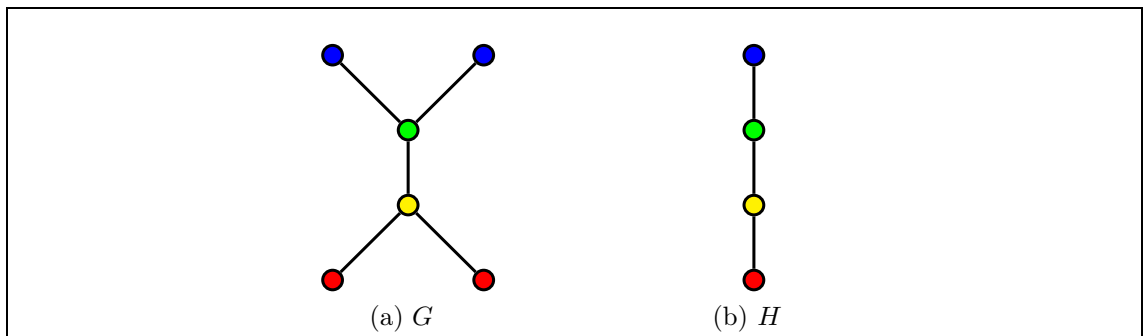

\centering
\subfloat[$G$]{
\begin{minipage}[c]{0.3\textwidth}
\centering
\includegraphics[height=0.75\textwidth]{ex3.1.2.eps}
\end{minipage}}
\subfloat[$H$]{
\begin{minipage}[c]{0.3\textwidth}
\centering
\includegraphics[height=0.75\textwidth]{ex3.1.3.eps}
\end{minipage}}
     \caption{A full homomorphism from $G$ to $H$.}
    \label{fig:fige4}
\end{figure}


\begin{example}
 There is a full homomorphism from $G$ to $H$ as in Figure~\ref{fig:fige4}, by mapping the vertices in graph $G$ to the vertex in the corresponding colour in graph $H$.
\end{example}

To our knowledge, the concept of full homomorphisms first appeared in 1966 in~\cite{Hedetniemi1966}. It has a close connection to other variants of graph homomorphisms, for example,
a full homomorphism is a surjective homomorphism in the sense of absence of isolated vertices in $H$. A homomorphism is an embedding if it is both a monomorphism and a full homomorphism.

\begin{obser}[\cite{Hell2004}]
There are some easy facts regarding to full homomorphisms:
\begin{enumerate}
 \item If $G$ is an induced subgraph of $H$ then the inclusion mapping $i:G\rightarrow H$ is a full injective homomorphism.
 \item Full injective homomorphisms are embeddings.
 \item There exists a full homomorphism $f:G\rightarrow H$ if and only if the vertices of $G$ can be partitioned into
independent sets $S_x$, $x\in V_H$, such that if $xy\notin E_H$ then no edge of $G$ joins the set $S_x$ to the set $S_u$,
and if $xy\in E_H$ then all vertices of $S_x$ are joined in $G$ to all vertices of $S_y$. 
\end{enumerate}
\end{obser}

Recall that graph homomorphism is a generalization of colouring, while full homomorphism is closely connected to another colour-related concept, `the achromatic number'~\cite{Hell1976}.

\subsubsection*{Point-determining graphs}
Full homomorphisms are also closely related to the notion of a point-determining graph, which was introduced by D.~Sumner~\cite{Sumner1973} and others in the 1970s. 

\begin{definition}
{\em Point-determining graphs}\index{point-determining graph} (also known as {\em mating-type graphs}\index{mating-type graphs}, {\em mating graphs}\index{mating graphs}, {\em M-graph}\index{M-graph} or {\em thin graphs}\index{thin graphs}~\cite{Hammack2011}) are graphs in which no two vertices have the same open neighbourhood. 
\end{definition}

\begin{example}
 In Fig.~\ref{fig:fige4}, the left graph is not a point-determining graph while the right graph is a point-determining graph. Moreover, there is a full homomorphism from the left graph to the right graph. 
\end{example}

\begin{definition}
  The \emph{thinness relation}\index{thinness relation} $S$ of $G$ is the equivalence relation on
  $V_G$ defined by $(x,y)\in S$ if and only if $N_G(x)=N_G(y)$.
  We denote by $\mathcal{S}$ the corresponding partition of $V_G$, and write
  $R_S\subseteq V_G\times\mathcal{S}$ for the relation that associates each vertex with its
  $S$-equivalence class, i.e., $(x,\beta)\in R_S$ if and only if $x\in\beta$.

  The \emph{point-determining graph of $G$}\index{point-determining graph of $G$}, (also called \emph{thin graph of $G$}\index{thin graph of a graph}), denoted by $G_{\mathrm{pd}}$\index{$G_{\mathrm{pd}}$}, is a graph with vertex
  $\mathcal{S}$, two equivalence classes $\sigma$ and
  $\tau$ of $S$ are adjacent in $G_{\mathrm{pd}}$ if and only if $(x,y)$ is an edge of
  $G$ with $x\in\sigma$ and $y\in\tau$.
\end{definition}

Note that $R_S$ is a full homomorphism of $G$ to $G_{\mathrm{pd}}$
\cite{Feder2008}. Also $G_{\mathrm{pd}}$ is indeed a point-determining graph~\cite{Hammack2011}.
We observe that $G$ is obtained from $G_{\mathrm{pd}}$
by vertex duplication. Thinness and the quotients w.r.t.\ the thinness
relation play a role in particular in the context of product graphs, see
\cite{Imrich2000}. In this context it is well known that $G$ can
be reconstructed from $G_{\mathrm{pd}}$ and knowledge of the $S$-equivalence
classes.

If there is a full homomorphism from graph $G$ to graph $H$, then there is a point-determining graph $K$ such that both $G$ and $H$
have a full homomorphism to $K$.
The correspondence between full homomorphisms and point-determining graphs is explored in greater detail in~\cite{Feder2008}.
In the following we summarize the results about point-determining graphs contained in~\cite{Sumner1973}.

\begin{lemma}[Sumner~\cite{Sumner1973}] \label{lem:pd}
Let $G$ be point-determining. If $N_G(a)=N_G(b)-c$ with $a\neq b$ and $N_G(d)=N_G(e)-a$ with $d\neq e$, then $d=c$.
\end{lemma}

\begin{proof}
Suppose that $d\neq c$. Since $N_G(d)=N_G(e)-a$ we have $e$ and $a$ are adjacent, because $N_G(a)=N_G(b)-c$, then $e\in N_G(a)\subset N_G(b)$. Hence $b\in N_G(e)-\{a\} = N_G(d)$. Since $d\neq c$, $d\in N_G(b)-\{c\}=N(a)$. So $a\in N_G(d)=N_G(e)-\{a\}$, a contradiction.
\end{proof}

\begin{pro} [Sumner~\cite{Sumner1973}]
If $G$ is nontrivial and point-determining, then there exists a vertex $x\in G$ such that $G-\{x\}$ is point-determining.
\end{pro}

\begin{proof}
Suppose no such point exists. Let $x\in G$ be the vertex with largest degree. Then since graph $G-\{x\}$ is not point-determining,
there exist vertices $p,q\in G-\{x\}$ with $p\neq q$ and $N_G(p)=N_G(q)-\{x\}$. Similarly there exist vertices $r,s\in G-\{p\}$ with $r\neq s$
such that $N_G(r)=N_G(q)-\{x\}$, by Lemma~\ref{lem:pd}, $r=x$, so $N_G(x)=N_G(s)-\{p\}$ and thus $d(x)=d(s)-1$, so that $d(s)>d(x)$, which is a contradiction.
\end{proof}

\section{Locally constrained homomorphisms}\label{sec:loc}

In this section we introduce three kinds of locally constrained homomorphisms: locally injective, surjective and bijective homomorphism. Locally constrained homomorphisms play a crucial rule in many application areas. A comprehensive survey~\cite{Fiala2008} is available summarizing the most important results about locally constrained homomorphisms in structural, algorithmic and applicational aspects.

\begin{definition}
For undirected graphs,
\begin{enumerate}
\item A homomorphism $f:G\to H$ is a {\em locally injective homomorphism}\index{locally injective homomorphism}, or {\em \LocInHomo}\index{\LocInHomo}, if for all $v\in V_G$, the restriction of the mapping $f$ to the domain $N_G(v)$ and range $N_H(f(v))$ is injective.
\item A homomorphism $f:G\to H$ is a {\em locally surjective homomorphism}\index{locally surjective homomorphism}, or {\em \LocSurHomo}\index{\LocSurHomo}, if for all $v\in V_G$, the restriction of the mapping $f$ to the domain $N_G(v)$ and range $N_H(f(v))$ is surjective.
\item A homomorphism $f:G\to H$ is a {\em locally bijective homomorphism}\index{locally bijective homomorphism}, or {\em \LocBiHomo}\index{\LocBiHomo}, if it is both a locally surjective and locally bijective homomorphism. That is if the restrictions to the neighbourhoods are the bijections.
\end{enumerate}
\end{definition}

There is an alternative and quite natural definition given in~\cite{Fiala2005}:

\begin{enumerate}
\item A homomorphism $f:G\to H$ is a {\em locally injective homomorphism}\index{locally injective homomorphism}, if for every edge $(u,v)$ of $H$, the subgraph
of $G$ induced by $f^{-1}(u)\cup f^{-1}(v)$ is a matching.
\item A homomorphism $f:G\to H$ is a {\em locally surjective homomorphism}\index{locally surjective homomorphism}, if for every edge $(u,v)$ of $H$, the subgraph
of $G$ induced by $f^{-1}(u)\cup f^{-1}(v)$ is a bipartite graph without isolated vertices.
\item A homomorphism $f:G\to H$ is a {\em locally bijective homomorphism}\index{locally bijective homomorphism}, if for every edge $(u,v)$ of $H$, the subgraph 
of $G$ induced by $f^{-1}(u)\cup f^{-1}(v)$ is a perfect matching.
\end{enumerate}

In accordance with the notation used for ordinary graph homomorphisms we write $G\LocInHom H$, $G\LocSurHom H$, $G\LocBiHom H$ for the existence of locally injective, surjective and bijective homomorphisms.

Compared to globally constrained homomorphisms, locally constrained homomorphisms have their own interesting properties and applications.
Figure~\ref{fig:loc} shows examples of locally injective ($G_i$ to $H$), locally bijective ($G_b$ to $H$), locally surjective homomorphism ($G_s$ to $H$).
It is easily seen that a locally injective (locally bijective) homomorphism is not necessarily an injective (bijective) homomorphism, but a locally surjective homomorphism between two connected graphs is a surjective homomorphism.

\begin{figure}[!ht]
\centering
\includegraphics[width=0.55\textwidth]{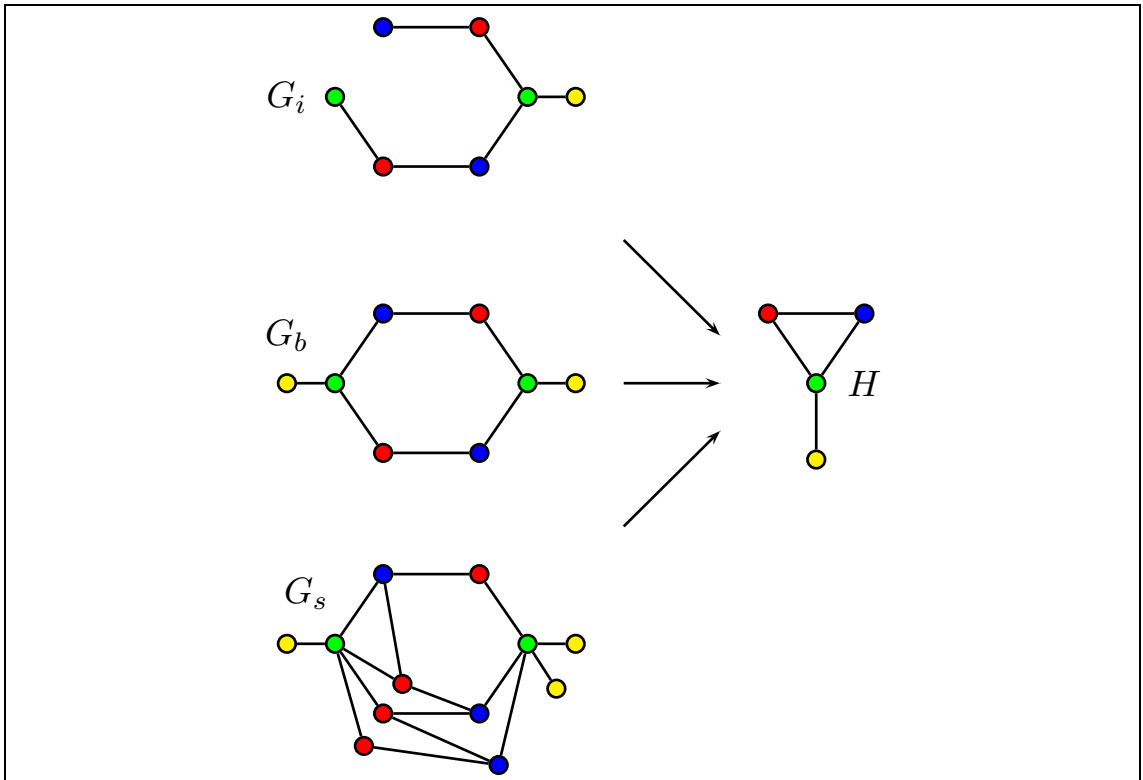} 
\caption{Locally injective, bijective, surjective homomorphisms~\cite{Fiala2008}.}
\label{fig:loc}
\end{figure}

Locally surjective homomorphisms were firstly introduced by Everett and Borgatti, who called them \emph{role colourings}\index{role colouring}~\cite{Everett1991}. They originated in the theory of social behaviour. The source graph $G$ represents relations between particular individuals of some group. The target graph $H$ is called the role graph. The task is to assign roles to individuals so that each person given a particular role has, among its neighbours, every role prescribed by the role graph at least once, while no other roles may appear in the neighbourhood~\cite{Fiala2008}, 
where individuals of the same social role relate to other individuals in the same way~\cite{Fiala2008}.
In this thesis, the other two kinds of locally constrained homomorphisms are of more interest to us, and so we accordingly introduce them in greater detail.

\subsubsection*{Locally injective homomorphisms} \label{sec:li}
The notion of locally injective homomorphism, sometimes called \emph{local monomorphism}\index{local monomorphism}~\cite{Nesetril1971}
or \emph{partial covering projection}\index{partial covering projection}~\cite{Fiala2001}, is a notion with an interesting history and wide application. It has a very close connection to the frequency assignment problem, which is a highly practical problem of interference-free frequency assignment for wireless networks~\cite{Fiala2008}.
Locally injective homomorphisms (also called local epimorphisms or partial covers) are also used in distance-constrained labelings of graphs~\cite{Fiala2001}
and as indicators of the existence of homomorphisms of derivative graphs (line graphs)~\cite{Nesetril1971}.

Observe that if there is a locally injective homomorphism from graph $G$ to graph $H$, then for any vertex $v\in V_G$,
$|N(v)|\leq |N(f(v))|$. 
Consequently, a locally injective homomorphism from a graph to itself preserves degrees, i.e.,
if $f$ is a locally injective homomorphism from a graph $G$ to itself, then for any vertex $v\in V_G$,
$d(v)=d(f(v))$. There is a classical result about locally injective homomorphisms from a graph to itself which was proved in 1971:

\begin{thm}[Ne\v{s}et\v{r}il~\cite{Nesetril1971}] \label{thm:locin-auto}
 Let $G$ be a connected finite graph. Then every locally injective homomorphism of $G$ into itself
 is an automorphism.
\end{thm}

\begin{proof}
 We prove the statement by induction on $|G|$.
 It is easily seen that for $|G|=2$ the statement of the theorem holds. Let $|G|>2$, and suppose that the statement
 of the theorem holds for all connected graphs $H$ with $|H|<|G|$. We assume there is a locally
 injective homomorphism $f$ from $G$ to itself which is not an automorphism, then $f(V_G) \subset V_G$ and the homomorphic image $f(G)$ is a connected graph.
 Take a vertex $v\in V_G \setminus f(V_G)$ and consider the graph $G-v$. Let $G^1,\dots,G^n$ be the disjoint connected components of $G-v$ and, because $G$ is a connected graph, for each $G^i$, there exists a vertex $v^i$ which is connected to $v$. Suppose the image of $v$ is in $G^j$ where $1\leq j\leq n$, then from the properties of locally injective homomorphism, for all $i$, $f(v^i)\in G^j$. Because of the connectivity, $f(G^i)\subset G^j$ for any $i$. It is easily seen that the restriction of $f$ to $G^j$
 is also a locally injective homomorphism, from the induction hypothesis it is an automorphism.
 Consider any vertex $w$ in the set of $N_G(v)\cap G^j$, then $f(N_{G^j}(w))=N_{G^i}(w)$. But $v\in V_G(w)$ and $v\notin f(N_G(w))$, then $f$ is not a locally injective homomorphism of $G$ to itself, which is a contradiction.
\end{proof}

\subsubsection*{Locally bijective homomorphisms} \label{sec:lb}

A locally bijective homomorphism is also locally injective and locally surjective. Hence,
any result valid for locally injective or locally surjective homomorphisms is also valid for locally bijective
homomorphisms. Locally bijective homomorphisms (also known as \emph{local isomorphisms}\index{local isomorphisms} or \emph{full covering projections}\index{full covering projection}~\cite{Fiala2001}) have important
applications in many areas, for example distributed computing, recognizing graphs by networks of
processors or constructing highly transitive regular graphs~\cite{Fiala2005}. 

For locally bijective homomorphisms the preimage classes of all vertices have the same size
and for locally surjective homomorphisms all the preimage classes have size at least one.
If there is a locally bijective homomorphism from graph $G$ to graph $H$, then either $|V_G|>|V_H|$
or else $G$ and $H$ are isomorphic.

Just as for graph isomorphism, a locally bijective homomorphism maintains vertex degrees and degrees
of neighbours and degrees of neighbours of neighbours and so on.
The existence of such a mapping from $G$ to $H$ therefore implies equality of the so-called
\emph{degree refinement matrices}\index{degree refinement matrix} of $G$ and $H$.

\subsubsection*{Degree decomposition matrices}
\label{sec:DRM}
Consider a locally bijective homomorphism $f:G\LocBiHom H$. By definition, for every vertex $x$, $f$ induces an isomorphism from the neighbourhood of $x$, $N_G(x)$, to the neighbourhood of $f(x)$, $N_H(f(x))$. In particular $x$ and $f(x)$ have the same degree. (For every vertex $x'\in N_G(x)$ the neighbourhood of $x'$ must be isomorphic to the neighbourhood of $f(x')$ that itself lies in the neighbourhood of $f(x)$.) It follows that locally bijective homomorphisms preserve not only the degrees of vertices, but also the degrees of neighbourhood vertices and so on. This property is well captured by the notion of degree refinement matrices~\cite{Fiala2005}, which we have already defined in Section~\ref{pra:DeMat}.

Recall a partition of the vertex set of a graph $G$ into disjoint classes is an {\em equitable partition}\index{equitable partition} if the vertices in the same class have the same numbers of neighbours in all classes of the partition.

The relation `being finer' defines a lattice on the set of all equitable partitions of a fixed graph.
Every finite graph admits a unique minimal equitable partition~\cite{Everett1991}.

Any equitable partition is characterized by the associated {\em degree decomposition matrix}\index{degree decomposition matrix} whose rows and columns are indexed by the blocks of the partition, and the entry in the $i$-th row and $j$-th column describes how many neighbours a vertex from the $i$-th block has in the $j$-th block.

Every finite graph $G$ admits a unique minimal equitable partition. In this case a canonical ordering can be imposed on the blocks, so the corresponding degree matrix, called the {\em degree refinement matrix}, $\drm(G)$, is also defined uniquely.

We put $G\Matrixeq H$ if and only if $\drm(G)=\drm(H)$. The relation $\Matrixeq$ is an equivalence relation on the class of finite connected graphs, $\ConnGraph$.

It is easy to observe that given a pair of graphs $(G,H)$ with $\drm(G)\neq \drm(H)$ there is no locally bijective homomorphism $G\LocBiHom H$.
Consequently the relation $\LocBiHom$ is a sub-relation of $\Matrixeq$.
The connections between $\Matrixeq$ and all three variants of locally constrained homomorphisms are captured by the following so-called Cantor-Bernstein type theorems.

\begin{thm}[Fiala, Kratochv\'{\i}l~\cite{Fiala2001}]
\label{thm:degeq1}
 If two graphs $G$ and $H$ share the same degree refinement matrix, then every locally injective homomorphism from $G$ to $H$ is locally bijective.
\end{thm}

\begin{thm}[Kristiansen, Telle~\cite{Kristiansen2000}]
\label{thm:degeq2}
 If two graphs $G$ and $H$ share the same degree refinement  matrix, then every locally surjective homomorphism from $G$ to $H$ is locally bijective.
\end{thm}

It is not necessarily the case that there is a locally bijective homomorphism between two graphs with the same degree refinement matrix, for example, $C_3,C_4$ and $C_5$ share the same degree refinement matrix $\{2\}$, but there are no locally bijective homomorphisms between them.

Because the degree refinement matrix is easy to compute, 
for many pair of graphs $(G,H)$ we can easily determine that a locally
bijective homomorphism from $G$ to $H$ does not exist.

\subsubsection*{Universal covers}
The notion of \emph{universal cover}\index{universal cover} is well established in the topology of
continuous spaces. 
 One requirement of a universal cover is that it be simply connected, which translates in terms of graph theory to the requirement of being acyclic and connected, that is a tree. In the discrete case, formally, the \emph{universal covering graph}\index{universal covering graph} (or \emph{universal cover}\index{universal cover}) of a graph $H$ is the only (possibly infinite) tree $T_H$ that admits a
locally bijective homomorphism to $H$.

If $G$ is a tree, then $G$ itself is the universal covering graph of $G$.
The universal covering graph of other connected graph is a countably infinite (but locally finite) tree.

The universal covering graph $T_G$ of a connected graph $G$ can be constructed as follows\footnote{This construction is from \textit{Covering graph. (2012, September 3). In Wikipedia, The Free Encyclopedia. Retrieved 09:51, February 20, 2014, from \url{http://en.wikipedia.org/w/index.php?title=Covering_graph&oldid=510586814}}.}. Choose an arbitrary vertex $r$ of $G$ as a starting point. Each vertex of $T_G$ is a non-backtracking walk that begins from $r$, that is, a sequence $w = (r, v_1, v_2, ..., v_n)$ of vertices of $G$ such that
$v_i$ and $v_{i+1}$ are adjacent in $G$ and $v_{i-1} \neq v_{i+1}$ for all $i$.
Then, two vertices of $T_G$ are adjacent if one is a simple extension of another: the vertex $(r, v_1, v_2, ..., v_n)$ is adjacent to the vertex $(r, v_1, v_2, ..., v_{n-1})$. Up to isomorphism, the same tree $T_G$ is constructed regardless of the choice of the starting point $r$. Thus the universal covering graph is unique up to isomorphism.

The covering map $f$ maps the vertex $(r)$ in $T_G$ to the vertex $r$ in $G$, and a vertex $(r, v_1, v_2, ..., v_n)$ in $T$ to the vertex $v_n$ in $G$.

Figure~\ref{fig:cover} gives an example of universal covering graph.

\begin{figure}[!ht]
\centering
\includegraphics{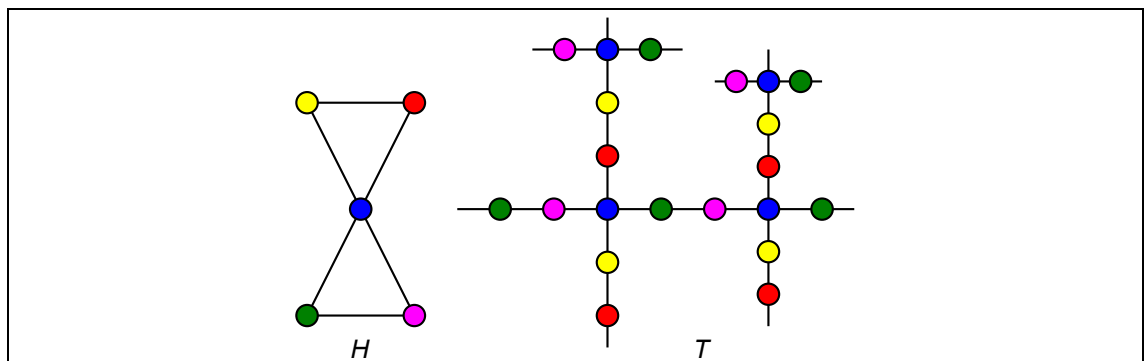} 
\caption{$H$ and its universal covering graph.}
\label{fig:cover}
\end{figure}

Suppose $f$ is a homomorphism from graph $G$ to graph $H$. 
If $f$ is locally injective, then the tree $T_G$ is a subtree of $T_H$;
if $f$ is locally surjective, then $T_H$ is a subtree of $T_G$;
if $f$ is locally bijective, then $T_G$ and $T_H$ are isomorphic.

The following theorem shows the connection between the degree refinement matrices and universal covers.

\begin{thm}[Leighton~\cite{Leighton1982}]
Two graphs $G$ and $H$ share the same degree refinement matrix if and only if their universal covers are
isomorphic.
\end{thm}

\section{Relations}
In the next chapter we will discuss the theory of relations in more detail.

\begin{definition}\label{def:muls}
  Let $G=(V_G,E_G)$ be a graph, $B$ a finite set, and $R\subseteq V\times
  B$ a binary relation, where for every element $b\in B$, we can find an
  element $v \in V_G$ such that $(v,b)\in R$. Then the graph $G\muls R$\index{$G\muls R$} has
  vertex set $B$ and edge set
\begin{equation*}
    E_{G\muls R} = \left\{ (a,b)\in B\times B \mid \textrm{ there is }
      (u,v)\in E_G \textrm{ such that } (u,a),(v,b)\in R \right\}.
\end{equation*}
\end{definition}

Graphs with loops are not always a natural model, however, so that
it may appear more appealing to consider the slightly modified definition.
\begin{definition}\label{def:mulw}
Let $G=(V_G,E_G)$ be a simple graph, $B$ a finite set, $R$ a binary relation,
where for every element $b\in B$, we can find an element $v \in V_G$ such that
$(v,b)\in R$. Then the (simple) graph $G\mulw R$\index{$G\mulw R$} has vertex set $B$ and edge
set
\begin{equation*}
  E_{G\mulw R} = \left\{ (a,b)\in B\times B \mid a\ne b \textrm{ and there is }
  (u,v)\in E_G \textrm{ and } (u,a),(v,b)\in R \right\}.
\end{equation*}
\end{definition}

We remark that these definitions remain meaningful for directed graphs,
weighted graphs (where the weight of an edge is the sum of the weights of its
pre-images) as well as relational structures. For simplicity, we restrict
ourselves to undirected graphs (with loops). Most of the results can be
directly generalized.

Graphs can be regarded as representations of symmetric binary
relations. Using the same symbol for the graph and the relation it
represents, we may re-interpret definition~\ref{def:muls} as a conjugation
of relations. $R^+$ is the \emph{transpose}\index{transpose} of $R$, i.e., $(u,x)\in R^+$
if and only if $(x,u)\in R$. The double composition $R^+\circ G\circ R$
contains the pair $(u,v)$ in $B\times B$ if and only if there are $x$ and
$y$ such that $(u,x)\in R^+$, $(y,v)\in R$, and $(x,y)\in E_G$. Thus
\begin{equation*}
  G\muls R = R^+\circ G\circ R.
\end{equation*}

Simple graphs, analogously, correspond to the irreflexive symmetric
relations. For any relation $R$, let $R^\iota$\index{$R^\iota$} denote its irreflexive
part, also known as the \emph{reflexive reduction}\index{reflexive reduction} of $R$. Since
Definition~\ref{def:mulw} explicitly excludes the diagonals, it can be
written in the form
\begin{equation*}
  G\mulw R = (R^+\circ G\circ R)^{\iota}.
\end{equation*}
We have $G\mulw R = (G\muls R)^{\iota}$, and hence $E_{G\mulw R}\subseteq
E_{G\muls R}$. The composition $G\muls R$ is of particular interest when
$G$ is also a simple graph, i.e., $G=G^{\iota}$.

If there is an $R$ such that $G\muls R=H$ holds, we say there is a \emph{relation}\index{relation} from $G$ to $H$.
If there is an $R$ such that $G\mulw R=H$ holds, we say there is a \emph{weak relation}\index{weak relation} from $G$ to $H$.

The main part of this contribution will be concerned with the solutions of
the equation $G\muls R=H$. The weak version, $G\mulw R=H$, will turn out
to have much less convenient properties, and will be discussed only briefly
in Section~\ref{sect:weak}.

\vspace{10pt}
Throughout this thesis we use the following standard notations and terms.

For relation $R\subseteq X\times Y$ we define by
$R(x)=\{p\in Y\mid (x,p)\in R\}$\index{$R(x)$} the \emph{image of $x$ under $R$}\index{image under relation} and
$R^{-1}(p)=\{x\in X\mid (x,p)\in R\}$\index{$R^{-1}(p)$} the
\emph{pre-image of $p$ under $R$}\index{pre-image under relation}.

The \emph{domain}\index{domain} of $R$ is defined by $\domain{R}=\{x\in
X\mid \exists p\in Y\text{ s.t. }(x,p)\in R\}$\index{$\domain{R}$}, and the \emph{image}\index{image} of
$R$ is defined by $\image{R}=\{p\in Y\mid \exists x\in X\text{
  s.t. }(x,p)\in R\}$\index{$\image{R}$}. We say that the \emph{domain of $R$ is full}\index{domain is full} if for
any $x\in X$ we have $R(x)\ne\emptyset$. Analogously, the \emph{image is
full}\index{image is full} if for any $p\in Y$ we have $R^{-1}(p)\ne\emptyset$.

If $R\subseteq X\times Y$ is a binary relation, then $R$ is
\emph{injective}\index{injective}, if for all $x$ and $z$ in $X$ and $y$ in
$Y$ it holds that if $(x,y)\in R$ and $(z,y)\in R$ then $x = z$.
$R$ is \emph{functional}\index{functional}, if for all $x$ in $X$, and $y$ and
$z$ in $Y$ it holds that if $(x,y)\in R$ and $(x,z)\in R$ then
$y = z$.

We denote by $I_G$\index{$I_G$}\index{$I_G$} the \emph{identity map on $G$}\index{identity map,identity}, i.e.,
$\{(x,x)\mid x\in V_G\}$.

\subsection*{Matrix multiplication}
The operation $\muls$ can also be formulated in terms of matrix
multiplication. To see this, consider the following variant of the
operation on weighted graphs.
\begin{definition}\label{mult_w}
If $G$ is a weighted graph, we use $w(x,y)$\index{$w(x,y)$} to denote the weight between
$x$ and $y$. Given a finite set $B$ and a binary relation $R\subseteq
V_G\times B$, $G\mulm R$ is defined as a weighted graph $H$ with vertex set
$B$, for any $u,v\in B$, $w(u,v)=\sum_{(x,u)\in R,(y,v)\in R} w(x,y)$.
\end{definition}
Ignoring the weights, the operations $\muls$ and $\mulm$ are equivalent.

Using the language of matrices, $G\mulm
 R=H$ can be interpreted as
matrix multiplication:
\begin{equation*}
  \textbf{W}_{G\mulm R} = \mathbf{R}^+ \mathbf{W}_G \mathbf{R}
\end{equation*}
where $\mathbf{R}$ is the matrix representation of the relation $R$,
i.e., $\mathbf{R}_{xu}=1$ if and only if $(x,u)\in R$, otherwise
$\mathbf{R}_{xu}=0$, $\mathbf{R}^+$\index{$\mathbf{R}^+$} denotes the transpose of
$\mathbf{R}$, and $\mathbf{W}_G$ is the matrix of edge weights of
$G$.

\subsection*{Homomorphisms and multihomomorphisms}

The notion of relations between graphs is in many ways similar (but not
equivalent) to the well-studied notion of graph homomorphism. The majority
of our results focus on similarities and differences between these two
concepts. 

Recall that a \emph{(strong) homomorphism}\index{strong homomorphism}\index{homomorphism} from a graph $G$ to a graph $H$ is a mapping
$f:V_G\to V_H$ such that for every edge $(x,y)$ of $G$, $(f(x), f(y))$ is
an edge of $H$. Note that strong homomorphisms require loops in $H$
whenever $(x,y)\in E_G$ and $f(x)=f(y)$. In contrast, we define $f$ is a \emph{weak
  homomorphism}\index{week homomorphism} if $(x,y)\in E_G$ implies that either $f(x)=f(y)$ or
$(f(x),f(y))\in E_H$. Every strong homomorphism from $G$ to $H$ induces
also a weak homomorphism, but not conversely. Weak homomorphisms play an important role in graph product, see~\cite{Imrich2000}.

A map $f:V_G\to V_H$ is, of course, a special case of a relation. This is
seen by setting $F=\{(x,f(x))\mid x\in V_G\}$. Hence, there is a strong
surjective homomorphism from $G$ to $H$ if and only if there is a
functional relation $F$ such that $G\muls F=H$. Another important connection to
the graph homomorphisms is the following simple lemma.
\begin{lemma} \label{reho}
  If $G\muls R = H$, and the domain of $R$ is full, then there is a
  homomorphism $f$ from $G$ to $H$ contained in $R$.
\end{lemma}

\begin{proof}
  If $G\muls R = H$, then take any functional relation $f\subseteq R$, we
  have $G\muls f \subseteq H$, where $f$ is a homomorphism from $G$ to $H$.
\end{proof}

Analogously, there is a \emph{weak surjective homomorphism}\index{weak surjective homomorphism} from $G$ to
$H$ if and only if there is a functional relation $F$ such that $G\mulw F=H$,
and there is a weak homomorphism from $G$ to $H$ if there is a functional
relation $F\subseteq V_G\times V_H$ such that $G\mulw F$ is a subgraph of
$H$. The existence of relations between graphs thus can be seen as a proper
generalization of weak or strong graph homomorphisms, respectively.

Relations between graphs can be regarded also as a variant of
multihomomorphisms. Multihomomorphisms are building blocks of
Hom-complexes, introduced by Lov{\'a}sz, and are related to recent exciting
developments in topological combinatorics~\cite{Matousek2003}, in particular
to deep results involved in proof of the Lov{\'a}sz hypothesis~\cite{Babson2006}.

A \emph{multihomomorphism}\index{multihomomorphism} $G\rightarrow H$ is a mapping $f:
V_G\rightarrow 2^{V_H} \setminus \{ \emptyset \}$ (i.e., associating a
nonempty subset of vertices of $H$ with every vertex of $G$) such that
whenever $\{ u_1,u_2\}$ is an edge of $G$, we have $(v_1,v_2)\in E_H$ for
every $v_1\in f(u_1)$ and every $v_2\in f(u_2)$.

Functions from vertices to sets can be seen as a representation of
relations. A relation with full domain thus can be regarded as
\emph{surjective multihomomorphism}\index{surjective multihomomorphism}, a multihomomorphism such that
pre-image of every vertex in $H$ is non-empty and for every edge $(u,v)$ in
$H$ we can find an edge $(x,y)$ in $G$ satisfying $u\in f(x)$, $v\in f(y)$.
Thus we call a relation with full domain also \emph{\FRelHomo}\index{\FRelHomo}.

\subsection*{Examples}

Similarly to graph homomorphisms, the equation $G\muls R=H$ (or $G\mulw R=H$
respectively) may have multiple solutions for some pairs of graphs $(G,H)$,
while there may be no solution at all for other pairs.

As an example, consider $K_2$ (two vertices $x,y$ connected by an edge) and
$C_3$ (a cycle of three vertices $u,v,w$). Denote $R_1=\{(u,x),(v,y)\}$,
$R_2=\{(v,x),(w,y)\}$, $R_3=\{(w,x),(u,y)\}$, then it is easily seen that
$C_3\muls R_i=K_2$ for each $1\leq i\leq 3$, i.e., the equation $C_3\muls
R=K_2$ has more than one solution.

On the other hand, there is no relation $R$ such that $K_2\muls R=C_3$.
Otherwise, each vertex of $C_3$ is related to at most one vertex of $K_2$,
since $C_3$ is loop free; hence there exists a vertex in $K_2$ which has no
relations to at least two vertices in $C_3$, w.l.o.g., one can assume
$(x,u), (x,v)\notin R$; then the definition of $\muls$ implies that there
is no edge between $u$ and $v$, which leads to a contradiction.

Because relations do not need to have full domain (unlike graph
homomorphisms), there is always a relation from a graph $G$ to its induced
subgraph $G[W]$.

Relations with full domain are not restricted to surjective
homomorphisms. As a simple example, consider paths $P_2$ with vertex set
$\{x,y\}$ and $P_3$ with vertex set $\{u,v,w\}$, respectively, and set
$R=\{(x,u),(x,w),(y,v)\}$. One can easily verify $P_2\muls R=P_3$ by direct
computation. Here, $R$ is not functional since $x$ has two images.

Note that, quite surprisingly, there is no relation $R$ satisfying
$P_3\muls R=P_4$. In fact $(P_2,P_3)$ is the only pair of paths such that
there is a relation from one to another and vice versa. This will be further
explained in Section~\ref{sub:dist}.



\chapter[R-Homomorphisms (Relations)]{\fontsize{32}{12}\selectfont R-Homomorphisms (Relations)}
\label{ch:relation}
\ifpdf
    \graphicspath{{Chapter3/Chapter3Figs/PNG/}{Chapter3/Chapter3Figs/PDF/}{Chapter3/Chapter3Figs/}}
\else
    \graphicspath{{Chapter3/Chapter3Figs/EPS/}{Chapter3/Chapter3Figs/}}
\fi



This chapter is organized as follows:

In Section~\ref{sect:basic} the basic properties of strong relations
between graphs are compiled. It is shown that relations compose and every
relation can be decomposed in a standard way into a surjective and an
injective relation (Corollary~\ref{DeCo}). We discuss some structural
properties of graphs preserved by relations.

Equivalence on the class of graphs induced by the existence of relations
between graphs is the topic of section~\ref{sect:retract}. We consider two
forms: strong relational equivalence, where relations are required to
be reversible, and weak relational equivalence. The equivalence classes of
strong relational equivalence are characterized in Theorem~\ref{thm:equi}.
To describe the equivalence classes of weak relational equivalence we
introduce the notion of an R-core of a graph and show that it is in many
ways similar to the more familiar construction of the graph core
(Theorem~\ref{thm:Rcore}). We explore in particular the differences between
core and R-core and provide an effective algorithm to compute the R-cores of
given graphs.

Section~\ref{sect:poset} is concerned with the partial order induced on
relations between two fixed graphs $G$ and $H$ by inclusion. Focusing on the special
case $G=H$, we describe the minimal elements of this partial order. In
Theorem~\ref{thm:auto1} we give a, perhaps surprisingly simple,
characterization of those graphs whose relations to
themselves are automorphisms.

R-retractions are defined in Section~\ref{sect:retraction} in analogy to
retractions. This naturally gives rise to a notion of R-reduced graphs that
we show coincides with the concept of graph cores. By reversing the
direction of relations, however, we obtain the concept of a cocore of a
graph, which does not have a non-trivial counterpart in the world of
ordinary graph homomorphisms, and we explore its properties. Finally, we give
a full list of the inclusion relations between these notions
and give examples to distinguish them.

The computational complexity of testing for the existence of a relation
between two graphs is briefly discussed in Section~\ref{sect:complexity}.
In Theorem~\ref{thm:complex} we describe the reduction of this problem to
the surjective homomorphism problem.

Finally, in Section~\ref{sect:weak} we briefly summarize the most important
similarities and differences between weak and strong relational
compositions.

\section{Basic properties}
\label{sect:basic}

\subsection*{Composition}

Recall that the composition of binary relations is associative, i.e.,
suppose $R\subseteq W\times X$, $S\subseteq X\times Y$, and $T\subseteq
Y\times Z$. Then $R\circ(S\circ T)=(R\circ S)\circ T$. Furthermore, 
transposition of relations satisfies $(R\circ S)^+ = S^+ \circ
R^+$. Interpreting a graph $G$ as a relation on its vertex set, we easily
derive the following identities.

\begin{lemma}[Composition law]\label{lem:compo}
  $(G\muls R)\muls S = G\muls(R\circ S)$.
\end{lemma}
\begin{proof}
\abovedisplayskip=-12pt
\belowdisplayskip=-12pt
\begin{equation*}
\begin{split}
(G\muls R)\muls S
 & = S^+\circ(R^+\circ G\circ R)\circ S =(S^+\circ R^+)\circ G\circ (R\circ S)\\
 & =(R\circ S)^+ \circ G\circ (R\circ S) = G\muls(R\circ S).
\end{split}
\end{equation*} 
\end{proof}

Now we show that every relation $R$ can be decomposed, in a standard way,
to a relation $R_D$ duplicating vertices and a relation $R_C$ contracting
vertices.

\begin{lemma} \label{dec}
  Let $R\subseteq X\times Y$ be a relation. Then there exists a subset $A$ of
  $X$, a set $B$, an
  injective relation with full domain $R_D\subseteq A\times B$ and a
  functional relation $R_C\subseteq B\times Y$, such that $R=I_A\circ
  R_D\circ R_C$, where $I_A$ is the identity on $X$ restricted to $A$.
\end{lemma}
\begin{proof}
  Put $A=\domain{R}$.
  Then the relation $I_A$ removes
  vertices in $X\setminus\domain{R}$. Therefore, it remains
  to show, that any relation $R\subseteq X\times Y$ with full
  domain can be decomposed into an injective relation
  $R_D\subseteq X\times B$ and a
  functional relation $R_C\subseteq B\times
  Y$. To see this, set $B=R$ and declare $(x,\alpha)\in R_D$ if and only if
  $\alpha=(x,p)\in R$ for some $p\in Y$, and $(\beta,q)\in R_C$ if and only if
  $\beta=(y,q)\in R$ for some $y\in X$. By construction $R_D$ is injective
  and $R_C$ is functional. Furthermore, $(x_0,p_0)\in R_D\circ R_C$ if and
  only if there is an $\alpha\in R$ that is simultaneously of the form
  $(x_0,p)$ and $(x,p_0)$, i.e., $x=x_0$ and $p=p_0$. Hence $(x_0,p_0)\in
  R$.
\end{proof}

Note that this decomposition is not unique. For instance, we could
construct $B$ from multiple copies of $R$. More precisely, let $B=R\times
\{1,2,\dots,k\}$, then we would set $\big(x,(\alpha,i)\big)\in R_D$
($1\leq i\leq k$) if and only if $\alpha=(x,p)\in R$ for some $p\in Y$,
etc.

The set $B$ as constructed in the proof of Lemma~\ref{dec} has
minimal size. To see this, it suffices to show that, given $B$ there
is a mapping from $B$ onto $R$. Since $R_D$ is injective and $R_C$ is
functional we may set
$$\alpha\in B\mapsto (R_D^{-1}(\alpha),R_C(\alpha)).$$
Since $R=I_A\circ R_D\circ R_C$ we conclude that the mapping is
surjective, and hence $|B|\geq |R|$.

According to Lemma~\ref{lem:compo}, the decomposition of $R$ in
Lemma~\ref{dec} can be restated as follows:
\begin{cor} \label{DeCo}
Suppose $G\muls R=H$. Then there is a set $B$, an injective relation
$R_D\subseteq \domain R\times B$ with full domain, and a surjective relation
$R_C\subseteq B\times \image R$ such that
$G[\domain R]\muls R_D\muls R_C=H$.
\end{cor}

In diagram form, this is expressed as
  \begin{figure}[!ht]
  \centering
  \includegraphics[width=0.25\textwidth]{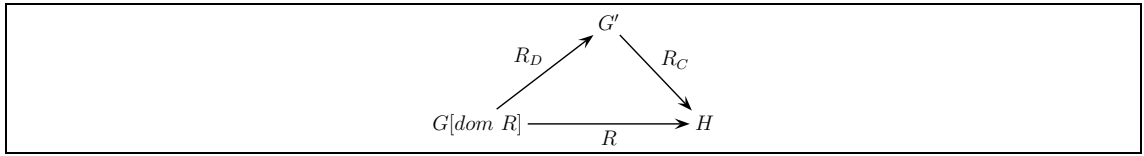} 
\caption{Commutative diagram.}
\label{fig:comum}
\end{figure}   

We remark that from the fact that relations compose it follows that
the existence of a relation implies a quasi-order on graphs that is related
to the homomorphism order. This order is studied more deeply in~\cite{Fiala2012}.

\subsection*{Structural properties preserved by relations}

In this subsection we investigate structural properties of $H$ that can be
derived from knowledge about certain properties of $G$ and the fact that
there is some relation $R$ such that $G\muls R=H$.

\subsubsection*{Connected components}

\begin{pro}\label{ccs}
  Let $G\muls R=H$ and denote by $H_1,\dots,H_k$ the (non-empty) connected components
  of $H$. Then there are relations $R_i\subseteq V_G\times V_{H_i}$ such
  that $G\muls R_i=H_i$ for each $1\leq i\leq k$ and $R=\bigcup_{i=1}^k
  R_i$. Furthermore, set $G_i=G[R^{-1}(V_{H_i})]$, then there are no edges
  between $G_i$ and $G_j$ for arbitrary $i\neq j$.
\end{pro}

\begin{proof}
  Define the restriction of $R$ to the connected components of $H$ as
  $R_i=\{(x,y)\in R\mid y\in V_{H_i}\}$. Clearly, $R$ is the disjoint union of
  $R_i$ and $G\muls R_i\subseteq H_i$. The definition of $\muls$
  implies $H=G\muls R= \left(\bigcup_i R_i\right)^+ \circ G\circ
  \left(\bigcup_j R_j\right) = \bigcup_i\bigcup_j R_i^+ \circ G\circ R_j$.
  Since for $i\neq j$, $R_i$ and $R_j$ relate vertices of $G$ to different connected
  components of $H$, we have $R_i^+ \circ G\circ R_j=\emptyset$. It follows
  that $H=\bigcup_i\bigcup_j R_i^+\circ G\circ R_j = \bigcup_i R_i^+\circ
  G\circ R_i = \bigcup_i G\muls R_i$. Hence $G\muls R_i=H_i$.

  Any edge between $G_i$ and $G_j$ would generate edges between $H_i$ and
  $H_j$, which causes a contradiction to our assumptions.
\end{proof}

Denote by $b_0(G)$ the number of connected components of $G$, then from
Proposition~\ref{ccs} we arrive at:
\begin{cor}\label{cor:concomp}
  Suppose both $G$ and $H$ do not have isolated vertices. If $G\muls R=H$
  and $R$ has full domain, then $b_0(G)\geq b_0(H)$.
\end{cor}
\begin{proof}
  Our notation is the same as in Proposition~\ref{ccs}. We claim for arbitrary
  connected component $C$ of graph $G$, there exists a unique $i$, such that
  $C$ is a connected component of $G_i$. Otherwise one can find two
  vertices $x,y\in C$, $x$ and $y$ adjacent,
  such that $x\in V_{G_i}$ and $y\in V_{G_j}$, since $G$ has no isolated
  vertices, which contradicts
  $E(G_i,G_j)=\emptyset$. Thus it follows $b_0(G)\geq b_0(H)$.
\end{proof}

From Corollary~\ref{cor:concomp}, we know that $H$ is connected whenever $G$
is connected. Conversely, the connectedness of $G$, however, cannot be deduced from the
connectedness of $H$. For example, consider $G=P_1\cup P_1$ with vertex set
$\{x_1,x_2,x_3,x_4\}$ and edges $\{x_1,x_2\}$ and $\{x_3,x_4\}$, and
$H=P_2$ with vertex set $\{v_1,v_2,v_3\}$. Set
$R=\{(x_1,v_1),(x_2,v_2),(x_3,v_2),(x_4,v_3)\}$. One can easily verify that
$G\muls R=H$. On the other hand, $H$ is connected but $G$ has 2 connected
components. The point here is, of course, that $R$ is not injective.

\subsubsection*{Colorings}
Graph homomorphisms of simple graphs can be seen as generalizations of
colourings. Thus, if $R$ is a
functional relation describing a vertex colouring, then $G\muls R\subseteq
K_k$. Conversely, if $G\muls R\subseteq K_k$, where $R$ has full domain,
then from Lemma~\ref{reho}, there exists a homomorphism from $G$ to
$K_k$, which is a colouring of $G$.
\begin{lemma} \label{colour} If $G$ is a simple graph and $R$ has full
  domain, then $\chi(G)\leq \chi(G\muls R)$.
\end{lemma}

\begin{proof}
  Suppose $G\muls R=H$ and the domain of $R$ is full, from Lemma~\ref{reho}
  we know $G\rightarrow H$, so $\chi(G)\leq \chi(G\muls R)$.
\end{proof}

\subsubsection*{Distances} \label{sub:dist}
\begin{obser} \label{path} If $P_k\muls R=G$, $G$ is a simple graph and the
  domain of $R$ is full, $P_k$ with the vertex set ${0,1,\dots,k}$, then
  there is a walk $[v_0,v_1,\dots,v_k]$ in $G$, where $(i,v_i)\in R$ for
  $0\leq i\leq k$.
\end{obser}

\begin{obser} \label{cycle} If $C_k\muls R=G$, $G$ is a simple graph and
  the domain of $R$ is full, then there is a closed walk
  $[v_0,v_1,\dots,v_{k-1}]$ in $G$, where $(i,v_i)\in R$ for $0\leq i\leq
  k-1$.
\end{obser}

Recall that $d_G(x,y)$ denotes the \emph{(canonical) distance}\index{canonical distance}\index{distance} on graph $G$, i.e.,
$d_G(x,y)$ is the minimal length of a path in graph $G$ that connects
vertices $x$ and $y$; if there is no path that connects vertices $x$ and $y$,
then the distance is infinite.

\begin{lemma}
  Suppose there exists a relation $R$ with full domain s.t.\ $G\muls R=H$,
  $x,y\in V_G$, $u,v\in V_H$ and $(x,u)\in R, (y,v)\in R$. If $x\neq y$,
  then $d_H(u,v)\leq d_G(x,y)$; If $x=y$ and $x$ is not an isolated vertex,
  then $d_H(u,v)\leq 2$.
\end{lemma}
\begin{proof}
  If $x=y$ and $x$ is not isolated, pick a vertex $z$ of graph $G$ which is
  adjacent to vertex $x$, and find a vertex $w\in H$ satisfying $(z,w)\in
  R$. Then $(w,u)\in E_H$ and similarly $(w,v)\in E_H$. So $d_H(u,v)\leq
  2$.

  If $x\neq y$, choose the shortest path $P=x,x_1,x_2,\dots, x_k,y$
  between $x$ and $y$, and find corresponding vertices
  $u_1,u_2,\dots,u_k\in H$ such that$(x_i,u_i)\in R$ for any $1\leq i\leq
  k-1$, it is easily seen that $(u,u_1)\in E_H$, $(u_i,u_{i+1})\in E_H$ and
  $(u_k,v)\in E_H$, then $d(u,v)\leq d(x,y)$.
\end{proof}

Hence we immediately get the following corollary about radius.

\begin{cor} \label{rad} Suppose $G\muls R=H$, $G$ and $H$ are connected
graphs, and $R$ has full domain, then $rad(H)\leq \max\{rad(G),2\}$.
\end{cor}
An analogous result holds for the diameters. In particular, if $G$
is not a complete graph, then $diam(G)\geq diam(G\muls R)$. As a result,
$k<l$ and $k\neq 1$ implies there is no relation from $P_k$ to $P_l$.
Also, as we state in the example section, $(P_1,P_2)$ is the only pair of paths
such that there is a relation from one to another and vice versa.

\label{subs:complete}

\subsubsection*{Complete graphs} 
Note that in this subsection we do not require that the domain of $R$ is
full.

\begin{pro}\label{prop:complete}

  Let $H$ be a simple graph. Then there exists a relation $R$ such that
  $K_k\muls R=H$ if and only if $H$ is a complete $m$-partite graph, where
  $m\leq k$.

\end{pro}

\begin{proof}
 We denote the vertices of $K_k$ by $1,\dots,k$.

 If $H$ is a complete $m$-partite graph, let $R=\{(i,u)\mid i=1,\dots,m,u\in U_i\}\}$, then
 for each $1\leq i\leq m$, all the vertices
  in $U_i$ are independent in $H$, and $u$ is adjacent to $v$ whenever
  $u\in U_i$ and $v\in U_j$ for distinct $i,j$. Hence it is easily
  seen that $K_k\muls R=H$.

   Conversely, if $K_k\muls R=H$, assume $\domain R=\{1,\dots,m\}$. We claim that $R$ is
  injective, otherwise $H$ would have loops. Thus $V_H$ is the disjoint union
  of $R(1),\dots,R(m)$. For any
  two distinct vertices $u,v$ in $R(i)$, $u$ and $v$ are independent in $H$,
  then $H$ is a $m$-partite graph.
  For distinct $i$ and $j$ every vertex in $R(i)$ are adjacent with every
  vertex in $R(j)$ whenever $R(i)\neq \emptyset$, which means $H$ is a complete $m$-partite graph.
\end{proof}


\subsubsection*{Subgraphs} \label{sect:subgraph}
Relations between graphs intuitively imply relations between local
subgraphs. In this section we make this concept more precise. Recall $N_G[x]$ was denoted to be
the \emph{closed neighbourhood}\index{closed neighbourhood} of $x$ in $G$ ($N_G[x]=N_G(x)\cup \{x\}$) and
$\overbar{N_G[x]}:= V_G\setminus N_G[x]$\index{$\overbar{N_G[x]}$} be the set of vertices that are
not adjacent (or identical) to $x$ in $G$. Furthermore,
$\overbar{G_x}:=G-N_G[x]$\index{$\overbar{G_x}$}, which is the induced subgraph of $G$
obtained by removing the closed neighbourhood of a vertex $x$.
Analogously, for a subset $S\subseteq V_G$, $\overbar S$ is the induced subgraph obtained by removing all vertices in $S$ and their neighbours.

Then we have the following result about relations between local subgraphs.

\begin{pro} Suppose $G\muls R=H$ and $(x,u)\in R$. Suppose there is no isolated vertex in $G$. If there is no loop in
  $x$, then $\overbar{G_x}\muls \widetilde{R} = \overbar{G_u}$ where
  $\widetilde{R}\subseteq \overbar{N_G[x]}\times\overbar{N_H[u]}$ is the
  restriction of $R$.
\end{pro}

\begin{proof}
  We claim that the pre-image of any vertex in $\overbar H_u$ is not
    in $N_G(x)$. Otherwise, there is a vertex $v\in \overbar H_u$ and a
    vertex $y\in N_G(x)$ such that $(y,v)\in R$; from $(x,u)\in R$ we deduce
    that $v$ and $u$ are adjacent, which contradicts to $v\in \overbar
    H_u$. Next we claim that there is no vertex in $\overbar H_u$ whose
    pre-image is $x$. Suppose $v\in \overbar H_u$ and $R^{-1}(v)=\{x\}$.
    Since $v$ is not isolated in $\overbar H_u$, there exists a
    vertex $v'$ which is adjacent with $v$ in $\overbar H_u$. Thus
    $R^{-1}(v')\cap N(x)\neq \emptyset$, a contradiction. Hence
    for any edge $(v,v')$ in $\overbar H_u$, we can find
    an edge $(y,y')$ in $\overbar G_x$ such that $(y,v),(y',v')\in
    R$. Thus $\overbar G_x\muls \widetilde{R}=\overbar{G_y}$.
\end{proof}






\begin{example}
Let us see an example. Let $G$ and $H$ be graphs as Figure~\ref{fige5}.
For a relation from $G$ to $H$, we consider the pre-image of
vertex $u$, because the domain of $R$ is full, and the symmetry of $P_4$, at least
either $(1,u)\in R$ or $(2,u)\in R$ holds. If $(1, u)\in R$, from the
above theorem, we know that there is an induced relation from $P_2$ to
$P_4$, but it is impossible. If $(2,u)\in R$ holds, also from above
theorem, there is an induced relation from $P_1$ to $P_4$, which is impossible.

\begin{figure}[!ht]
\centering
\subfloat[$G$]{
\begin{minipage}[c]{0.4\textwidth}
\centering
\includegraphics[width=0.5\textwidth]{dupl1.eps}
\end{minipage}
}
\subfloat[$H$]{
\begin{minipage}[c]{0.4\textwidth}
\centering
\includegraphics[width=0.5\textwidth]{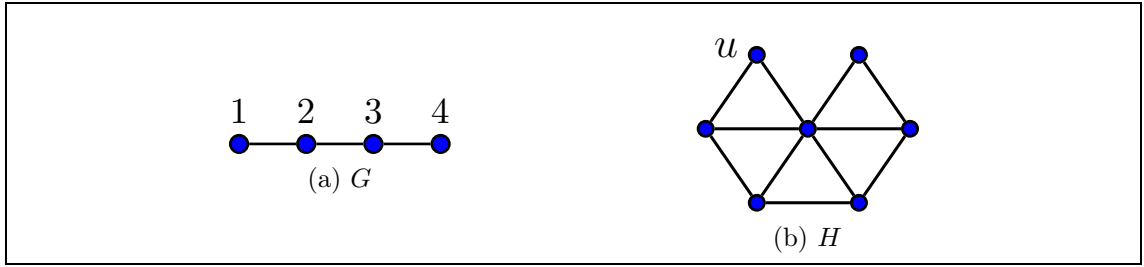}
\end{minipage}
}
 \caption{There is no relation from $G$ to $H$.}
    \label{fige5}
\end{figure}
\end{example}

Note that from Corollary~\ref{rad}, the radius of $H$ is smaller than the radius of $G$, so there is no relation $R$ such that
$H\muls R=G$. Therefore there is neither a relation from $G$ to $H$, nor one from $H$ to $G$. Note that the two solutions can be different.

\begin{pro} \label{pro:subgraph}
  Suppose $G\muls R=H$ and $S$ and $D$ are subsets of $V_G$ and $V_H$,
  respectively, such that
  \begin{enumerate}
   \item $G[S] \muls R|_{(S\times D)}=H[D]$;
   \item $R|_{(S\times D)}$ has full domain on $S$;
   \item there is no isolated vertex in $\overbar D$.
  \end{enumerate}

  Then $\overbar S\muls\widetilde{R}=\overbar D$,
  where $\widetilde{R} = R|_{(\overbar S\times \overbar D)}$ is the
  corresponding restriction of $R$.
\end{pro}

\begin{proof}
  Obviously, $\overbar S\muls\widetilde{R}$ is an induced subgraph of $\overbar
  D$. We need to show the reverse inclusion. Given $u\in V_{\overbar{D}}$
  and $x\in R^{-1}(u)$, we first show that there are two possibilities:
  \begin{enumerate}
  \item $x$ is a vertex of $\overline{S}$.
  \item $x$ is an isolated vertex of $S$.
  \end{enumerate}
  Assume, to the contrary, $x$ is either a non-isolated vertex of $S$ or $x$ is in the
  neighbourhood of some vertex of $S$. In both cases there is $y\in S$
  connected by an edge to $x$. Consequently there is a vertex $v\in D$, such
  that $v\in R(y)$, and $v$ is connected by an edge to $u$. It follows
  $u\notin V_{\overbar{D}}$, a contradiction.

  Now consider an arbitrary edge $(u,v)\in E_{\overbar{D}}$. We have
  $(x,y)\in E_G$ such that $u\in R(x)$ and $v\in R(y)$.
  If any of $x$ and $y$, say $x$, is an isolated vertex in $S$, then $y$ is not in $S$.
  Thus $y$ is in the neighbourhood of $S$, which contradicts the two possibilities.
  So both $x$ and $y$ are not isolated in $S$, from the above deduction $x,y$ are in the case 1, i.e.,
  they are vertices of $\overline{S}$.
  Consequently $\overbar S\muls\widetilde{R}$ has precisely the same edges
  as $\overbar D$. Because $\overbar{D}$ has no isolated vertices and
  thus every vertex is an endpoint of some edge, we know that the vertex set
  of $\overbar S\muls\widetilde{R}$ is same as the vertex set of $\overbar D$. Hence
  $\overbar S\muls\widetilde{R}=\overbar D$.
\end{proof}
This result is of particular practical use in the special case where $S$
and $D$ consist of a single vertex. When looking for a relation $R$ such
that $G\muls R=H$ one can remove a vertex including its neighbourhood from
$G$ as well as the prospective image including the neighbourhood from $H$
and then solve the problem on the subgraphs.

\section{Relational equivalence}
\label{sect:retract}
Recall that in Section~\ref{sec:hom-equ} we define graphs $G$ and $H$ are \emph{homomorphically equivalent}\index{homomorphically equivalent} (or
\emph{hom-equivalent}\index{hom-equivalent}) if there exists homomorphisms $G\to H$ and $H\to
G$. It is well known that every equivalence class of the homomorphism order
contains a minimal representative that is unique up to isomorphism: the
\emph{graph core}\index{graph core}~\cite{Hell2004}.

We define similar equivalences implied by the existence of
(special) relations between graphs. In this section, we require all
relations to have full domain unless explicitly stated otherwise. With
this condition we will show that these equivalences produce a rich structure
closely related to but distinct from the structure of homomorphism
equivalences.

This may come as a surprise: the equivalence implied by the existence of
surjective homomorphisms is not interesting. Consider two graphs $G$ and
$H$ and suppose there are surjective homomorphisms $f:G\to H$ and $g:H\to
G$. Since every vertex in $V_G$ has at most one image under $f$, we have
$|V_G|\geq |V_H|$. Analogously $|V_H|\geq |V_G|$, and hence
$|V_G|=|V_H|$. Thus $f$ and $g$ are both bijective, and $G$ is isomorphic
to $H$.

\subsection*{Reversible relations} \label{revers}

\begin{definition}
  A relation $R$ is \emph{reversible with respect to a graph $G$}\index{reversible} if
  $(G\muls R)\muls R^+=G$.
\end{definition}

Recall that in Proposition~\ref{dec} we proved that $R$ can be decomposed.
Suppose $R=R_D\circ R_C$, where $R_D$ and $R_C$
are constructed as in the proof of Proposition~\ref{dec}. $R_C$ is a functional relation and $R_D$ is an injective
relation with full domain.
Then we can get the following result with respect to reversible relation.

\begin{pro} \label{reve}   
$R$ is reversible with respect to $G$ if and only if for every $\alpha$ and $\beta$
  satisfying $R_C(\alpha)=R_C(\beta)$ we have
  $N_{G\muls R_D}(\alpha)= N_{G\muls R_D}(\beta)$.
\end{pro}
\begin{proof}
  We set $G_1=G\muls R_D$, If $R_C(\alpha)=R_C(\beta)$ implies
  $N_{G_1}(\alpha)= N_{G_1}(\beta)$, then $H\muls R_C^+=G_1$.
  Therefore,
  \begin{align*}
     G\muls R\muls R^+ & = H\muls R^+ = H\muls (R_D\muls R_C)^+\\
                       & = H\muls R_C^+\muls R_D^+ = G_1\muls R_D^+ = G.
   \end{align*}

  The first equation is by assumption, the second is by Proposition~\ref{dec}, the third one
  is by the property of transpose, the fourth is from Lemma~\ref{lem:compo},
  the last equation is because $R_D$ is injective.
  Thus by definition, $R$ is reversible.

  Conversely, since $R$ is
  reversible, i.e., $H\muls R^+=G$, setting $G_2 = H\muls R_C^+$ gives
  $G_2\muls R_D^+=G$. Hence $G_1\muls R_C\muls R_C^+=G_2$ and $G_2\muls
  R_D^+\muls R_D=G_1$. From $I_{G_1}\subseteq R_C\muls R_C^+$ we conclude
  $G_1\subseteq G_2$, and similarly $I_{G_2}\subseteq R_D^+\muls R_D$ yields
  $G_1\supseteq G_2$. Hence $G_1=G_2$. $R_C^+$ is injective, hence
  $\alpha,\beta\in V_{G_2}=V_{G_1}$ has the same open neighbourhood whenever
  the pre-image of $\alpha$ and $\beta$ under $R_C^+$ coincide,
  i.e. $R_C(\alpha)=R_C(\beta)$.
\end{proof}

$R_D$ is an injective relation, hence one can easily get
$N_{G\muls R_D}(\alpha)=R_D(N_G(x))$ provided that $(x,\alpha)\in R_D$. On the
other hand, if we define $R$ to be the image of $R_D$ as in the proof of
Proposition~\ref{dec}, then $R_C(\alpha)=R_C(\beta)$ implies there are two
distinct vertices $x,y\in V_G$, s.t. $(x,u),(y,u)\in R$, where
$u=R_C(\alpha)=R_C(\beta)$, and verse visa. Using Proposition~\ref{reve}
we thus obtain
\begin{pro}
  A relation $R$ is reversible with respect to $G$ if and only if for every
  two vertices $x$ and $y$ such that $R(x)\cap R(y) \neq \emptyset$ we have
  $N_G(x)=N_G(y)$.
\end{pro}

\subsection*{Strong relational equivalence}

\begin{definition}
  Two graphs $G$ and $H$ are \emph{strongly relationally equivalent}\index{strong relationally equivalent}, denoted by $G
  \backsim_s H$\index{$\backsim_s$}, if there is a relation $R$ such that $G\muls R=H$ and
  $H\muls R^+=G$.
\end{definition}

\begin{lemma}
  Relational equivalence is an equivalence relation on graphs.
\end{lemma}
\begin{proof}
  By definition $\backsim_s$ is symmetric. Because $G\muls I_G=G$, relation
  $\backsim_s$ is reflexive. Suppose $G\backsim_s G'$ and $G'\backsim_s
  G''$. Thus there are relations $R$, $S$, such that
  $G'=G\muls R$, $G''=G'\muls S$, $G=G'\muls R^+$, and $G'=G''\muls S^+$. By
  the composition law (Lemma~\ref{lem:compo}) it follows that
  $G''=G\muls(R\circ S)$ and $G=G''\muls(S^+\circ R^+)=G''\muls(R\circ S)^+$, i.e, $G\backsim_s
  G''$. Hence $\backsim_s$ is transitive.
\end{proof}

Recall that the \emph{thinness relation}\index{thinness relation} $S$ of $G$ is the equivalence relation on
  $V_G$ defined by $(x,y)\in S$ if and only if $N_G(x)=N_G(y)$.
  We denote by $\mathcal{S}$ the corresponding partition of $V_G$, and write
  $R_S\subseteq V_G\times\mathcal{S}$ that associates each vertex with its
  $S$-equivalence class, i.e., $(x,\beta)\in R_S$ if and only if $x\in\beta$.
In this context it is well known that $G$ can
be reconstructed from $G_{pd}$ (the point-determining graph of $G$) and knowledge of the $S$-equivalence
classes. In fact, we have
\begin{equation}
\label{eq:thin-1}
  G_{pd}\muls R^+_S=G.
\end{equation}

\begin{thm} \label{thm:equi}
  $G$ and $H$ are in the same equivalence class w.r.t.\ $\backsim_s$ if and
  only if their point-determining graphs are isomorphic.
\end{thm}
\begin{proof}
  Assume $G\backsim_s H$. From Equation(\ref{eq:thin-1}) we know that
  $G\backsim_s G_{pd}$, $H\backsim_s H_{pd}$, so $G_{pd}\backsim_s
  H_{pd}$.
  Now we claim that $G_{pd}$ and $H_{pd}$ are isomorphic.
  Suppose $G_{pd}\muls R=H_{pd}$, then the pre-image of $R$ is unique.
  Otherwise, there exist distinct vertices $x,y\in V_{G_{pd}}$ such that
  $R(x)=R(y)$, then $N_{G_{pd}}(x)=N_{G_{pd}}(y)$, contradicting
  thinness. Likewise, the pre-image of $R^{+}$ is unique, i.e., the image
  of $R$ is unique. Hence $R$ is one-to-one.
\end{proof}

\begin{figure}[!ht]
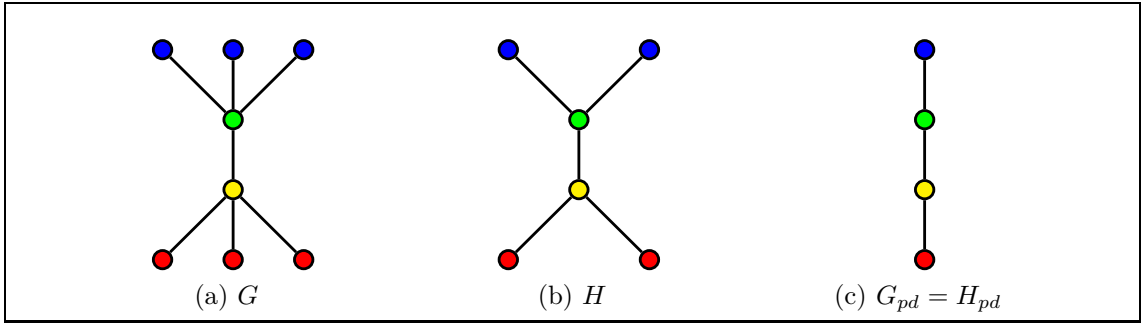

\centering
\subfloat[$G$]{
\begin{minipage}[c]{0.3\textwidth}
\centering
 \includegraphics[height=0.7\textwidth]{ex3.1.1.eps}
\end{minipage}}
\subfloat[$H$]{
\begin{minipage}[c]{0.3\textwidth}
\centering
 \includegraphics[height=0.7\textwidth]{ex3.1.2.eps}
\end{minipage}}
\subfloat[$G_{pd}=H_{pd}$]{
\begin{minipage}[c]{0.3\textwidth}
\centering
 \includegraphics[height=0.7\textwidth]{ex3.1.3.eps}
\end{minipage}}
\caption{Non-isomorphic graphs $G$ and $H$ with isomorphic point-determining graphs.}
\label{fig:thin}
\end{figure}


The example in Fig.~\ref{fig:thin} shows that point-determining graphs can be isomorphic
while $G$ and $H$ themselves are not isomorphic. Relational equivalence
thus is coarser than graph isomorphism (surjective homomorphic equivalence)
but stronger than homomorphic equivalence.

\subsection*{Weak relational equivalence} \label{sec:weakeq}

\begin{definition}
  Two graphs $G$ and $H$ are \emph{(weak) relationally equivalent}\index{weak relationally equivalent}\index{relationally equivalent}, denoted by $G
  \backsim_w H$\index{$\backsim_w$}, if there are relations $R$ and $S$ such that $G\muls R=H$
  and $H\muls S=G$.
\end{definition}

Note that $\backsim_w$ is the same as $\Releq$ in Chapter~\ref{Ch:order}.

\begin{lemma}
  Weak relational equivalence is an equivalence relation on graphs.
\end{lemma}
\begin{proof}
  By definition $\backsim_w$ is symmetric. Because $G\muls I_G=G$, relation
  $\backsim_w$ is reflexive. Suppose $G\backsim_w G'$ and $G'\backsim_w
  G''$. Thus there are relations $R$, $S$, $R'$, and $S'$, such that
  $G'=G\muls R$, $G''=G'\muls R'$, $G=G'\muls S$, and $G'=G''\muls S'$. By
  the composition law (Lemma~\ref{lem:compo}) it follows that
  $G''=G\muls(R\circ R')$ and $G=G''\muls(S'\circ S)$, i.e, $G\backsim_w
  G''$. Hence $\backsim_w$ is transitive.
\end{proof}

Strong relational equivalence implies weak relational equivalence. To see
this, simply observe that the definition of the weak form is obtained from
the strong one by setting $S=R^+$.

\begin{figure}[!ht]
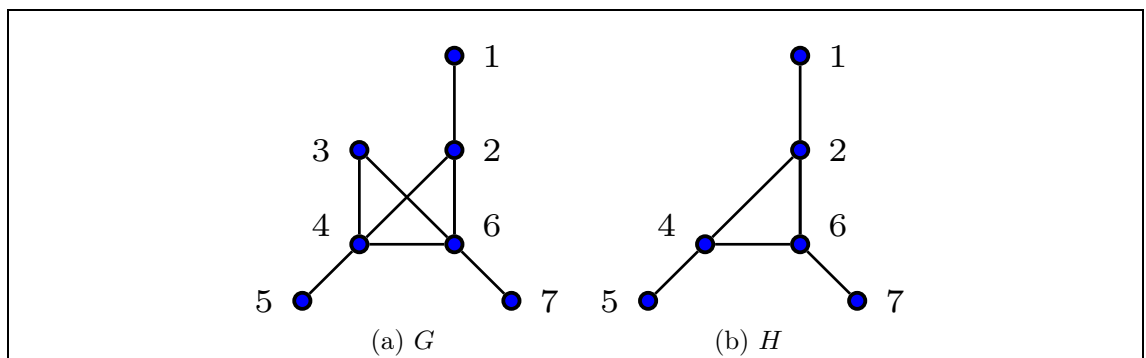

\centering
\subfloat[$G$]{\label{figa:weak}
\begin{minipage}[c]{0.3\textwidth}
\centering
\includegraphics[width=0.9\textwidth]{weak_nonstrong.eps}
\end{minipage}}
\subfloat[$H$]{
\begin{minipage}[c]{0.3\textwidth}
\centering
\includegraphics[width=0.9\textwidth]{weak_nonstrong2.eps}
\end{minipage}}
\caption{$G$ and $H$ are weakly relationally equivalent but have
  non-isomorphic  graphs.}
\label{fig:weak}
\end{figure}


The converse is not true, as shown by the graphs $G$ and $H$ in
Fig.~\ref{fig:weak}: It is easy to see that their point-determining graphs are different
and thus $G$ and $H$ are not strongly relationally equivalent. However,
they are relationally equivalent. To get relation from $G$ to $H$ contract
vertices 2 and 3 and keep other vertices on place, i.e.,
$$R=\{(1,1),(2,2),(3,2),(4,4),(5,5),(6,6),(7,7)\}.$$ To get relation
from $H$ to $G$, duplicate 5 and 7 and contract them together to 3,
$$S=\{(1,1),(2,2),(4,4),(5,5),(6,6),(7,7),(5,3),(7,3)\}.$$
Consequently, weak relational equivalence is coarser than strong relational
equivalence.

\subsection*{R-cores}
\label{subsec:R-core}

A graph is an \emph{R-core}\index{R-core}, if it is the minimal graph (in the number of
vertices) in its equivalence class of $\backsim_w$.

This notion is analogous to the definition of graph cores. In this section
we show properties of R-cores that are similar to the properties of graph
cores. To this end we first need to develop a simple characterization of
R-cores.

Again we start from a decomposition of relations. Consider a relation
$R$ such that $G\muls R = H$. We seek for pair of relations $R_1$ and
$R_2$ such that $R=R_1\circ R_2$. In contrast to Lemma~\ref{dec}, however,
we now look for a decomposition so that the graph $G'=G\muls R_1$ is
smaller (in the number of vertices) than $G$.
  \begin{figure}[!ht]
  \centering
  \includegraphics[width=0.25\textwidth]{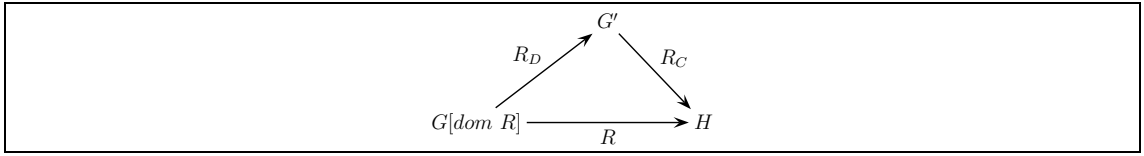} 
\caption{Commutative diagram.}
\label{fig:comum2}
\end{figure}   

The existence of such a decomposition follows from a translation of
the well-known Hall Marriage Theorem~\cite{Schrijver2003} to the language of
relations. We say that the relation $R\subseteq A\times B$ satisfies the
\emph{Hall condition}, if for every $S\subseteq A$ we have $|S|\leq |R(S)|$.

We have already proved Hall's marriage theorem in~\ref{thm:hallo}. Here we restate it in terms of relations.

\begin{thm}[Hall's marriage theorem] \label{thm:hall}
  If $G\muls R = H$ and $R$ satisfies the Hall condition, then $R$
  contains a monomorphism $f:G\rightarrow H$.
\end{thm}

\begin{lemma} \label{lem:nohall} If $G\muls R = H$ and relation $R$ does
  not satisfy the Hall condition, then there are relations $R_1$ and $R_2$
  such that $R=R_1\circ R_2$, and the number of vertices of graph
  $G'=G\muls R_1$ is strictly smaller than the number of vertices of $G$.
\end{lemma}
\begin{proof}
  Without loss of generality assume that $V_G\cap V_H=\emptyset$. If $R$
  does not satisfy the Hall condition, then there exists a vertex set $S\subset
  V_G$ such that $|S|>|R(S)|$. Now we define relations $R_1$ and $R_2$ as
  follows:
  \begin{equation*}
    R_1(x)=
    \begin{cases} R(x) \text{   for $x\in S$,}\\
      x \text{   otherwise,}
    \end{cases}
\qquad\qquad\qquad
    R_2(x)=
    \begin{cases} x \text{   for $x\in R(S)$,}\\
      R(x) \text{   otherwise.}
    \end{cases}
  \end{equation*}
Obviously $R_1\circ R_2 = R$ and $|V_{G'}|=|V_G|-(|S|-|R(S)|)<|V_G|$.
\end{proof}

This immediately gives a necessary, but in general not sufficient,
condition for a graph to be an R-core.

\begin{cor}\label{cor:nohall2}
  If $G$ is an R-core, then every relation $R$ such that $G\muls R=G$
  satisfies the Hall condition and thus contains a monomorphism.
\end{cor}
\begin{proof}
  Assume that there is a relation $R$ that does not satisfy the Hall
  condition. Then there is a graph $G'$, $|V_{G'}|<|V_G|$, and relations
  $R_1$ and $R_2$ such that $G\muls R_1=G'$ and $G'\muls
  R_2=G$. Consequently $G'$ is a smaller representative of the equivalence
  class of $\backsim_w$, a contradiction with $G$ being R-core.
\end{proof}

To see that the condition of Corollary~\ref{cor:nohall2} is not sufficient,
consider a graph $G$ consisting of two independent vertices. It is not an R-core,
because it can retract to an isolated vertex. However, the relation is required
to be full, any relation from $G$ to itself satisfies the Hall condition.

Next we show that R-cores are, up to isomorphism, unique representatives of
the equivalence classes of $\backsim_w$.

  \begin{figure}[!ht]
  \centering
  \includegraphics[width=0.35\textwidth]{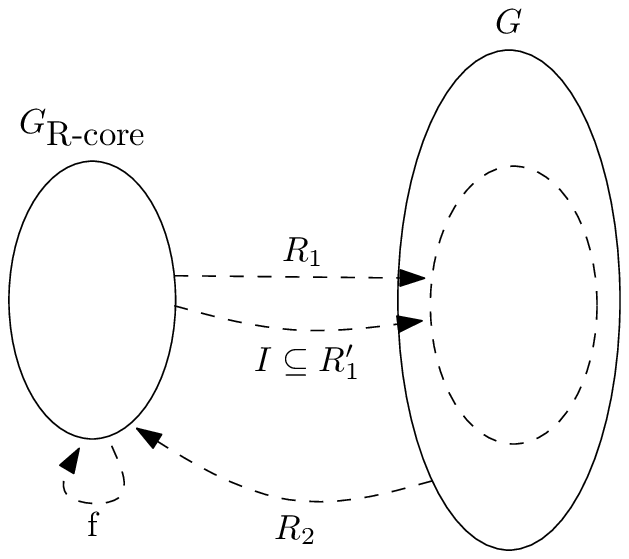} 
\caption{Construction of an embedding from $G_\text{R-core}$ to $G$.}
\label{fig:rcore}
\end{figure}   

\begin{pro}\label{pro:iso}
  If both $G$ and $H$ are R-cores in the same equivalence class of
  $\backsim_w$, then $G$ and $H$ are isomorphic.
\end{pro}

\begin{proof}
  Because both $G$ and $H$ are R-cores, we know that $|V_G|=|V_H|$.

  Consider relations $R_1$ and $R_2$ such that $G\muls R_1=H$ and $H\muls
  R_2=G$. Applying Lemma~\ref{lem:nohall} we know that $R_1$ satisfies the
  Hall condition. Otherwise there would be a graph $G'$ with $|V_{G'}|<
  |V_G|$ so that $G'$ is relationally equivalent to both $G$ and $H$
  which contradicts the fact that $G$ and $H$ are R-cores. Similarly, we can
  show that $R_2$ also satisfies the Hall condition.

  From Theorem~\ref{thm:hall} we know that there is a monomorphism $f$ from
  $G$ to $H$, and monomorphism $g$ from $H$ to $G$. It follows that the number
  of edges of $G$ is not larger than the number of edges of $H$ and
  vice versa. Because $G$ and $H$ have the same number of edges and
  same number of vertices, $G$ and $H$ must be isomorphisms.
\end{proof}

It thus makes sense to define a construction analogous to the core of a graph.
\begin{definition}
  $H$ is an \emph{R-core of a graph $G$}\index{R-core} if $H$ is an R-core and $H\backsim_w
  G$.
\end{definition}

All R-cores of graph $G$ are isomorphic as an immediate consequence of
Proposition~\ref{pro:iso}. We denote the (up to isomorphism) unique R-core
of graph $G$ by $G_{\text{R-core}}$\index{$G_{\text{R-core}}$}..

\begin{lemma}
  $G_{\text{R-core}}$ is isomorphic to a (not necessarily induced) subgraph
  of $G$.
\end{lemma}

\begin{proof}
  Take any relation $R$ such that $G_{\text{R-core}}\muls R=G$. By the same
  argument as in Corollary~\ref{cor:nohall2}, there is a monomorphism
  $f:G_{\text{R-core}}\rightarrow G$ contained in $R$. Consider the image
  of $f$ on $G$.
\end{proof}

\begin{thm}
  \label{thm:Rcore}
  $G_{\text{R-core}}$ is isomorphic to an induced subgraph of $G$.
\end{thm}

\begin{proof}
  Fix $R_1$ and $R_2$ such that $G_{\text{R-core}}\muls R_1=G$ and
  $G\muls R_2=G_{\text{R-core}}$.

  $R=R_1\circ R_2$ is a relation such that $G_{\text{R-core}}\muls
  R=G_{\text{R-core}}$. By Corollary~\ref{cor:nohall2}, $R$ contains a
  monomorphism $f:G_{\text{R-core}}\rightarrow G_{\text{R-core}}$. Because
  such a monomorphism is a permutation, there exists an $n$ such that $f^n$ is
  identity.

  Put $R_1'=R^{n-1}\circ R_1$. Because $R^{n}$ contains the identity and
  $R^{n}=R_1'\circ R_2$, it follows that for every
  $x\in V_{G_{\text{R-core}}}$, there is a vertex
  $I(x)\in V_G$ such that $I(x)\in R'_1(x)$ and $x\in R_2(I(x))$.
  In other words, $(x,I(x))\in R'_1(x)$ and $(I(x),x)\in R_2$.
  It is easily seen that $I$ is a mapping from vertex set of $G_{\text{R-core}}$
  to vertex set of $G$. We take the induced subgraph of $G$ induced by the image of $I$,
  denoted by $G'$, and claim that $I$ is an isomorphism of $G_{\text{R-core}}$ to $G'$.

  We first show that for two vertices $x\neq y$, we have $I(x)\neq I(y)$ and
  thus $I$ is vertex injective (thus both vertex and edge injective). Assume, to the contrary, that there are two vertices $x\neq y$ such that $I(x)=I(y)$.
  Consider an arbitrary vertex $z$ in the neighbourhood of $x$. It follows
  that $I(z)$ must be in the neighbourhood of $I(x)$ and consequently
  $z$ is in the neighbourhood of $y$. Thus the neighbourhoods of $x$ and $y$
  are the same. By Theorem~\ref{thm:equi}, however, we know that the
  R-core is a point-determining graph (because weak relational equivalence is coarser
  than strong relational equivalence), a contradiction.

  Then we claim that $I$ is a homomorphism of $G_{\text{R-core}}$ to $G'$, i.e.,
  for every edge $(x,y)\in E_{G_{\text{R-core}}}$ we also have edge
  $(I(x),I(y))\in E_G$. This is because $I\subset R'_1$, then
  $G_{\text{R-core}}=G_{\text{R-core}}\muls I\subset G_{\text{R-core}}\muls R'_1=G$.

  By the construction, $I$ is vertex surjective.
  Finally we prove that $I$ is also a edge surjective homomorphism by claiming
  $I^{-1}$ is a homomorphism of $G'$ to $G_{\text{R-core}}$.
  We see $I^{-1}\subset R_2$ by observing $x \in R_2(I(x))$ implies $(I(x),x)\in R_2$ for all $x\in$
  vertex set of $G_{\text{R-core}}$. Because $I$ as well as $I^{-1}$ is one-to-one,
  $G'=G'\muls I^{-1}\subset G'\muls R_2=G_{\text{R-core}}$.
  $(I(x),I(y))\in G'$ corresponds to an edge $(x,y)\in
  G_{\text{R-core}}$.

  Hence we have proved $I$ is a bijective (both edge and vertex) homomorphism, thus $I$ is an isomorphism.
\end{proof}

We close the section with an algorithm for computing the R-core of a graph. In
contrast to graph cores, where the computation is known to be NP-complete,
there is a simple polynomial algorithm for R-cores.

Observe that the R-core of a graph containing isolated vertices is isomorphic
to the disjoint union of the R-core of the same graph with the isolated
vertices removed and a single isolated vertex. The R-core of a graph without
isolated vertices can be computed by Algorithm~\ref{alg:Rcore}.

\begin{algorithm}[htb]
\caption{The R-core of a graph}
\label{alg:Rcore}
\begin{algorithmic}[1]
  \REQUIRE ~~\newline
  Graph $G$ with loops allowed and without isolated vertices,
  vertex set denoted by $V$, neighbourhoods $N_G(i)$, $i\in V$.
  \FOR {$i\in V$}
  \STATE $W(i)= \emptyset$
  \STATE found = FALSE
  \FOR {$j\in V\setminus \{i\}$}
  \IF {$N(j)\subseteq N(i)$}
  \STATE  $W(i):= W(i)\cup N(j)$
  \ENDIF
  \IF {$N(i)\subseteq N(j)$}
  \STATE found = TRUE
  \ENDIF
  \ENDFOR
  \IF {$W(i)=N(i)$ \AND found}
  \STATE delete $i$ from $V$
  \STATE $N(i)= \emptyset$
  \ENDIF
  \ENDFOR
\RETURN  The R-core $G[V]$ of $G$.

\end{algorithmic}
\end{algorithm}

The algorithm removes all vertices $v\in G$ such that (1) the neighbourhood
of $v$ is union of the neighbourhood of some other vertices $v_1,v_2,\dots, v_n$
and (2) there is vertex $u$ such that $N_G(v)\subseteq N_G(u)$.

It is easy to see that the resulting graph $H$ is relationally equivalent
to $G$. Condition (1) ensures the existence of a relation $R_1$ such that
$H\muls R_1=G$, while condition (2) ensures the existence of a relation
$R_2$ such that $G\muls R_2=H$.

We need to show that $H$ is isomorphic to $G_{\text{R-core}}$. By Theorem
\ref{thm:Rcore} we can assume that $G_{\text{R-core}}$ is an induced
subgraph of $H$ that is constructed as an induced subgraph of $G$.

We also know that there are relations $R_1$ and $R_2$ such that
$G_{\text{R-core}}\muls R_1=H$ and $G\muls R_2=G_{\text{R-core}}$. By the
same argument as in the proof of Theorem~\ref{thm:Rcore} we can assume both
$R_1$ and $R_2$ to contain an (restriction of) identity.

Now assume that there is a vertex $v\in V_H\setminus
V_{G_{\text{R-core}}}$. We can put $u=R_2(v)$ and because $R_2$ contains an
identity we have $N_G(v)\subseteq N_G(u)$. We can also put
$\{v_1,v_2\ldots v_n\}$ to be the set of all vertices such that $v\in
R_1(v_i)$. It follows that the neighbourhood of $v$ is the union of
the neighbourhoods of $v_1,v_2,\ldots, v_n$ and by the construction of $H$ we have $v\notin
V_H$, a contradiction.

\newpage
\section{The partial order Rel\texorpdfstring{$(G,H)$}{Rel(G,H)}}
\label{sect:poset}

\subsection*{Basic properties}

For fixed graphs $G$ and $H$ we consider the partial order $\mathrm{Rel}(G,H)$.
The vertices of this partial order are all relations $R$ such
that $G\muls R=H$. We put $R_1\leq R_2$ if and only if $R_1\subseteq R_2$.

This definition is motivated by Hom-complexes, see~\cite{Matousek2003}. In
this section we show the basic properties of this partial order and
concentrate on minimal elements in the special case of $\mathrm{Rel}(G,G)$.

\begin{pro}
  Suppose $G\muls R'=H$, $G\muls R''=H$ and $R'\subseteq R''$, then any
  relation $R$ with $R'\subseteq R\subseteq R''$ also satisfies $G\muls
  R=H$.
\end{pro}
\begin{proof}
  From $R'\subseteq R\subseteq R''$ we conclude $G\muls R'\subseteq G\muls
  R\subseteq G\muls R''$. Hence $G\muls R'=G\muls R''$ implies $G\muls
  R=H$.
\end{proof}

Hence it is possible to describe the partial order $\mathrm{Rel}(G,H)$ by
listing minimal and maximal solutions $R$ of $G\muls R=H$ w.r.t.\ set
inclusion.

For example, if $G$ is $P_3$ with vertices $v_0,v_1,v_2,v_3$ and $H$ is
$P_1$ with vertices $x_0,x_1$, it is easily seen that
$R''=\{(v_0,x_0),(v_2,x_0),(v_1,x_1),(v_3,x_1)\}$ is a maximal solution of
$G\muls R=H$ and $R'=\{(v_0,x_0),(v_1,x_1)\}$ is a minimal solution,
because $R'\subset R''$, then all the relations $R$ with $R'\subseteq
R\subseteq R''$ satisfy $G\muls R=H$. We note that minimal and maximal
solutions need not be unique.

\subsection*{Solutions of \texorpdfstring{$G\muls R = G$}{G* R = G}}

For simplicity, we say that a relation $R$ is an \emph{automorphism}\index{automorphism} of $G$
if it is of the form $R=\{(x,f(x))\mid x\in V_G\}$ and $f:V_G\to V_G$
is an automorphism of $G$.

We will see that conditions related to thinness again play a major role in
this context. Recall that $G$ is point-determining if no two vertices have the same
neighbourhood, in other words, $N_G(x)=N_G(y)$ implies $x=y$. Here we need an even
stronger condition:
\begin{definition}
  A graph $G$ satisfies \emph{condition N}\index{condition N} if $N_G(x)\subseteq N_G(y)$
  implies $x= y$.
\end{definition}
In particular, graphs satisfying condition N are point-determining.

\begin{pro}
  For a given graph $G$, the set $\mathrm{Rel}(G,G)$ of all relations
  satisfying $G\muls R = G$ forms a monoid.
\end{pro}
\begin{proof}
  Firstly, $R,S\in \mathrm{Rel}(G,G)$ implies $G\muls R=G$ and
  $G\muls S=G$ and thus $G\muls (R\circ S)=G$, so that $R\circ S\in
  \mathrm{Rel}(G,G)$. Furthermore, for any $R,S,T\in \mathrm{Rel}(G,G)$ we have $(R\circ S)\circ T=R\circ(S\circ T)$.
Finally, the identity relation $I_G$ is a left and right
  identity for relational composition: $I_G\circ R=R\circ I_G=R$.
\end{proof}

A relation $R\subset V_G\times V_G$ can be interpreted as a directed graph
$\vec R$ with vertex set $V_G$ and a directed edge $u\rightarrow v$ if and
only if $(u,v)\in R$. Note that $\vec R$ may have loops. We say that $v\in
V_G$ is \emph{recurrent}\index{recurrent} if and only if there exists a walk (of length at
least 1) from $v$ to itself. Let $S_G$ be the set of all the recurrent
vertices. Furthermore, we define an equivalence relation $\xi$ on $S_G$ by
setting $(u,v)\in\xi$ if there are both a walk in $\vec R$ from $u$ to $v$
and from $v$ to $u$. The equivalence classes w.r.t.\ $\xi$ are denoted by
$\vec R/\xi=\{D_1,D_2,\dots,D_m\}$. We furthermore define a binary
relation $\geq$ over $\vec R/\xi$ as follows: if there is a walk from a
vertex $u$ in $D_i$ to a vertex $v$ in $D_j$, then we say $u\geq v$. It is
easily seen that $\geq$ is reflexive, antisymmetric, and transitive, hence
$(\vec R/\xi,\geq)$ is a partially ordered set. W.l.o.g.\ we can assume
$\{D_1, D_2,\dots,D_s\}$ are the maximal elements w.r.t.\ $\geq$. Now let
$G_r=G[D_1\cup\dots\cup D_s]$ be the subgraph of $G$ induced by these
maximal elements.

\begin{lemma}\label{lem:aut1}
For arbitrary $x\in V_G$, there exist an $l\in \mathbb{N}$ and a
recurrent vertex $v$ such that $(v,x)\in R^l$, where $R^l$ is $l$-times
composition of relation $R$.
\end{lemma}
\begin{proof}
Set $x_0=x$ and choose $x_i\in R^{-1}(x_{i-1})$ for all $i\ge1$. Since
$|V_G|<\infty$, there are indices  $j,k\in \mathbb{N}$, $j<k$,
$x_j=x_k$. Then $x_j$ is a recurrent vertex. The lemma follows by setting
$l=j$ and $v=x_i$.
\end{proof}

\begin{lemma}\label{lem:aut2}
For every $v\in V_{G_r}$, $R^{-1}(v)\subseteq V_{G_r}$.
\end{lemma}
\begin{proof}
  We prove this by two steps:

  First, we claim that for any $x\in R^{-1}(v)$, $x$ is recurrent.
  Suppose $x\in R^{-1}(v)$ is not recurrent. Lemma~\ref{lem:aut1} implies
  that there is an $l\in \mathbb{N}$ and a recurrent vertex $w$ such that
  $(w,x)\in R^l$. Hence the definitions of $E$ and $\geq$ imply $[w]\geq
  [v]$, where $[v]$\index{$[v]$} denotes the equivalence class (w.r.t.\ $E$) containing
  the vertex $v$. Since $[v]$ is maximal w.r.t.\ $\geq$, we have
  $[v]=[w]$. Consequently, there exists an index $k\in \mathbb{N}$ such
  that $(v,w)\in R^k$. On the other hand, we have $(x,x)=(x,v)\circ
  (v,w)\circ (w,x)\in R^{k+l+1}$. Thus, $x$ is recurrent, a contradiction.
  Therefore, every vertex $x\in R^{-1}(v)$ is recurrent.

  Second, we prove $x$ is a maximal element in the partial order.
  Hence $[x]\geq
  [v]$ together with the maximality of $[v]$ gives $[x]=[v]$, and thus
  $x\in V_{G_r}$.
\end{proof}

\begin{lemma}\label{lem:aut3}
  For every $x\in V_G$, there is $l\in \mathbb{N}$ such that, for arbitrary
  $i\geq l$, there exists $u\in V_{G_r}$ satisfying $(u,x)\in R^i$.
\end{lemma}

\begin{proof}
  From Lemma~\ref{lem:aut1} and Lemma~\ref{lem:aut2} we conclude that it is
  sufficient to show that for an arbitrary recurrent vertex $v$ there is a
  $k\in \mathbb{N}$ and $w\in V_{G_r}$ such that $(w,v)\in R^k$. The lemma
  now follows easily from the finiteness of $V_G$.
\end{proof}

Now we introduce a lemma which connects relation $R$ and the neighbourhood of vertex in a graph.
The lemma is simple but very useful.
\begin{lemma}\label{lem:aut4}
 Suppose $G\muls R = H$. For any $u,v\in V_H$, if $R^{-1}(u)\subseteq R^{-1}(v)$, then $N_H(u)\subseteq N_H(v)$.
\end{lemma}

\begin{proof} 
 Suppose there exist $u,v\in V_H$, $R_{-1}(u)\subseteq R_{-1}(v)$ but $N_H(u)\nsubseteq N_H(v)$,
 then there exist a vertex $w$ of $H$, $w\in N_H(u)$ but $w\notin N_H(v)$. Because $(u,w)$ is an edge
 in $H$ and $G\muls R = H$, we can find $x,y\in V_G$ such that $x\in R_{-1}(w), y\in R_{-1}(u)$ and $(x,y)$ is an edge in $G$.
 $R_{-1}(u)\subseteq R_{-1}(v)$ implies $y\in R_{-1}(v)$. We already have $(x,w)\in R$ and $(x,y)$ is an edge in $G$,
 by the definition of $\muls$ we have $(v,w)$ is an edge of $H$, which contradicts to $w\notin N_H(v)$.
\end{proof}


From these lemmata we can deduce:
\begin{thm} \label{thm:auto1}
All solutions of $G\muls R=G$ are
automorphisms if and only if $G$ has property N.
\end{thm}
\begin{proof}
  Suppose there are distinct vertices $x,y\in V_G$ such that $N_G(x)\subseteq
  N_G(y)$. Then $R=I_G\cup(x,y)$, which is not functional, satisfies $G\muls
  R=G$. Thus $G\muls R=G$ is also solved by relations that are not
  automorphisms of $G$. This proves the `only if' part.

  Conversely, suppose $G$ has property N. \textbf{Claim:} There is a
  $k\in\mathbb{N}$ such that $R^k\cap (V_{G_r}\times V_{G_r})=I_{G_r}$.

  For each $v_i\in V_{G_r}$ there is a walk of length $s_i\ge1$ from $v_i$
  to itself. Hence $(v_i,v_i)\in R^{s_i}$. Let $s$ be the least common
  multiple of all the $s_i$. Then $(v_i,v_i)\in R^s$ for all $v_i\in V_{G_r}$.
  Define $Q:=R^s\cap(V_{G_r}\times V_{G_r})$. Thus $I_{G_r}\subseteq Q$ and
  moreover $Q^j\subseteq Q^{j+1}$ for all $j\in\mathbb{N}$. Since $V_{G_r}$
  is finite, there is an $n\in\mathbb{N}$ such that $Q^{n+1}=Q^{n}$, and
  hence $Q^{2n}=Q^n$. Let us write $R^{-i}(v):=\{u\in V_G\mid (u,v)\in
  R^i\}$. For $v\in V_{G_r}$ we have $R^{-i}(v)\subset V_{G_r}$ (from Lemma~\ref{lem:aut2}) and hence $Q^{-n}(v)=R^{-sn}(v)$ for all $v\in V_{G_r}$.
  If $Q^n\neq I_{G_r}$, then there are two distinct vertices $u,v\in
  V_{G_r}$, such that $(u,v)\in Q^n$. From Lemma~\ref{lem:aut4} $N_G(u)\nsubseteq N_G(v)$ and $G=G\muls
  R^{sn}$ allow us to conclude that $R^{-sn}(u)\nsubseteq R^{-sn}(v)$ and
  $R^{-sn}(v)\nsubseteq R^{-sn}(u)$. Hence, there is a vertex $w$, such
  that $(w,u)\in Q^n$ and $(w,v)\notin Q^n$. From $(u,v)\in Q^n$ and
  $(w,u)\in Q^n$ we conclude $(w,v)\in Q^n\circ Q^n=Q^{2n}$, contradicting
  to $Q^{2n}=Q^n$. Therefore $Q^n=I_{G_r}$. Setting $k=sn$ now implies the
  claim.

  Finally, we show $V_{G_r}=V_G$. For any $v\in V_G\setminus V_{G_r}$,
  Lemma~\ref{lem:aut3} implies the existence of $w\in V_{G_r}$ and $m\in
  \mathbb{N}$ such that $(w,v)\in R^{mk}$. However, we have claimed
  $R^{-k}(w)=\{w\}$, hence $R^{-mk}(w)=\{w\}$. From Lemma~\ref{lem:aut4}, this, however, implies
  $N_G(w)\subseteq N_G(v)$ and thus contradicts property N. Therefore,
  $V_G=V_{G_r}$ and moreover $R^k=I_G$. This $R$ is an automorphism.
\end{proof}

\section{R-Retractions}
\label{sect:retraction}

Recall that an important special case of ordinary graph homomorphisms are
homomorphisms to subgraphs, and in particular so-called retractions. Let
$H$ be a subgraph of $G$, a \emph{retraction}\index{retraction} of $G$ to $H$ is a
homomorphism $r: V_G\rightarrow V_H$ such that $r(x)=x$ for all $x\in
V_H$.

We introduced the graph cores in Section~\ref{sect:retract} as minimal
representatives of the homomorphism equivalence classes. The classical and
equivalent definition is the following: A \emph{(graph) core}\index{core} is a graph
that does not retract to a proper subgraph. Every graph $G$ has a unique
core $H$ (up to isomorphism), hence one can speak of $H$ as \emph{the core
  of}\index{the core} $G$, see~\cite{Hell2004}.

Here, we introduce a similar concept based on relations between graphs.
Again to obtain a structure related to graph homomorphisms, in this section
we require all relations to have full domain unless explicitly stated
otherwise.

\begin{definition}
  Let $H$ be a subgraph of $G$. An \emph{R-retraction}\index{R-retraction} of $G$ to $H$ is a
  relation $R$ such that $G\muls R=H$ and $(x,x)\in R$ for all $x\in
  V_H$. If there is an R-retraction of $G$ to $H$ we say that $H$ is a
  \emph{retract}\index{retract} of $G$.
\end{definition}

\begin{lemma}\label{lem:ret}
  If $H$ is an R-retract of $G$ and $K$ is an R-retract of $H$, then $K$ is
  an R-retract of $G$.
\end{lemma}

\begin{proof}
  Suppose $T$ is an R-retraction of $H$ to $K$ and $S$ is an
  R-retraction of $G$ to $H$. Then $(G\muls S)\muls T=G\muls (S\circ
  T)=K$. Furthermore $(x,x)\in T$ for all $x\in V_K\subseteq V_H$, and
  $(u,u)\in S$ for all $u\in V_H$, hence $(x,x)\in S\circ T$ for all $x\in
  V_k$. Therefore $S\circ T$ is an R-retraction from $G$ to $K$.
\end{proof}

Hence, the following definition is meaningful.
\begin{definition}
  A graph is \emph{R-reduced}\index{R-reduced} if there is no R-retraction to a proper
  subgraph.
\end{definition}
Thus, we can also speak about `the R-reduced graph of a graph $G$' as the
smallest subgraph on which it can be retracted. We will see below that the
R-reduced graph of a graph is always unique up to isomorphism.

We remark that R-reduced graphs differ from R-cores introduced
in Section~\ref{sect:retract}, thus we chose an alternative name used
also in homomorphism setting (cores are also called reduced graphs).

\begin{lemma}
  Let $G$ be a graph with loops and $o$ a vertex of $G$ with a loop on
  it. Then the R-reduced graph of $G$ is the subgraph induced by $\{o\}$.
\end{lemma}

\begin{proof}
  Let $O$ be the graph induced by $\{o\}$, and $R=\{(x,o)\mid x\in V_G\}$, then
  it is easily seen $R$ is an R-retraction of $G$ to $O$. Moreover, since
  $O$ has only one vertex, there is no R-retraction to its
  subgraphs. So $O$ is an R-reduced graph of $G$.

  Conversely, let $H$ be an R-reduced graph of $G$ and denote by $R$ the
  R-retraction from $G$ to $H$. Then a loop of $G$ must generate a loop of
  $H$ via $R$, denote it by $O$\index{$O$}. Similarly to above, we see $O$ is an
  R-retract of $H$, hence it is also an R-retract of $G$ (by Lemma~\ref{lem:ret}).
  Therefore the definition of R-reduced graph implies
  $H=O$.
\end{proof}

In the remainder of this section, therefore, we will only consider graphs
without loops.

\begin{lemma} \label{lem:core}
  If $G$ is R-reduced, then $G$ has property N.
\end{lemma}
\begin{proof}
  Suppose there are two distinct vertices $x,y\in V_G$ with $N_G(x)\subseteq
  N_G(y)$ and consider the induced graph $G/x:=G[V_G\setminus\{x\}]$ obtained
  from $G$ by deleting the vertex $x$ and all edges incident with $x$. The
  relation $R=\left\{(z,z)\mid z\in V_G\setminus\{x\}\right\}\cup\{(x,y)\}$
  satisfies $G\muls R=G/x$: the first part is the identity on $G/x$ and
  already generates all necessary edges in $G/x$. The second part
  transforms edges of the form $(x,z)\in E_G$ to edges $(y,z)$. Since $R$
  has full domain and contains the identity relation restricted to
  $G/x$, it is an R-retraction of graph $G$, and hence $G$ is not R-reduced.
\end{proof}

\begin{pro} \label{pro:core}
  A graph $G$ is R-reduced if and only if it has no relation to a proper
  subgraph.
\end{pro}

\begin{proof}
  The `if' part is trivial. Now we suppose that $H$ is a proper induced
  subgraph of graph $G$ with the minimal number of vertices such that there
  is a relation $R$ satisfying $G\muls R=H$. Then $H$ does not have a
  relation to a proper subgraph of itself. We claim that $H$ has property
  N; otherwise, one can find a vertex $u\in V_H$ and construct a retraction
  from $H$ to $H/{u}$ as in Lemma~\ref{lem:core}, which causes a
  contradiction. Denote $\widetilde{R}=R\cap (V_H\times V_H)$, then
  $K=H\muls\widetilde{R}$ is a subgraph of $H$. From our assumptions on $H$ we
  obtain $K=H$. By virtue of Theorem~\ref{thm:auto1}, $\widetilde{R}$ is
  induced by an automorphism of $H$. Hence $R\circ \widetilde{R}^{+}$ is again
  a relation of $G$ to $H$ that contains the identity on $H$, i.e., it is
  an R-retraction.
\end{proof}

Since graph cores are induced subgraphs and retractions are surjective they
also imply relations. Proposition~\ref{pro:core} is also a consequence of
this fact. We refer to~\cite{Hell2004} for a formal proof.

\begin{figure}[!ht]
\centering
   \begin{tabular}{cc}
    \includegraphics[width=0.5\textwidth]{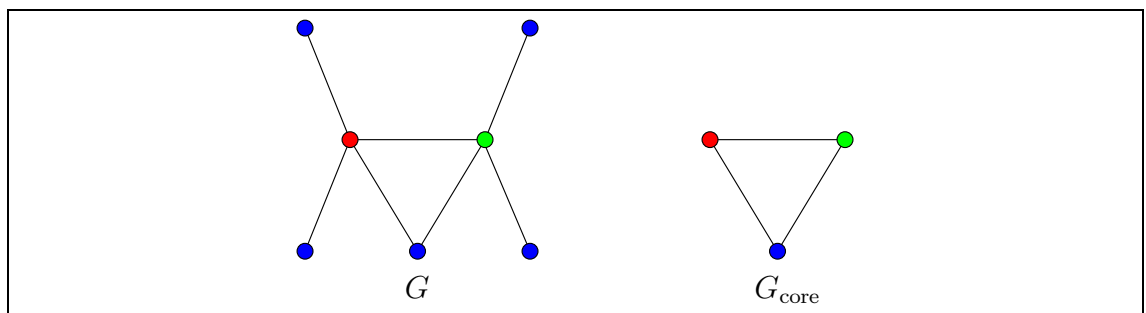} \\
    \hspace{0.85cm}   $G$ \hspace{4cm} $G_\text{core}$
    \end{tabular}
\caption{A graph $G$ and its core.}
\label{fig:core}
\end{figure}

We call $R$ an \emph{minimal R-retraction}\index{minimal R-retraction} if there is no
R-retraction $R'$ such that $R\supset R'\supset I_H$.

\begin{lemma} \label{fun}
  Let $H$ be an R-retract of $G$. Then any minimal
  R-retraction of $G$ to $H$ is functional.
\end{lemma}
\begin{proof}
  Suppose $R$ is a minimal R-retraction of $G$ to $H$. If $R$ is not
  functional, then there exist distinct $x,y\in V_H$ such that
  $(u,x),(u,y)\in R$. Hence we could always pick a vertex from $\{x,y\}$
  which is different of $u$, w.l.o.g.\ suppose it is $x$. Then $R\setminus(u,x)$ is
  an R-retraction, which contradicts minimality. To see this, set
  $R'=R\setminus(u,x)$, then $R\supset R'\supset I_H$ and moreover $H=G\muls I_H
  \subseteq G\muls R'\subseteq G\muls R=H$, and thus $G\muls R'=H$.
\end{proof}

\begin{pro}
  A graph is R-reduced if and only if it is a graph core.
\end{pro}
\begin{proof}
  If $H$ is R-reduced from $G$ there is an R-retraction from $G$ to $H$
  which can be chosen minimal and hence by Lemma~\ref{fun} is functional
  and hence is a homomorphism retraction. Conversely, a homomorphism
  retraction is also an R-retraction. Hence the R-reduced graphs coincide
  with the graph cores.
\end{proof}

\begin{pro}
  Suppose $H$ is the core of graph $G$. If $H\muls R=K$ then there is a
  relation $R'$ such that $G\muls R'=K$. If $K\muls S=G$, then there is a
  relation $S'$ such that $K\muls S'=H$.
\end{pro}
\begin{proof}
  Since $H$ is the core of graph $G$, there is a relation $R_1$ such that
  $G\muls R_1=H$. If $H\muls R=K$ we have $G\muls R_1\muls R=K$ and
  $R'=R_1\circ R$ satisfies $G\muls R'=K$. If $K\muls S=G$ we have
  $K\muls S\muls R_1=H$ and $S'=S\circ R_1$ satisfies $K\circ S'=G$.
\end{proof}

\subsection*{Cocores}

In the classical setting of maps between graphs, one can only consider
retractions from a graph to its subgraphs, since graph homomorphisms of an
induced subgraph to the original graph are just the identity map. In the
setting of relations between graphs, however, it appears natural to
consider relations with identity restriction between a graph and an induced
subgraph. This gives rise to the notions of R-coretraction and R-cocore in
analogy with R-retractions and R-reduced graphs.

\begin{definition}
  Let $H$ be a subgraph of graph $G$. An \emph{R-coretraction}\index{R-coretraction} of $H$ to $G$ is
  a relation $R$ such that $H\muls R=G$ and $(x,x)\in R$ for all $x\in
  V_H$. We say that $H$ is an \emph{R-coretract}\index{R-coretract} of $G$.
\end{definition}

\begin{lemma}
  If $H$ is an R-coretract of a graph $G$ and $K$ is an R-coretract of $H$, then
  $K$ is an R-coretract of $G$.
\end{lemma}
\begin{proof}
  Suppose $T$ is an R-coretraction of $K$ to $H$ and $S$ is an
  R-coretraction of $H$ to $G$. Then $(K\muls T)\muls S = K\muls(T\circ
  S)=G$. Furthermore $(x,x)\in T$ for all $x\in V_K\subseteq V_H$, and
  $(v,v)\in S$ for all $v\in V_H$, hence $(x,x)\in T\circ S$ for all $x\in
  V_K$. Therefore $T\circ S$ is an R-coretraction from $K$ to $G$.
\end{proof}

Hence, the following definition is meaningful.
\begin{definition}
  An R-coretract $H$ of a graph $G$ is an \emph{R-cocore of $G$}\index{R-cocore} if $H$
  does not have a proper subgraph that is an R-coretract of $H$ (and hence
  of $G$).
\end{definition}

\begin{figure}[!ht]
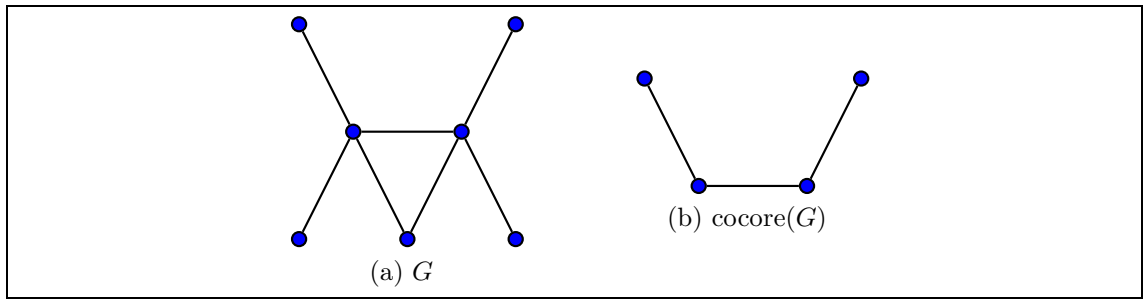

\centering
\subfloat[$G$]{
\begin{minipage}[c]{0.3\textwidth}
\centering
\includegraphics[width=0.7\textwidth]{core1.eps}
\end{minipage}}
\subfloat[$\mathrm{cocore}(G)$ ]{
\begin{minipage}[c]{0.3\textwidth}
\centering
\includegraphics[width=0.7\textwidth]{cocore1.eps}
\end{minipage}}
 \caption{A graph and its cocore.}
  \label{fig:cocore}
\end{figure}

Clearly, the reference to $G$ is irrelevant: A graph $G$ is an
\emph{R-cocore}\index{R-cocore} if there is no proper subgraph of $G$ that is an
R-coretract of $G$. Similarly, we call $R$ a \emph{minimal
  R-coretraction}\index{minimal R-coretraction} of $H$ to $G$ if there exists no R-coretraction $R'$,
such that $R'\subset R$.

\begin{lemma}\label{lem:cocore}
  Let $H$ be an R-coretract of graph $G$, and let $R$ be a minimal
  R-coretraction of $H$ to $G$. Then the restriction of $R$ to $H$ equals to
  $I_H$.
\end{lemma}
\begin{proof}
  Suppose $R\cap (V_H\times V_H)\neq I_H$ and define $R_1=R\setminus
  \{(x,y)\in R:x,y\in V_H,x\neq y\}$. Then $H\muls R_1\subseteq H\muls
  R=G$. We claim that $H\muls R_1=H\muls R$ and thus $R_1$ is an
  R-coretraction of $H$ to $R$, contradicting the minimality of $R$.

  To prove this claim, it is sufficient to show that any edge $e\in E_G$ is
  contained in $H\muls R_1$. If $e$ is not incident with any vertex in
  $V_H$ or $e\in E_H$, the conclusion is trivial. So we only need to
  consider $e=(z,u)$ with $z\in E_H$ and $u\in V_G\setminus V_H$. Since
  $G=H\muls R$, one can find $x_1,x_2\in V_H$ such that $(x_1,z),(x_2,u)\in
  R$ and $(x_1,x_2)\in E_H$. Because $H\subseteq H\muls\big(I_H\cup
  (x_1,z)\big)\subset H\muls\big(R\cap (V_H \times V_H)\big)=H$, we get
  $N_H(x_1)\subseteq N_H(z)$. It follows that $(z,x_2)\in E_H$ and
  hence $e=(z,u)\in G\muls R_1$.
\end{proof}

Like R-reduced graphs, R-cocores also satisfy a stringent condition on their
neighbourhood structure.

\begin{definition} \label{nei}
  A graph $G$ satisfies property N* if, for every vertex $x\in V_G$,
  there is no subset $U_x\subseteq V_G\setminus\{x\}$ such that
  \begin{equation*}
    N_G(x)=\bigcup_{y\in U_x} N_G(y)
  \end{equation*}
\end{definition}
In other words, no neighbourhood can be represented as the union of
neighbourhoods of other vertices of graph $G$.

\begin{pro} \label{cocore} $G$ is an R-cocore if and only if $G$ has
  property N*.
\end{pro}
\begin{proof}
  Consider a vertex set $U_x$ as in Definition~\ref{nei} and suppose
  that there is a vertex $x\in V_G$ such that $N_G(x)=\bigcup_{y\in U_x}
  N_G(y)$. Then the relation $R:=I\setminus(x,x)\cup\{(y,x):y\in U_x\}$ is an
  R-coretraction from $G/x$ to $G$. Thus $G$ is not an R-cocore.

  Conversely, suppose that $G$ is not an R-cocore, let $H$ be a coretract
  of $G$, and denote by $R$ a minimal R-coretraction of $H$ to $G$. Then,
  by Lemma~\ref{lem:cocore}, $R\cap (V_H\times V_H)=I_H$. Consider a
  vertex $v\in V_G\setminus V_H$ and set
  $R^{-1}(v)=\{x_1,\dots,x_i\}$. Then $N(v)=\bigcup_i N(x_i)$, contradicting
  property N*.
\end{proof}

\begin{lemma} \label{thin-cocore1}
 An R-cocore of graph $G$ is isomorphic to an R-cocore of $G_{\text{pd}}$.
\end{lemma}
\begin{proof}
 Suppose graph $H$ is an R-cocore of graph $G$, then there exists a coretraction $R$
 such that $H\muls R = G$. By the definition of $G_{\text{pd}}$, we know
 there exists an R-retraction $R'$ such that $G\muls R'=G_{\text{pd}}$.
 Then $H\muls (R\circ R') = G$. It is easily seen that $R\circ R'$ is a coretraction.
 Also $H$ is the smallest subgraph of $G_{\text{pd}}$ which has a coretraction
 to $G_{\text{pd}}$, otherwise, suppose there are smaller graph $K$ such that $K$
 has a coretraction to $G_{\text{pd}}$, then $K$ also has a coretraction to $G$,
 which contradicts to the assumption that $H$ is an R-cocore.
 Therefore, $H$ is an R-cocore of $G_{\text{pd}}$.

 Conversely, if $H$ is an R-cocore of graph $G_{\text{pd}}$, then it is also
 an R-cocore of $G$.
\end{proof}

We call a subset $B$ of $V_G$ a \emph{basis}\index{basis} of graph $G$, if for any $x\in V_G$, there
 exists a subset $U_x$ of $B$ such that $N_G(x)=\bigcup_{y\in U_x} N_G(y)$. Of course
 a basis of a graph always exists, because $V_G$ is a basis of $G$. We call a basis of graph $G$
 with minimal cardinality a \emph{minimal basis}\index{minimal basis}. Because we consider only finite graphs,
 a minimal basis always exists. Now we prove that minimal basis of a point-determining graph $G$ is unique.

\begin{lemma} \label{basis}
 Minimal basis of a point-determining graph $G$ is unique.
\end{lemma}

\begin{proof}
  Assume there are two distinct minimal basis $A,B$. Neither
  contains the other by their minimality. Because $A,B$ are distinct,
  there exists an element $a\in A$ and $a\notin B$.

  Since $B$ is a basis of graph $G$, by the definition, there is a set $U_a\subset B$ such that
  $N_G(a)=\bigcup_{b\in U_a} N_G(b)$, thus $|N_G(a)|\geq |N_G(b)|$
  for all $b\in U_a$. Since $G$ is a point-determining graph, any two vertices of $G$ do not
  have the same neighbourhood, thus $|N_G(a)|>|N_G(b)|$
  for all $b\in U_a$. Because $A$ is also a basis of graph $G$, then for any $b\in U_a$,
  $N_G(b)$ can be represented as the union of the neighbourhood of some elements of $A\setminus\{a\}$,
  $a$ is not contained since $|N_G(a)|>|N_G(b)|$. Thus $N_G(a)=\bigcup_{b\in U_a} N_G(b)$ can
  be represented as the union of the neighbourhood of some elements of $A\setminus\{a\}$, which contradicts
  the minimality of $A$.
\end{proof}

\begin{lemma} \label{thin-cocore2}
R-cocore of a point-determining graph is unique.
\end{lemma}

\begin{proof}

 From Lemma~\ref{basis}, we know that the basis of a point-determining graph is unique. We claim that the vertex set of
 any R-cocore of graph $G$ is the minimal basis of $G$. First,
 if $H$ is an R-cocore, then $V_H$ is a basis of graph $G$. Because by the definition of R-cocore,
 there is a coretraction $R$ such that $H\muls R = G$. For any vertex $x\in V_G$, we assign
 $U_x=R^{-1}\subset V_H$, it is easily seen that $N_G(x)=\bigcup_{y\in U_x} N_G(y)$.

 Then we prove $V_H$ is the minimal basis. Suppose there exists a subset $K$ of $V_G$ is a basis,
 and the cardinality of $K$ is smaller than $V_G$. By the definition of basis, for any $x\in V_G$
 we fix the set $U_x$, then there is an R-coretraction from $H$ to $G$, induced by $R(x)=\{y\in V_G\mid x\in U_y\}$.
 That is because for any vertex $x\in K$, we have $U_x=\{x\}$. It contradicts to the assumption that
 $H$ is a cocore.

 Therefore the R-cocore of a point-determining graph is unique.
\end{proof}

Because $G_{\text{pd}}$ is a point-determining graph, it follows from Lemma~\ref{thin-cocore2} that $G_{\text{pd}}$
has unique R-cocore for any graph $G$. And from lemma~\ref{thin-cocore1} we immediately get the
following proposition.

\begin{pro}
  R-cocore of $G$ is unique up to isomorphism.
\end{pro}

These results allow us to construct an algorithm that computes the cocore
of given graph $G$ in polynomial time. First observe that the cocore of a
graph $G$ that contains isolated vertices is the disjoint union of the cocore
of the graph $G'$ obtained from $G$ by removing isolated vertices and the
graph consisting of a single isolated vertex. It is thus sufficient to
compute cocores for graphs without isolated vertices in Algorithm
\ref{alg1}.

\begin{algorithm}[htb]
\caption{The cocore of a graph}
\label{alg1}
\begin{algorithmic}[1]
  \REQUIRE ~~\\
  Graph $G$ with loops and without isolated vertices specified by its
  vertex set $V$ and the neighbourhoods $N_G(i)$, $i\in V$.
  \FOR{$i\in V$}
    \STATE $W(i)= \emptyset$
    \FOR {$j\in V\setminus \{i\}$}
      \IF {$N(j)\subseteq N(i)$}
        \STATE  $W(i):= W(i)\cup N(j)$
      \ENDIF
     \ENDFOR
    \IF {$W(i)=N(i)$}
      \STATE delete $i$ from $V$
      \STATE $N(i)= \emptyset$
    \ENDIF
  \ENDFOR
\RETURN $G[V]$, the cocore of $G$.
\end{algorithmic}
\end{algorithm}

\begin{pro}
  Suppose $H$ is a cocore of $G$. If $K\muls R=H$, then there is a relation
  $R'$ such that $K\muls R'=G$. If $G\muls S=K$, then there is a relation
  $S'$ such that $H\muls S'=K$.
\end{pro}
\begin{proof}
  Since $H$ is a cocore of $G$, there exists an R-coretraction $R_1$ such
  that ${H\muls R_1=G}$. If $K\muls R=H$, then letting $R'=R\circ R_1$ implies
  $K\muls R'=G$. If $G\muls S=K$, we have $H\muls R_1\muls S=K$. Let
  $S'=R_1\circ S$, then $H\muls S'=K$.
\end{proof}

We can find a graph which is not an R-core, and meanwhile it does not satisfy the
Hall condition. To see that, consider the graph $G$ in the example of Fig.~\ref{fig:weak}.
Obviously $R=\{(1,1),(2,2),(3,2),(4,4),(6,6),(5,3),(7,3),(5,5),(7,7)\}$
satisfies $G\muls R=G$. For the subset $S=\{2,3\}$ of $V(G)$, $R(S)=\{2\}$,
so $|R(S)|<|S|$, it does not satisfy the Hall condition. But not every graph which
are not R-cores do not satisfy the Hall condition. That means the condition of
Corollary~\ref{cor:nohall2} is not sufficient. To see this, we
consider a graph consisting of two independent isolated vertices.

Note that relations from an R-core or R-cocore to itself also satisfy the Hall condition.

\begin{pro}
If $G$ is an R-core, and $R$ is a relation satisfying $G\muls R = G$,
then any monomorphism $f\subset R$ satisfies $G\muls f = G$.
\end{pro}

\begin{proof}
From corollary~\ref{cor:nohall2} we know that $R$ which is a relation from
an R-core to itself satisfies the Hall condition
thus it contains a monomorphism. Suppose  $G\muls f \neq G$, then $G\muls f \subset G$.
By iteration we have $G\muls f^m \subset G$ for any $m$.
On the other hand, $f$ is a permutation, then there exists a $n$ such that $f^n$ is identity, thus
$G\muls f^n = G$. A contradiction.
\end{proof}




Without the constraint of $R$ with full domain, we have:
\begin{pro}
All solutions of $G\muls R =G$ satisfy the Hall condition if and only if
$G$ is an R-cocore.
\end{pro}

\begin{proof}
If $G$ is an R-cocore, then all the solutions of $G\muls R = G$ are with full domain.
Note that $G$ is an R-cocore thus it is also an R-core,
from Corollary~\ref{cor:nohall2} we know that all solutions $R$ satisfy the
Hall condition.

Conversely, if all solutions satisfy the Hall condition, then all the solutions have full domains,
Hence $G$ is an R-cocore, otherwise we can find an induced subgraph $H$ of $G$ and a relation $R$
satisfying $H\muls R =G$. This relation $R$ is also the solution of $G\muls R = G$ in general (without constraint of
full domain) case.
\end{proof}


Actually, it is very easy to get the following result which characterize the cases when all the solutions satisfy the Hall condition:
\begin{pro}
 If $G\muls R = G$, and all solutions satisfy the Hall condition, then either of the two cases happens:
\begin{enumerate}
 \item $G$ is an $R$-core;
 \item There is a solution contains a relation with non-full domain.
\end{enumerate}
\end{pro}

\subsection*{Inclusion relation}

Figure~\ref{fig:inclu} shows the inclusion relation of cores, R-cores, R-cocores, graphs with property N and point-determining graphs: the set of cores is a proper subset of the set of graphs with property N, which is a proper subset of R-cocores. The latter is a proper subset of R-cocores, the set of R-cores is a proper subset of the set of point-determining graphs.

  {\floatstyle{plain}\restylefloat{figure}%
   \begin{figure}[ht!]
   \centering
  \includegraphics[width=0.6\textwidth]{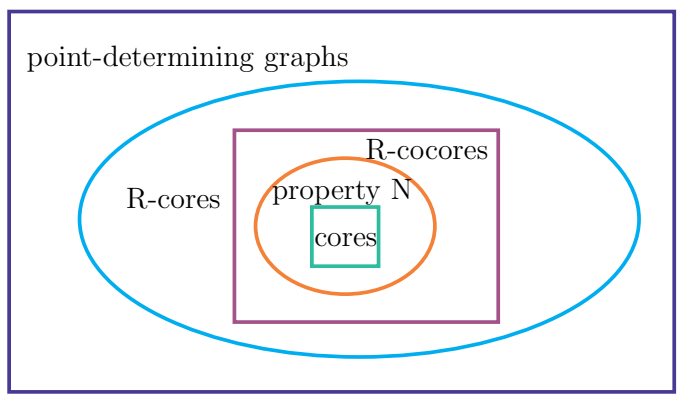} 
\caption{Inclusion digram}
\label{fig:inclu}
   }{\end{figure}}

\begin{example}
We also give a list of examples to illustrate the proper inclusion relations, some of the examples have been already shown or explained before:
\begin{enumerate}
\item $K_3$ is a core, then of course it has property N and is an R-core, R-cocore and point-determining graph.
\item $C_6$ has property N but is not a core. Its core is $P_2$.
\item $P_4$ is an R-cocore but not with property N.
\item The graph $H$ in Figure~\ref{fig:glo} is an R-core but not an R-cocore, its R-cocore is $P_4$.
\item The graph in Figure~\ref{figa:weak} is a point-determining graph but not an R-core.
\end{enumerate}
\end{example}
%
%
%
%
%
%
%

\section{Computational complexity}
\label{sect:complexity}

In this section we briefly consider the complexity of computational
problems related to graph homomorphisms.
For the readers who are not familiar with computational complexity,
please refer to Appendix A. The \emph{homomorphism problem}\index{homomorphism problem}
$\Homo(H)$\index{$\Homo(H)$} takes as input some finite $G$ and asks whether there is a
homomorphism from $G$ to $H$. The computational complexity of the
homomorphism problem is fully characterized. It is known that $\Homo(H)$ is
NP-complete if and only if $H$ has no loops and contains odd cycles. All
the other cases are polynomial, see~\cite{Hell2004}.

The analogous problem for relations between graphs can be phrased as
follows: The \emph{(full) relation problem}\index{full relation problem}\index{relation problem} $\FulRel(H)$\index{$\FulRel(H)$} takes as input
some finite $G$ and asks whether there is a relation with full domain from
$G$ and asks whether there is a relation from $G$ to $H$. We show that
this problem can be easily converted to a related problem on surjective
homomorphisms. The \emph{surjective homomorphism problem}\index{surjective homomorphism problem} $\sHom(H)$\index{$\sHom(H)$}
takes as input some finite $G$ and asks whether there is a surjective
homomorphism from $G$ to $H$.

Let $\leq_\text{P}^\text{Tur}$ indicate polynomial time Turing reduction.

\begin{thm} \label{thm:complex}
For finite $H$ our relation problem sits in the following relationship.
\begin{equation}
  \Homo(H)\leq_\text{P}^\text{Tur} \FulRel(H)
  \leq_\text{P}^\text{Tur} \sHom(H)\,.
\end{equation}
\end{thm}

\begin{proof}
  First we show that $\Homo(H)$ is polynomially reducible to $\FulRel(H)$.
  If there is a homomorphism from $G$ to $H$, then there is also a surjective
  homomorphism from $G+H$ to $H$. On the other hand, suppose $G$ has no
  homomorphism to $H$, from Lemma~\ref{reho} we conclude that $G+H$ has no
  relation to $H$ since $G$ has no relation to $H$.

  Now we show that the relation problem $\FulRel(H)$ is polynomially reducible to $\sHom(H)$.
  From Corollary~\ref{DeCo} we know $G\muls R = H$ if and only if there is
  a graph $G'=G\muls R_D$ which has a full homomorphism to $G$ and has a
  surjective homomorphism to $H$.

  We construct $G''$, by duplicating all the vertices of $G$ precisely
  $|V_H|$ times. We claim that $G'$ exists if and only if
  $G''$ has a full homomorphism to $G$. We can put
  $G'=G''$ because the surjective homomorphism can easily undo the
  redundant duplications. This gives the polynomial reduction from $\FulRel(H)$ to
  $\sHom(H)$.
\end{proof}

To our knowledge, $\sHom(H)$ is not fully classified. A recent survey of
the closely related complexity problem concerning the existence of vertex
surjective homomorphisms~\cite{Bodirsky2012} provides some arguments why the
characterization of complexity is likely to be hard. We observe that the
existence of a homomorphism from $G$ to $H$ is equivalent to the existence
of a surjective homomorphism from $G+H$ to $H$. Thus $\sHom(H)$ is clearly
hard for all graphs for which $\Homo(H)$ is hard, i.e., for all loop-less
graphs with odd cycles.

Testing the existence of a homomorphism from a fixed $G$ to $H$ is
polynomial (there is only a polynomial number $|V_H|^{|V_G|}$ of possible
functions from $G$ to $H$). Similarly the existence of a relation from a fixed
$G$ to $H$ is also polynomial. In fact, an effective algorithm exists.
For fixed $G$ there are finitely many point-determining graphs $T$ which $G$ has relation
to. The algorithm thus first constructs the point-determining graph of $H$ and then, using
a decision tree recognizes all isomorphic copies of all point-determining graphs $G$ has 
relation to.

\section{Weak relational composition}
\label{sect:weak}

In this section we will briefly discuss the `loop-free' version, i.e.,
equations of the form $G\mulw R=H$.

Most importantly, there is no simple composition law analogous to
Lemma~\ref{lem:compo}. The expression
\begin{equation*}
(G\mulw R)\mulw S = (S^+\circ (R^+\circ G\circ R)^{\iota}\circ S)^\iota
\end{equation*}
does not reduce to relational composition in general. For example, let
$G=K_3$ with vertex set $V=\{x,y,z\}$ and consider the relations
$R=\{(x,1),(z,1),(y,2)\}\subseteq \{x,y,z\}\times\{1,2\}$ and
$S=\{(1,x')(1,z')(2,y')\}\subseteq \{1,2\}\times \{x',y',z'\}$. One can easily
verify
\begin{equation*}
(G\mulw R)\mulw S = P_2 \ne G\mulw(R\circ S)= K_3
\end{equation*}
The most important consequence of the lack of a composition law is that
R-retractions cannot be meaningfully defined for the weak
composition. Similarly, the results related to R-equivalence heavily rely
on the composition law.

Nevertheless, many of the above results, in particular basic properties derived
in Section~\ref{sect:basic}, remain valid for the weak composition
operation. As the proofs are in many cases analogous, we focus here mostly
on those results where strong and weak composition differ, or where we need
different proofs. In particular, Lemma~\ref{dec} also holds for the weak
composition. Thus, we still have a result similar to Corollary~\ref{DeCo},
but the proof is slightly different.

\begin{cor}
Suppose $G\mulw R=H$. Then there is a set $C$, an injective relation
$R_D\subseteq \domain R\times C$, and a surjective relation
$R_C\subseteq C\times \image R$ such that
$G[\domain R]\mulw R_D\mulw R_C=H[\image R]$.
\end{cor}
\begin{proof}
From Proposition~\ref{dec} we know $R=I'\circ R_D\circ R_C\circ I''$.
And we know $G[\domain R]\mulw R_D=G[\domain R]\muls R_D$.
From the properties of $\mulw$, we have
\begin{align*}
G[\domain R]\mulw R &=(R^+\circ G[\domain R]\circ R)^l\\
&=((R_D\circ R_C)^+\circ G[\domain R]\circ R_D\circ R_C)^l\\
&=(R_C^+\circ R_D^+\circ G[\domain R]\circ R_D\circ R_C)^l\\
&=(R_C^+\circ (R_D^+\circ G[\domain R]\circ R_D)\circ R_C)^l\\
&=(R_C^+\circ G[\domain R]\muls R_D\circ R_C)^l\\
&=(R_C^+\circ G[\domain R]\mulw R_D\circ R_C)^l\\
&=G[\domain R]\mulw R_D\mulw R_C\\
&=H[\image R].\qedhere
\end{align*}
\end{proof}

Assume $G\mulw R=H$ and let $H_1,\dots,H_k$ be the connected components of
$H$. From the definition of $\mulw$ and $\muls$, if we denote
$\widetilde{H}=G\muls R$, then $\widetilde{H}$ could be decomposed into the union
of connected components $\widetilde{H}_i$($1\leq i\leq k$), such that
$(\widetilde{H}_i)^\iota=H_i$. Hence the conclusion of the proposition
\ref{ccs} also holds true for weak relations.

Lemma~\ref{colour} does not hold for weak relations. For example, there is a
weak relation of $K_5$ to $K_3$, but $\chi(K_5)=5>\chi(K_3)=3$.

Lemma~\ref{path} and Lemma~\ref{cycle} do not hold for weak relations. For
example, if $G$ is a graph consisting of a single isolated vertex isolated,
then $P_3\mulw R=G$ and $C_3\mulw R=G$, but there are no walk in $G$.

With respect to complete graphs, weak relational composition also behaves
differently from strong composition. If $K_k\mulw R=H$ then $R(i)$ can
contain more than one vertex in $V_H$. Comparing to Proposition
\ref{prop:complete}, we also obtain a different result:

\begin{thm}
  There is a relation $R$ such that $K_k\mulw R = H$ if and only if every
  connected component of $\overline{H}$ is a complete graph, and the number
  of connected components of $\overline{H}$ containing at least 2 vertices
  is at most $k$.
\end{thm}

\begin{proof}
  If every connected component of $\overline{H}$ is a complete graph,
  we denote the vertex sets of the connected components containing at least
  $2$ vertices by $H_1,\dots, H_m$, $m\leq k$ and the vertices of $K_k$ by
  $1,\dots,k$. Let $R={\{(i,u)\mid i=1,\dots,k,u\in V_{H_i}\}}\cup
  \{(j,v):1\leq j\leq k,v\in V_H\setminus \bigcup_{i=1}^m V_{H_i}\}$. One
  easily checks that $K_k\mulw R=H$.

  Conversely, let $R$ be a relation satisfying $K_k\mulw R=H$. Consider
  the set $U_i=\{u\in V_H\mid R^{-1}(u)=\{i\}\}$. Then $u$ and $v$ are not
  adjacent for arbitrary $u,v\in U_i$, while $u$ is adjacent to $w$ for
  every $w\in V_H\setminus U_i$. Hence
  $\overbar{H}(U_i)$ is a connected component of $\overbar{H}$, which is
  also a complete graph. Given $w\in V_H\setminus \bigcup_{i=1}^m U_i$,
  $R^{-1}(w)$ must have at least 2 vertices in $K_k$, hence $w$ is adjacent
  to every vertex in $H$ except itself; in other words, $w$ is an
  isolated vertex in $\overbar{H}$. Therefore the number of connected
  components of $\overbar{H}$ containing at least 2 vertices is no more
  than $k$.
\end{proof}

The results in subsection~\ref{revers} also remain true for weak relations.




\chapter[Constrained Homomorphisms Orders]{\fontsize{28}{10}\selectfont Constrained Homomorphisms Orders} 
\label{Ch:order}
\ifpdf
    \graphicspath{{Chapter4/Chapter4Figs/PNG/}{Chapter4/Chapter4Figs/PDF/}{Chapter4/Chapter4Figs/}}
\else
    \graphicspath{{Chapter4/Chapter4Figs/EPS/}{Chapter4/Chapter4Figs/}}
\fi

While homomorphism orders have been a fruitful research topic for several decades,
relatively little effort has been put into the study of orders induced by constrained homomorphisms. The aim of this chapter is to establish some of the properties of these orders and their similarities and differences. We survey existing results and provide several new results about cores, universality, gaps and dualities for constrained homomorphisms. 

Due to the nature of our topic, many of the results turn out to be simple observations. With a systematic analysis of related orders we can often establish the properties of more complex structures by using the properties of easier structures. We give several non-trivial results.

We spend the first two sections developing new methods that we shall use throughout the chapter.
In Section~\ref{sec:universality} we introduce notions of past-finite-universal and future-finite-universal orders, and provide a practical way to prove past-finite universality or future-finite universality of orders and then from these to build universality orders.
In Section~\ref{sec:pastfinite} we give a standard approach for the characterization of density, gaps and duality of past-finite or future-finite orders.

In Section~\ref{sec:homoorder} we recall some results about the homomorphism order. We also give a new and significantly easier proof of the universality of graph homomorphisms (simplifying the one given by the Hubi\v cka and Ne\v set\v ril~\cite{Hubicka2005}). For the first time we show that the homomorphism order is universal even on the class of oriented cycles. We use this result to resolve the universality of locally constrained homomorphisms.
In later sections we summarize the properties of constrained homomorphism orders.
In Section~\ref{sec:fullhomo} we characterize gaps in the full homomorphism order and give a new and simple proof of the existence of left duals (simplifying the results independently obtained by Feder and Hell~\cite{Feder2008} and Ball, Ne\v set\v ril, Pultr~\cite{Ball2007,Ball2010}).
In Section~\ref{sec:locally-constrained} we prove the universality of locally injective homomorphisms on connected graphs
and give a partial characterization of gaps. We also prove the universality of all three kinds of locally constrained homomorphisms on graphs. 

In Section~\ref{sec:line} we prove that for every $d\geq 3$ the homomorphism order on the class of line graphs of graphs with maximum degree $d$ is universal.

In Section~\ref{sec:relations}, using notions derived from Chapter~\ref{ch:relation}, we characterize cores for relation orders. We also describe their gaps and dualities by extending the corresponding results for surjective homomorphism orders.

\subsubsection*{Notation}
The most generic way to think of structures induced by constrained homomorphisms is in terms of category theory and relational structures~\cite{Pultr1980}.
Because we are however interested in specific examples of orders induced by specific types of constrained homomorphisms, we believe that it
makes the thesis more readable to use the standard language of graph theory and graph homomorphisms with just the following extension.

We use the abbreviations we assigned to each constrained homomorphism (such as \FullHomo{}\index{\FullHomo{}} for the full homomorphisms) to denote the variant of the order of interest. We write $G \AnyHom H$\index{$\AnyHom$} and $G \AnyLeq H$\index{$\AnyLeq$} to denote the existence of $\AnyHomo$ from $G$ to $H$. For example, $G\FullHom H$\index{$\FullHom$} denote the existence of a full homomorphism.

We use the analogous notation for other concepts introduced above. For example, just as the core is the smallest representative of its equivalence class of $\Homeq$, the $*$-core\index{$*$-core} is the smallest representative of its equivalence class of $\AnyHomeq$\index{$\AnyHomeq$} (that is, the equivalence induced on $\DiGraphs$ by $\AnyHom$).

\section{Universal orders}
\label{sec:universality}

It is a well-known fact that every finite order $(P,\leq_P)$ can be represented by the inclusion order on finite sets. More precisely, for every finite order $(P,\leq_P)$ there exists an embedding $\PtoInclussion$ assigning to every element $x\in P$ a finite set $\PtoInclussion(x)$ such that
 $$x\leq_P y\hbox{ if and only if } \PtoInclussion(x)\subseteq \PtoInclussion(y) \hbox{ for every $x,y\in P$}.$$
The easiest way to obtain this embedding is to put $\PtoInclussion(x)=\downarrow x.$

The choice of $(P,\leq_P)$ is arbitrary and thus we always assume that $P$ is a subset of a fixed countable set, for example $\mathbb N$ (the set of natural numbers). Consequently the order formed by all finite sets of natural numbers ordered by the inclusion contains every finite order as a suborder. 

Countable orders containing every finite order as a suborder are {\em finite-universal}\index{finite-universal}. We denote by $\Pfin(A)$\index{$\Pfin(A)$} the set of all finite subsets of $A$.
We can now state the following folklore fact.
\begin{prop} [The existence of future-finite orders]
\label{pro:finiteuniv}
For any countably infinite set $A$, $(\Pfin(A),\subseteq)$ is a finite-universal order.
\end{prop}

Finite-universality orders have a rich structure. It immediately follows that they are of infinite dimension and contain finite chains, antichains and decreasing chains of arbitrary length. While finite-universal orders can be seen as `rich' specimens, they are rather easy to construct. There are only countably many finite orders (up to isomorphism). Thus one can consider an order formed from the disjoint union of all finite orders. However this is no longer true when countable universality is considered.
A countable order is {\em universal}\index{universal} if it contains every countable order as a suborder.
The existence of universal orders seems to be a counter-intuitive fact. Countable orders have a rich structure and there are uncountably many of them, yet all of them can be `packed' into a single countable structure.

The easiest way to show the existence of a universal order is by the \Fraisse{} Theorem~\cite{Fraisse1986} which leads to an implicit construction of such order. We will give a constructive proof of its existence. 
Having an explicit representation makes the embedding into (constrained) homomorphism orders easy. 

\vspace{0.4cm}
We build a universal order in two layers.
\vspace{-0.4cm}

\paragraph{1.} First we turn our attention to a `lesser' notion of universality.
Recall that in Chapter~\ref{ch:graph} we defined that a countable order is {\em past-finite}\index{past-finite} if every down-set is finite. Similarly a countable order is {\em future-finite}\index{future-finite} if every up-set is finite.
Again, we say a countable order is {\em past-finite-universal}\index{past-finite-universal}\footnote{
`Past-finite' or `future-finite' is a standard name used in the context of study of causal sets (that is an alternative name for locally finite partial orders). We use these names due to the apparent lack of a standard name of these properties in the order theoretical context. Note that past finiteness is different from well-quasi-ordering.} (or {\em future-finite-universal}\index{future-finite-universal}) if it contains every past-finite (or future-finite) order as a suborder. We now strengthen Proposition~\ref{pro:finiteuniv}.

\begin{lem}
\label{lem:pastfiniteuniv}
For any countably infinite set $A$, $(\Pfin(A),\subseteq)$ is a past-finite-universal order.
\end{lem}
\begin{proof}
Fix a past-finite ordered set $(P,\leq_P)$. Without loss of generality we assume that $P\subseteq A$. Consider the same mapping as in the proof of Proposition~\ref{pro:finiteuniv}, that is to assign to every $x\in P$ a set $E(x) = \downarrow x$. It is easy to verify that $E$ is an embedding from $E:(P,\leq_P)\to(\Pfin(A),\subseteq)$.
\end{proof}
Because every past-finite order turns to be a future-finite order when the direction of inequalities is reversed, we immediately obtain:
\begin{corollary}
\label{cor:futurefiniteuniv}
Let $A$ be any countably infinite set.
Then $(\Pfin(A),\supseteq)$ is a future-finite-universal order.
\end{corollary}

%
%
%
%
\paragraph{2.}
For a given order $(P,\leq_P)$ we construct its {\em subset order}\index{subset order}, $(\Pfin(P),\subsetLeq{P})$\index{$(\Pfin(P),\subsetLeq{P})$}, by putting $$A\subsetLeq{P} B\iff \hbox{ for every }a\in A \hbox{ there exists }b\in B \hbox{ such that }a\leq_P b$$
(in this case we say $A$ {\em is dominated by}\index{dominated by} $B$ in $(P,\leq_P)$).

It is easy to see that for every order $(P,\leq_P)$ the relation $\subsetLeq{P}$ (on finite subsets of $P$) is transitive and reflexive and thus forming a quasi-order.
We again choose a unique representative for every equivalence class to get an order. $A\in \Pfin(P)$ is a representative if for every $a,b\in A$, $a\neq b$ the elements $a$ and $b$ are incomparable in $(P,\leq_P)$ (that is, $A$ is an antichain in $(P,\leq_P)$).

As a consequence of Lemma~\ref{lem:pastfiniteuniv} we easily obtain:
\begin{corollary}
\label{cor:subsetfinite}
Let $(P,\leq_P)$ be a countably infinite discrete order (that is the relation $\leq_P$ is empty). Then $(\Pfin(P),\subsetLeq{P})$ is a past-finite-universal order.
\end{corollary}

If $(P,\leq_P)$ is a countably infinite discrete order, then for any $A,B\subseteq P$, $A\subsetLeq{P} B$ if and only if for any $a\in A$, there exists a $b\in B$ such that $a=b$, this is true if and only if $A\subseteq B$. For this reason,
the order $(\Pfin(A),\subseteq)$ is in fact a special case of subset orders. 

More surprisingly, the subset orders can be layered to create a universal order.

\begin{thm}
\label{thm:univ}
For every future-finite-universal order $(F,\leq_F)$, the order $(\Pfin(F),\subsetLeq{F})$ is universal.
\end{thm}
\begin{proof}
Take an arbitary order $(P,\leq_P)$. Without loss of generality we can assume that $P\subseteq \mathbb{N}$. This
way we enforce the linear order $\leq$\footnote{Note that here $\leq$ does not mean the existence of homomorphisms.} on the elements of $P$. Note that the linear order $\leq$ is completely unrelated to the order $\leq_P$.
We can think of $\leq$ as a specification of the time of creating the elements of $P$.

We define two new orders on the elements of $P$: $\leq_f$\index{$\leq_f$} (\emph{forwarding order}\index{forwarding order}) and $\leq_b$\index{$\leq_b$} (backwarding order)\index{backwarding order}:
\begin{enumerate}
\item We put $x\leq_f y$ if and only if $x\leq_P y$ and $x \leq y$.
\item We put $x\leq_b y$ if and only if $x\leq_P y$ and $x \geq y$.
\end{enumerate}

Thus we decompose the order $(P,\leq_P)$ to $(P,\leq_f)$ and $(P,\leq_b)$.
For every element $x\in P$ both of the sets $\{y\mid y\leq_f x\}$ and $\{y\mid x\leq_b y\}$ are finite.
It follows that $(P,\leq_f)$ is past-finite and $(P,\leq_b)$ is future-finite.

Since $(F,\leq_F)$ is future-finite-universal, there is an embedding $E: (P,\mathop{\leq_b}) \to (F,\leq_F)$. For every $x\in P$, we put $$U(x)=\{E(y)\mid y\leq_f x\}.$$
We show that $U:(P,\leq_P) \to (\Pfin(F),\subsetLeq{F})$ is an embedding.

First we show that $U(x)\subsetLeq{F} U(y)$ implies $x\leq_P y$.
From the definition of $\subsetLeq{F}$ we know that there is an element
$w\in P$, $E(w)\in U(y)$, such that $E(x)\leq_F E(w)$. By the definition of $U$, $E(w)\in U(y)$ if and only if $w\leq_f y$. By the definition of $E$, $E(x)\leq_F E(w)$ if and only if $x\leq_b w$.
It follows that $x\leq_b w\leq_f y$ and thus also $x\leq_P w\leq_P y$ and consequently $x\leq_P y$.

To show that $x\leq_P y$ implies $U(x)\subsetLeq{F} U(y)$ we consider two cases.
\begin{enumerate}
\item When $x\leq y$. Then $U(x)\subseteq U(y)$ and thus $U(x)\subsetLeq{F} U(y)$.
\item When $x>y$. Then we take any $w\in P$, $E(w)\in U(x)$. From the construction of $U(x)$ we know that $w\leq_f x\leq_P y$.
If $w\leq y$, then $E(w)\in U(y)$. In the other case we have $w\leq_b y$ and thus $E(w)\leq_F E(y)$. It follows that $U(x)$ is dominated by $U(y)$.\qedhere
\end{enumerate}
\end{proof}

It is easily seen that the embeddings constructed in proofs of Theorem~\ref{thm:univ} have the property that the embedding of vertex $x$ depends only on vertices $\{y\mid y<x\}$. Such embeddings are known as {\em on-line embeddings}\index{on-line embeddings} because they can be constructed inductively without a priori knowledge of the whole order. See for example~\cite{Hubicka2011}. All embeddings shown in this chapter are in fact on-line.

\begin{example}
\label{example:poset}
Consider the order $(P,\leq_P)$ as specified by Figure~\ref{fig:poset}. 
Denote by $(P,\leq_f)$ the order consisting of forwarding inequalities in ${(P,\leq_P)}$ and by ${(P,\leq_b)}$ the order of backwarding inequalities.
\begin{figure}
\centerline{\includegraphics{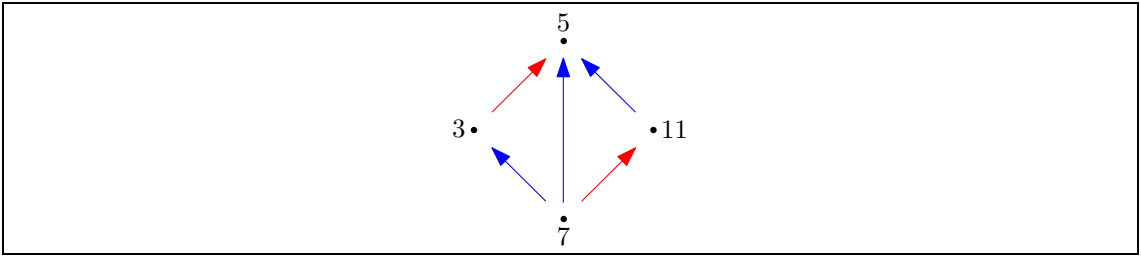}}
\caption{The order $(P,\leq_P)$. Blue lines represent backwarding edges}
\label{fig:poset}
\end{figure}
By Corollary~\ref{cor:futurefiniteuniv} we can obtain the following on-line embedding $f$ from
$(P,\leq_b)$ to $(\Pfin(\mathbb P),\supseteq)$, for any element $x\in P$, $f(x)=\{y\mid x\leq_b y\}$:
$$f(3)=\{3\},f(5)=\{5\},f(7)=\{3,5,7\},f(11)=\{5,11\}.$$
By Theorem~\ref{thm:univ} we get the following on-line embedding $U$ from
$(P,\leq_P)$ to $(\Pfin(\Pfin(\mathbb P)),\subsetLeq{\supseteq})$, for any element $x\in P$,
$U(x)=\{f(y)\mid y\leq_f f(x)\}$:
$$U(3)=\{\{3\}\},U(5)=\{\{5\},\{3\}\},$$ $$U(7)=\{\{3,5,7\}\},U(11)=\{\{3,5,7\},\{5,11\}\}.$$
\end{example}

We remark that the universality of the order $(\Pfin(F),\subsetLeq{F})$ and the proof of Theorem~\ref{thm:univ} are shown in several works on universal orders. The precise formulations differ, but the key ideas remain. To the best of our knowledge Hedrl\'in first introduced equivalent structure as an example of the universal order in~\cite{Hedrlin1969}. In~\cite{Nesetril2000} the structure was re-discovered to show the universality of the homomorphism order of set-systems (for the first time, its universality was shown by on-line embedding). In~\cite{Hubicka2004} it is used to show the universality of the homomorphism order of oriented paths (Theorem~\ref{thm:paths}). Finally in~\cite{Hubicka2011} a more streamlined variant of the same structure appears as a tool to show the universality of multiple naturally defined orders.

Here we use a more layered approach with the notions of future-finite-universal and past-finite-universal sets. The existence of past-finite-universal structures has been shown in a less constructive way in~\cite{Droste2005}. We apply these notions to several universality results on the constrained homomorphism order.


\section[Properties of past-finite and future-finite orders]{\fontsize{20}{10}\selectfont Past-finite and future-finite orders}
\label{sec:pastfinite}
Many constrained homomorphism orders we consider have finite down-sets or up-sets. We thus generally characterize the orders into three categories: universal, past-finite universal or future-finite universal.

In this section we make some basic observations about past-finite and future-finite orders. Both types of orders are special cases of locally-finite order~\cite{Britz2001,Droste2005}. An order $(P,\leq_P)$ is said to be {\em locally-finite}\index{locally finite} if for every $a,b\in P$, $a\leq_P b$, the interval $[a,b]_P=\{c\mid a\leq_P c\leq_P b\}$ is finite. This immediately gives the following result.

\begin{prop}[Density of locally-finite orders]
\label{prop:futurefinitegraps}
If $(P,\leq_P)$ is a past-finite, future-finite or locally-finite order, then $(P,\leq_P)$ is not dense.

For every $u\in P$ that is not maximal in $(P,\leq_P)$, there is an element $v\in P$ such that $(u,v)$ is a gap in $(P,\leq_P)$.
For every $u'\in P$ that is not minimal in $(P,\leq_P)$, there is a $v'\in P$ such that $(v',u')$ is a gap in $(P,\leq_P)$.
\end{prop}

The concept of duality was studied for some past-finite and future-finite partial orders, see \cite{Hell2013,Ball2010,Feder2008,Xie2006}. We simplify our presentation by the following simple observation that holds for all future-finite and past-finite orders.

The \emph{generalized finite duality pair}\index{generalized finite duality pair} $(\mathcal{F},\mathcal{D})$\index{$(\mathcal{F},\mathcal{D})$} in $(P,\leq_P)$ can be seen as an equation:
$$\bigcup_{f\in \mathcal{F}}\uparrow f=P\setminus \bigcup_{d\in\mathcal{D}}\downarrow d.$$

This is depicted in Figure~\ref{fig:duality}. 

\begin{figure}[ht]
\centerline{\includegraphics{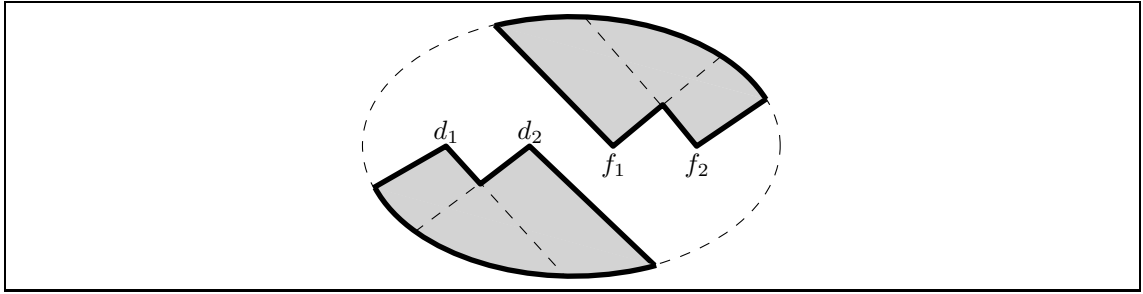}}
\caption{Generalized duality pair $(\{f_1,f_2\},\{d_1,d_2\})$.}
\label{fig:duality}
\end{figure}

If we assume that $(P,\leq_P)$ is future-finite, it follows that left hand side of the equation is always finite, while the right hand side may be infinite. Obviously $(\mathcal{F},\mathcal{D})$ can be a duality pair only if both sides of the equation describe a finite set and thus only for relatively degenerated choices of set $\mathcal{D}$ one can find a set $\mathcal{F}$ such as $(\mathcal{F},\mathcal{D})$ is a generalized finite duality pair.

In the case of future-finite orders we will thus ask for the characterization of sets $\mathcal{F}$ with $\mathcal{D}$ forming a generalized finite duality pair $(\mathcal{F},\mathcal{D})$. We show that the existence of finite duality pairs is closely related to the existence of gaps.
\begin{prop}[Dualities in a future-finite order]
\label{prop:futurefiniteduals}
Let $(P,\leq_P)$ be a future-finite order and
$\mathcal{F}$ a finite subset of $P$.
There is a set $\mathcal{D}$ such that $(\mathcal{F},\mathcal{D})$
is a generalized finite duality pair in $(P,\leq_P)$ if and only if
\begin{enumerate}
\item for every element $a$ in $\bigcup_{f\in \mathcal{F}}\uparrow f$ there are only finitely many elements $b$ in $P$ such that $(b,a)$ is a gap in $(P,\leq_P)$,
\item $(P,\leq_P)$ has only finitely many maximal elements.
\end{enumerate}
\end{prop}
The situation is schematically depicted in Figure~\ref{fig:finiteduality}.

\begin{figure}
\centerline{\includegraphics{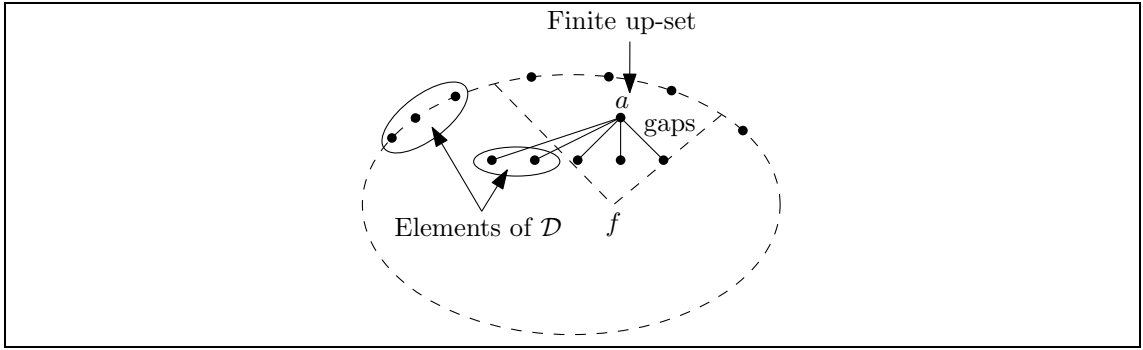}}
\caption{Dualities in a future-finite order.}
\label{fig:finiteduality}
\end{figure}
\begin{proof}

Put $$A=\bigcup_{f\in \mathcal{F}}\uparrow f,$$ 
$$M=\{m\mid \hbox {$m$ is a maximal element of $(P,\leq_P)$ and }m\notin A\},\hbox{ and}$$
$$\mathcal{D}=\{b\mid b\in P\setminus A \hbox { and there exists $a\in A$ such that $(b,a)$ is a gap in $(P,\leq_P)$} \}\cup M.$$

It follows from Proposition~\ref{prop:futurefinitegraps} and the future-finiteness of $(P,\leq_P)$ that $\mathcal{D}$ is the smallest set such that $(\mathcal{F},\mathcal{D})$ is a generalized finite duality pair.

The conditions 1 and 2 in the statement of this proposition are precisely the ones assuring that $\mathcal{D}$ is finite,
thus $(\mathcal{F},\mathcal{D})$ is a generalized finite duality pair.
\end{proof}

We say that there are {\em all right generalized finite dualities}\index{all right generalized finite dualities} in $(P,\leq_P)$ if
for every finite set $\mathcal{F}\subseteq P$ there is a finite $\mathcal{D}\subseteq
P$ such that $(\mathcal{F},\mathcal{D})$ is a generalized finite duality pair.

\begin{corollary}
\label{cor:dualities} 
If $(P,\leq_P)$ is a future-finite order such that for every $a\in P$ there is only finitely many $b\in P$ such that $(b,a)$ is a gap in $(P,\leq_P)$,
then there are all right generalized finite dualities in $(P,\leq_P)$.
\end{corollary}

In an order $(P,\leq_P)$, every element $a\in P$ with only one $b\in P$ such that $(a,b)$ is a gap in $(P,\leq_P)$ is called {\em left realization} or {\em left realization of $b$}\index{left realization}.
Symmetrically {\em right realization}\index{right realization} is the element $a\in P$ with only one $b\in P$ such that $(b,a)$ is a gap.

The following observation (that was a special case about full homomorphisms used by Feder and Hell in~\cite{Feder2008}) follows from the construction of $\mathcal{D}$ shown above.
\begin{prop}
\label{prop:realization}
Let $(P,\leq_P)$ be a future finite order and $\mathcal{F}=\{f\}$ a set consisting of single element of $P$. For inclusion minimal $\mathcal{D}$ such that $(\mathcal{F},\mathcal{D})$ is a generalized finite duality pair, the only elements $r\in \mathcal{D}$ such that $(r,f)$ is a gap in $(P,\leq_P)$ are left realizations of $f$.
\end{prop}

From the symmetry, in the context of past-finite orders the question about the existence of duality pairs $(\mathcal{F},\mathcal{D})$ for a given $\mathcal{D}$ is more interesting than the converse question. As we will show, some of the orders we consider have duality for every choice of $\mathcal{F}$, while others do not.

We also remark that one of the main motivations for graph dualities --- the connection to computational complexity--- does not hold in the case of past-finite and future-finite orders; the set of objects described by dualities is finite and thus it is always easy to write algorithms recognizing them for a fixed generalized finite duality pair. Yet those types of question are interesting, see, for example~\cite{Ball2010, Ball2007} where the existence of such dualities is related to the concept of Ramsey lists, see also~\cite{Feder2008,Xie2006}.

\section{Homomorphism orders} \label{sec:homoorder}

For two graphs $G$ and $H$, we put a relation $G\to H$\index{$G\to H$} if there is a graph homomorphism of $G$ to $H$.
The relation $\to$ is reflexive (identity is a homomorphism) and transitive (composition of two homomorphisms is still a homomorphism). Thus the existence of homomorphisms induces a quasi-order on the class of all finite directed graphs. We denote this quasi-order by $(\DiGraphs,\leq)$\index{$(\DiGraphs,\leq)$}. When speaking of orders, we will thus use $G\leq H$ in the same sense as $G\Hom H$. This allows us to write shortly $G<H$\index{$G<H$}, $G>H$\index{$G>H$} and $G\geq H$\index{$G\geq H$} with the obvious meanings.
In this section we first outline the main properties of quasi-orders induced by the existence of homomorphisms on $\Graphs$ and $\DiGraphs$; for a detailed discussion, see~\cite[Chap.~3]{Hell2004}. Also we revisit universality on restricted graph classes and give a new proof of universality of the homomorphism order on the class of oriented cycles.

There are two standard ways to transform a quasi-order to an order---by identifying equivalent objects, or by choosing a particular representative for each equivalence class. We do the latter; cores fit the purpose perfectly. 
Recall that in Chapter~\ref{ch:homo} we defined that
a (finite) directed graph is a {\em graph core}\index{graph core} when it is the smallest graph (in the number of vertices) in its equivalence class of $\Homeq$ (also see~\cite{Hell2004}).
It can be shown that every equivalence class of $\Homeq$ contains, up to isomorphism, a unique core. The {\em core of graph $G$}\index{core} is an up to isomorphism unique graph $H$ such that $H\Homeq G$ and $H$ is a core. Further the core of graph $G$ is always an induced subgraph of $G$, see~\cite{Hell2004}. 
Consequently it makes sense to speak of cores in both $\DiGraphs$, $\Graphs$ and in any other class of graphs which is  hereditary (i.e. closed under taking induced subgraphs). We shall note that the graph core is usually defined in terms of retractions. We use an alternative, but equivalent, definition for the sake of carrying the term more easily over into the context of constrained homomorphism orders. The relation $\to$ induces an order on the class of all isomorphism types of cores. We take the liberty of denoting this order in the same way as the quasi order, that is by $(\DiGraphs,\leq)$\index{$(\DiGraphs,\leq)$}, and implicitly assume that the class $\DiGraphs$ is restricted to cores. We call this order the {\em homomorphism order of directed graphs}\index{homomorphism order of directed graphs}. We will make similar assumptions on all other constrained homomorphism orders discussed thoroughout the chapter.

Both homomorphism orders $(\DiGraphs,\leq)$ and $(\Graphs,\leq)$ have been extensively studied.


\subsection*{Universality}

It is a relatively easy exercise to find an infinite antichain, a descending chain and an increasing chain in $(\Graphs, \leq)$. This suggests that the homomorphisms induce a rich structure on the class of graphs. It is however a non-trivial result that every countable order can be found as a suborder of $(\DiGraphs, \leq)$. The initial result has been proved in the even stronger setting of category theory, see~\cite{Pultr1980}. Subsequently it has been shown that many restricted classes of graphs and similar structures (such as homomorphisms of set systems~\cite{Nesetril2000}, oriented trees~\cite{Hubicka2005}, oriented paths~\cite{Hubicka2004}, orders and lattices~\cite{Lehtonen2008}) admit this universality property. 

The following result was the culmination of a series of papers in the context of categories.

\begin{thm}[\cite{Pultr1980}]
Any countable order is isomorphic to a suborder of $(\DiGraphs,\leq)$.
\end{thm}

This is a non-trivial result and it can be strengthened to say that $(\Graphs,\leq)$ is universal.

\subsection*{Density and duality}

Recall in Chapter~\ref{ch:graph} we introduced the definition of density, gaps and duality of an order. In this section we consider these properties for two orders: $(\Graphs, \leq)$ and $(\DiGraphs, \leq)$.

\subsubsection*{Density and gaps}

As a special case, the order of rational numbers, $(\mathbb{Q},\leq)$, can be embedded into the universal order $(\Graphs, \leq)$. It easily follows that there are 
graphs $G < H$ such that the set of graphs $\{H'|G<H'<H\}$ contains infinitely many mutually non-isomorphic graph cores. Consider, for example, the images of 0 and 1 in $(\Graphs, \leq)$. On the other hand, there are no graphs strictly in between $K_1$ (that is the graph formed by a single isolated vertex) and $K_2$ (that is, two vertices connected by an edge). This is written as the following theorem:

\begin{thm} [Density~\cite{Nesetril1999}]
 Given two simple graphs $G_1,G_2$ with $G_1 < G_2$, $G_1 \leq O$ and $G_1 \leq K_1$, there is a graph $G$ satisfying $G_1 < G < G_2$.
\end{thm}

There are three proofs of this result: a probabilistic proof~\cite{Nesetril1999}, a constructive proof via products due to M.\ Perles and J.\ Ne\v{s}et\v{r}il~\cite{Nesetril1999a}, and the third one relies on a result about duality which we will introduce later in this section.

All gaps in both $(\Graphs,\leq)$ and $(\DiGraphs,\leq)$ have been characterized~\cite{Welzl1982, Nesetril2000}.
It follows that $(K_1, K_2)$ is the only gap in $(\Graphs,\leq)$.

By contrast, the order $(\DiGraphs,\leq)$ has many gaps. In this case, the existence of $H'$ can be shown only for graphs $H$ containing an oriented cycle. 
We also mention that density in other classes of graphs has also been studied, for example, C.\ Tardif~\cite{Tardif1999} proved the class of vertex-transitive graphs is dense.

\subsubsection*{(Generalized) duality}

The structure of gaps in $(\DiGraphs,\leq)$ is related to another concept of graph dualities. 
We say that a pair of directed graphs $(F,D)$ is a {\em simple duality pair}\index{simple duality pair} in class $\K$ if
$$G\nrightarrow D \hbox{ if and only if } F\to G \hbox { for all } G\in \K.$$ An example of a such duality pair is given in Figure~\ref{fig:dualpair}.

\begin{figure}[!ht]
\centering
\includegraphics[width=0.7\textwidth]{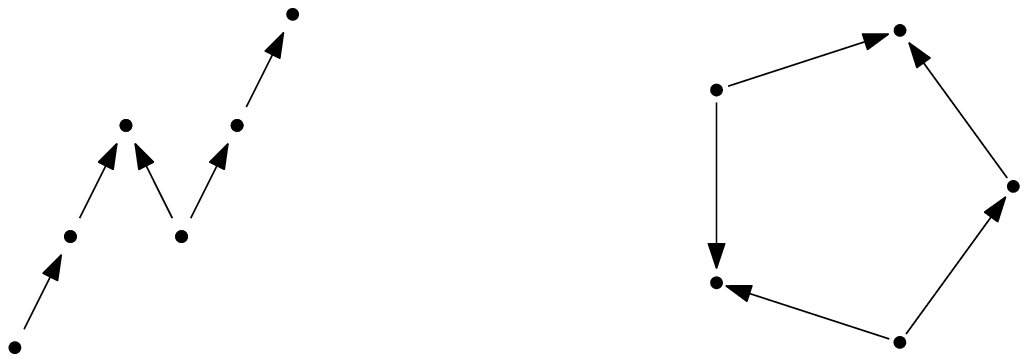}
\caption{Duality pair.}
\label{fig:dualpair}
\end{figure}

For homomorphism dualities the class $\K$ is usually implicitly $\Graphs$ or $\DiGraphs$.
In both cases dualities are in one-to-one correspondence with gaps~\cite{Nesetril2000}. 
It is easy to see that $(K_2,K_1)$ is a simple duality pair and in fact it is the only duality pair in $(\Graphs,\leq)$. In $(\DiGraphs,\leq)$ there exists a duality pair $(F,D)$ if and only if $F$ is an oriented tree (that is, when core of $F$ contains no oriented cycle).

Graph duality is an important concept of graph theory. It provides better understanding of other structural properties, such as maximal antichains~\cite{Nesetril2003}, computational complexity properties of homomorphism testing, etc. Several variants of graph duality have been studied.

Considering classes $\K\subset \DiGraphs$ leads to the notion of restricted dualities~\cite{Nesetril2008}.
For two finite sets of directed graphs $\mathcal{F}$ and $\mathcal{D}$ we say that $(\mathcal{F},\mathcal{D})$ \index{$(\mathcal{F},\mathcal{D})$} is a {\em generalized finite duality pair} in $\K$ if for any directed graph $G\in \K$ there exists $F\in \mathcal{F}$ such that $F\to G$ if and only if $G\to D$ for no $D\in \mathcal{D}$~\cite{Foniok2008}.
Recently the infinite-finite dualities have been considered and partially characterized. In a {\em generalized infinite-finite duality pair}\index{generalized infinite-finite duality pair} the set $\mathcal{F}$ can be infinite. See~\cite{Erdos2012} for details.

\subsubsection*{Maximal antichains}

From the definition, a simple duality pair is a maximal antichain of size two. Since every other digraph admits either a homomorphism to $H$ or a homomorphism from $F$. It is not known whether there are other maximal antichains of size $t\geq 2$ in $(\DiGraphs,\leq)$. The following propostion characterises all maximal antichains of size one in 
${(\DiGraphs,\leq)}$.

\begin{pro}[Antichain~\cite{Hell2004}]
The order ${(\DiGraphs,\leq)}$ has exactly three maximal antichains of size one---$\{\overrightarrow{P_1},\overrightarrow{P_2},\overrightarrow{P_3}\}$.
\end{pro}

However, in contrast to $(\DiGraphs,\leq)$, $(\Graphs,\leq)$ has no maximal antichains of any finite size $t>1$.

\subsection*{Restricted homomorphism universality}

As a warm-up, we first recall the most common restricted graph classes which are already defined in Chapter~\ref{ch:graph}. We denote by $\Path$\index{$\Path$} the class of all finite oriented paths, by $\DiCycle$\index{$\DiCycle$} the class of directed graphs formed all $\overrightarrow{C}_k$, $k\geq 3$, and $\DiCycles$\index{$\DiCycles$} the class of directed graphs formed by disjoint union of finitely many graphs in $\DiCycle$. Similarly, denote by $\Cycle$\index{$\Cycle$} the class of graphs formed by all $C_k$, $k\geq 3$,
denote by $\Cycles$\index{$\Cycles$} the class of graphs formed by disjoint union of finitely many graphs in $\Cycle$.
Finally denote by $\mathbb{P}$\index{$\mathbb{P}$} the class of all odd prime numbers.

It is a classical result that the order $(\Graphs,\leq)$ is a universal order. In fact graph homomorphisms are universal also in the categorical sense~\cite{Pultr1980}. More recently it has been shown that the universality of the homomorphism order is a relatively persistent property and the homomorphism order remains universal even for very restricted classes of (directed) graphs. In~\cite{Hubicka2004} the order $(\Path,\leq)$ is shown to be universal. This result can be used to show the universality of the homomorphism order in many naturally defined classes of (directed) graphs, such as 3-colourable graphs, planar graphs, series-parallel graphs, etc. Dichotomy results on the classes of graphs specified by chromatic and achromatic number are provided in~\cite{Nesetril2007}.

The arguments in the proofs of those results build on techniques used in the proof of the following theorem (first stated in~\cite{Hubicka2004}, an alternative and more streamlined proof appeared in~\cite{Hubicka2011}).                                                                                                                                                                                                                                                                                                                                                                                                                                                                                                                                                                                                                                                                                                                                                                                                                                                                                                                                                                                                                                                                

\begin{thm}[Hubi\v{c}ka, Ne\v{s}et\v{r}il~\cite{Hubicka2004,Hubicka2011}]
\label{thm:paths}
The order $(\Path,\leq)$ is universal.
\end{thm}

These techniques however can not be easily applied in the context of locally constrained homomorphisms. In particular,  locally constrained homomorphisms do not allow any sort of `flipping' operation necessary for non-trivial path homomorphisms.

We prove the universality of $(\DiCycles,\leq)$ with a new argument replacing `flipping' by `rolling'.
As a pleasant surprise, this result is significantly easier to prove (compared to the ones appearing in~\cite{Nesetril2006,Lehtonen2008,Lehtonen2010,Kwuida2011,Samal2013}) and has similar applications as Theorem~\ref{thm:paths}

\begin{lem}
\label{lem:cycles}
The order $(\DiCycle,\leq)$ is future-finite-universal.
\end{lem}
\begin{proof}
By Corollary~\ref{cor:futurefiniteuniv} the order $(\Pfin(\mathbb{P}),\supseteq)$ is future-finite-universal.

For the finite set $A\subset \mathbb{P}$ we put $$p(A)=\prod_{p\in A} p.$$
It immediately follows that $$\PrimesToCycles(A)=\overrightarrow{C}_{p(A)}$$
is an embedding from $(\Pfin(\mathbb{P}),\supseteq)$ to $(\DiCycle,\leq)$;
there is a homomorphism $\overrightarrow{C}_k\to \overrightarrow{C}_l$ if and only if $l$ divides $k$.
\end{proof}

\begin{thm}
\label{thm:cycles}
The order $(\DiCycles,\leq)$ is universal.
\end{thm}

\begin{proof}
By Lemma~\ref{lem:cycles}, the order $(\DiCycle,\leq)$ is isomorphic to $(\Pfin(\mathbb{P}),\mathop{\supseteq})$ that is future-finite-universal. It is easy to see that the order $(\DiCycles,\leq)$ is actually isomorphic to the subset order $(\Pfin(\Pfin(\mathbb{P})),\subsetLeq{\supseteq}$)\footnote{By $\subsetLeq{\supseteq}$\index{$\subsetLeq{\supseteq}$} we mean $\subsetLeq{P}$ where $\leq_P$ is $\supseteq$.}. The latter is known to be universal from Theorem~\ref{thm:univ}.
\end{proof}
\begin{example}
Consider the embedding of order given in Example~\ref{example:poset}. We can further construct an embedding $\PrimesToCycles$ from
$(\Pfin(\Pfin(\mathbb P)),\subsetLeq{\supseteq})$ to $(\DiCycles,\leq)$ as follows:
$$\PrimesToCycles(3)=\overrightarrow{C}_3,\PrimesToCycles(5)=\overrightarrow{C}_5\oplus \overrightarrow{C}_3,\PrimesToCycles(7)=\overrightarrow{C}_{105},\PrimesToCycles(11)=\overrightarrow{C}_{105}\oplus \overrightarrow{C}_{55}.$$
Here $\oplus$\index{$\oplus$} denote the disjoint union of two graphs.
\end{example}

\section[Embedding and Monomorphism orders]{\fontsize{19}{10}\selectfont Embedding and monomorphism orders}
\label{sec:embedding}
Now we are ready to start exploring the properties of individual orders that are of interest to us. As a warm-up we start with the very restricted orders induced by monomorphisms and embeddings. 

First we show that these orders are past-finite for a simple reason: For two
directed graphs $G$ and $H$, there is an embedding (or monomorphism) from $G$ to $H$ if
and only if the number of vertices of $G$ is not greater than the number of
vertices of $H$. We immediately get the following proposition.
\begin{prop}[Past-finiteness]
\label{prop:embed-finite}
The orders $(\DiGraphs,\EmbedLeq)$ and \\${(\DiGraphs,\MonoLeq)}$ are past-finite.
\end{prop}

We can show past-finite-universality of $(\DiGraphs,\EmbedLeq)$ and $(\DiGraphs,\MonoLeq)$ by Lemma~\ref{lem:pastfiniteuniv}.
There are infinitely many mutually incomparable connected graphs (isomorphism types). 
For example, every element of $S\in \Pfin(\mathbb{N})$ can be represented by the disjoint union of cycles $C_i$, $i\in S$.
\begin{prop}[Universality]
\label{prop:embed-univ}
The orders $(\Cycles,\EmbedLeq)$, ${(\DiCycles,\EmbedLeq)}$, ${(\Cycles,\MonoLeq)}$ and ${(\DiCycles,\MonoLeq)}$ are past-finite-universal.
\end{prop}

Finally we show that dualities in the context of monomorphisms and embeddings are very simple.
\begin{prop}[Dualities]
\label{prop:embeddual}
For any finite set of directed graphs $\mathcal{D}$, there is a finite set of directed graphs
$\mathcal{F}_E$ such that $(\mathcal{F}_E,\mathcal{D})$ is a generalized finite E-duality pair
and a finite set of directed graphs $\mathcal{F}_M$ such that $(\mathcal{F}_M,\mathcal{D})$ is a generalized finite M-duality pair.
\end{prop}
\begin{proof}
We apply Corollary~\ref{cor:dualities} (flipped for past-finite orders). Observe that for an E-gap or an M-gap $(G,H)$, graphs
$G$ and $H$ differ by at most one vertex or edge. Thus for every $G$ there is only finitely many $H$ such that $(G,H)$ is an E-gap or an M-gap.
Both embedding and monomorphism orders have only one minimal element (the empty graph).
\end{proof}
We however show a direct proof which gives a better understanding of the structure of $\mathcal{D}$ and we will extend it later.
\begin{proof}[Alternative proof]
We prove the statement only for embeddings. Monomorphisms follow in a completely analogous way.

Fix a set $\mathcal{D}$ and denote by $n$ the maximal number of vertices of graph in $\mathcal{D}$.

We construct a set $\mathcal{F}'_E$ as a set of all graphs $F$ such that $|V_F|\leq n+1$ and $F\EmbedHom D$ for no $D\in \mathcal {D}$. Construct $\mathcal{F}_E$ as
the set of all mutually non-isomorphic minimal elements of
$(\mathcal{F}'_E,\EmbedLeq)$.


It is easy to see that $(\mathcal{F}_E,\mathcal{D})$ is a generalized finite E-duality pair.
(Obviously $\mathcal{F}_E$ is finite.)
Consider a graph $G$ such that there is no $D\in \mathcal{D}$ satisfying $G\EmbedHom D$. There are two
cases.
\begin{enumerate}
\item If $|V_G|\leq n+1$, then $G\in \mathcal{F}'_E$ and thus there is $F\in \mathcal{F}_M$,
$F\EmbedHom G$.
  \item If $|V_G|>n+1$, we consider any subgraph of $G$ on $n+1$ vertices. \qedhere
\end{enumerate}
\end{proof}
Now consider $\mathcal{D}=\{D\}$ where $D$ has $n$ vertices and let $(\mathcal{F},\mathcal{D})$ be an
E-duality pair as constructed above. By Proposition~\ref{prop:realization} only graph $F\in \mathcal{F}$
with $n+1$ vertices can be the right E-realization of $D$. We show that except for the degenerate case there
is at most one such graph $F$.

A graph $G$ is {\em vertex transitive}\index{vertex-transitive} if for any distinct vertices $v,w\in V_H$, there exist a graph automorphism $f$ such that $f(v)=w$.
\begin{prop}[Realizations]
Graph $G$ is a right E-realization if and only if it is vertex transitive. 
\end{prop}

\begin{proof}
It is easy to observe that vertex transitive graphs are right E-realizations. We show the opposite implication.

Consider a graph $G$ that is a right E-realization. This means that removing
any vertex from $g$ leads to the same graph.
 It follows that in-degrees and
out-degrees of vertices are the same: if not, removing one vertex can not lead
to the same graph as removing different vertices with different degrees. Fix
vertices $v$ and $v'$ and denote by $H$ and $H'$ graphs created from $G$ by
removing $v$ and $v'$ respectively. By definition $H$ and $H'$ are
isomorphic. Denote by $d$ the in-degree of vertices in $G$ and (the out-degree
must be the same). It follows that both $H$ and $H'$ have $d$ vertices with
in-degree $d-1$ and the removed vertex is connected to all of them. The
isomorphism of $H$ to $H'$ must map vertices of in-degree $d-1$ to vertices of
in-degree $d-1$. The same holds for out-degrees. Therefore the isomorphism can be
extended to an automorphism of $G$ mapping $v$ to $v'$. The choice of $v$ and $v'$
was arbitrary and thus it follows that $G$ is vertex transitive.

Consider graph $H$. Because it has a (right) E-realization, we know that there is $d$ such that
$H$ has $d$ vertices of in-degree $d-1$ and the rest of the vertices are of in-degree $d$
and the same holds for out-degrees.
This also specifies the vertex extending $H$ to its realization. The only degenerate
case is when $d=0$ and thus there is no vertex of degree $d-1$. In the case
of single vertex graph the realizations are all graphs on 2 vertices.
\end{proof}

%
%

\section{Full homomorphism orders}
\label{sec:fullhomo}

As a first non-trivial case we consider full homomorphisms (mappings that are both
edge and non-edge preserving). Knowledge about the full homomorphism order is somewhat developed: 
dualities in the full homomorphism order have been studied by Ball, Ne\v{s}et\v{r}il and Pultr in
\cite{Ball2007,Ball2010}. The problem of the existence of a full homomorphism from $G$ to $H$ is also known as
{\em full $H$-colouring}\index{full $H$-colouring} problem and in this language it has been independently studied by Feder and Hell in~\cite{Feder2008} and, more recently, by Hell and Hern{\'a}ndez-Cruz~\cite{Hell2013}. In both cases the motivation was
the correspondence of full homomorphisms to a certain matrix partition problem, see~\cite{Feder2008}.

We closely relate the full homomorphism order to the embedding order.
First we show that cores in full homomorphisms correspond to point-determining graphs which were studied in the 1970s by Sumner~\cite{Sumner1973} (c.f. Feder and Hell~\cite{Feder2008})
and consequently we show that the cores are ordered by embeddings.

%
%

In directed graph $G$, the {\em out-neighbourhood}\index{out-neighbourhood} of a vertex $v\in G_V$, denoted by $N^\rightarrow_G(v)$\index{$N^\rightarrow_G(v)$}, is the set of all vertices $v'$ of $G$ such that there is an edge from $v$ to $v'$ in $G$. Similarly {\em in-neighbourhood}\index{in-neighbourhood} of $v$, denoted by $N^\leftarrow_G(v)$, is the set of all vertices $v'$ of $G$ such that there is an edge from $v'$ to $v$ in $G$. We say that two vertices $v$ and $v'$ have the same neighbourhoods if both the in- and out-neighbourhoods match. Recall that {\em point-determining graphs} (also known as {\em mating-type graphs}, {\em mating graphs}, {\em M-graphs} or {\em thin graphs}) are graphs in which no two vertices have the same neighbourhoods. If we start with any graph $G$, and identify vertices with the same neighbourhoods, we obtain a point-determining graph we denote by $G_{\mathrm{pd}}$\index{$G_{\mathrm{pd}}$}.

It is easy to observe $G_{\mathrm{pd}}$ is always an induced subgraph of $G$ and moreover that for every directed graph $G$, $G_{\mathrm{pd}}\FullLeq G\FullLeq G_{\mathrm{pd}}$ and thus $G\FullHomeq G_{\mathrm{pd}}$. This motivates the following proposition:
\begin{prop}[Cores]
\label{prop:F-core}
Graph $G$ is an F-core if and only if it is point-determining.
\end{prop}
\begin{proof}
Recall that $G$ is an F-core if it is minimal (in number of vertices) in its
equivalence class of $\FullHomeq$. If $G$ is a F-core, $G_{\mathrm{pd}}$ can
not be smaller than $G$ and thus $G=G_{\mathrm{pd}}$.

It remains to show that every point-determining graph is an F-core.
Consider two point-determining graphs $G\FullHomeq H$ that are not isomorphic. There are 
full homomorphisms $f:G\FullHom H$ and $g:H\FullHom G$. Because injective full
homomorphisms are embeddings, it follows that either $f$ or $g$ is not
injective. Without loss of generality, assume that $f$ is not injective. Consider $u,v\in V_G$, $u\neq v$, such
that $f(u)=f(v)$. Because full homomorphisms preserve both edges and non-edges
it is easy to see that $N^\leftarrow_G(u)=N^\leftarrow_G(v)$ and $N^\rightarrow_G(u)=N^\rightarrow_G(v)$, a contradiction with the assumption that $G$
is point-determining.
\end{proof}
Now we apply the observations of Section~\ref{sec:embedding} to get the following proposition:

\begin{prop}
\label{prop:thincmp}
For F-cores $G$ and $H$ we have $G\FullLeq H$ if and only if $G\EmbedLeq H$.
\end{prop}
\begin{proof}
Every embedding is also a full homomorphism. In the opposite direction
consider a full homomorphism $f:G_{\mathrm{pd}}\to H_{\mathrm{pd}}$. By the same
argument as in the proof of Proposition~\ref{prop:F-core} we obtain that $f$ is injective.
\end{proof}

\begin{corollary}[Past-finiteness]
\label{cor:fullhomsize}
For every $G\FullLeq H$ we also have $|V_{G_{\mathrm{pd}}}|\leq |V_{H_{\mathrm{pd}}}|$.
Thus the order $(\DiGraphs,\FullLeq)$ is past-finite.
\end{corollary}
\begin{corollary}[Universality]
The order $(\DiCycles,\FullLeq)$ is past-finite-universal.
\end{corollary}
\begin{proof}
Observe that oriented cycles are point-determining graphs and apply Proposition~\ref{prop:embed-univ}.
\end{proof}

Sumner~\cite{Sumner1973} has defined a {\em nucleus}\index{nucleus} of a point-determining graph $G$
as a set of vertices $G^0\subseteq G$ such that for every $v\in G^0$ the graph created from $G$ by
removing $v$ is also point-determining.
All subgraphs that can be induced on a nucleus of connected point determining graphs have been
characterized. It was shown that every (non-trivial) unoriented connected point determining graph has
nucleus of size at least 2. This is a non-trivial result with an alternative proof appearing in~\cite{Feder2008}.

Because all these notions were considered on undirected graphs only, we now prove some of these results for directed graphs.
We recently became aware that the same work was also done independently in~\cite{Hell2013}:

\begin{prop}
\label{prop:Fcoregap}
Every F-core $G$ with at least 2 vertices contains a $F$-core $H$ with $|V_G|-1$ vertices as an induced subgraph.
\end{prop}
While the Proposition is proved as Theorem 1 of \cite{Hell2013}, we give a simple self-contained proof.
\begin{proof}
If there is a vertex $v$ of $G$ such that the graph $H$ created from $G$ by removing $v$ is point-determining, by Proposition \ref{cor:fullhomsize} there is no $F$-core strictly between $G$ and $H$ and we are done.

Assume to the contrary that for every vertex $v\in V_F$ there are two vertices
$\{v_1,v_2\}$ such that neighbourhoods of $v_1$ and $v_2$ become the same after
removing $v$. Denote by $G'$ the graph induced on vertices of $G$ by those
pairs. Thus $G'$ is a graph with $|V_G|$ edges.

Take any path $p_1,p_2,\ldots, p_n$ in $G'$. Denote by $v_1,\ldots, v_{n-1}$ the
vertices such that neighbourhoods of $p_i$ and $p_{i+1}$ differ only by $v_i$.
By construction of $G'$ all vertices $v_i$ are unique.
It follows that $p_1$ and $p_n$ differs by all vertices $v_1$,\ldots,
$v_{n-1}$. Consequently $G'$ is a tree and thus it has at most $|V_G|-1$
edges. A contradiction.
\end{proof}

Very similar ideas give the characterization of gaps in the full homomorphism order.
\begin{thm}[Gaps]
\label{thm:fullgap}
If $G$ and $H$ are F-cores and $(G,H)$ is an F-gap, then $G$ is created from $H$
by removing one vertex. Consequently the gaps in full homomorphism orders coincide
with the gaps in embedding order.
\end{thm}
\begin{proof}
Assume to the contrary that there are F-cores $G$ and $H$ such that $(G,H)$ is an F-gap,
but the size of $G$ differs from $H$ by more than one vertex.

By Proposition \ref{prop:thincmp} we know that $G$ is a subgraph of $H$. We can
assume that $G$ is induced on $H$ by $V_G$. Denote by $A$ the set of vertices
$V_H\setminus V_G$. By our assumption no vertex of $A$ can be removed from $H$
so that the resulting graph stays point-determining. 

For every vertex $v\in A$ denote by $\{v_1,v_2\}$ a pair of vertices such
that their neighbourhoods differ only by $v$. Denote by $H'$ the graph induced on
$V_H$ by those pairs.

Because $G$ is point-determining, we know that every edge of $H'$ contains at
least one vertex of $A$. Denote by $N_{H'}[A]$ the closed neighbourhood of $A$
in $H'$.

We say that vertex $v$ {\em determines}\index{determines} pair $\{v_1,v_2\}$ if the neighbourhoods of $v_1$ and
$v_2$ differ only by $v$.
On the induced subgraph of $H'$ on vertices of $N_{H'}[A]$ we can apply the same analysis as in Proposition
\ref{prop:Fcoregap}. It follows that there is at least one vertex $v$ of $N_{H'}[A]$ which
does not determine any pair of vertices of $N_{H'}[A]$. We know that $v\notin A$ and
thus there is an edge $(v,v')$ in $H'$, $v'\in A$.

Denote by $H''$ the graph created from $H$ by removing $v'$. First observe that
$G$ is an induced subgraph of $H''$; it is induced on vertices $(V_G\setminus\{v\})\cup \{v'\}$.
(neighbourhoods of $v$ and $v'$ differ only by a vertex in $A$.)

Finally observe that $H''$ is point-determining. By the choice of $v$ no pair of neighbourhoods 
vertices $u$, $u'$ such that $u\in A$ became the same. Also no pair of neighbourhoods in
$(V_G\setminus\{v\})\cup \{v'\}$ is equivalent because $G$ is point-determining. This is a contradiction
with the assumption that $(G,H)$ is an F-gap.
\end{proof}

Now we are ready to give a very simple proof of the existence of generalized dualities.
\begin{thm}
\label{thm:fulldual}
For every finite set of directed graphs $\mathcal{D}$ there is a finite set of graphs
$\mathcal{F}$ such that $(\mathcal{F},\mathcal{D})$ is a generalized finite F-duality pair.
\end{thm}
\begin{proof}
We again apply Corollary~\ref{cor:dualities} (flipped for past-finite orders). 
By Theorem~\ref{thm:fullgap}, for every F-gap $(G,H)$, graphs $G$ and $H$ differ by at most one vertex. Thus for every $G$ there are only finitely many $H$ such that $(G,H)$ is an F-gap. The empty graph is the unique minimum of the full homomorphism order.
\end{proof}

Essentially the same duality result (for undirected graphs)
was proved by Ball, Ne\v set\v ril and Pultr~\cite{Ball2010} (stated in terms
of generalized finite dualities) and by Feder and Hell~\cite{Feder2008} (stated
as a bound on the obstruction to full $H$- colourability). The second was
extended to directed graphs by Hell and Herm\'anez-Cruz~\cite{Hell2013}
independently of our proof (\cite{Ball2010} speaks of relational structures so their approach covers the case of directed graphs too). A simple argument for the bound also appears in~\cite{Xie2006}. It is interesting to observe that all proofs are different.
Both \cite{Ball2010} and \cite{Xie2006} use a Ramsey type argument. Our argument
is based on embeddings of point-determining graphs and is analogous to~\cite{Feder2008}, but we use a simpler argument in Proposition~\ref{prop:Fcoregap}
than the one needed to characterize the nucleus of a point determining graph. 
In~\cite{Feder2008,Hell2013} the results are however further strengthened.
For $\mathcal{D}$ consisting of single connected graph, one can show that
$\mathcal{F}$ contains at most two graphs of size $n+1$ (in addition to smaller graphs).

Thus one might say that proving dualities for full homomorphisms is a popular
task. Before becoming aware about prior research, the authors also enjoyed
work on an independent proof. To give a further contributution we present it here.
 We believe it has its own merit --- it does not need to use Ramsey type
arguments, analyse gaps or characterize the nucleus. It does however lead to a
weaker description of the set $\mathcal{F}$ within the duality pair.
\begin{proof}[Alternative proof]
%
%
We adapt the alternative proof of Proposition~\ref{prop:embeddual}.

Without loss of generality, we assume that $\mathcal{D}$ is a set of F-cores.
Fix a set $\mathcal{D}$, denote by $n$ the maximum size of a graph in $\mathcal{D}$ and construct a set $\mathcal{F}'$ as union of the set of all F-cores $F$ such that $|V_F|\leq n+1+k\leq n+1+\binom{n+1}{2}$ and there is $F\FullHom D$ for no $D\in \mathcal{D}$. Finally construct $\mathcal{F}$ as the set of all mutually non-isomorphic minimal elements of $(\mathcal{F}',\FullLeq)$.

Obviously $\mathcal{F}$ is finite because it contains graphs of bounded size.
To see that $(\mathcal{F},\mathcal{D})$ is a generalized finite F-duality pair consider the
less trivial case where for a given $G$ there is no $D\in \mathcal{D}$ such that $G\to D$.

If $|V_G|\leq n$ we have $G\in \mathcal{F'}$ and thus there is $F\in \mathcal{F}$, $F\FullLeq G$.

Consider $|V_G|>n$. Denote by $A$ an arbitrary subset of $V_G$ such that $|A|=n+1$.
Denote by $G[A]$\index{$G[A]$} the subgraph of $G$ induced by $A$.
We construct a point-determining subgraph of $V_G$ containing $A$ in the following way.
Put $$A_0=A.$$ Enumerate by $p_1,p_2,\ldots, p_k$ all pairs $(u,v)$ of vertices of $G[A]$ such that
$u\neq v$ and $N_{G[A]}(u)=N_{G[A]}(v)$.

Now for $i=1,\dots, k$ put $p_i=(u,v)$ and 
\begin{enumerate}
\item $A_i=A_{i-1}$ if $N_{G[A_{i-1}]}(u)\neq N_{G[A_{i-1}]}(v)$;
\item $A_i=A_{i-1}\cup \{v'\}$ otherwise. $v'$ is a vertex of $G$ which is in the neighbourhood of $u$ and not in the neighbourhood of $v$ (or vice versa). Such vertex exists because $G$ is point-determining.
\end{enumerate}
It is easy to see that $|A_k|\leq n+1+k\leq n+1+\binom{n+1}{2}$ and moreover $G[A_k]$ is point-determining and thus
its isomorphic copy is contained in $\mathcal{F'}$ leading to existence of $F\in \mathcal{F}$ such
that $F\FullLeq G[A_k]\FullLeq G$.
\end{proof}

%

By the results above it may seem that the structure of the full homomorphism order
is basically the same as the structure of the embedding order. We close the section by
pointing out one surprising difference. 

The main result of \cite{Feder2008} and \cite{Hell2013} shows that
in homomorphism order for a given graph $H$ there are at most two
right $F$-realizations of $H$. There are examples of graphs for all three cases: 
no realizations, one realization or two realizations. As we have shown in
Proposition~\ref{prop:realization}, a similar situation does not hold in the
embedding order, where (up to a degenerate case) there is always one realization.

\section{Surjective homomorphism orders}

Vertex and edge surjective homomorphisms are very natural kinds of special
homomorphisms. While they are generally less studied and understood than
homomorphisms, there is a lively area of research with respect to the computational complexity problem of surjective  colouring,
see a recent survey~\cite{Bodirsky2012}.

From our point of view, surjective homomorphisms are similarly easy as
embeddings and monomorphisms. We consider three cases---vertex surjective
homomorphisms, edge surjective homomorphisms, and surjective homomorphisms (that are both edge and vertex surjective).
\begin{prop}[Cores]
\label{prop:shomo-core}
Every finite graph is an S-core and VS-core.

ES-core of graph $G$ is created from $G$ by removing all isolated vertices.
\end{prop}
\begin{proof}
Consider graph $G$ and its S-core or VS-core $H$. By definition, there are
surjective homomorphisms $f:G\VSurHom H$ and $g:H\VSurHom G$. By
surjectivity of $f$ we have $|V_G|\geq |V_H|$. From the surjectivity of $g$ we have
$|V_G|\leq |V_H|$ and thus $|V_G|=|V_H|$. It follows that both $f$ and $g$ are
monomorphisms. Consequently $|E_G|\geq |E_H|\geq |E_G|$ and thus $f$ and $g$
are also isomorphisms.

On graphs without isolated vertices edge surjective homomorphisms are
surjective homomorphisms. From the fact that removing isolated vertices does not affect
any edges we know that if $G'$ is created from $G$ by removing (some) isolated
vertices, then $G\SurHomeq G'$.
\end{proof}

It follows that $(\DiGraphs,\SurLeq)$ is actually a suborder of
$(\DiGraphs,\SurLeq)$ that is induced on $(\DiGraphs,\ESurLeq)$ by the class of
all ES-cores (that is graphs without isolated vertices). 
All the basic properties of locally surjective homomorphism order are almost immediate:

\begin{prop}[Future-finiteness]
\label{prop:shomo-univ}
The orders ${(\DiGraphs,\SurLeq)}$, \\${(\DiGraphs,\ESurLeq)}$ and ${(\DiGraphs,\VSurLeq)}$ are future-finite.
\end{prop}
\begin{proof}
To see that $(\DiGraphs,\VSurLeq)$ is future-finite, fix graph $G$ and consider its up-set.
By vertex surjectivity it is easily seen that the size of every graph $H$
in the up-set of $G$ is no greater than the number of vertices of graph $G$.

$(\DiGraphs,\ESurLeq)$ is a suborder of $(\DiGraphs,\SurLeq)$ that is a non-induced suborder of $(\DiGraphs,\VSurLeq)$.
\end{proof}
\begin{prop}[Universality]
The orders $(\DiCycle,\SurLeq)$, $(\DiCycle,\ESurLeq)$ and $(\DiCycle,\VSurLeq)$ are future-finite-universal.
\end{prop}
\begin{proof}
Observe that any homomorphism between two oriented cycles is also an surjective homomorphism.
The universality follows from Lemma~\ref{lem:cycles}.
\end{proof}

\begin{prop}[Gaps]
\label{prop:surgaps}
$(G,H)$ is a VS-gap if and only if $|G|=|H|$ and $G$ is created from $H$ by removing an edge, or
$|G|=|H|+1$, $G\VSurHom H$ and there is no way to add an edge to $G$ preserving this property.

$(G,H)$ is a S-gap if and only if $|G|=|H|+1$, $|E_G|=|E_H|$, $G\SurHom H$.

ES-gaps correspond to S-gaps for pairs of graphs with no isolated vertices.
\end{prop}
\begin{proof}
It is easy to see that in all cases $|G|$ is at most $|H|+1$.
Otherwise a graph inbetween can be constructed by partly concatenating vertices
of $G$ as given by the surjective homomorphism $G\SurHom H$. The extra condition
given prevents the existence of a graph inbetween $G$ and $H$ in this case.

For a vertex surjective mapping, there are gaps $(G,H)$ with $|G|=|H|$: the
mapping must be a monomorphism. If the graphs $G$ and $H$ differ by precisely one
edge, they represent a gap in the monomorphism order.
In the edge surjective case however the
mapping must be an embedding and there are no gaps in the embedding order such that
$|G|=|H|$.
\end{proof}

\begin{prop}[Dualities]
\label{prop:surdual}
For every finite set of directed graphs $\mathcal{F}$ there are finite sets of directed graphs
$\mathcal{D}_S$, $\mathcal{D}_{VS}$ and $\mathcal{D}_{ES}$, such that $(\mathcal{F},\mathcal{D}_S)$ is a generalized finite S-duality pair,
$(\mathcal{F},\mathcal{D}_{VS})$ is a generalized finite VS-duality pair, and
$(\mathcal{F},\mathcal{D}_{ES})$ is a generalized finite ES-duality pair.
\end{prop}
\begin{proof}
Again the existence of dualities follows from Corollary~\ref{cor:dualities} and from the characterization of gaps (Proposition~\ref{prop:surgaps}).
Observe that a single vertex with a loop on it and the empty graph are the only maximal elements of the vertex surjective homomorphism
order. In edge surjective homomorphisms there are three maximal elements; single vertex, vertex with a loop and the empty graph.
All gaps are the pairs of graphs that differ by 1 in number of vertices.
\end{proof}

Note that loops on vertices have to be allowed to make finite dualities possible: 
for graphs without loops the set $\mathcal{D}$ would need to contain cliques of
arbitrary sizes, because all cliques are maximal elements in surjective homomorphism
order on (directed) graphs without loops. Thus there are no dualities in the case of
loopless graphs.

\section[Locally constrained homomorphism orders]{\fontsize{17}{8}\selectfont Locally constrained homomorphism orders}
\label{sec:locally-constrained}
In this section we consider three forms of locally constrained homomorphisms: locally bijective, locally injective and locally surjective mappings. All three are defined by a particular behavior on the neighbourhoods of vertices and thus share many of common properties. There are also a number of other non-trivial relations between the three kinds of homomorphisms. See \cite{Fiala2008} for a recent survey of this area. 
The orders of locally constrained homomorphisms have been studied in \cite{Fiala2005}. We review and extend these results.

To our knowledge, locally constrained homomorphisms have not been considered for
directed graphs. Moreover for locally constrained homomorphisms the
orders induced on connected graphs differ significantly from the
orders induced on $\Graphs$. For example, it is easy to observe that a locally
surjective homomorphism between two connected graphs is always surjective. This
is not true for non-connected graphs. In this section we will restrict our
attention to properties of locally constrained homomorphism orders on connected
undirected graphs. We however still consider loops. We denote by $\ConnGraph$\index{$\ConnGraph$}
the class of all finite undirected connected graphs. The properties of these
orders on $\Graphs$ can be derived easily from our results by considering
mappings between individual connected components. Directed graphs would
require more generalizations, because some of the underlying concepts we use
are formulated for undirected graphs only.

The following results show the main correspondence between all three types of locally constrained homomorphisms.

\begin{thm}[Fiala, Maxov\'a \cite{Fiala2006}]
If a graph $G$ admits a locally injective homomorphism $f$ to a finite and connected graph $H$ as well as a locally surjective homomorphism $g$ to $H$, then all locally constrained homomorphisms between $G$ and $H$ are locally bijective.
\end{thm}
Or in the language of homomorphism orders we get:
\begin{thm}[Fiala, Paulusma, Telle \cite{Fiala2005}]
The order ${(\ConnGraph,\LocBiLeq)}$ is the intersection of orders $(\ConnGraph,\LocSurLeq)$ and $(\ConnGraph,\LocInLeq)$.
\end{thm}

\subsection*{Orders of degree refinement matrices}
Recall that in Section~\ref{sec:DRM} we defined degree refinement matrices, which are closely connected to locally constrained homomorphisms.

We put $G\Matrixeq H$\index{$G\Matrixeq H$} if and only if $\drm(G)=\drm(H)$. The relation $\Matrixeq$ is an equivalence relation on the class of finite connected graphs, $\ConnGraph$.

From Theorem~\ref{thm:degeq1} and Theorem~\ref{thm:degeq2}
it follows that in a single equivalence class of $\Matrixeq$ all three locally constrained homomorphisms coincide.
We thus first look for the common properties of these orders within a single equivalence class.
Observe that 2-regular undirected graphs are undirected cycles and thus $\Cycle$ is the class of graphs with degree refinement matrix $(2)$. Consequently by Proposition~\ref{thm:cycles} there is a rich, future-finite-universal order within a single equivalence classes of $\Matrixeq$.

We define \emph{acyclic graphs}\index{acyclic graph} are trees with optional loops attached its vertices.
We can summarize the properties of the orders within a single equivalence class of $\Matrixeq$ as follows:
\begin{thm}
\label{thm:singleclass}
Let $\K$ be an equivalence class of $\Matrixeq$. Then
\begin{enumerate}
 \item If $\K$ consists only of acyclic graphs, then $\K$ is trivial (i.e. consisting of only one acyclic graph up to isomorphism).
 \item If $\K$ contains graphs with a cycle, then the orders $(\K,\LocBiLeq)$, $(\K,\LocSurLeq)$ and $(\K,\LocInLeq)$ are all equivalent.
They have the following properties:
\begin{enumerate}
 \item \label{item:a} all three orders are future-finite,
 \item \label{item:aa} all three orders are future-finite-universal,
 \item \label{item:b} for every pair of graphs $G,H\in \K$ there exists a graph $C\in \K$ such that $C\LocBiLeq G$ and $C\LocBiLeq H$, and
 \item \label{item:c} there are no minimal elements.
\end{enumerate}
\end{enumerate}
\end{thm}
Note that the minimal elements within equivalence classes of $\Matrixeq$ may be established when infinite graphs are considered.
For every equivalence class $\K$ containing a graph with a cycle, there exists a graph $C_{\K}$ (known as the universal cover) such that for every $G\in \K$, $C_{\K}\LocBiLeq G$,
see \cite{Leighton1982}.
\begin{proof}
The first part follows from the fact that acyclic graphs are uniquely described by their degree refinement matrices. 
See \cite{Fiala2005b} for a formal proof.

The equivalence of orders $(\K,\LocBiLeq)$, $(\K,\LocSurLeq)$ and $(\K,\LocInLeq)$ follows from Theorems \ref{thm:degeq1} and Theorem \ref{thm:degeq2}.

\ref{item:a} follows from Proposition~\ref{prop:shomo-univ}.

\ref{item:aa} follows from the fact that the same construction as used in the proof of Proposition~\ref{prop:embed-univ}
can be applied to any equivalence class of $\Matrixeq$ that contains at least one graph $G$ with a cycle.
Construct for every $k$ a graph $G_k$ created from $G$ by multiplying the length of the cycle $k$ times and duplicating its
neighbourhood accordingly so $\drm(G_k)=\drm(G)$.

A non-trivial method of constructing graphs $C$ for \ref{item:b} is shown in
\cite{Leighton1982}, and alternate proofs in terms of Bass-Serre theory appear in
\cite{Bass1990} and \cite{Neumann2011}. 

\ref{item:c} follows from the fact that for every graph $G$ containing a cycle $C$, one can construct a graph $G'\LocBiLeq G$
that contains a cycle on $2|C|$ vertices by sufficiently expanding $G$. \qedhere
\end{proof}

\subsection*{Locally surjective and locally bijective orders}
For connected graphs $G$ and $H$, it is easy to observe that both locally bijective and locally surjective homomorphisms are also surjective homomorphisms (but
not vice versa; a surjective homomorphism may not be locally surjective). 
By Proposition \ref{prop:shomo-core} we have:
\begin{prop}[Cores]
\label{prop:locsurbi-core}
Every finite connected graph is LS-core and LB-core.
\end{prop}

By Proposition \ref{prop:shomo-univ} we have:

\begin{prop}[Future-finiteness]
\label{pro:lochomo-fini}
The orders $(\ConnGraph,\LocBiLeq)$ and $(\ConnGraph,\LocSurLeq)$ are future-finite.
\end{prop}
We denote by $\UnorPath$\index{$\UnorPath$} the class of all unoriented paths.
By an analogous argument as in Lemma~\ref{lem:cycles} give:
\begin{prop}
\label{pro:lochomo-univ}
The order $(\Cycle,\LocBiLeq)$ $(={(\Cycle,\LocSurLeq)}={(\Cycle,\LocInLeq))}$ is future-finite-universal.
The order ${(\UnorPath,\LocSurLeq)}$ is future-finite-universal.
\end{prop}
\begin{proof}
It is easy to observe that for all three types of locally constrained homomorphisms there is a homomorphism from $C_k\to C_l$ if and only if $l$ divides $k$.
The same holds for unoriented paths with $k$ and $l$ edges ordered by the locally surjective homomorphisms.
The rest of the statement follows the same way as Lemma~\ref{lem:cycles}.
\end{proof}
Here we make an exception in our focus on connected graphs and state explicitly the following universality result.
It can be proved analogously as Theorem~\ref{thm:cycles} by applying Proposition~\ref{pro:lochomo-univ}
and Theorem~\ref{thm:univ}.
\begin{corollary}[Universality]
The order $(\Cycles,\LocBiLeq)$ $(={(\Cycles,\LocSurLeq)}={(\Cycles,\LocInLeq))}$ is universal.
\end{corollary}

We know of no reasonable characterization of gaps in the locally bijective order. We however
can show that:

\begin{lem}[Gaps]
\label{lem:locbigaps}
If $G$ is graph with a cycle, then there are infinitely many graphs $H$, $\drm(G)=\drm(H)$ such
that $(H,G)$ is a LB-gap, and this pair $(H,G)$ is also a LI-gap and a LS-gap.
\end{lem}
\begin{proof}

We first consider locally bijective homomorphisms.

As an example, consider $C_3$ (cycle on 3 vertices). Then for every prime $p$
the pair $(C_{3p},C_p)$ is a gap. The general case is similar. In a locally
bijective homomorphism $f:H\LocBiHom G$, the pre-images of all vertices of $G$ have
the same size (and in locally surjective homomorphisms they are all non-empty).
It remains thus to construct for every $p$ a graph on $p|H|$ vertices with a
locally bijective homomorphism to $H$. This can be done by considering any
cycle in $G$ and expanding its length $p$ times and sufficiently duplicating
all vertices connected to it.

Now observe that those LB-gaps are also LS-gaps. Assume to the contrary we have a LB-gap $(H,G)$
which is not a LS-gap. It means that there is graph $G'$ such that $H\LocSurHom G'$ and
$G'\LocSurHom G$, and at least one of those locally surjective homomorphisms is not locally
bijective. It follows that even their composition is also not locally bijective, a
contradiction with Theorem~\ref{thm:degeq2}.

LI-gaps follow in a completely analogous way.
\end{proof}
\begin{lem}[Gaps of locally surjective homomorphism order]
\label{lem:locsurgaps}
If $G$ is an acyclic graph, then there are infinitely many graphs $H$ such that $(H,G)$ is a LS-gap.
\end{lem}
\begin{proof}
Let us first consider the simple case of path $P_{n+1}$ with $n$ edges. 
It is not difficult to see that every path $P_{pn+1}$ where $p$ is a prime number
forms a gap: the only graphs in between $P_n$ and $P_{pn+1}$ and locally surjective
homomorphism of two paths must map the initial vertex to the initial or terminal vertex 
and the terminal vertex to the initial or terminal vertex. Also it can not flip
in the middle of the path.

The same argument can be used for acyclic graphs in general. Ignoring loops,
those are trees and thus they have at least two leaf vertices of degree
1. Denote two leaves of the tree as initial and terminal vertex and connect $p$
copies of the acyclic graph into a sequence to form a gap.
\end{proof}

\begin{prop}[Dualities]
\label{prop:locbidual}
There are no generalized finite duality pairs in $(\ConnGraph,\LocBiLeq)$
and in $(\ConnGraph,\LocSurLeq)$.
\end{prop}
\begin{proof}
We apply Proposition~\ref{prop:futurefiniteduals}. 

In the locally bijective case the analysis is easy. If $\drm(G)\neq \drm(H)$, then there is no locally bijective
homomorphism from $G$ to $H$, and
because there are infinitely many degree refinement matrices and thus there
are infinitely many maximal elements in $(\ConnGraph,\LocBiLeq)$.

In the case of locally surjective homomorphism order, the maximal element is a vertex with
loop. Because there are infinitely many gaps below any graph with a cycle
(Lemma~\ref{lem:locbigaps}) as well as below any acyclic graph (Lemma~\ref{lem:locsurgaps}),
there are no duality pairs in both orders.
\end{proof}
%

\subsection*{Locally injective homomorphism orders}

The surjectivity argument can not be applied to locally injective homomorphisms. We will show that in fact the locally injective homomorphism order has some properties that are in sharp contrast with the locally surjective and locally bijective homomorphism orders.

To characterize cores of the locally injective homomorphism order we apply the classical result of Theorem~\ref{thm:locin-auto} (also a special case of Theorem~\ref{thm:degeq1}):
If $G$ is a finite connected graph, then every locally injective homomorphism $f:G\LocInHom G$ is an automorphism of $G$.
\begin{corollary}[Cores]
\label{pro:locin-core}
Every finite connected graph is a LI-core.
\end{corollary}
\begin{proof}
The composition of the locally injective homomorphisms $f:G\LocInHom H$ and $f':H\LocInHom G$ is a locally injective homomorphism $f\circ f':G\LocInHom G$. By Theorem~\ref{thm:locin-auto},
$f\circ f'$ is surjective and thus $f'$ is surjective and it follows that $|V_H|\geq |V_G|$. By the same argument on $f'\circ f$, we have $f$ is surjective and thus $|V_G|\geq |V_H|$. Consequently $|V_G|=|V_H|$ and both $f$ and $f'$ are isomorphisms.
\end{proof}

For the first (and last) time we are able to show a variant of Welzl's density theorem~\cite{Welzl1982}. This shows that the locally injective homomorphism order on undirected connected graphs is not locally finite.

\begin{thm}[Density]
\label{thm:density}
Let $G$ and $H$ be connected graphs such that $\drm(G)\neq \drm(H)$, $G\LocInHom H$ and $H$ has no vertices of degree 1. Then:
\begin{enumerate}

\item[(a)] There exists a connected graph $F$, such that $G\LocInL F\LocInL H$, $\drm(F)\neq \drm(G)$ and $\drm(F)\neq \drm(H)$.

\item[(b)] When $G$ has no vertices of degree 1 and $H$ has at least one cycle with a vertex of degree greater than 2, then $F$ can be constructed to have no vertices of degree 1 and contain a cycle with a vertex of degree greater than 2.

\end{enumerate}
\end{thm}
\begin{proof}
Fix a locally injective homomorphism $f:G\to H$. Because $\drm(G)\neq \drm(H)$ we know that $f$ is not locally bijective. Denote by $u$ a vertex such that $f(u)$ is strictly injective (that means injective but not bijective).

First construct $F_1$ from a graph $G$ by extending a single vertex $u'$ and a single edge $(u,u')$. It is easy to observe that $F_1$ satisfies the conditions given on $F$ by (a). Obviously there are locally injective homomorphisms $G\LocInHom F_1$ ($G$ is a subgraph of $F_1$) and $F_1\LocInHom H$ (by extending the homomorphism $f$ that is strictly injective at vertex $u$). $\drm(F_1)$ differs from $\drm(H)$ because $H$ has no vertex of degree 1. Suppose $\drm(G) = \drm(H)$, from Theorem~\ref{thm:degeq1} there is a locally bijective homomorphism from $G$ to $F_1$. Because every locally bijective homomorphism is surjective. we have $|G|\geq|F_1|$ that is a contradiction with $|G|<|F_1|$. Thus $\drm(G)\neq \drm(H)$. 

Now we extend our construction to meet the requirements of (b).
Denote by $v$ the vertex of $H$ that lies on a cycle and has degree greater than 2. Because $H$ is connected, there is a path of length $l$ from $f(u)$ to $v$. Finally denote by $c$ the length of the cycle containing $v$.

We call the graph created by connecting a path of length $a$ to a cycle of length $b$ an {\em $(a,b)$-lasso}\index{$(a,b)$-lasso}.

Now construct $F_{k,k'}$ as disjoint union of graph $G$ and $(l+ck,ck')$-lasso with vertex $u$ and the single vertex of degree 1 in the lasso identified.
See Figure~\ref{fig:lasso}.

\begin{figure}
\centerline{\includegraphics{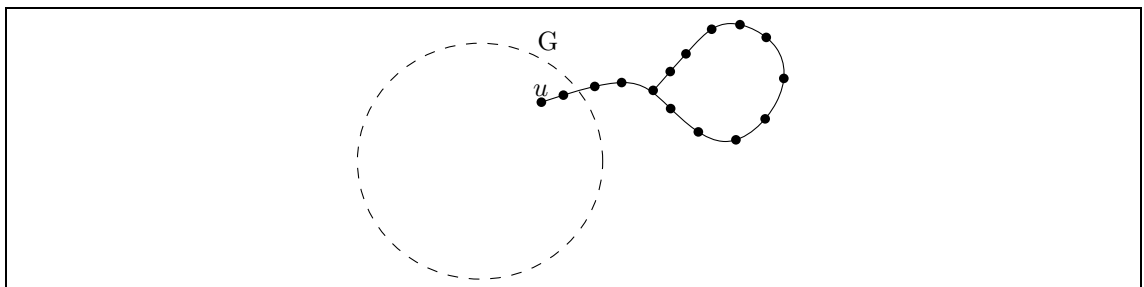}}
\caption{Graph $G$ with vertex $u$ extended by a (5,11)-lasso.}
\label{fig:lasso}
\end{figure}

It is easy to observe that for every $k,k'\geq 1$ there is a locally injective homomorphism $G\to F_{k,k'}$ ($G$ is an induced subgraph of $F_{k,k'}$). Similarly there is a locally injective homomorphism $F_{k,k'}\to H$ that can be constructed by extending $f$. The length of the cycle and the path has been chosen in a way so the first $l$ vertices can be mapped into the path from $f(u)$ to $v$ in $H$ and the remaining vertices are mapped on the cycle containing $v$ in $H$.

For a given vertex of degree $2$ denote by a distance pair $(a,b)$ the distances to nearest vertices of degree different than two, such that $a\leq b$. Two vertices of degree 2 with different distance pairs can not be in the same block of degree refinement matrix. Consequently $\drm(F_{k,k'})\neq \drm(F_{m,m'})$ whenever $k\neq m$ or $k'\neq m'$. It follows that there exists a choice of $k$ and $k'$ such that $\drm(F_{k,k'})\neq \drm(G)$ and $\drm(F_{k,k'})\neq \drm(H)$. We put $F=F_{k,k'}$ for such choice of $k$ and $k'$.
\end{proof}

The constraints given on graph $H$ may seem artificial. However they lead to a full characterization of graphs $H$ such that there is $G$, $\drm(G)\neq \drm(H)$ and $(G,H)$ is a gap.

It remains to explore the cases where $H$ has neither a vertex of degree 1, nor a cycle with vertex of degree greater than 2. We know that $H$ is not a tree and thus it contains a cycle. Because $H$ is connected it follows that all such graphs $H$ are cycles.
Graphs with a locally injective homomorphism to a cycle are either cycles or paths. When $G$ is a path of length $l$, the graphs $F$ can be chosen as a path of length $l+1$. It is easy to observe that $G\LocInHom F\LocInHom G$ and the degree matrices of $G,H$ and $F$ are mutually disjoint.
When $G$ is a cycle, we have $\drm(G)=\drm(H)$ and thus we can not hope for density result in general. As an example, consider $G$ to be a cycle of length $4$ and $H$ to be a cycle of length $8$.

Combining all results together we get:
\begin{thm}[Gaps]
\label{thm:locingap}
Let $H$ be a connected graph.
\begin{itemize}
\item[(a)] There exists a connected graph $G$ such that $\drm(G)=\drm(H)$ and $(G,H)$ is a gap in $(\ConnGraph,\LocInLeq)$ if and only if $H$ contains cycles.
\item[(b)] There exists a connected graph $G$, such that $\drm(G)\neq \drm(H)$ and $(G,H)$ is a gap if and only if $H$ has at least one vertex of degree 1.
\end{itemize}
\end{thm}
\begin{proof}
When $H$ is acyclic, then its equivalence class is trivial. Consequently one implication of $(a)$ is true.

We prove the other implication. As a simple consequence of Theorem~\ref{thm:degeq1}, for any choice of $G'$ and $H$ with $\drm(G')=\drm(H)$, graph $G$ such that $G'\LocInHom G\LocInHom H$ must have $\drm(G)=\drm(G')=\drm(H)$. We thus seek the gaps in the order $(\ConnGraph,\LocInLeq)$ restricted to a single equivalence class of $\Matrixeq$. By Theorem~\ref{thm:singleclass}, this order is future-finite. Consequently there are only finitely many choices of $G$ and every maximal choice of $G$ (in $(\ConnGraph,\LocInLeq)$) produces a gap.

One implication of $(b)$ is a consequence of Theorem~\ref{thm:density} and the discussion concerning cycles given below.

In the other direction when $H$ has a vertex $v$ of degree 1, consider graph $G=H\setminus v$ and we show it has the desired properties. Assume, to the contrary that there is graph $F$ such that $G\LocInL F\LocInL H$. 

First assume that $H$ is acyclic. In this case $F$ is also acyclic and the locally injective homomorphisms are embeddings. Consequently $F$ is a proper subgraph of $H$ and $G$ is a proper subgraph of $F$. A contradiction.
 
Assume that $H$ contains a cycle. Denote by $G'$ the largest induced subgraph of $G$ containing no vertices of degree 1. Denote by $g$ the locally injective homomorphism from $G$ to $F$ and by $f$ the locally injective homomorphism from $F$ to $H$. Then the locally injective homomorphisms $g\circ f$ partialized to $G'$ may only use vertices of $G'$ and thus, by Theorem~\ref{thm:locin-auto}, it is an automorphism. It follows that any component of $G\setminus G'$ must be mapped to acyclic component of $H\setminus G'$. Such mappings are embeddings, too. Consequently $g\circ f$ is an embedding and thus $g$ is also an embedding. Again, it follows that $F$ is isomorphic either to $G$ or $H$, which causes a contradiction.
\end{proof}

Further developing these ideas, it is possible to show the universality of locally constrained order on connected graphs.
 
\begin{thm}[Universality]
\label{thm:lochomouniv}
$(\ConnGraph,\LocInLeq)$ is a universal order.
\end{thm}
Fix an order $(P,\leq_P)$. For brevity we assume that
$P=\{p_1,p_2,\ldots\}$ consists of prime numbers greater than or equal to 5. We construct an on-line
embedding from $(P,\leq_P)$ to $(\ConnGraph,\LocInLeq)$ by representing the
elements by cycles similarly as in Theorem~\ref{thm:cycles}. 
Main complication lies in the fact that we need to connect
cycles representing an element of $(P,\leq_P)$ to connected graph.

We build the representations from gadgets containing the cycles of desired lengths.
{\em Sunlet on vertices $v_0,v_1,\ldots, v_n$ and $v'_1,v'_2,\ldots, v'_n$}\footnote{It is basically the same as $n$-sunlet $S_n$ in Definiton~\ref{def:sunlet}, just the vertices are marked.}\index{sunlet} is a graph on vertices $\{v_0,v_1,\ldots,$ $v_n,v'_1,v'_2,\ldots, v'_n\}$
containing a cycle on vertices $v_0,v_1,\ldots, v_n$ and edges $(v_0,v'_0)$, $(v_1,v'_1)$, \ldots, $(v_n,v'_n)$.
Vertices $v_0,v_1,\ldots, v_n$ are called {\em internal vertices}\index{internal vertices} while vertices $v'_1,v'_2,\ldots, v'_n$ are {\em external}\index{external}.
The length of sunlet corresponds to the length of cycle on its internal vertices.
For every $n\in P$ put: $$l(n)=\{p\mid p\in P,p\leq n \hbox{ and } p\geq_P n\},$$
$$p(n)=\prod_{p\in l(n)}p.$$
Denote by $H(n)$ the graph constructed as a disjoint union of two sunlets:
\begin{enumerate}
\item a left sunlet on vertices $l_{n,0},l_{n,1},l_{n,2},\ldots, l_{n,2^n p(n)-1}$ and \\$l'_{n,0},l'_{n,1},l'_{n,2},\ldots, l'_{n,2^n p(n)-1}$,
\item a right sunlet on vertices $r_{n,0},r_{n,1},r_{n,2},\ldots, r_{n,2^n3-1}$ and \\$r'_{n,0},r'_{n,1},r'_{n,2},\ldots, r'_{n,2^n3-1}$,
\end{enumerate}
with an additional edge $(l'_{n,p(n)},r'_{n,0})$.

Graphs $E(p)$, $p\in P$ for the sample order $(P,\leq_P)$ shown in Figure~\ref{fig:poset2} are schematically
depicted in Figure~\ref{fig:representation}.

\begin{lem}
For $n>n'\in P$ there exists a locally injective homomorphism $H(n)\LocInHom H(n')$ if and only
if $n\leq_P n'$. If such locally injective homomorphism exists, then it
maps vertex $l_{n,0}$ of $H(n)$ to vertex $l_{n,p(n)}$ of $H(n')$.
\end{lem}
\begin{proof}

Both $H(n)$ and $H(n')$ consist of left and right sunlets connected by an edge.
The only vertices of degree 3 are internal vertices of sunlets.
Consequently every locally injective homomorphism $f:H(n)\LocInHom H(n')$ must
map sunlets of $H(n)$ to sunlets of $H(n')$.

The length of the left sunlet is always divisible by at least one prime number greater or equal to 5
and is not divisible by 3.
The length of the right sunlet is always multiple of powers of 2 and 3. 
Consequently every locally injective
homomorphism $f:H(n)\LocInHom H(n')$ must map the left sunlet of $H(n)$ to the left sunlet of
$H(n')$ and the right sunlet of $H(n)$ to the right sunlet of $H(n')$.
Such homomorphism exists if and only if $p(n)$ divides $p(n')$ and that is
true, by definition of $p(n)$, if and only of $n\leq n'$ and $n\leq_P n$.
\end{proof}

\begin{figure}
\centerline{\includegraphics{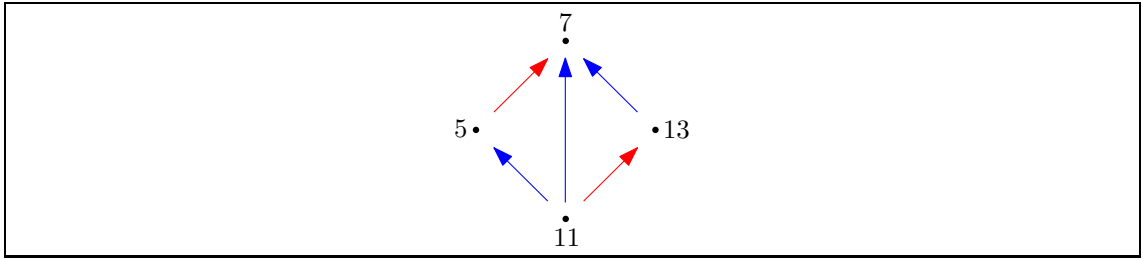}}
\caption{The order $(P,\leq_P)$.}
\label{fig:poset2}
\end{figure}
\begin{figure}
\centerline{\includegraphics{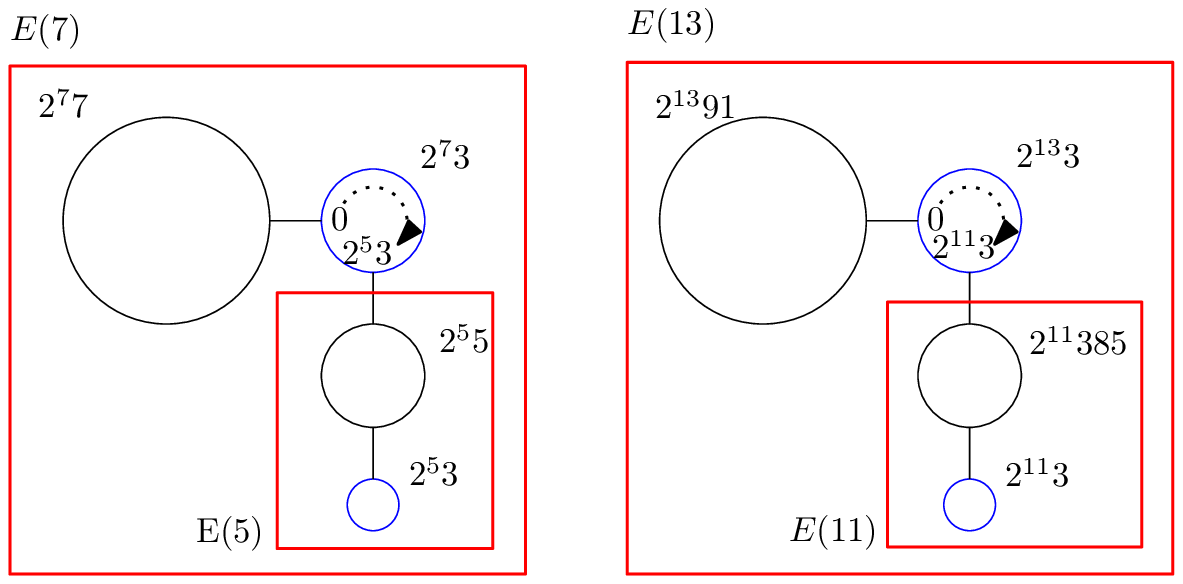}}
\caption{Representation of order $(P,\leq_P)$. Black circles denote left sunlets, blue circles right sunlets.}
\label{fig:representation}
\end{figure}

\begin{proof}[Proof (of Theorem~\ref{thm:lochomouniv})]
Denote by $E(n)$ the graph created by the disjoint union of $H(n)$ and graphs $E(k)$, $k<n, p_k\leq_P p_n$,
with additional edges between $l_{k,0}$ and $r_{n,2^{k-1} 3}$.

We claim that $E$ is an embedding from $(P,\leq_P)$ to $(\ConnGraph,\LocInLeq)$.

We show by induction on $n'$ that for every $n\leq n'$
\begin{enumerate}
\item there is a locally injective homomorphism $f_{n,n'}:E(n)\to E(n')$ if and only if $p_n\leq_P p_{n'}$
and moreover $f_{n,n'}$ maps $l_{n,0}$ to $l_{n',p(n)}$;
\item there is a locally injective homomorphism $f_{n',n}:E(n')\to E(n)$ if and only if $p_{n'}\leq_P p_{n'}$. 
\end{enumerate}
We consider the following individual cases:
\begin{enumerate}
\item $p_n$ is incomparable with $p_n'$ in $(P,\leq_P)$. The left cycle of $H(n)$ have length that is multiple of $p_n$, but the left cycle of any $H(n'')$, $n\leq_P n'$ does not.
Consequently there is no homomorphism $H(n)\LocInHom H(n')$. The opposite direction follows in complete analogy.
\item $n\leq n', n\leq_P n'$. $E(n')$ contains copy of $E(n)$.
\item $n\leq n', n\geq_P n'$. We build homomorphism $f_{n',n}$ as follows:
\begin{enumerate}
\item $f_{n',n}(l_{n',i})=l_{n,(p(n)+i)\mod 2 p(n)}$.
\item $f_{n',n}(r_{n',i})=r_{n,i\mod 2^n3}$.
\item For every $n''<n$, $n''\leq_P n'$, the vertices of copy of $E(n'')$ in $E(n')$ are mapped to copy of $E(n'')$ in $E(n)$.
\item For every $n''>n$, $n''\leq_P n'$, the vertices of copy of $E(n'')$ in $E(n')$ are mapped to $E(n)$ by $f_{n'',n}$.\qedhere
\end{enumerate}
\end{enumerate}
\end{proof}

\begin{prop}[Dualities]
For a finite set of undirected connected graphs $\mathcal{D}$ there is a finite set of undirected connected graphs
$\mathcal{F}$ such that $(\mathcal{F},\mathcal{D})$ is a generalized finite LI-duality pair
if and only if $\mathcal{D}$ consists of acyclic graphs.
\end{prop}
\begin{proof}
Every duality pair must be also a duality pair when restricted to a given
equivalence class of degree refinement matrices. By the same argument as in
Proposition~\ref{prop:locbidual} this is not possible for equivalence classes
consisting of infinitely many elements.

The only equivalence classes consisting of finitely many elements are those corresponding
to acyclic graphs.

Let $\mathcal{F}$ be a given set of acyclic connected undirected graphs.
Denote by $n$ the maximal number of vertices of graph in $\mathcal{F}$.
$\mathcal{D}$ can be formed as the set of all acyclic graphs $T$ 
on at most $n+1$ elements such that there is no locally bijective homomorphism
from $F\in\mathcal{F}$ to $\mathcal{T}$.

Because the locally injective homomorphism order restricted to acyclic graphs
is an embedding order, it remains to show that for every graph $G$ containing
a cycle there is graph $D\in \mathcal {D}$ such that there is an locally
injective homomorphism from $D\LocInHom G$. This follows from fact that $\mathcal{F}$
contains a path on $n+1$ vertices.
\end{proof}


\subsection*{Remarks on the orders of DRMs}
Fiala, Paulusma and Telle~\cite{Fiala2005} have considered the orders implied by locally constrained homomorphisms on  degree refinement matrices.
Denote the class of all degree refinement matrices of graphs in $\ConnGraph$ by $\Matrices$\index{$\Matrices$}. We define three relations $\LocBiHom$, $\LocSurHom$ and $\LocInHom$\index{$(\Matrices,\AnyHom)$} as follows: for two matrices $M,N\in \Matrices$ we have $M\AnyHom N$ if there exist graphs $G\in\Graphs$ with $\drm(G)=M$ and $H\in\Graphs$ with $\drm(G)=H$ such that $G\AnyHom H$ holds for the appropriate local constraint.

As stated earlier $(\Matrices,\LocBiHom)$ is a trivial order where no distinct
elements are comparable. We further show the following two results that shed
more light on the interplay among the individual variants of locally constrained
homomorphisms:

\begin{thm}[Fiala, Paulusma, Telle \cite{Fiala2005}]
\label{thm:matrixorder}
For every constraint $*=LB$, $LI$, $LS$ the relation $(\Matrices,\AnyLeq)$ is an order. It arises as a factor order $(\Graphs,\AnyLeq)$ when we unify graphs that have the same degree refinement matrices.

\end{thm}
\begin{corollary}[Fiala, Paulusma, Telle \cite{Fiala2005}]
The intersection of orders $(\Matrices,\LocSurLeq)$ and $(\Matrices,\LocInLeq)$ is a trivial order $(\Matrices,\LocBiLeq)$.
\end{corollary}

This gives an unexpected insight: while all three locally constrained
homomorphisms agree within the equivalence classes of degree refinement
matrices, they fully disagree inbetween; the existence of a locally surjective
homomorphism rules out the existence of locally injective homomorphism.

\section{The intervals of the line graph order} \label{sec:line}
Line graph homomorphisms are another variant of homomorphisms, which defines the mapping on edges.
In this section, we consider the line graph order.
Recently D.~E.~Roberson proposed a systematic study of homomorphism order on the class
of line graphs~\cite{Roberson}.
Here, a \emph{line graph}\index{line graph} of an undirected graph $G$, denoted by $L(G)$\index{$L(G)$}, is graph
$H=(V_H,E_H)$ such that $V_H=E_G$ and two vertices of $L(G)$ are adjacent if
and only if their corresponding edges share a common endpoint in $G$. Because edges of $G$ play the role of vertices of $L(G)$, we will refer vertices of line graphs as \emph{nodes}\index{node}.

The classical Vizing theorem gives an insight into the structure of the homomorphism
order of line graphs in terms of \emph{chromatic index}\index{chromatic index} $\chi'(G)$\index{$\chi'(G)$}
that is the chromatic number of $L(G)$, i.e. the minimum number of colours needed to colour edges of a graph $G$ such that edges with a common vertex receive different colours:

\begin{thm}[\textbf{Vizing theorem}~\cite{Vizing1964}]
For any graph $G$ of maximum degree $d$ it holds that $\chi'(G) \le d+1$.
\end{thm}

Since the line graph of a graph with a vertex of degree $d$ contains a $d$-clique, the Vizing theorem splits graphs into two classes. {\em Vizing class 1}\index{Vizing class 1} contains the graphs whose chromatic index is the same as the maximal degree of a vertex, while {\em Vizing class 2}\index{Vizing class 2} contains the remaining graphs.

The approach taken by Roberson~\cite{Roberson} divides the class of line graphs into
intervals. By $[K_{n},K_{n+1})_\mathcal{L}$\index{$[K_{n},K_{n+1})_\mathcal{L}$} we denote the class of all line
graphs $L(G)$ such that $K_{n}\leq L(G)<K_{n+1}$. The line graphs in each interval have a particularly simple characterization:

\begin{corollary}\label{cor:Vizing}
The intervals $[K_{d},K_{d+1})_\mathcal{L}$ consist of line graphs of graphs whose maximum degree is $d$. 
\end{corollary}

\begin{proof}
The existence of a $(d+1)$-edge colouring is equivalent with $L(G)\le K_{d+1}$. Note that for the Vizing class 1 we have $L(G)\le K_d \le K_{d+1}$.

As $G$ contains a vertex of degree $d$, we have $K_d\le L(G)$, indeed $K_d\subseteq L(G)$. On the other hand, a clique on $d+1\ge 4$ vertices can be formed only from $d+1$ edges sharing a common vertex, hence $K_{d+1} \not\le L(G)$. The same argument used for $K_d\le L(G)$ implies that $G$ contains a vertex of degree $d$.
\end{proof}

Line graphs can be considered as almost perfect graphs (a {\em perfect graph}\index{perfect graph} is a
graph in which the chromatic number of every induced subgraph equals to the size
of the largest clique of that subgraph). The homomorphism order of the class of
perfect graphs is a trivial chain, since the core of every perfect graph is a
clique. The almost-perfectness of the class of line graphs suggests that the
homomorphism order of this class may be more constrained in its structure
than the homomorphism order of graphs in general, and indeed many of the results about properties of the homomorphism order can not be easily restricted to the line graphs.

Roberson, in~\cite{Roberson}, showed that the homomorphism order of line graphs
contains many gaps. This is the first important difference from the structure of
the homomorphism order of graphs which was shown (up to one exception) to be
dense by Welzl~\cite{Welzl1982}. Roberson also asked whether every interval
$[K_d,K_{d+1})_L$, $d\geq 3$ contains infinitely many incomparable elements.
The answer is trivially negative for graphs with maximal degree 1 and 2. We give
an affirmative answer to this problem. Indeed, we show the following (the proof is at the end of this section):

\begin{thm}
\label{thm:main}
The homomorphism order of line graphs is universal 
on every interval $[K_d,K_{d+1})_{\mathcal L}$ for $d\geq 3$.
\end{thm}

This further develops the results on the universality of the homomorphism order of special classes of graphs (see e.g.~\cite{Hubicka2004,Hubicka2005,Nesetril2006,Nesetril2007}), and on universal partially ordered structures in general (see e.g.~\cite{Johnston1956,Hedrlin1969,Hubicka2005a,Lehtonen2008,Lehtonen2010,Hubicka2011,Kwuida2011}).

As a special case, universality of the interval $[K_3,K_4)_L$ follows from the
construction given by \v{S}\'amal~\cite{Samal2013}. This is not an obvious
observation --- one has to carefully check that for the graphs constructed in~\cite{Samal2013}
the existence of a circulation coincides with the existence of a homomorphism of
line graphs. Our proof use a new approach based on a new divisibility argument which we have introduced for a similar occasion in Section~\ref{sec:universality}. This argument leads to 
a simpler construction without the need of 
complex gadgets (Blanu\v sa snarks) used by \v{S}\'amal~\cite{Samal2013}.

Let $(P,\leq_P)$\index{$(P,\leq_P)$} be the order where $P$ consists of all finite sets of
integers and for $A,B\in P$ we put $A\leq_P B$ if and only if for every $a\in
A$ there is $b\in B$ such that $b$ divides $a$. 

\begin{thm}
\label{thm:universal}
The order $(P,\leq_P)$ is a universal order.
\end{thm}

By the {\em divisibility order}\index{divisibility order}, denoted by $(\mathbb{Z},\leq_d)$\index{$(\mathbb{Z},\leq_d)$}, we mean an order where elements are natural numbers and $n$ is smaller than $m$ if $n$ is divisible by $m$.

\begin{lemma}
\label{lem:futurefiniteuniv}
The divisibility order $(\mathbb{Z},\leq_d)$ is future-finite-universal.
\end{lemma}

\begin{proof}
Denote by $\mathbb{P}$\index{$\mathbb{P}$} the set of all odd prime numbers\footnote{Here we could use the set of all prime numbers, choose the odd ones is for consistency of the thesis}. Apply Lemma \ref{lem:pastfiniteuniv} for $A=\mathbb P$.
Observe that $A\in \Pfin(\mathbb P)$ is a subset of $B\in \Pfin(\mathbb P)$ if and only if $\prod_{p\in A} p$ divides $\prod_{p\in B} p$.
\end{proof}

Then by Theorem~\ref{thm:univ} we prove Theorem~\ref{thm:universal}.

\subsection*{Dragon graphs}
\label{sec:dragon}
We use a simple gadget called $d$-dragon which is also used in several constructions developed by Roberson~\cite{Roberson}. In our constructions, the parameter $d$ specifies the maximal degree of a vertex:
\begin{defi}
For $d\geq 3$, the \emph{$d$-dragon}\index{dragon}\footnote{The name is derived from the visual similarity of this graph to a kite, which in Czech is `dragon'.}, denoted by $D_d$\index{$D_d$}, is the graph created from $K_{d+1}$ by replacing one of its edges by a path on 3 vertices.
\end{defi}
The 3-dragon is depicted in Figure~\ref{fig:dragon}. 

\begin{figure}[!ht]
	\centering
	\includegraphics[width=0.3\textwidth]{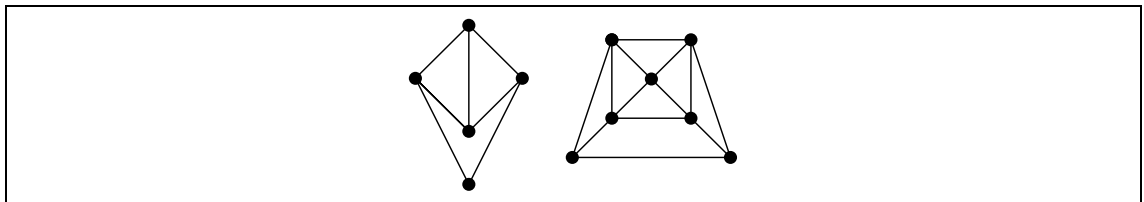}
	\caption{The 3-dragon $D_3$ and its line graph $L(D_3)$.}
	\label{fig:dragon}
\end{figure}



We proceed by a simple lemma about edge-colourings of dragons.

\begin{lemma}\label{lem:chrInd_Dragon}
For all $d\geq 3$ it holds that $D_d$ is Vizing class 2 graph, i.e. its chromatic index is $d+1$.
\end{lemma}

\begin{proof}
By Vizing theorem, $L(D_d)$ is $(d+1)$-colourable.
We prove that $L(D_d)$ is not $d$-colourable. The number of edges of $D_d$ is $\frac{d(d+1)}2+1=\frac{d^2+d+2}2$, the number of vertices is $d+2$. We use the fact that every $k$-edge-colouring yields a decomposition of the graph into $k$ disjoint matchings. We consider two cases:
\begin{enumerate}
\item If $d$ is odd, then the maximum size of a matching in $D_d$ is $\frac{d+1}2$, so the
partition contains at least $d+1$ matchings. Thus the chromatic number of $L(D)$ is $d+1$.
\item If $d$ is even, then the maximum size of a matching in $D_d$ is $\frac{d+2}{2}$. However note that there is a vertex with degree 2, so there are at most 2 maximal matchings. The others matchings have the size at most $\frac{d}2$.
It follows that $d$ matching can cover at most $2(\frac{d+2}2)+(d-2)\frac{d}2=\frac{4+d^2}2$ edges. For $d\geq 4$ we have $\frac{4+d^2}2\leq \frac{d^2+d+2}2$ and thus the partition contains at least $d+1$ matchings, i.e. colour classes. \qedhere
\end{enumerate} 
\end{proof}

In our construction we will use the fact that $d$-dragons are cores
as we show in the following lemma.

\begin{lemma} \label{lem:core_dragon}
The graph $L(D_d)$, $d\geq 3$, is a core.
\end{lemma}

\begin{proof}
For $d=3$ observe that $L(D_3)$, depicted in Figure~\ref{fig:dragon}, is not 3-colourable, 
while each of its induced subgraphs is. Hence the statement holds for $d=3$. 

For $d\ge 4$ denote the vertices of $D_d$ by $1,2,\dots, d,d+1,d+2$, where vertices 
$1,2,\dots,d+1$ have degree $d-1$ and the vertex $d+2$ has degree 2.
The vertices of degree $d$ correspond to $d$-cliques in $L(D_d)$. 
Each pair of those $d$-cliques share at most one node that corresponds to
the edge connecting the original pair of vertices. 
Note that the shared node is unique for each such pair.
Observe also that there are no other $d$-cliques in $L(D_d)$.
This follows from the fact that the only way to create a $4$-clique in a line graph
is by a vertex of degree at least 4. See Figure~\ref{fig:4-dragon}.

\begin{figure}[!ht]
	\centering
	\includegraphics{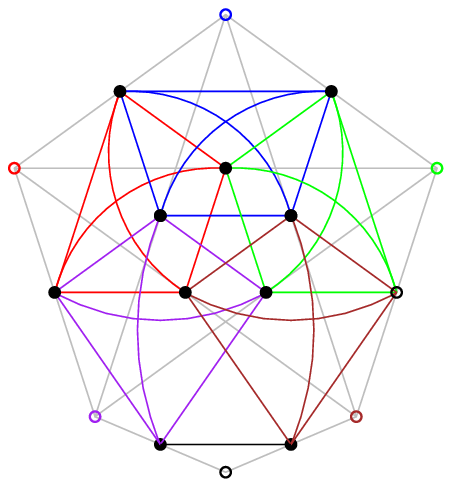}
	\caption{Line graph of 4-dragon with cliques corresponding to the neighbourhoods of two vertices distinguished.}
	\label{fig:4-dragon}
\end{figure}

Consider a homomorphism $f:L(D_d)\to L(D_d)$. Every homomorphism must map a $d$-clique
to a $d$-clique, and thus it defines a vertex mapping $f':\{1,2,\ldots ,d+1\}\to \{1,2,\ldots ,d+1\}$ in $D_d$.

\begin{figure}[ht!]
	\centering
	\includegraphics[width=0.6\textwidth]{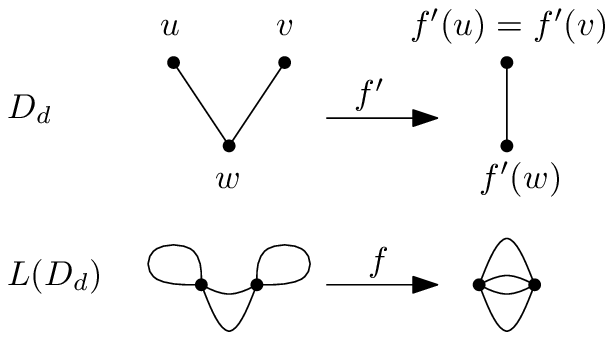}
	\caption{Mapping two $d$-cliques to the same target.}
	\label{fig:d-clique}
\end{figure}

Assume that there are distinct $u,v\in \{1,\ldots,d+1\}$ such that $f'(u)=f'(v)$, see Figure~\ref{fig:d-clique}.
Take any $w \in \{2,\ldots, d\}\setminus\{u,v\}\ne\emptyset$. Because the node shared 
by the cliques corresponding to $u$ and $w$ is unique, it is different from the node shared by the cliques corresponding to $v$ and $w$. Consequently,
the cliques corresponding to $f'(u)=f'(v)$ and $f'(w)$ share at least two nodes. 
Since distinct $d$-cliques of $L(D_d)$ may share at most one node, it follows
that $f'(w)=f'(u)$. Hence $f'$ is either a bijection 
or a constant function. 
On the other hand, $f'$ can not be a constant function by Lemma~\ref{lem:chrInd_Dragon},
as otherwise such mapping would yield an edge colouring of $D_d$ by $d$ colours.

Since $f'$ is a bijection on vertices $\{1,2,\dots,d+1\}$ of $D_d$, the mapping $f'$ must be a bijection on the edges between these vertices. The only way to get a homomorphism $f$ of the whole $D_d$ is to extend the mapping bijectively also on edges $\{1,d+2\}$ and $\{d+1,d+2\}$.
By this argument we have proved that $f$ is an isomorphism.
\end{proof}

\subsection*{Indicator construction}
\label{sec:indicator}

We briefly describe the indicator technique, called often the `arrow construction' in~\cite{Hell2004}.
Informally, this construction means replacing every edge of a given graph $G$ by a copy of graph $I$
(an indicator) with two distinguished vertices identified with the endpoints of the edge.
Figure~\ref{fig:circle} (left) shows result of indicator construction on graph
in Figure~\ref{fig:sunlet} with indicator shown in Figure~\ref{fig:gadget} (left). We give a precise
definition of this standard notion:

\begin{figure}[!ht]
	\centering
	\includegraphics[width=0.3\textwidth]{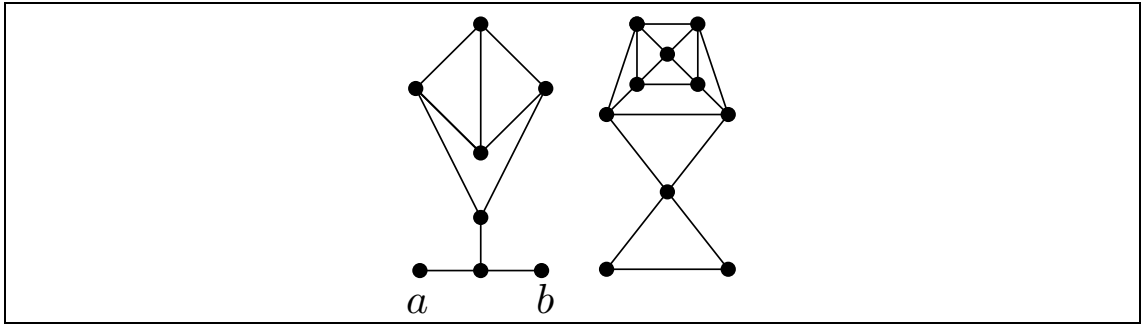}
	\caption{Indicator $I_3(a,b)$ and its line graph.}
	\label{fig:gadget}
\end{figure}

An {\em indicator}\index{indicator} is any graph $I=(V_I,E_I)$ with two distinguished vertices $a$,$b$.

Given a graph $G=(V_G,E_G)$, we denote by $G*I(a,b)$\index{$G*I(a,b)$}\footnote{Here the *-operation is different from the one in relations.} the graph $H=(V_H,E_H)$,
where each edge is replaced by an extra copy of $I(a,b)$, where the vertices $a$ and $b$
are identified with the original vertices.

Formally, to obtain $V_H$ we first take the Cartesian product 
$E_G\times V_I(a,b)$ and factorize this set by the equivalence relation $\sim$
consisting of the following pairs:
$$((x,y),a)\sim((x,y'),a),$$
$$((x,y),b)\sim((x',y),b),$$
$$((x,y),b)\sim((y,z),a).$$

In other words, the vertices of $H$ are equivalence classes of the equivalence $\sim$.
For a pair $(e,x)\in E\times V_I$, the symbol $[e,x]$\index{$[e,x]$} denotes its equivalence class.

Vertices $[e,x]$ and $[e',x']$ are adjacent in $H$ if $e=e'$ and $\{x,x'\}\in E_I$, we add no other edges.

\subsection*{Final construction}
\label{sec:final}

It is a standard technique to use the indicator construction to represent the class
of graphs which is known to be universal (such as oriented paths)
within another class of graphs (such as planar graphs) by using an appropriate
rigid indicator (see~\cite{Hubicka2004}). It is then possible to show that the structure induced
by the homomorphism order is preserved by the embedding via the indicator construction.

While our construction also uses indicators, the application is not so direct.
It is generally impossible to have an indicator that would turn a graph into
a line graph. We use the indicator to make graphs more rigid with respect
of homomorphisms of their line graphs and model the divisibility order
directly.

Our basic building blocks are the following:

\begin{defi}
\label{def:sunlet}
The \emph{$n$-sunlet graph}\index{sunlet graph}, denoted by $S_n$\index{$S_n$}, is the graph on $2n$ vertices obtained by attaching $n$ pendant edges to a cycle $C_n$, see Figure~\ref{fig:sunlet}.
\end{defi}
\begin{figure}[!ht]
	\centering
	\includegraphics{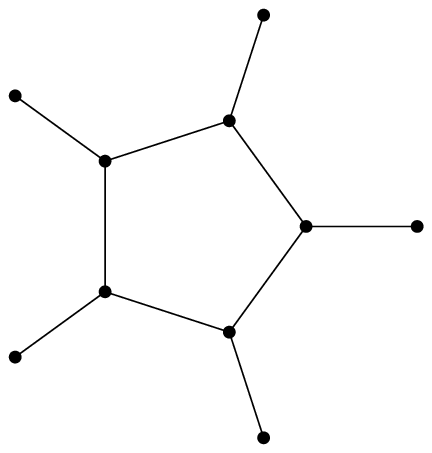}
	\caption{The 5-sunlet $S_5$.}
	\label{fig:sunlet}
\end{figure}

\begin{defi} For $d\geq 3$ the \emph{indicator}\index{indicator} $I_d(a,b)$ is the graph created from the disjoint union of the dragon $D_d$ and a path on vertices $a,c,b$, where the vertex $c$ is connected by an edge to the vertex of degree 2 in $D_d$, see Figure~\ref{fig:gadget}.
\end{defi}


The desired class of graphs to show universality of the interval $[K_{d},K_{d+1})_\mathcal{L}$, $d\ge 3$ consists of graphs $S_n*I_d(a,b)$ for 
$n\ge 3$. We abbreviate $S_n*I_d(a,b)$ by the symbol $G_{n,d}$.
An example, the graph $G_{5,3}$, is shown in Figure~\ref{fig:circle}. By squares are
indicated vertices of degree three of the original sunlet graph. The three incident edges 
are in the line graph drawn as the triplets joined by the dashed triangles.

\begin{figure}[!ht]
	\centering
	\includegraphics[width=0.8\textwidth]{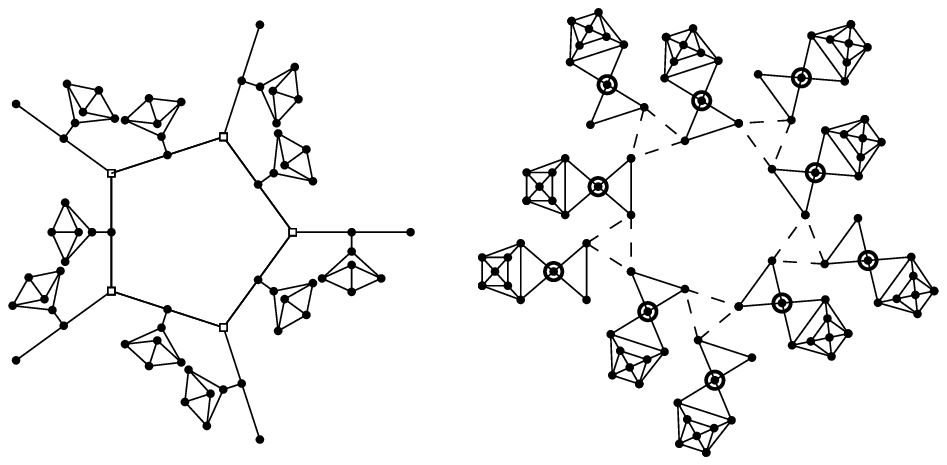}
	\caption{The graph $G_{5,3}=S_5*I_3(a,b)$ and its line graph.}
	\label{fig:circle}
\end{figure}

By Corollary~\ref{cor:Vizing} the graph $L(G_{n,m})$ is in the interval $[K_m,K_{m+1})_{\mathcal L}$ for every $n\geq 3$ and $m\geq 3$.

It remains to show the following property of the graphs $G_{n,m}$.
\begin{prop}
\label{prop:homo}
For every $d\geq 3$, $n\geq 3$, $n'\geq 3$ there is a homomorphism from $L(G_{n,d})$ to $L(G_{n',d})$ if and only if $n$ is divisible $n'$.
\end{prop}

One direction of the proposition is trivial. If $n$ is divisible by $n'$ then the 
homomorphism is given by a homomorphism from $S_n$ to $S_{n'}$ that cyclically wraps 
the bigger cycle around the smaller cycle.
We call this homomorphism \emph{cyclic}\index{cyclic}.

The other implication is a consequence of the following two lemmas.

The nodes of $L(G_{n,d})$ corresponding to the edges connecting the 
dragons with the vertices $c$ are called {\em special}\index{special}.
In Figure~\ref{fig:circle} they are highlighted by circles.

\begin{lemma} 
\label{lem:special}
For $d\geq 3$, $n\geq 3$, $n'\geq 3$ every homomorphism $f:L(G_{n,d})\to L(G_{n',d})$ must map special vertices to special vertices.
\end{lemma}

\begin{proof}
In $L(G_{n,d})$ the only 2-connected components of chromatic number $4$ are the
line graphs of dragons. It follows that the image of every line graph of a dragon must be in a line graph of a dragon. By Lemma~\ref{lem:core_dragon}, the line
graphs of dragons are cores, thus for any special node $u$ holds that its two neighbours $v$ and $w$ in the associated dragon $D_d$ in $L(G_{n,d})$ must be mapped into some dragon $D_d$ from $L(G_{n',d})$, but bijectively onto the two neighbours of the attached special node. (See Figure~\ref{fig:gadget}.)

Since $u,v$ and $w$ form a triangle, the only way to complete a triangle containing 
$f(v)$, $f(w)$ is to map $u$ to the adjacent special node, as such triangle cannot be completed inside the dragon.
\end{proof}

A triangle in $L(G_{n,d})$ is called a {\em connecting triangle}\index{connecting triangle} if it originates from an original node of degree three in $S_m$.
In Figure~\ref{fig:circle} the connecting triangles are denoted by dashed lines.

\begin{lemma}
\label{lem:triangles}
For every $d\geq 3$, $n\geq 3$ and $n'\geq 3$, every homomorphism $f:L(G_{n,d})\to L(G_{n',d})$ must map connecting triangles to connecting triangles.
\end{lemma}

\begin{proof}
Consider an arbitrary connecting triangle of $G_{n,d}$. It has the property that each of its node is adjacent to a special node. By Lemma~\ref{lem:special}, special nodes are preserved by the homomorphism $f$.
The only triangles with this property in $G_{n',d}$ are precisely the connecting triangles.
\end{proof}

\begin{proof}[Proof of Proposition~\ref{prop:homo}]
By Lemma~\ref{lem:triangles}, the connecting triangles map to the connecting triangles. 
The other triangle in the cyclic configuration are those containing a special node. 
By Lemma~\ref{lem:special} these triangles can not map to connecting triangles.
Consequently, two connecting triangles joined by an edge can not map into
one connecting triangle and thus the homomorphism $f$ must be cyclic.
\end{proof}

\begin{proof}[Proof of Theorem~\ref{thm:main}]
We apply Theorem~\ref{thm:universal} and show an embedding of $(P,\leq_P)$ 
to the homomorphism order of the interval $[K_d,K_{d+1})_{\mathcal L}$.
For the chosen $d\geq 3$ we assign to every $A\in P$ a line graph $L(d,A)$ consisting of the disjoint union of graphs $L(G_{a,d}), a\in A$.
Since any homomorphism must map connected components to connected components, 
we know by Proposition~\ref{prop:homo}
that $L(d,A)$ allows a homomorphism to $L(d,B)$ if and only if $A\leq_P B$.
\end{proof}

\subsection*{Concluding remarks}
\label{sec:concluding}

Our results confirm that the homomorphism order of line graphs is rich. It is interesting that our embedding
 differs considerably from the one used in the proof of the universality of
oriented paths by Hubi\v{c}ka and Ne\v{s}et\v{r}il~\cite{Hubicka2005}.

Our construction is based on the retrospect of a homomorphism $f:L(G)\to L(H)$
to a binary relation $f'\subseteq V_G\times V_H$. We put $f'(v)=v'$ if
all edges adjacent to $v$ are mapped by $f$ to the edges 
adjacent $v'$ to $H$. This mapping is not always well defined. In particular:
\begin{itemize}
\item[(a)] The image of an edge adjacent to a vertex of degree $1$ of $G$ is contained
in the set of edges adjacent to two different vertices $u$, $v$ connected by an edge in $H$. In this sense $f'$ is not a function.
\item[(b)] Edges adjacent to a vertex $v$ of degree $3$ in $G$ correspond to a triangle
in $L(G)$. Because the line graph of a triangle is also a triangle, the image of these edges
may thus map to a line graph of a triangle. In this case $f'(v)$ is not defined.
\end{itemize}
The basic idea behind the proof of Lemma~\ref{lem:core} is the fact that $f'$ (if
it is a function) is close to a graph homomorphism $G\to H$ with two
main differences:
\begin{itemize}
\item[(c)] It may happen that $f'(u)=f'(v)$ for two adjacent vertices of $G$.
\item[(d)] For vertices of degree at least $3$ the mapping $f'$ is locally injective with the exception of $(c)$.
\end{itemize}
Recall that a homomorphism $h:G\to H$ is {\em locally injective}\index{locally injective} if the restriction of the
mapping $h$ to the domain consisting of the vertex neighbourhood of $v$ and
range consisting of the vertex neighbourhood of $v$ is injective. In one
direction, every locally constrained homomorphism $h:G\to H$ yields a
homomorphism $h':L(G)\to L(H)$. Our observations above show that this direction can be reversed in special cases.

It is thus not a surprise that our universality proof is based on the ideas
developed for the proof of universality of locally injective homomorphisms. We
get closer to graph homomorphisms by means of the indicator construction. In
the proof of Proposition~\ref{prop:homo} we consider a mapping $f'':V_G\to V_H$
that retrospects a homomorphism $L(G*I_d(a,b))\to L(H*I_d(a,b))$ in a similar way
as $f'$. This mapping is a locally injective homomorphism with the exception of
$(b)$ and vertices of degree 2, where the local injecitivity is not enforced.
For this reason we use sunlets instead of cycles and our
embedding of the divisibility order is based on the fact that there is a
locally injective homomorphism from between sunlet graphs $S_n$ and $S_m$ if and only if
$m$ divides $n$.

With this insight it is not difficult to see that our construction can be
altered to form $3$-regular graphs. This can be done by adding a cycle
connecting pendant vertices to every sunlet. For degrees $d>3$, the edges
connecting both inner and outer can be turned into a multi-edges of a given
degree, but also the indicator needs to be modified to become $d$-regular
except for the two vertices of degree 1. By replacing the edge connecting
dragon with vertex $c$ by a clique of the corresponding degree and by 
adding a separate copy of a dragon to all but one vertex (the one connected to the base).

Several properties of locally injective homomorphism order
are given by degree refinement matrices of the graphs considered~\cite{Fiala2005}. 
As a special case, for $d$-regular graphs $G$ the degree refinement matrix
is trivial consisting of only one value $d$.

Roberson \cite{Roberson} shows that the homomorphism order line graphs is dense
above every line graph of $K_n$, $n\geq 2$. Our construction gives many extra
pairs with infinitely many different graphs strictly in between. Further such pairs can be obtained
by an application of Theorem~\ref{thm:density} in Section~\ref{sec:locally-constrained}.

Assume that $G$ and $H$ have no vertex of degree 2 and the degree of
vertices are either bounded by $d-1$, or bounded by $d$ and moreover that $G$ and $H$ are in Vizing class 1. Then the graph $F$ given by Theorem~\ref{thm:density} can be easily extended to $F'$, where an extra edge is added to every vertex of degree $2$. It is also easy to see that it cannot have vertices of degree greater than $d$. From the proof of Theorem~\ref{thm:density} it follows that $F$ is in Vizing class 1.

Consequently, $L(F'*I_d(a,b))$ is strictly in between $L(G*I_d(a,b))$ and
$L(H*I_d(a,b))$. Afterwards, a graph strictly in between $G*I_d(a,b)$ and $F'*I_d(a,b)$
can be again constructed by applying Theorem~\ref{thm:density} on $G$ and $F$.
This construction provides new examples of dense pairs in the homomorphism order of
line graphs.

There are more results about the locally constrained homomorphism order that seem
to suggest a strategy to attack problems about the homomorphism order of line graphs.
It appears likely that the characterization of gaps in the locally injective homomorphism
order will give new gaps in the homomorphism order of line graphs.
The proof of universality of the locally constrained order of connected graphs (Theorem~\ref{thm:lochomouniv}) in Section~\ref{sec:locally-constrained} can be translated to yet another proof of the universality of the homomorphism order of line graphs,
this time however the graphs used are connected but not $d$-regular --- since locally
constrained homomorphism order is not universal on $d$-regular connected graphs.
This suggests the question of whether the homomorphism order of line graphs is universal on the class of finite connected $d$-regular graphs.

Finally, Leighton's construction of a common covering for graphs~\cite{Leighton1982}
may give an insight into a way of constructing a suitable product for line graphs.

\section{The orders of Relations}
\label{sec:relations}
Relations between graphs have been introduced in Chapter~\ref{ch:relation} and show an interesting correspondence to graph homomorphisms. In this section we consider properties of the order induced by the existence of relations.

Recall that for directed graphs $G$ and $H$ we say that a binary relation is a {\em relation from $G$ to $H$}\index{relation}, or {\em \RelHomo}\index{\RelHomo}, if there is a relation $R\subseteq V_G\times V_H$ such that (1) for every $(u,v)\in E_G, (u,u')\in R, (v,v')\in R$ we also have $(u',v')\in E_H$, (2) for every $(u',v')\in E_H$ we have some $(u,v)\in E_G, (u,u')\in R, (v,v')\in R$ and (3) for every $v'\in V_H$ we have some $v\in V_G$ such that $(u,v)\in H$. 
We say that there is a {\em relation from $G$ to $H$ with full domain}\index{relation with full domain}, or {\em \FRelHomo}\index{\FRelHomo} if there is relation $R$ from $G$ to $H$ and for every $v\in V_G$ there is some $v'$ such that $(v,v')\in R$.

From the decomposition lemma, the order induced by relations is an interesting mix of the surjective homomorphism order and the transposed full homomorphism order:
$\RelLeq$ is the transitive closure of $\FullGeq \cup \SurLeq$.

Moreover, without the full domain condition, we are for the first time allowed to forget parts of graphs $G$ while looking for a relation $G\RelHom H$\footnote{$\RelHom$ is the same as $\backsim_w$ in Chapter~\ref{ch:relation}.}.
For this reason we first show properties of order induced by relations with full domain ($\FRelLeq$).

We know from Section~\ref{subsec:R-core} that R-cores are the unique minimal elements of the equivalence classes of $\FReleq$. Now we need to investigate the cores of $(\DiGraphs,\FRelLeq)$.

\subsection*{PR-cores}
$\RelLeq$ is the transitive closure of $\FRelLeq$ and $\EmbedGeq$.
Consequently $\Releq$ is coarser than $\FReleq$ and we immediately obtain:
\begin{obs}
Every PR-core is an R-core.
\end{obs}
\begin{lem}
\label{lem:PRcore}
Let graphs $G$ and $H$ such that $G\Releq H$. If $H$ is a PR-core, then $H$ is
isomorphic to induced subgraph of $G$.
\end{lem}
\begin{proof}
  Fix $R_1:H\to G$ and $R_2:G\to H$. Denote by $G_1$ the graphs induced
  on domain of $R_2$ by $G$. It is easy to see that $R=R_2\circ R_1$ is a relation
  for $G_1\RelHom G$ and thus also $G_1\Releq G$.

  $R_2$ is also a relation from $G_1$ to $H$ with full domain. Put
  $R'_1=R_1\cup V_H\times V_{G_1}$ and $H_1$ the graph induced by $H$
  on domain of $R'_1$. It is easy to see that $H_1\Releq G_1\Releq G$
  and thus $H_1=H$ from minimality of $H$.

  Consequently $H$ is the R-core of $G_1$ and thus by Theorem~\ref{thm:Rcore}
  $H$ is isomorphic to induced subgraph of $G_1$ that is induced subgraph of $G$.
\end{proof}

From Lemma~\ref{lem:PRcore} and the observation that all relations involved can contain an identity we obtain:

\begin{thm}
Graph $G$ is a PR-core if and only if it is an R-cocore.
\end{thm}

\subsection*{Properties of \texorpdfstring{$(\DiGraphs,\RelLeq)$ and $(\DiGraphs,\FRelLeq)$}{Lg}}

\begin{prop}[Future-finiteness]
\label{prop:relfuturefinite}
The orders $(\DiGraphs,\RelLeq)$ and \\$(\DiGraphs,\FRelLeq)$ are future-finite.
\end{prop}
\begin{proof}
We have shown that both R-cores and PR-cores are point-determining graphs. For given graph $G$
it remains to show that there are only finitely many point-determining graphs $G'$ such that
there is relation $R:G\to G'$.

Fix graph $G$. Consider point-determining $G'$ and relation $R:G\to G'$. For two distinct vertices $u,v\in V_G'$,
consider their pre-images $R^{-1}(u)$ and $R^{-1}(v)$. Because $G'$ is point-determining,
we know that $R^{-1}(u)\neq R^{-1}(v)$. Pre-images are subsets of $V_G$ and there are only $2^{|G|})$ such subsets. Consequently, the size of $G'$ is bounded by $2^{|G|}$.
\end{proof}

The relations with non-full domains are quite diffrent from the common notion of graph homomorphism,
so it is not completely trivial to find an easy antichain. We given a simple example.
Let $\overline{G}$\index{$\overline{G}$} denote the complement of graph $G$. In our directed graph settings with loops,
the complement of a graph without loops is graph with loops everywhere.

\begin{lem} \label{lem:complement_cycle}
For any $n,m\geq 4$, $\overline{C_n}$ has a relation to $\overline{C_m}$ if and only if $m=n$.
\end{lem}

\begin{proof}
One direction is trivial; there is always relation from $\overline{C_n}$ to $\overline{C_n}$.

Assume that $m\neq n$ and the existence of relation
$R:\overline{C_n}\to\overline{C_m}$. Every vertex $v$ in a complement of cycle
is not connected precisely to two additional vertices $(u,u')$. Every vertex
$v$ has unique pair $(u,u')$. By duplicating a vertex one obtain vertices that
are not connected to the same pair of vertices. By the unification of vertices
(that was not created by duplication of the same vertex) one gets vertex
connected to everything. It follows that these operations are ``useless'' in an
attempt to change $\overline{C_n}$ to $\overline{C_m}$. Finally observe that
removing a vertex will break the cycle and turn it into a path.
\end{proof}

\begin{thm}[Universality]
The orders $(\DiCycle,\FRelLeq)$ and $(\Graphs,\RelLeq)$ are future-finite-universal.
\end{thm}

\begin{proof}

We know that $\SurLeq$ is a sub-relation of $\FRelLeq$ which is in turn a sub-relation of $\leq$.
On the class of oriented cycles we have $\SurLeq=\leq$.
By Proposition \ref{prop:shomo-univ} we know that $(\DiCycle,\FRelLeq)$ is future-finite-universal.

For every $S\in \Pfin(\mathbb{N})$ we denote by $E(S)$ the graph created as the complement of the disjoint union of $E(S)=C_{i+4}, i\in S$. 
By the same argument as in the proof of Lemma~\ref{lem:complement_cycle} we get that $E(S)\RelLeq E(S')$ if and only if $S'\subseteq S$.
by Corollary~\ref{cor:futurefiniteuniv} we know that the class of complements of disjoint unions of cycles is future-finite-universal and thus $(\Graphs,\RelLeq)$ is future-finite universal.
\end{proof}





\begin{prop}[Gaps]
Fix a directed graph $H$. There are only finitely many directed graphs $G$ such that $(G,H)$ is an R-gap or PR-gap.
\end{prop}

\begin{proof}
Assume that $(G,H)$ is an R-gap, then $G$ and $H$ are R-cores, and there exists a relation $R$ from $G$ to $H$. We give an upper bound on the number of vertices of $G$ based on the number of vertices of $H$.

We say that vertices $u$ and $v$ (of $G$) are equivalent if $R(u)=R(v)$.
This equivalence has at most $2^{|H|}$ classes. Assume that there are two different vertices, $u$, and $v$ in the same
equivalence class. Denote by $G_{u,v}$\index{$G_{u,v}$} the graph created by unification of
those two vertices. Since $(G,H)$ is an R-gap, $G_{u,v}$ must be R-equivalent to $H$.
And from the fact that there is a relation from $H$ to $G_{u,v}$, we know that there are at most $2^{|H|}$ vertices in $G_{u,v}$ which have distinct neighbourhoods.
Each neighbourhood of a vertex in $G_{u,v}$ can correspond to at most three different neighbourhoods in $G$ (if a vertex $w$ is connected in $G_{u,v}$ to the vertex
created by unification of $u$ and $v$, in $G$ the vertex $w$ may be connected to one of $u$ and $v$ or to both of them).
It turns out that if $G$ has at least $2^{|H|}+1$ vertices (so at least one equivalence class is non-trivial) than there are at
most $6^{|H|}$ different neighbourhoods. Because $H$ is point determining
we know it has no more than $6^{|H|}$ vertices.

Bounds in the case of partial relations follow in a completely analogous fashion.
\end{proof}

%

\begin{prop}[Dualities]
\label{prop:freldual}
For every finite set of directed graphs $\mathcal{F}$ there is finite set of
directed graphs $\mathcal{D}_{R}$, such that $(\mathcal{F},\mathcal{D}_{R})$
is a generalized finite R-duality pair, and a finite set of directed graphs
$\mathcal{D}_{PR}$, such that $(\mathcal{F},\mathcal{D}_{PR})$ is a generalized
finite PR-duality pair.
\end{prop}

\begin{proof}
By Proposition~\ref{prop:relfuturefinite} the orders are future-finite,
by Proposition~\ref{prop:freldual} there are only finitely many gaps below every graph.
The maximal elements in both $(\DiGraphs,\FRelLeq)$ and $(\DiGraphs,\RelLeq)$
are a single vertex and a vertex with loop. With these observations, we apply Corollary \ref{cor:dualities}. 
\end{proof}

\section{Concluding Remarks}
Our task of analysing variants of graph homomorphisms and proving their order-theoretic properties is open-ended. We have considered eleven variants of graph
homomorphisms and answered questions related to cores, density, universality and
dualities. We have found numerous references among previous research showing that many of
these problems have been considered in different contexts and terminologies. Naturally our
approach can be extended to new kinds of mappings (for example quantum
homomorphisms studied in \cite{Roberson}), to wider or narrower classes of
structures (such as graphs without loops, undirected graphs, general relational
structures \cite{Ball2010}, partial orders \cite{Kwuida2011} or minor-closed
classes \cite{Nesetril2006}) and to new
properties, such as the existence of lattice operations.

We believe that the framework we have developed shows interesting connections between
individual areas and sheds more light on some earlier results. For example, 
rather simple observations about the structure of full
homomorphisms better explain earlier results on dualities. Our
techniques seems to apply relatively smoothly to new special mappings and suggest 
a general line of attack to follow when analysing these properties. A
systematic approach also led to a greatly simplified proof of the universality of
the homomorphism order with applications to locally constrained homomorphisms
and line graphs. We were able to
show that the question of universality of oriented paths, solved in
\cite{Hubicka2005}, is a lot more complex than universality of oriented cycles.

We have discussed many types of partial order --- future-finite, past-finite,
universal orders, with or without dualities. Among these the one closest to the
homomorphism order is the order of locally injective homomorphisms on
unoriented connected graphs. Only in this case have we been able to show 
universality and the existence of dualities, and partially characterize the gaps (the techniques have been used in the proofs are completely different to the ones used in the analogous results of the homomorphisms orders). The
only property it seems to lack is an interesting notion for cores. In fact we have not been able
to identify any mapping with a notion of core that would be NP-complete to compute just like the graph core is. Our most complex example is given in Section \ref{sec:relations}). 
The full characterization of gaps in the locally constrained homomorphism orders seems to be an interesting open problem.

In some cases we have given only a partial characterizations of gaps,  in particular in
the case of relations. It seems that an approach similar to the one for full
homomorphisms may lead to a better description. Also, we did not fully resolve
questions about realizations that seem to be difficult already in the
case of the full homomorphism order.






\backmatter 

\appendix
\renewcommand{\thesection}{A.\arabic{section}}

\renewcommand{\chaptermark}[1]{\markboth{\thechapter\ #1}{}}
\renewcommand{\sectionmark}[1]{\markright{\thesection\ #1}}
\fancyhf{}
\fancyhead[LE,RO]{\bfseries\thepage}
\fancyhead[LO]{\bfseries\leftmark}
\fancyhead[RE]{\bfseries\leftmark}

\chapter[Appendix A: Computational complexity]{Appendix A:\\ Computational complexity}
\ifpdf
    \graphicspath{{Appendix1/}{Apeendix1}}
\else
    \graphicspath{{Appendix1/}{Appendix1}}
\fi



The aim of this appendix is to give a brief exposition of the necessary background from computational complexity that we assume in the main body of the thesis. This introduction will be limited to the bare essentials, with no intention
of comprehensiveness.  For a thorough treatment of computational complexity theory one may refer to the monographs of Papadimitriou~\cite{Papadimitriou2003} and Arora and Barak~\cite{Arora2009}.

Computational complexity is a branch of theoretical computer science. In contrast to the design and analysis of algorithms,
computational complexity theory focuses on the hardness of computational problems.  That is, it tries to answer the question why some problems
cannot be computed \emph{efficiently}\index{efficiently}. 
In a \emph{computational problem}\index{computational problem}, an input is given and one is asked to return an output satisfying some property.
For example, in the \emph{factorization problem}\index{factorization problem}, the input is a positive integer $a$, and the output is required to be all the prime factors of $a$.
Another example is the \emph{reachability problem}\index{reachability problem}: given a graph $G$ and two vertices $v_1$,$v_n\in V_G$, is there a path from $v_1$ to $v_n$? The input is the pair of vertices ($v_1$, $v_n$), and the output is yes or no. 

We consider two kinds of computational problems. 
If the output of a computational problem is either yes or no, we call it \emph{decision problem}\index{decision problem}. The reachability problem is a decision problem.
If a computational problem has a single output (of a total function) for every input, but the output is more complex than that of a decision problem,
that is, it isn't just yes or no, we call it a \emph{function problem}\index{function problem}. The factorization problem is a function problem. 


A particular solution to a computational problem is called an \emph{algorithm}\index{algorithm}. An algorithm is a detailed step-by-step procedure for solving a problem.
For example, reachability can be solved by the so-called \emph{search algorithm}~\cite{Papadimitriou2003}. This algorithm works as
follows: Throughout the algorithm we maintain a set of nodes, denoted by $S$. Initially, $S=\{v_1\}$. Each node can be either marked or unmarked.
Node $v_i$ is marked means that $v_i$ has been in $S$ at some point in the past or it is presently in $S$. Initially only $v_1$ is marked.
At each iteration of the algorithm, we choose a node $v_i\in S$ and remove it from $S$. We then process one by one all edges $(v_i,v_j)$ out of $v_i$.
If node $v_j$ is unmarked, then we mark it, and add it to $S$. This process continues until $S$ becomes empty. At this point, we answer yes if node $v_n$
is marked, and no if it is not marked.
It is clear that this algorithm solves reachability problem. 
Moreover, it works efficiently. Here we can only give a rough explanation. We claim that we only spend about $n^2$ operations processing
edges out of the chosen nodes, because there can be at most $n^2$ edges in a graph~\footnote{It is usual to also consider space complexity.
However in this thesis we only discuss time complexity.}. 



In order to introduce formally what is the time complexity of an algorithm, we need a model of computation. 
In 1936 Alan Turing invented the Turing machine, which is a mathematical model of a real machine. 
Since Turing machines are easy to analyse mathematically, and are believed to be as powerful as any other model of computation, besides, it is
the most commonly used model in complexity theory and it can express an arbitrary algorithm.


\begin{defi}[\cite{Arora2009}]
A \emph{(deterministic) Turing machine}\index{deterministic Turing machine} is a quadruple $M=(K,\Sigma,\delta,s)$. Here $K$ is a
finite set of \emph{states}\index{states}; $s\in K$ is the \emph{initial state}\index{initial state}. $\Sigma$ is a finite set of \emph{symbols}\index{symbols}
(we say $\Sigma$ is the \emph{alphabet}\index{alphabet} of $M$). We assume that $K$ and $\Sigma$ are disjoint sets. $\Sigma$ always contains the special
symbols: the \emph{blank}\index{blank} and the \emph{first symbol}\index{first symbol}. Finally, $\delta$ is a transition function, which maps $K\times
\Sigma$ to $(K\times \{h,\text{`yes'},\text{`no'}\})\times \Sigma \times\{\leftarrow, \rightarrow$, -- $\}$.
We assume that $h$ (the \emph{halting state}\index{halting state}), `yes' (the \emph{accepting state}\index{accepting state}), `no' (the \emph{rejecting
state}\index{rejecting state}), and the \emph{cursor directions}\index{cursor directions} $\leftarrow$ for `left', $\rightarrow$ for `right', and -- for `stay', are not in $K\cup \Sigma$.
\end{defi}

When the Turing machine starts to work, the initial state is $s$. The string is initialized to a first symbol, followed by a input string $x$.
The cursor is pointing to the first symbol. And then the machine takes a step according to $\delta$, changing its states, printing a symbol, and
moving the cursor, then continue take another step, and another. Only when the machine reaches one of the three halting states: $h$, `yes', `no',
the machine does not continue. If this happens, we say that the machine has \emph{halted}\index{halted}. At this time, if the states `yes' has been
reached, we say the machine \emph{accepts}\index{accept} its input; if `no' has been reached, we say the machine \emph{rejects}\index{reject} its input.
If $h$ has been reached, the output is the string of $M$ at the time of halting.

For decision problems, the machine always reaches either `yes' or `no'.
We define the collection of all the inputs of Turing machine $M$ which reach the states `yes' as the \emph{language}\index{language} decided by $M$.
For the function problem, the machine always reaches $h$, we say this kind of Turing machine computes functions.
We use the terms of languages and decision problems interchangeably, as in~\cite{Arora2009}.


\subsubsection*{Time complexity}

\begin{defi}
 Let $M$ be a (deterministic) Turing machine that halts on all inputs. The \emph{running time}\index{running time} or \emph{time complexity}\index{time complexity} of $M$ is the function $f:\mathbb{N}\rightarrow \mathbb{N}$, where $f(n)$ is the maximum number of steps that $M$ halts on any input of length $n$. 
\end{defi}

In the anaysis of complexity, usually we only need to know the order of $f(n)$, thus
the complexity of an Turing machine is often expressed using \emph{big $\mathcal{O}$ notation}\index{big O notation}: $f(n)$ is $\mathcal{O}(g(n))$ if there exists a $n_0$ and $c>0$ such that $n\leq n_0$ implies
$f(n)\leq cg(n)$. The following proposition explains why this estimation works.

\begin{pro}[\cite{Papadimitriou2003}]
 Suppose that a (deterministic) Turing machine $M$ decides a language $L$ within time $f(n)$, where $f$ is a proper function. Then there is a precise Turing machine $M'$, which decides the same language in time $\mathcal{O}(f(n))$.
\end{pro}

\subsubsection*{Complexity classes}

\begin{defi}
 Let $t:\mathbb{N}\rightarrow \mathbb{R^+}$ be a function. Define the \emph{time complexity class}\index{time complexity class}, $\textsc{TIME}(t(n))$, to be the collection of
all languages that are decidable by an $\mathcal{O}(t(n))$ time Turing machine.
\end{defi}

\begin{defi}
\emph{\textsc{P}}\index{P} is the class of languages that are decidable on a polynomial time deterministic Turing machine. In other words,
$$\textsc{P}=\bigcup_k \textsc{TIME}(n^k).$$
\end{defi}

In a deterministic Turing machine, the set of rules prescribes at most one action to be performed for any given situation. A non-deterministic Turing machine, by contrast, may have a set of rules that prescribes more than one action for a given situation.A \emph{nondeterministic Turing machine}\index{nondeterministic Turing machine} is a quadruple $N=(K,\Sigma,\Delta,s)$, where $K$, $\Sigma$, $s$ are as normal Turing
machine, and the transition function $\Delta$ is a relation
$\Delta\subset (K\times \Sigma) \times [(K\times \{h,\text{`yes'},\text{`no'}\})\times \Sigma \times\{\leftarrow, \rightarrow$, -- $\}]$.

The set of languages decided by nondeterministic Turing machines within time $f$ is denoted $\textsc{NTIME}(f(n))$. 
\emph{\textsc{NP}}\index{NP} is the class of languages that are decidable on a polynomial time nondeterministic Turing machine, In other words,
$$\textsc{NP}=\bigcup_k \textsc{NTIME}(n^k).$$

We would like to point out that \textsc{NP} does not mean non-polynomial.


It is easily seen that $\textsc{P}\subseteq \textsc{NP}$, since any language that is  decidable by a deterministic Turing machine can also be decided
by a nondeterministic Turing machine.

\subsubsection*{Polynomial time Turing reductions} 

Let $\Sigma$ be an alphabet. We define $\Sigma^*:=\{x_1,x_2,\dots,x_k\mid k\geq 0 \text{ and each } x_i\in \Sigma\}$. 

\begin{defi}
 A function $f:\Sigma^*\to \Sigma^*$ is a \emph{polynomial time computation function}\index{polynomial time computation function} if some polynomial time Turing machine $M$ exists that halts with
just $f(w)$ on its tape, when started on any input $w$.
\end{defi}

\begin{defi}
Language $L\subseteq \Sigma^*$ is \emph{polynomial time reducible}\index{polynomial time reducible}, to language $L'\subseteq \Sigma^*$, written $L \leq_\text{P}^\text{Tur} L'$, if
a polynomial time computation function $f:\Sigma^*\to \Sigma^*$ exists, where for every $x$, $x\in L$ if and only if $f(x)\in L'$.
\end{defi}




\subsubsection*{NP-complete problems}

A central aim of the study of computational complexity is to sort out which problems can be solved in polynomial time and which cannot. In this context,
we introduce the class of \textsc{NP}-complete problems, which are those the least likely to be in \textsc{P}. 

\begin{defi}
A problem $H$ is \emph{\textsc{NP}-hard}\index{NP-hard} (nondeterministic polynomial time hard) if and only if every \textsc{NP} problem is polynomial
time Turing reducible to $H$.

\emph{\textsc{NP}-complete problem}\index{NP-complete problem} is a \textsc{NP} problem that is \textsc{NP}-hard.
\end{defi}

When we need to prove that a problem is NP-complete, we need only to prove that (1) it is in \textsc{NP}, and (2) show that it is polynomial time reducible to a problem
already known to be \textsc{NP}-complete. The hard part of that was finding the first example of an \textsc{NP}-complete problem: that was done by Steve Cook in Cook's Theorem~\cite{Cook1971}.

\bibliographystyle{Classes/jmb}
\renewcommand{\bibname}{References} 
\bibliography{References/references} 
\printindex 
\pagestyle{plain}

\begin{flushleft}{\Large Bibliographische Daten} \\\rule{145mm}{0.2mm}\end{flushleft}
\mytitle \\
(\mytitleDE)\\ \vspace{12pt}
\melast, \mefirst\\
Universit{\"a}t Leipzig, Dissertation, \year \\
157 Seiten, 51 Abbildungen, 89 Referenzen\\

\end{document}